\documentclass[11pt]{amsart}
\usepackage{amsfonts,latexsym,amsthm,amssymb,amsmath,amscd,euscript,tikz, tikz-cd}
\usepackage[alphabetic, msc-links, bibtex-style, nobysame]{amsrefs}
\usepackage{stackengine}
\usepackage{framed}
\usepackage{empheq}
\usepackage{xfrac}
\usepackage[makeroom]{cancel}
\usepackage{faktor}
\usepackage{braket}
\usepackage{pgf,tikz,pgfplots}
\usepackage{tikz-3dplot}
\usetikzlibrary{3d,calc}
\usepgfplotslibrary{groupplots}
\pgfplotsset{compat=1.18}
\usepgfplotslibrary{fillbetween} 
\usepackage{mathrsfs}
\usetikzlibrary{arrows}
\definecolor{wrwrwr}{rgb}{0.3803921568627451,0.3803921568627451,0.3803921568627451}
\definecolor{rvwvcq}{rgb}{0.08235294117647059,0.396078431372549,0.7529411764705882}
\definecolor{mblue}{rgb}{0.2, 0.3, 0.8}
\definecolor{morange}{rgb}{1, 0.5, 0}
\definecolor{mgreen}{rgb}{0.1, 0.4, 0.2}
\definecolor{mred}{rgb}{0.5, 0, 0}
\definecolor{ForestGreen}{RGB}{34,139,34}
\usepackage{float}

\numberwithin{equation}{section}

\usepackage{lmodern}
\usepackage{supertabular}
\usepackage{amssymb}
\usepackage{enumerate}
\usepackage{stmaryrd}
\usepackage{bbm}
\usepackage{mathtools}
\usepackage{setspace}
\usepackage{tikz,tikz-cd}
\usetikzlibrary{matrix,calc,positioning,arrows,decorations.pathreplacing,patterns,knots}

\newcommand{\la}{\langle}
\newcommand{\rg}{\rangle}

\newtheorem{theorem}{{Theorem}}[section]
\newtheorem*{theorem*}{Theorem}
\newtheorem{lemma}[theorem]{Lemma}
\newtheorem{conjecture}[theorem]{Conjecture}
\newtheorem{proposition}[theorem]{Proposition}

\newtheorem{corollary}[theorem]{Corollary}
\newtheorem*{corollary*}{Corollary}

\theoremstyle{definition}
\newtheorem{definition}{Definition}
\newtheorem{remark}{Remark}

\usepackage{lmodern,url,enumerate,mathtools}
\usepackage[hmargin = 1in,vmargin=1in]{geometry}
\usepackage{graphicx}
\usepackage{subcaption}

\usepackage{hyperref}
    \hypersetup{colorlinks=true,citecolor=ForestGreen,linkcolor = blue,urlcolor =black,linkbordercolor={1 0 0}}

\newcommand{\ve}{\varepsilon}

\newcommand{\mr}[1]{{\rm #1}}
\newcommand{\mres}{\mathbin{\vrule height 1.6ex depth 0pt width
0.13ex\vrule height 0.13ex depth 0pt width 1.3ex}}

\newcommand{\cA}{\mathcal{A}}\newcommand{\cB}{\mathcal{B}}

\newcommand{\cH}{\mathcal{H}}
\newcommand{\cI}{\mathcal{I}}\newcommand{\cJ}{\mathcal{J}}
\newcommand{\cL}{\mathcal{L}}
\newcommand{\cM}{\mathcal{M}}
\newcommand{\cP}{\mathcal{P}}
\newcommand{\cQ}{\mathcal{Q}}\newcommand{\cR}{\mathcal{R}}

\newcommand{\cW}{\mathcal{W}}

\newcommand{\bB}{\mathbb{B}}

\newcommand{\bR}{\mathbb{R}}
\newcommand{\bS}{\mathbb{S}}

\newcommand{\bZ}{\mathbb{Z}}

\newcommand{\nc}{\newcommand}

\nc{\on}{\operatorname}
\nc{\p}{\partial}
\nc{\ol}{\overline}
\nc{\ul}{\underline}
\nc{\pa}{\partial}

\nc{\pb}{\partial_b}
\nc{\pc}{\partial_c}
\nc{\pd}{\partial_d}
\nc{\pe}{\partial_e}
\nc{\pf}{\partial_f}
\nc{\pg}{\partial_g}
\nc{\ph}{\partial_h}
\nc{\pari}{\partial_i}
\nc{\pj}{\partial_j}
\nc{\pk}{\partial_k}
\nc{\pl}{\partial_l}
\nc{\pell}{\partial_\ell}
\nc{\parm}{\partial_m}
\nc{\pn}{\partial_n}
\nc{\po}{\partial_o}
\nc{\pp}{\partial_p}
\nc{\pq}{\partial_q}
\nc{\pr}{\partial_r}
\nc{\ps}{\partial_s}
\nc{\pt}{\partial_t}
\nc{\pu}{\partial_u}
\nc{\pv}{\partial_v}
\nc{\pw}{\partial_w}
\nc{\px}{\partial_x}
\nc{\py}{\partial_y}
\nc{\pz}{\partial_z}

\numberwithin{equation}{section}
\makeatletter
\@addtoreset{equation}{section}
\makeatother

\allowdisplaybreaks

\usepackage{accents}
\renewcommand{\underbar}[1]{\underaccent{\bar}{#1}}

\title{Area-minimizing capillary cones}
\date{\today}

\author[Benjy Firester]{Benjy Firester}
\address{
Department of Mathematics, MIT \newline 
{\href{mailto:benjyfir@mit.edu}{benjyfir@mit.edu}}}
\author[Raphael Tsiamis]{Raphael Tsiamis}
\address{
Department of Mathematics, Columbia University \newline
{\href{mailto:r.tsiamis@columbia.edu}{r.tsiamis@columbia.edu}}}
\author[Yipeng Wang]{Yipeng Wang}
\address{
Department of Mathematics, Columbia University \newline
{\href{mailto:yw3631@columbia.edu}{yw3631@columbia.edu}}}

\begin{document}

\begin{abstract}
We construct non-flat minimal capillary cones with bi-orthogonal symmetry groups for any dimension and contact angle.
These cones interpolate between rescalings of a singular solution to the one-phase problem and the free-boundary cone obtained by halving a Lawson cone along a hyperplane of symmetry.
The existence and uniqueness of such cones is proved by solving a nonlinear free boundary equation parametrized by the contact angle and obtaining monotonicity properties for the solutions.
The constructed cones are minimizing in ambient dimension $8$ or higher, for appropriate contact angles, demonstrating that the regularity theory for minimizing capillary hypersurfaces can have singularities in codimension $7$ and completing the capillary regularity theory for contact angles near $\pi/2$.
We further develop the connection between capillary hypersurfaces and solutions of the one-phase problem, consequently producing new examples of singular minimizing free boundaries for the Alt-Caffarelli functional.
\end{abstract}

\maketitle

\vspace*{-0.5in}

\tableofcontents

\section{Introduction}

Given a smooth manifold with boundary $X^{n+1}$ and functions $g \in C^1(X)$ and $\sigma \in C^1(\partial X; (-1,1))$, the \textit{Gauss free energy} associated to an open set $E \subset X$ is 
\[
    \cA(E) := \cH^n(\partial^* E \cap \mathring{X}) - \int_{\partial^* E \cap \partial X} \sigma(x) \, d \cH^n + \int_E g(x) \, dx.
\]
This functional measures the energy of an incompressible fluid occupying the region $E$ inside a container $X$, with the functions $\sigma(x)$ and $g(x)$ describing, respectively, the adhesive (wetting) energy of the fluid and its gravitational energy.
The stationary points of the Gauss free energy represent equilibria of such systems, and the interface $M := \partial E \cap \mathring{X}$ is called a \textit{capillary surface}.
The condition $\delta \cA(E) = 0 $ for a stationary point requires $M$ to have prescribed mean curvature $g$ and boundary contact angle $\cos^{-1}(\sigma)$, meaning that
\[
H_M = g \quad \text{in } \; M \qquad \text{and} \qquad \cos \theta = \sigma(x) \quad \text{at } \; x \in \partial M.
\]
We produce minimizing critical points of the Gauss free energy with boundary singularities of codimension 7.

The blow-up limits of minimizing critical points of $\cA$ for $(X^{n+1},\p X)$ are area-minimizing surfaces in $\bR^{n+1}_+$ meeting the boundary hyperplane $\Pi := \p \bR^{n+1}_+$ at constant angle $\theta$.
We therefore define the capillary functional $\cA^\theta$ for $X^{n+1} = \bR^{n+1}_+$ with the Euclidean metric as the specialization of $\cA$ obtained from $g \equiv 0$ and $\sigma \equiv \cos(\theta)$.
For the model container $X^{n+1} = \bR^{n+1}_+$ given by the upper half-space $\{ p \in \bR^{n+1} : p_{n+1} \geq 0 \}$, the capillary energy $\cA^{\theta}$ has a monotonicity formula (see~\cite{simon-capillary}*{\S~2} and~\cites{Kagaya-Tonegawa, capillary-regularity}) saturated by a $1$-homogeneous critical point of $\cA^{\theta}$, which is called the \textit{tangent cone} of the capillary surface at a given point.
Combined with the $\ve$-regularity theory for capillary hypersurfaces developed by De~Philippis-Maggi~\cites{dephilippis-maggi}, this property allows us to apply Federer's dimension reduction technique and reduce the study of regularity and singularities of $\cA^{\theta}$-minimizers to the analysis of non-flat minimizing capillary cones in the threshold dimension.
\begin{theorem}\label{thm:regularity}
    For every angle $\cos^{-1}(\sigma) = \theta$, there exists a positive integer $d(\theta)$ such that any hypersurface of dimension $n < d(\theta)$ minimizing the $\cA^{\theta}$-capillary functional is smooth.
    There exist constants $\ve_1, \ve_2 > 0$ such that:
    \begin{enumerate}[$(i)$]
        \item for small contact angles $\sigma\in (-1, -1+\ve_1) \cup (1-\ve_1 ,1)$, the critical dimension satisfies $5 \leq d(\theta) \leq 7$,
        \item for large contact angles $|\sigma| \leq \ve_2$, the critical dimension satisfies $d(\theta) = 7$.
    \end{enumerate}
\end{theorem}
In both of the above cases, the novel upper bound $d \leq 7$ is proved by explicit construction of 7-dimensional area-minimizing capillary cones with an isolated singularity, in Theorems~\ref{thm:cones-are-minimizing-smalltheta} and~\ref{thm:cones-are-minimizing-nearpi/2}.

The small contact angle case $(i)$ of Theorem~\ref{thm:regularity} is a physically significant point of study, corresponding to very \textit{hydrophobic} or \textit{hydrophilic} interactions.
Liquid-solid interfaces with these properties play crucial roles in engineering and chemistry including efficient heat diffusion for computer chips and medical diagnostics.
The capillary surface theory also has important applications to the comparison geometry of scalar curvature on manifolds with boundary~\cite{gromov-capillary}.
The construction of capillary surfaces with prescribed mean curvature (capillary $\mu$-bubbles) extends the Schoen-Yau slicing method from the positive mass theorem and has led to recent breakthroughs such as the polyhedron comparison theorem, due to C.~Li~\cite{li-polyhedron}, and the extrinsic Penrose inequality, due to Eichmair-K\"orber~\cite{extrinsic-penrose}.
We refer the reader to~\cites{ko-yao-I, ko-yao-II, yujie-wu} for other applications of these techniques.
When extending this strategy to higher dimensions, it is essential to understand the regularity of the constructed object as well as its limitations, as obtained in Theorem~\ref{thm:regularity}.

Regarding Theorem~\ref{thm:regularity}, Almgren and Taylor had previously shown that the 
singular set has codimension at least $3$~\cites{almgren, taylor}.
For contact angle $\theta = \frac{\pi}{2}$, the surface $M$ is called a free-boundary minimal surface and satisfies the sharp estimate $\dim \on{sing}(M \cap \partial X) \leq n-7$ due to Gr\"uter-Jost~\cite{gruter-jost}.
Recently, Chodosh-Edelen-Li~\cite{improved-regularity} extended this regularity result to angles $\theta \in ( \frac{\pi}{2} - \ve, \frac{\pi}{2}]$, for a dimensional constant $\ve(n)$.
Combining our construction with their result proves the dimensional threshold $d(\theta) = 7$ for a range of contact angles close to $\frac{\pi}{2}$.
The construction of singular, but not necessarily stable, solutions to the capillary problem is a classical question that has received great attention; major results in the field have examined interior singularities~\cites{concus-finn-acta, concus-finn-1,concus-finn-2}.
To the best of our knowledge, this paper presents the first non-trivial examples of boundary singularity behavior for minimizers of the capillary problem~\cite{edelen-owr-2025}.

For small contact angle $\theta$, the capillary problem is intimately connected to the one-phase Bernoulli problem, which can be viewed as its linearization as $\theta \downarrow 0$.
The one-phase Bernoulli free boundary problem models steady configurations of an incompressible inviscid fluid with a free surface.
Such systems are given by critical points of an energy functional with a kinetic term and a volume term encoding the effect of the exterior pressure, by analogy with the Gauss free energy.
Mathematically, this functional is known as the Alt-Caffarelli energy; in addition to its description of fluids, it has many known connections to the theory of minimal surfaces~\cites{traizet, jerison-kamburov}.
In their originating paper~\cite{alt-caffarelli}, Alt and Caffarelli showed the regularity of the free boundary in dimension 2 and developed a partial regularity theory in higher dimensions, proving that the axisymmetric conical free boundary is not minimizing in dimension 3.
Weiss~\cite{weiss} proved the existence of a critical dimension admitting minimizing solutions with an isolated singularity along the free boundary; Caffarelli-Jerison-Kenig and Jerison-Savin proved that this critical dimension is at least $5$~\cites{caffarelli-jerison-kenig, jerison-savin}.
On the other hand, De~Silva-Jerison~\cites{desilva-jerison-cones} exhibited an axisymmetric cone in 7 dimensions which is minimizing, giving an upper bound on the critical dimension comparable to the Simons cone for minimal surface regularity theory.
Prior to the present paper, no other singular minimizing solutions were known~\cites{generic-regularity-II, one-phase-simon-solomon}.

Leveraging the regularity theory for the one-phase problem,~\cite{improved-regularity} proved the existence of a small $\theta_0 > 0$ such that minimizing cones for $\cA^{\theta}$, where $\theta \in (0, \theta_0)$, are flat in dimension $n \leq 4$.
Using similar ideas, they strengthened the Weiss monotonicity formula for capillary surfaces~\cite{weiss-monotonicity}.
Their regularity result was extended to arbitrary contact angles by Pacati-Tortone-Velichkov~\cite{singular-capillary} under the additional assumption of graphicality or boundary mean convexity.

We construct capillary cones for any dimension $n$ and arbitrary capillary angle $\theta$ exhibiting \textit{bi-orthogonal} $O(n-k)\times O(k)$ symmetry.
These symmetries reduce the existence of capillary cones to an appropriate free boundary ODE whose symmetric graph is the desired minimal capillary cone.
The condition that the cone meets the upper half-space is a first boundary (Dirichlet) condition, and the prescribed angle is the second boundary (Neumann) condition.
Our construction interpolates between rescalings of a singular one-phase solution, for small $\theta$, and the free-boundary cone, with $\theta = \frac{\pi}{2}$, obtained by halving a Lawson cone $C(\bS^{n-k-1} \times \bS^k) \subset \bR^{n+1}$ along a hyperplane of symmetry that bisects one of the spheres equatorially.
\begin{theorem}\label{thm:capillary-cones}
    For every $n \geq 3$, every  $1 \leq k \leq n-2$, and all $\theta \in ( 0 , \frac{\pi}{2}]$, there exists a unique non-flat, $O(n-k) \times O(k)$-invariant graphical minimal capillary cone of contact angle $\theta$ in $\bR^{n+1}_+$ with an isolated singularity at the origin, which we denote by $\mathbf{C}_{n,k,\theta}$.
    The association $\theta \mapsto \mathbf{C}_{n,k,\theta}\cap \bS^n_+$ defines a smooth map and the cones $\mathbf{C}_{n,k,\theta}$ have the following convergence properties:
    \begin{enumerate}[$(i)$]
        \item As $\theta \downarrow 0$, the graphing functions converge, after rescaling by $\tan \theta$, to an $O(n-k) \times O(k)$-invariant solution of the one-phase problem in $\bR^n$ whose free boundary is the non-flat, $O(n-k) \times O(k)$-invariant cone $\partial \Gamma_{n,k}$, where
        \[
        \Gamma_{n,k} := \{ (x,y) \in \bR^{n-k} \times \bR^k : |y| \leq t_{n,k} \sqrt{|x|^2 + |y|^2} \}, 
        \]
        with an isolated singularity at the origin.
        \item As $\theta \uparrow \frac{\pi}{2}$, the capillary cones converge in $C^{\infty}_{\on{loc}}$ to $\mathbf{C}_{n,k,\frac{\pi}{2}}$, the free boundary minimal cone obtained by halving the Lawson cone $C(\bS^{n-k-1} \times \bS^k)$ in the upper half-space.
        Notably, the links of the cones $\mathbf{C}_{n,k,\theta}$ converge smoothly, as manifolds-with-boundary inside $\bS^n_+$, to the Clifford hypersurface $\bS^{n-k-1} \times \bS^k$ halved by an equatorial plane of symmetry in the $\bS^k$-factor.
    \end{enumerate}
\end{theorem}
\begin{remark}\label{rmk:remark-after-existence}
In the process of proving Theorem~\ref{thm:capillary-cones}, we will also show that the map $\theta \mapsto \mathbf{C}_{n,k,\theta} \cap \bS^n_+$ of links is smooth, and the positive phases of $\mathbf{C}_{n,k,\theta}$ are nested decreasingly in $\theta$.
The cones $\mathbf{C}_{n,k,\theta}$ are topologically $C(\bS^{n-k-1} \times \bS^{k}_+)$, so unlike the Lawson cones, we do not have the full symmetry group $O(n-k-1)\times O(k+1)$ because the second factor only contributes its upper hemisphere.
Notably, while both $\mathbf{C}_{7,1,\tfrac{\pi}{2}}$ and $\mathbf{C}_{7,5,\tfrac{\pi}{2}}$ are half of the Lawson cone $C(\bS^1 \times \bS^5)$, they are halved along different non-equivalent planes.
Furthermore, we will see that $\mathbf{C}_{7,1,\theta}$ and $\mathbf{C}_{7,5,\theta}$ may behave differently for small $\theta$ (see Remark~\ref{rem:minimizingNotStrict}): $ \mathbf{C}_{7,1,\theta}$ is axisymmetric and transitions from strictly minimizing (for sufficiently small $\theta$) to only strictly stable, whereas $\mathbf{C}_{7,5,\theta}$ is $O(2)\times O(5)$-invariant and may not be minimizing even for small $\theta$.
\end{remark}

The capillary cones constructed in Theorem~\ref{thm:capillary-cones} are minimizing for $\cA^{\theta}$ in ambient dimension $8$ or higher (so $n \geq 7$), for appropriate ranges of $k$ and the angle $\theta$.
\begin{theorem}\label{thm:cones-are-minimizing-smalltheta}
For every $n \geq 7$ and $1 \leq k \leq n-2$, there exists $\theta_{n,k} > 0$ such that for all $\theta \in (0, \theta_{n,k})$, the cone $\mathbf{C}_{n,k,\theta} \subset \bR^{n+1}_+$ is one-sided strictly minimizing for $\cA^{\theta}$, and strictly stable.
Moreover, for $n \leq 100$ and $k$ in the ranges
\begin{itemize}
    \item $n\geq 7$ and $k \leq n-5$,
    \item $n\geq 9$ and $k \leq n-4$,
    \item $n\geq 12$ and $k \leq n-3$,
    \item $n\geq 20$ and $k \leq n-2$,
\end{itemize}
there exists a $\theta_{n,k} > 0$ such that for all $\theta \in (0, \theta_{n,k})$, the cone $\mathbf{C}_{n,k,\theta} \subset \bR^{n+1}_+$ is minimizing for $\cA^{\theta}$.
Notably, the cones $\mathbf{C}_{7,1,\theta}$ and $\mathbf{C}_{7,2,\theta}$ are $\cA^{\theta}$-minimizing in $\bR^8_+$, for $\theta \in (0, \theta_*)$.
\end{theorem}
Our methods indicate that Theorem~\ref{thm:cones-are-minimizing-smalltheta} holds in all dimensions $n \geq 7$, for the $1 \leq k \leq n-2$ specified in the bullets above.
For mathematical rigor, we only state the result for dimensions $n \leq 100$ where we have verified the minimizing property.
As $n \gg 7$, the proof strategy becomes easier to implement.
Theorem~\ref{thm:cones-are-minimizing-smalltheta} treats all $O(n-k) \times O(k)$-invariant cones in dimension $n \geq 7$, except for twenty pairs $(n,k)$ between $7 \leq n \leq 19$ where a more refined construction is needed.
The methods presented in this paper do not rule out that some of these pairs may also be $\cA^{\theta}$-minimizing; however, strong numerical evidence suggests certain exceptional pairs are not strictly minimizing.
These pairs $(n,k)$ are studied in our upcoming work~\cite{FTW_MinimizingII}.

For $n \leq 6$ and small contact angle $\theta$, the cones $\mathbf{C}_{n,k,\theta}$ are all unstable by the criterion of Caffarelli-Jerison-Kenig and a computation of Hong~\cite{hong-singular}.
In an upcoming paper~\cite{FTW_Stability}, we show that for any angle $\theta \in (0, \frac{\pi}{2}]$, the cones $\mathbf{C}_{n,k,\theta}$ are unstable for $n \leq 6$ and stable for $n \geq 7$.
The proof of Theorem~\ref{thm:cones-are-minimizing-smalltheta} is inspired by the argument of De~Silva-Jerison~\cite{desilva-jerison-cones}, reducing the construction of barriers to the linear problem in the dimensions where singularities occur.
Our constructions therefore exhibit new singular minimizing and strictly stable cones for the one-phase problem.
\begin{theorem}\label{thm:one-phase-cones}
    For every $n \geq 3$ and $1 \leq k \leq n-2$, there exists an $O(n-k) \times O(k)$-invariant solution of the one-phase problem, given by
    \[
    U_{n,k}(x,y) := c_{n,k}\sqrt{|x|^2 + |y|^2} \cdot f_{n,k} \left( \frac{|y|}{\sqrt{|x|^2 + |y|^2}} \right), \qquad f_{n,k}(t) := {}_2 F_1 \left( \frac{n-1}{2}, - \frac{1}{2} ; \frac{k}{2} ; t^2 \right)
    \]
    with free boundary given by the cone $\Gamma_{n,k} = \{ (x,y) : |y| \leq t_{n,k} \sqrt{|x|^2+ |y|^2} \}$.
    For every $n \geq 7$ and $1 \leq k \leq n-2$, this solution is strictly stable and strictly minimizing from below; for every $n \geq 7$ and $k$ in the ranges specified in Theorem~\ref{thm:cones-are-minimizing-smalltheta}, it is globally minimizing.
\end{theorem}
Our proof provides a general construction of sub- and supersolutions.
This argument refines the strategy of De~Silva-Jerison~\cite{desilva-jerison-cones}, notably recovering the minimality of the one-phase cone $U_{7,1}$, previously the only known example of a minimizing singular one-phase cone.
The estimates that establish the minimizing property build on the work in our companion paper~\cite{FTW-stability-one-phase}, where we obtain the strict stability of all $1$-homogeneous solutions $U_{n,k}$ for $n \geq 7$.

The other extreme case consists of the free boundary cones where $\theta = \tfrac{\pi}{2}$, where the half Lawson cones give strictly minimizing examples.
We prove that the capillary cones with angle near $\tfrac{\pi}{2}$ constructed in Theorem~\ref{thm:capillary-cones} are minimizing, with the exception of the halved Lawson cones $C(\bS^5 \times \bS^1_+)$ and $C(\bS^1 \times \bS^5_+)$, which are strictly stable and one-sided minimizing, but not area-minimizing.
The free-boundary case $\theta = \frac{\pi}{2}$ is special because reflection about the boundary plane $\Pi$ produces a minimal hypersurface in $\bR^{n+1}$, so it is natural to ask if the stronger free-boundary theory extends to nearby capillary angles, where such a property is not available.
Notably, Fern\'andez-Hauswirth-Mira~\cite{immersed-minimal-annuli} constructed non-rotationally symmetric capillary minimal annuli in $\bB^3$ for contact angles $\theta < \frac{\pi}{2}$, thereby disproving the capillary analogue of the critical catenoid conjecture~\cite{wente-critical-capillary}.
We also refer the reader to recent works of Naff and Zhu~\cites{naff-zhu-capillary-I, zhu-capillary-II} on the behavior of capillary minimal surfaces in $\bB^3$ as the contact angle varies.

\begin{theorem}\label{thm:cones-are-minimizing-nearpi/2}
For $n \geq 7$ and $1 \leq k \leq n-2$ with $(n,k) \not\in \{ (7,1), (7,5) \}$, there exists an $\ve_{n,k} > 0$ such that for all $\theta \in ( \frac{\pi}{2} - \ve_{n,k}, \frac{\pi}{2}]$, the cone $\mathbf{C}_{n,k,\theta} \subset \bR^{n+1}_+$ is strictly stable and strictly minimizing for $\cA^{\theta}$.
For $(n,k) \in \{ (7,1) , (7,5) \}$, there exists a maximal $\ve^* > 0$ such that for all $\theta \in ( \frac{\pi}{2} - \ve^*, \frac{\pi}{2}]$, the cones $\mathbf{C}_{n,k,\theta}$ are strictly stable and one-sided strictly minimizing, but not minimizing for $\cA^{\theta}$.
\end{theorem}
For $(n,k) = (7,1)$, there is a transition angle $\theta_*$ such that all cones $\mathbf{C}_{n,k,\theta}$ are minimizing for $\theta \in (0,\theta_*)$, and are strictly stable, but not minimizing, for $\theta$ close to $\frac{\pi}{2}$.
For $\theta = \frac{\pi}{2}$, Sim\~oes~\cite{simoes} and F.H.~Lin~\cite{lin} proved that the cone $C( \bS^1 \times \bS^5) \subset \bR^8$ is one-sided area-minimizing, therefore strictly stable, but not globally area-minimizing.
Consequently, $\mathbf{C}_{n,k,\theta}$ cannot be globally minimizing for $(n,k) \in \{ (7,1), (7,5) \}$ and $\theta \uparrow \frac{\pi}{2}$.
In particular, examining the transition angle $\theta_*$ where $\mathbf{C}_{7,1,\theta}$ changes from non-minimizing to minimizing results in an example of a minimizing, but not strictly minimizing cone (for $\cA^{\theta}$).
No such example is known for area-minimizing cones in Euclidean space; see Remark~\ref{rem:minimizingNotStrict} for further details.
The pair $(n,k) = (7,5)$ is exceptional, in the sense of Theorem~\ref{thm:cones-are-minimizing-smalltheta}, so the behavior of $\mathbf{C}_{n,k,\theta}$ as $\theta \downarrow 0$ is not clear; however, we still prove that the cones $\mathbf{C}_{7,5,\theta}$ are one-sided strictly minimizing and strictly stable as $\theta \downarrow 0$ and $\theta \uparrow \frac{\pi}{2}$.
The area-minimizing properties of the exceptional cones $\mathbf{C}_{n,k,\theta}$ for small angles $\theta$, as well as of the general cones $\mathbf{C}_{n,k,\theta}$, will form part of an upcoming work~\cite{FTW_MinimizingII}.

We conjecture that for the ranges of $(n,k)$ considered in Theorem~\ref{thm:cones-are-minimizing-smalltheta},  except $(n,k) \in \{ (7,1),(7,5) \}$, the cones constructed in Theorem~\ref{thm:capillary-cones} are minimizing for all $\theta \in (0, \frac{\pi}{2}]$.
Notably, we expect that all cones $\mathbf{C}_{7,2,\theta}$ are minimizing for $\cA^{\theta}$ in $\bR^8$, exhibiting minimizing capillary hypersurfaces with an isolated singularity in dimension $8$, for any contact angle.
The regularity Theorem~\ref{thm:regularity} follows from the minimizing properties of $\mathbf{C}_{7,1,\theta}$ and $\mathbf{C}_{7,2,\theta}$, exhibiting examples with the critical dimension for minimizers with isolated singularities in minimal surface theory.
\begin{conjecture}\label{conjecture:7-dim}
For a minimizer of the capillary energy $M \subset (X^{n+1}, \partial X, g,\sigma)$, the dimension of the boundary singular set satisfies
\[
\dim ( \on{sing} (M) \cap \partial X) \leq n - 7.
\]
This bound is strict and exemplified by $\mathbf{C}_{7,2,\theta}$ for any $\theta \in (0, \frac{\pi}{2}]$, which we conjecture is minimizing.
\end{conjecture}
This conjecture is proved for the range $|\sigma| < \epsilon$ as the second item of Theorem~\ref{thm:regularity} by combining the results of Chodosh-Edelen-Li~\cite{improved-regularity} with Theorem~\ref{thm:cones-are-minimizing-nearpi/2}.
When the wetting term satisfies $|1-\sigma| < \epsilon$ or $|1 + \sigma| < \epsilon$, corresponding to small contact angles, our methods show that the regularity theory for the capillary problem aligns with the regularity theory for the one-phase problem, where it is widely conjectured that the critical regularity dimension for the one-phase problem is $d(0) = 7$~\cites{caffarelli-jerison-kenig, desilva-jerison-cones, jerison-savin}.
Our techniques, combined with the results of Chodosh-Edelen-Li~\cite{improved-regularity}, formalize the equivalence between the regularity of the one-phase problem and of the capillary problem for small contact angles, where we indeed exhibit examples with singularities of codimension $7$.
In upcoming works~\cites{FTW_MinimizingII,FTW_Stability}, we will provide further evidence and results towards this regularity theory and Conjecture~\ref{conjecture:7-dim}.

\subsection{Outline of the argument}
We now outline the structure of the paper and the proof strategies for the main theorems.
In Section~\ref{sec:prelim}, we describe the variational framework of the capillary problem and define the notions of stable and minimizing critical points.
In particular, we frame the problem in the setting of integral currents and Caccioppoli sets.
These definitions allow us to formalize the connections between the one-phase problem and the capillary problem for contact angles $\theta \downarrow 0$.

In Section~\ref{section:o(n-k)-cones}, we construct the cones $\mathbf{C}_{n,k,\theta}$ and prove Theorem~\ref{thm:capillary-cones}.
The imposed symmetry reduces the existence of such a cone to a free boundary ODE whose boundary condition encodes the capillary angle.
For the existence and uniqueness, we prove that the capillary angle produced by the symmetric graph of a solution to this ODE are in bijective, monotonic correspondence with the initial value of the solution.
This argument requires two steps: first, showing that for any initial data below a Lawson-type barrier, the solution to the ODE exists until it reaches $0$ (the Dirichlet boundary condition); second, that the family of solutions sweeps out all possible contact angles bijectively (the Neumann condition).
The main tool is Lemma~\ref{lemma:psi-detect-blow-up}, which obtains a monotonicity formula for the various solutions based on separation from the Lawson solution.
Using this monotone quantity, we determine the zero-or-blow-up behavior of solutions in terms of their initial value in Proposition~\ref{prop:general-a-behavior}.

Section~\ref{sec:minimizing} proves Theorems~\ref{thm:cones-are-minimizing-smalltheta} and~\ref{thm:cones-are-minimizing-nearpi/2}, from which we can conclude Theorem~\ref{thm:regularity}.
Our approach works in two regimes, for contact angles near $0$ or $\frac{\pi}{2}$, for which different techniques are needed.
In dimension $n \geq 7$, for any $k$ and small contact angle, we prove that the cones $\mathbf{C}_{n,k,\theta}$ are one-sided strictly minimizing and strictly stable by constructing subsolutions whose rescalings foliate the region below the cone.
For $(n,k)$ in the range of Theorem~\ref{thm:cones-are-minimizing-smalltheta}, we can further construct strict supersolutions whose dilates foliate the region above the cone.
As a consequence, we obtain the minimality of the corresponding solutions to the one-phase problem for Theorem~\ref{thm:one-phase-cones}.
In sections~\ref{subsection:subsolution} and~\ref{prop:supersolution}, we construct these barriers and reduce the verification of their properties to a number of inequalities related to the parameter $\beta$, whose details are found in the \cite{FTW-stability-one-phase}.

Finally, in Section~\ref{section:pi/2} we discuss the minimizing properties of the cones $\mathbf{C}_{n,k,\theta}$ for contact angles near $\frac{\pi}{2}$.
For all $n \geq 7$ and $k$ except $(n,k) \in \{ (7,1), (7,5) \}$, we prove that the cones are strictly minimizing by constructing barrier surfaces $S_{\pm} \subset E_{\pm}$ asymptotic to $\mathbf{C}$ from either side, with a signed mean curvature and smaller (resp.~larger) contact angle along their free boundary on $\Pi$.
The surfaces $S_{\pm}$ are built by joining, over a large annulus, a normal graph decaying to the cone with the appropriate power and contact angle together with a compact ``cap'' profile obtained by perturbing known sub-calibrations of Lawson's cones~\cites{sharp-stability-plateau, lawson-liu}.
A careful gluing argument shows that the surfaces $S_{\pm}$ are star-shaped, so their rescalings foliate $E_{\pm}$ and prove the strict minimality of the cone.
On the other hand, the cones $\mathbf{C}_{7,1,\theta}$ and $\mathbf{C}_{7,5,\theta}$ admit a calibration on one of the sides $E_{\pm}$, hence they are one-sided strictly minimizing and strictly stable.
We then use the construction of Lawlor~\cite{lawlor} to produce an $\cA^{\theta}$-decreasing competitor for the cones $\mathbf{C}_{7,1,\theta}$ and $\mathbf{C}_{7,5,\theta}$, showing that they are not globally area-minimizing.

The above construction proves Theorem~\hyperref[thm:regularity]{\ref{thm:regularity} $(ii)$} and establishes the sharp boundary regularity of the capillary problem near the free-boundary regime.
A similar strategy was employed in Savin-Yu~\cite{savin-yu-AP-cones}*{\S~7} for proving the minimality of axisymmetric cones of the Alt-Phillips problem with exponent $\gamma \to 1$.
We also refer the reader to~\cite{mooney-yang-inventiones}, where the sharp dimension for the anisotropic Bernstein theorem is obtained by a foliation argument.

\smallskip
\noindent\textbf{Acknowledgements.}
We are grateful to Chao Li for helpful discussions.
We are also thankful to Daniela De~Silva and Ovidiu Savin for valuable insights.

\section{Preliminaries}\label{sec:prelim}

We introduce the fundamental variational notions related to minimizing and stable behaviors for the capillary energy $\cA^{\theta}$ and the Alt-Caffarelli functional $\cJ$.
These will be used to establish the minimality of the cones we construct in Theorem~\ref{thm:capillary-cones}.
We develop general strategies to prove minimality as well as related results and constructions for similar PDEs.
Constructing tangent cones and establishing their minimality is a crucial step for regularity theory and becomes more difficult in non-linear free boundary problems; these include optimal transport, the Monge-Amp\`ere equation, and the Alt-Phillips problem~\cites{CollinsTong,BT,BTF, savin-yu-AP-cones}.

\subsection{Stable and minimizing capillary cones}\label{subsection:definitions}

We will formulate the capillary problem in the context of integer-multiplicity rectifiable currents with the standard notions of mass, see for instance~\cite{simon-gmt}.
We highlight the works of~\cites{dephilippis-maggi, capillary-allard, capillary-varifolds, capillary-regularity, capillary-regularity}, which have developed the regularity theory of capillary surfaces via weak formulations involving currents and capillary varifolds.

For a point $p \in \bR^{n+1}_+$, we let $z := p_{n+1} = \la p , e_{n+1} \rg$ and $z' := \pi_{e_{n+1}^{\perp}}(p)$, so $p =(z',z) \in \bR^n \times \bR_+$.
We consider surfaces in the upper half-space $\bR^{n+1}_+ = \{ z \geq 0 \}$, with boundary plane $\Pi := \partial \bR^{n+1}_+ = \{z = 0\}$.
For a set of locally finite perimeter $E \subset \bR^{n+1}_+$ (a Caccioppoli set), we define the capillary energy in an open set $U \subset \bR^{n+1}$ by
\begin{equation}\label{eqn:capillary-energy}
    \cA^{\theta}(E;U) := P ( E; U \cap \bR^{n+1}_+) - \cos \theta \, \cH^n ( \partial^* E \cap \Pi \cap U).
\end{equation}
The functional $P( E ; U  \cap \bR^{n+1}_+ ) = \cH^n ( \partial^* E \cap U \cap \bR^{n+1}_+)$ denotes the relative perimeter, $\partial^* E$ is the reduced boundary of $E$, and $(\partial^* E \cap \Pi) \cap U$ is the trace of $E$ on $\Pi$.

For an open set $\Omega \subset \bR^{n+1}$, we let $\cI_m (\Omega)$ denote the space of integer-multiplicity rectifiable $m$-currents in $\Omega$.
We write $\mathbf{M}(S)$ for the mass of an $n$-current and $S \mres U$ for its restriction to an open set, so $\mathbf{M} (S \mres U) = \cH^n(S \cap U)$ if $S$ is a $C^1$-submanifold.
We take the boundary plane $\Pi \ni z'$ to be oriented by the standard $n$-vector $dz'_1 \wedge \cdots \wedge d z'_n$.
For a current $S \in \cI_n( U \cap \bR^{n+1}_+)$, we define the $\theta$-capillary energy
\[
\cA^{\theta}(S; U) := \mathbf{M} ( S \mres (U \cap \bR^{n+1}_+) ) - \cos \theta \cdot \inf \{ \mathbf{M}(T) : T \in \cI_n(U \cap \Pi) , \; (\partial T + \partial S) \mres U = 0 \}.
\]
The condition $\partial T = - \partial S$ in $U$ ensures that the contact line is obtained as the common boundary of $S$ and $T$, with opposite orientations.
If $E \subset \bR^{n+1}_+$ is a set of locally finite perimeter, then we set $S := \partial [\![E]\!] \mres (U \cap \bR^{n+1}_+)$ and $T := \partial [\![ E ]\!] \mres (U \cap \Pi)$ in the above definition, whereby 
\[
\mathbf{M}( S \mres (U \cap \bR^{n+1}_+)) = \cH^n (\partial^* E \cap U \cap \bR^{n+1}_+), \qquad \mathbf{M}(T \mres (U \cap \Pi)) = \cH^n( \partial^* E \cap U \cap \Pi),
\]
so we recover the definition~\eqref{eqn:capillary-energy}.

When $U = \bR^{n+1}_+ = \{ z \geq 0 \}$ is an open half-space, we denote $\cA^{\theta}(E) := \cA^{\theta}(E; \bR^{n+1}_+)$.
A Caccioppoli set $E \subset \bR^{n+1}_+$ is said to minimize $\cA^{\theta}$, if for every pre-compact open set $U \subset \bR^{n+1}_+$,
\begin{equation}\label{eqn:cacciopoli-minimizing}
\cA^{\theta}( E ; U ) \leq \cA^{\theta} ( F  ; U)
\end{equation}
is satisfied for any Caccioppoli set $F \subset \bR^{n+1}_+$ such that $E \triangle F \Subset U$.
In this case, the (reduced) boundary $M := \partial^*E\cap U \cap \bR^{n+1}_+$ is called a \textit{minimizing capillary surface} with contact angle $\theta$.
In particular, $E$ must be a critical point of the capillary functional $\cA^{\theta}$; computing the first variation of $\cA^{\theta}$ along flows of $E$ that are compactly supported inside $U$ and keep $\Pi$ fixed, we arrive at the defining equations for a minimal surface with capillary free boundary,
\begin{equation}\label{eqn:capillary-fbms}
    H_M = 0 \quad \text{ in } \; M, \qquad \la \nu_M , e_{n+1} \rg = \cos \theta \quad \text{at } \; \partial M.
\end{equation}
The set $\partial M = \partial^* E \cap \Pi$ is called the \textit{free boundary} of the surface $M$.

The definition of currents minimizing $\cA^{\theta}$ is analogous to the condition~\eqref{eqn:cacciopoli-minimizing}, by generalizing the admissible competitors to pairs $(S,T)$ with
\[
(S - \partial [\![ E_+ ]\!] ) \mres (\bR^{n+1}_+ \setminus U) = 0 , \qquad (T - \partial [\![ E_+ ]\!]) \mres ( \Pi \setminus U) = 0, \qquad (\partial T + \partial S) \mres U = 0.
\]
If $E \subset \bR^{n+1}_+$ is a $1$-homogeneous set minimizing the functional $\cA^{\theta}$ in every compact subset of $\bR^{n+1}$, we say that $\mathbf{C} := \partial^* E \subset \{ z \geq 0 \}$ is a \textit{minimizing capillary cone}.
Capillary cones saturate the monotonicity formula for capillary minimal surfaces, analogous to the standard monotonicity formula for minimal surfaces, and model blow-up limits.
For a capillary cone ${\mathbf{C}}$, we can define $E_\pm$ as the two components of $\bR^n_+ \setminus \mathbf{C}$ where $E_+$ is the region \textit{above} ${\mathbf{C}}$, containing points $p = (z',z)$ with unbounded $z \geq 0$, for any fixed $z' \in \Pi$.
We then interpret~\eqref{eqn:capillary-fbms} with the understanding that the cone normal $\nu_{\mathbf{C}}$ points into $E_+$.

If $E$ is a stationary point of the functional $\cA^{\theta}$ such that $\mathbf{C} := \partial E$ has an isolated singularity at the origin, then we have
\[
\cot \theta \, A_{\mathbf{C}}(\eta, \eta) = - \cos \theta \, H_{\partial  \mathbf{C}} \qquad \text{on } \; \partial  \mathbf{C} \setminus \{ 0 \}
\]
where $A_{\partial  \mathbf{C}}$ is the second fundamental form and $H_{\partial  \mathbf{C}}$ denotes the mean curvature of $\partial  \mathbf{C}$ pointing towards the complement of $\partial E \cap \Pi$ in $\Pi= \{ z = 0 \}$, see~\cite{singular-capillary}*{Lemma~2.1}.

The capillary stationary point ${\mathbf{C}}$ is called \textit{stable} if the second variation of the functional $\cA^{\theta}$ is non-negative in the direction of any function $\varphi \in C_c^{\infty}(\bR^{n+1} \setminus \{ 0 \})$.
The stability condition was formulated by Ros-Souam~\cites{ros-souam}; combining their expression with the above relation, we arrive at the stability form
\begin{equation}\label{eqn:stability-form}
    Q_{\theta}(\varphi) := \int_{\mathbf{C}} (|\nabla_{\mathbf{C}} \varphi|^2 - |A_{\mathbf{C}}|^2 \varphi^2) \, d \cH^n  - \cos \theta \int_{\partial  \mathbf{C}} H_{\partial  \mathbf{C}} \, \varphi^2 \, d \cH^{n-1}
\end{equation}
whose non-negativity over all $\varphi \in C_c^{\infty}(\bR^{n+1}_+ \setminus \{ 0 \})$ is equivalent to the stability of $E$.

Let $\mathbf{C} \subset \bR^{n+1}_+$ be a capillary cone of contact angle $\theta$ with isolated singularities such that the link $\Sigma := \mathbf{C} \cap \bS^n_+$ is a smooth manifold with boundary $\partial \Sigma = \Sigma \cap (\Pi \cap \bS^n)$; we say that ${\mathbf{C}}$ is a regular cone.
We say that ${\mathbf{C}}$ is \textit{strictly stable} if it satisfies the stronger condition
\begin{equation}\label{eqn:strictly-stable}
    Q_{\theta}(\varphi) \geq \mu \int_{\mathbf{C}} \frac{\varphi^2}{R^2} \, d \cH^n \qquad \text{for some } \; \mu > 0, \qquad \text{for all } \; \varphi \in C_c^{\infty} (\bR^{n+1} \setminus \{ 0 \})
\end{equation}
where $R = |p|$.
The $1$-homogeneity of the cone ${\mathbf{C}}$ makes~\eqref{eqn:strictly-stable} equivalent to the positivity of the first eigenvalue of the Jacobi operator on the link $\Sigma := \mathbf{C} \cap \bS^n_+$ with Neumann boundary condition along $\partial ( \mathbf{C} \cap \bS^n_+)$.
The notion of strict stability is crucial in the study of singularities for minimal submanifolds; see, for example, the constructions of submanifolds with isolated singularities, due to Smale~\cite{smale}, and with prescribed singular sets, due to Simon~\cite{prescribed-singular}.

\subsection{The one-phase problem}

We recall some fundamental notions and properties of the one-phase free boundary problem.
The study of the one-phase problem has attracted great interest in differential equations for many decades, with foundational work done by Caffarelli, Jerison, Kenig, and others~\cites{alt-caffarelli, caffarelli-jerison-kenig, caffarelli-friedman , caffarelli-salsa, weiss}.
The one-phase problem has many connections with the theory of minimal surfaces, and the question of finding counterparts of theorems about minimal surfaces in the theory of free boundaries has received great attention.
In low dimensions, minimal surface methods such as the Weierstrass representation and gluing techniques leading to classification results and new examples~\cites{entire-hairpins, jerison-kamburov, traizet, hollow-vortices , n-dim-catenoid ,  hines-kolesar-mcgrath}.

Given an open set $B$ in Euclidean space $\bR^n$ and a non-negative function $u$ on the domain $B$, the \textit{Alt-Caffarelli functional} is defined as  
\begin{equation}\label{eqn:alt-caffarelli}
    \cJ(u ,B) := \int_B |\nabla u|^2 + \chi_{ \{ u > 0 \} } 
\end{equation}
We say that a function $u$ is a \textit{minimizer} for the Alt-Caffarelli functional in a domain $B$, if
\begin{equation}\label{eqn:minimizer-one-phase}
    \cJ(u,B) \leq \cJ(v,B) \qquad \text{for any function} \; v \in H^1(B) \; \text{with} \; v = u \; \text{on} \; \partial B.
\end{equation}
We say that $u$ is \textit{one-sided minimizing from above} in $B$ (resp.~\textit{one-sided minimizing from below}) if the comparison~\eqref{eqn:minimizer-one-phase} holds whenever, in addition, $v \geq u$ (resp.~$v \leq u$).
We say that $u$ is a \textit{global (one-sided) minimizer} in $\bR^n$ if the corresponding property holds in balls $B^n_R$ for all $R>0$.

A critical point of the Alt-Caffarelli functional solves the one-phase free boundary problem
\begin{equation}\label{eqn:one-phase-problem}
    \begin{cases}
        \Delta u = 0 & \text{in } \; \{ u > 0 \}, \\
        |\nabla u| = 1 & \text{on } \; \partial \{ u > 0 \}.
    \end{cases}
\end{equation}
The critical point $U$ of $\cJ$ is called \textit{stable}, if the second variation of the functional $\cJ$ at $U$ is non-negative on every annulus $B^n_{r_2} \setminus \bar{B}^n_{r_1}$ for $0 < r_1 < r_2$.
The positivity of the second variation with respect to inner perturbations in an annulus $B^n_{r_2} \setminus \bar{B}^n_{r_1}$ was computed in~\cite{caffarelli-jerison-kenig}, and amounts to
\[
\int_{B^n_{r_2} \setminus \bar{B}^n_{r_1}} |\nabla \phi|^2 \geq \int_{ \partial ( B^n_{r_2} \setminus \bar{B}^n_{r_1} )} H \phi^2, \qquad \text{for any } \; \phi \in C^2(B^n_{r_2} \setminus \bar{B}^n_{r_1}).
\]
where $H$ denotes the mean curvature of the free boundary $\partial \{ U > 0 \}$ with respect to the outward-pointing normal vector.
A standard application of the Hopf lemma shows that $H>0$ when $U$ is not the half-space solution, meaning that the free boundary is strictly mean-convex.

The stability of a homogeneous solution $U$ can be recast as a variational problem on the spherical link of the positive phase, i.e., the spherical domain $\Omega_S := \{ U > 0 \} \cap \bS^{n-1}$.
The first eigenvalue is computed via the minimization
\[
\Lambda := \inf_v \frac{\int_{\Omega_S} |\nabla v|^2 - \int_{\partial \Omega_S} H v^2}{\int_{\Omega_S} v^2}
\]
and the solution $U$ is stable (in any bounded domain that avoids the origin) if and only if $\Lambda \geq - ( \frac{n-2}{2})^2$.
We say that $U$ is \textit{strictly stable} if this inequality is strict, meaning that $\Lambda > - ( \frac{n-2}{2})^2$.

A useful reformulation of the stability condition was given in~\cite{jerison-savin}: stability is equivalent to the existence of a subsolution for the linearization of the one-phase problem at $U$.
Accordingly, the cone $U$ is \textit{strictly stable} if the above subsolution is strict.
We refer the reader to~\cites{jerison-savin, hardt-simon-DSJ, one-phase-simon-solomon} for a detailed discussion of stability and strict stability properties and their implications.

\subsection{Previous results on the one-phase problem}

A major open question for the Alt-Caffarelli functional is the regularity of the free boundary for minimizers.
As in the case of minimal surfaces, the presence of a monotonicity formula~\cite{weiss} reduces the question to determining the lowest dimension in which a homogeneous minimizing solution with an isolated singularity can appear.
Caffarelli-Jerison-Kenig proved that stable cones are flat in dimension $3$, and Jerison-Savin proved this result in dimension $4$~\cites{caffarelli-jerison-kenig, jerison-savin}.
On the other hand, the De~Silva-Jerison cone~\cite{desilva-jerison-cones} is an axially symmetric homogeneous energy-minimizing free boundary in dimension $7$.
Hong~\cite{hong-singular} later considered more general symmetric examples and computed another stable cone in seven dimensions.
In recent years, many important regularity results have appeared, including foliations of minimizers, generic regularity, and uniqueness of blow-up limits~\cites{hardt-simon-DSJ , desilva-jerison-cpam, one-phase-bernstein, one-phase-simon-solomon , uniqueness-one-phase, SMP-one-phase , continuity-up-to-boundary, generic-regularity-I, generic-regularity-II}.

The one-phase problem is closely related to the capillary minimal graph problem for small contact angle $\theta$: the Alt-Caffarelli functional $\cJ$ arises as the linearization of the capillary functional $\cA^{\theta}$ for $\theta \downarrow 0$ (see, for example,~\cite{caffarelli-friedman})
This link has recently been utilized to obtain regularity results for the capillary problem~\cites{improved-regularity, singular-capillary, regularity-with-obstacle, capillary-regularity}.
Notably, Chodosh-Edelen-Li~\cite{improved-regularity} proved that, after an appropriate rescaling, minimizing capillary graphs with small contact angle sub-converge to one-homogeneous minimizers of the Alt–Caffarelli functional.
Using this result and the regularity theory of the one-phase problem, they proved prove that in ambient dimension $\leq 5$, stable capillary cones with contact angle $\theta \in (0, \ve_0)$ are flat.

Developing these connections further, we prove that the $O(n-k) \times O(k)$-invariant non-flat capillary cones that we construct are minimizing in dimension $n \geq 7$ for small contact angle, motivated by the De~Silva-Jerison axisymmetric minimizing one-phase cone~\cite{desilva-jerison-cones}.
For the capillary problem, Pacati-Tortone-Velichkov proved that stable axisymmetric capillary cones are flat in dimension $\leq 7$, by analogy to the result of Cafarelli-Jerison-Kenig~\cite{caffarelli-jerison-kenig}.
We also note the constructions of De~Silva-Jerison-Shahgolian and Edelen-Spolaor-Velichkov~\cites{hardt-simon-DSJ, SMP-one-phase}, which produce analogues of the Hardt-Simon foliation for minimizing cones of the one-phase problem.

It is interesting to highlight a further connection between the construction of Theorem~\ref{thm:capillary-cones} and the theory of free boundary problems.
Generalizing the Alt-Caffarelli functional, the \textit{one-phase Alt-Phillips} free boundary problem concerns the study of non-negative minimizers of the functional
\[
\cJ_{\gamma}(u) := \int_{B_1} |\nabla u|^2 + u^{\gamma} \chi_{ \{ u > 0 \} }, \qquad \text{where } \; \gamma \in (-2,2)
\]
among functions with non-negative boundary datum on $\partial B_1$, cf.~\cite{alt-phillips}.
For $\gamma = 0$, one recovers the standard one-phase problem, whereas $\gamma = 1$ corresponds to the \textit{obstacle problem}.
Recently, De~Silva-Savin studied the Alt-Phillips problem for negative exponents $\gamma$ and proved that the functional $\cJ_{\gamma}(u)$ $\Gamma$-converges to the Alt-Caffarelli energy, for $\gamma \uparrow 0$, and to the perimeter functional for $\gamma \downarrow -2$.
In particular, this implies that minimizers of the latter problem $\Gamma$-converge to free boundary minimal surfaces~\cites{negative-alt-phillips, uniform-density}.
Interestingly, the $\Gamma$-convergence properties of the Alt-Phillips functional imply that its regularity behavior witnesses ``jumps'' depending on the parameter $\gamma$: we have $d(\gamma) \to 8$ as $\gamma \downarrow -2$ and $d(\gamma) \to 3$ as $\gamma \uparrow 1$ (by the minimality of the radial cone, due to~\cite{savin-yu-AP-cones}, and the classification of two-dimensional homogeneous solutions, by~\cite{alt-phillips}), while $d(0) \geq 5$ by~\cite{jerison-savin}.
On the other hand, in Theorem~\ref{thm:capillary-cones} we produce capillary cones interpolating between free boundary minimal cones and (after appropriately rescaling) homogeneous solutions of the one-phase problem.
Our Conjecture~\ref{conjecture:7-dim} expresses that the capillary problem should not witness the same jumps, meaning that $d(\theta) = 7$ for all $\theta$.

Recently, Savin-Yu constructed minimizing axisymmetric cones for the Alt-Phillips problem with exponent $\gamma \uparrow 1$ in dimension $d \geq 4$ by exploiting the convergence to the obstacle problem~\cite{savin-yu-AP-cones}.
Their argument shares some similarities with our proof of Theorem~\ref{thm:cones-are-minimizing-smalltheta}, upon carefully estimating the deviation of the operators from those of the model problem.
Earlier, De~Silva-Savin~\cites{negative-alt-phillips, uniform-density} had analyzed viscosity solutions of the Alt-Phillips problem for negative powers (notably, when $\gamma \downarrow -2$ and the problem approximates the free-boundary perimeter functional) using a calibration method.
Our proof of the minimality of the cones $\mathbf{C}_{n,k,\theta}$ near the free-boundary regime $\theta \uparrow \frac{\pi}{2}$ (Theorem~\ref{thm:cones-are-minimizing-nearpi/2}) relies on constructing a capillary sub-calibration.

\subsection{Foliations and minimality}\label{subsection:foliations-minimality}

We turn our attention to minimal capillary hypersurfaces that are expressed as graphs of positive functions over a domain $\Omega \subset \Pi = \{ z=0\}$, of the form $\{ z = u(z') \}$.
The minimality of graphical surfaces can be obtained using foliations by sub- and supersolutions; this technique is especially applicable to the regime of small contact angles, where minimizers are graphical as discussed in~\cite{improved-regularity}.
The fact that a solution is minimizing if it supports a foliation has been used in many constructions of minimizing cones with isolated singularities~\cites{bdgg, desilva-jerison-cones, mooney-yang, savin-yu-AP-cones}, we refer the reader also to Savin-Yu~\cite{savin-yu-AP-cones}*{\S~2.2} for a detailed discussion on these techniques for the Alt-Philips problem.
We now define foliations of graphical solutions to the capillary and one-phase problems in a unified setting which will be utilized to prove minimality for capillary cones with large contact angle.

We recall that critical points $u$ of the Alt-Caffarelli functional are solutions of the one-phase free boundary problem~\eqref{eqn:one-phase-problem}.
For a capillary hypersurface in $\bR^{n+1}_+$ expressed as the graph of a function $u \geq 0$ over $\bR^n$, the capillary condition becomes
\begin{equation}\label{eqn:capillary-condition-graph}
    \la \nu, e_{n+1} \rg = \cos \theta
\end{equation}
along $\partial \{ u > 0 \}$ over the graph of the function.
Here, $\nu = \frac{(- \nabla u, 1)}{\sqrt{1 + |\nabla u|^2}}$ is the upward-pointing unit normal vector.
The equation~\eqref{eqn:capillary-condition-graph} means that
\begin{align*}
\frac{1}{\sqrt{1+ |\nabla u|^2}} &= \frac{\la ( - \nabla u, 1), e_{n+1} \rg}{\sqrt{1 + |\nabla u|^2}} = \frac{1}{\sqrt{1 + \tan^2 \theta}},
\end{align*}
which implies that $|\nabla u | = \tan \theta$ along $\{ u = 0 \}$.
By analogy with~\eqref{eqn:one-phase-problem}, the capillary free boundary problem with contact angle $\theta$ takes the form
\begin{equation}{\label{eqn:capillary-fbp}}
\begin{cases}
    \cM(u) := \on{div}\left(\dfrac{\nabla u}{\sqrt{1+|\nabla u|^2}}\right)=0,&\text{in }\{u>0\}\\
    |\nabla u|=\tan\theta,&\text{on }\partial\{u>0\}.
\end{cases}
\end{equation}
Given a solution $u$ of the problem~\eqref{eqn:capillary-fbp}, its graph produces a capillary minimal surface with contact angle $\theta$ at the free boundary.

\begin{definition}
A function $u \geq 0$ is called a \textit{subsolution} (resp.~\textit{supersolution}) for the capillary free boundary problem~\eqref{eqn:capillary-fbp} on a domain $D$ if it satisfies
\begin{enumerate}[$(i)$]
    \item $\cM(u) \geq 0$ (resp.~$\cM(u) \leq 0$) in $\{ u > 0 \} \cap D$,
    \item $|\nabla u| \geq \tan \theta$ (resp.~$|\nabla u| \leq \tan \theta$) along the free boundary $\partial \{ u > 0 \} \cap \mathring{D}$.
\end{enumerate}
A subsolution (resp.~supersolution) is called \textit{strict}, if at least one of the above inequalities holds strictly at some point.
\end{definition}

We recall that the graph of a function $u$ has mean curvature $H = - \cM(u)$ with respect to the upward normal, so the interior inequality amounts to the graphs of subsolutions (resp.~supersolutions) being mean-concave (resp.~mean-convex).
The boundary inequality means that the contact angle produce is $\geq \theta$, i.e.,
$\la \nu, e_{n+1} \rg = \frac{1}{\sqrt{1 + |\nabla u|^2}} \leq \cos \theta$ for a subsolution (resp.~angle $\leq \theta$ and $\la \nu, e_{n+1} \rg \geq \cos \theta$ for a supersolution).

Sub- and supersolutions of the one-phase free boundary problem~\eqref{eqn:one-phase-problem} are defined analogously (cf.~\cites{desilva-jerison-cones,savin-yu-AP-cones}), with the corresponding notion of strictness for the interior or the boundary condition of~\eqref{eqn:one-phase-problem}.

\begin{definition}\label{def:lowerLeaf}
     Let $u$ be a global solution of the problem~\eqref{eqn:capillary-fbp} (respectively,~\ref{eqn:one-phase-problem}).
     A \textit{lower foliation} of $u$ is a family of functions $\{ V_{\lambda}\}_{\lambda \in (0, +\infty)}$ satisfying the following properties:
     \begin{enumerate}[$(i)$]
         \item For each $\lambda \in (0,+\infty)$, we have
         \[
         0 \leq V_{\lambda} \leq u \quad \text{in } \; \bR^n, \qquad V_{\lambda} < u \quad \text{in } \; \overline{ \{ V_{\lambda} > 0 \} }.
         \]
         \item For each $\lambda \in (0,+\infty)$, the function $V_{\lambda}$ is a subsolution of the problem~\eqref{eqn:capillary-fbp} (resp.~\ref{eqn:one-phase-problem}).
         \item The map $(\lambda,z') \mapsto V_{\lambda}(z')$ is continuous and the map $\lambda \mapsto \overline{ \{ V_{\lambda} > 0 \} }$ is locally continuous in the Hausdorff distance.
         \item As $\lambda \to 0$, we have $V_{\lambda} \to u$ locally uniformly in $\{ u > 0 \}$.
         For any compact set $K \subset \bR^n$, there is a $\lambda_K$ such that $V_{\lambda} =0$ on $K$, for all $\lambda > \lambda_K$.
     \end{enumerate}
     Each function $V_{\lambda}$ is called a \textit{lower leaf}.
\end{definition}

\begin{definition}\label{def:upperLeaf}
    Let $u$ be a global solution of the problem~\eqref{eqn:capillary-fbp} (respectively,~\ref{eqn:one-phase-problem}).
    An \textit{upper foliation} of $u$ is a family of functions $\{ W_{\lambda} \}_{\lambda \in (0, +\infty)}$ satisfying the following properties:
    \begin{enumerate}[$(i)$]
        \item For each $\lambda \in (0,+\infty)$, we have
        \[
        W_{\lambda} > 0 \quad \text{in } \; \bR^n, \qquad W_{\lambda} > u \quad \text{in } \; \overline{ \{ u > 0 \} }.
        \]
        \item For each $\lambda \in (0,+\infty)$, the function $W_\lambda$ is a supersolution of the problem~\eqref{eqn:capillary-fbp} (resp.~\ref{eqn:one-phase-problem}).
        \item The map $(\lambda, z') \mapsto W_{\lambda}(z')$ is continuous.
        \item As $\lambda \to 0$, we have $W_{\lambda} \to u$ locally uniformly in $\{ u > 0 \}$; as $\lambda \to + \infty$, we have $W_{\lambda} \to + \infty$ locally uniformly.
    \end{enumerate}
    Each function $W_{\lambda}$ is called an \textit{upper leaf}.
\end{definition}
For the capillary problem~\eqref{eqn:capillary-fbp}, the significance of the functions $V_{\lambda}$ (respectively, $W_{\lambda}$) is that their graphs (over their positive phase) are mean-concave (respectively, mean-convex) and have contact angle $\geq \theta$ (respectively, $\leq \theta$) along their free boundary.

For a capillary hypersurface that is expressible as the graph of a function $u$, meaning
$\{ (z' ,z) \in \bR^n \times \bR_+ : z = u(z') \}$, we consider the Caccioppoli sets
\[
E_+(u) := \{ (z', z) \in \bR^n \times \bR_+ : z > u(z') \}, \qquad E_-(u) := \{ (z',z) \in \bR^n \times \bR_+ : z < u(z') \}.
\]
\begin{lemma}\label{lemma:capillary-foliations-imply-minimizing}
Let $u$ be a global solution of the problem~\eqref{eqn:capillary-fbp} with $\cA^{\theta}(E_{\pm}(u) ; B_1) < + \infty$.
\begin{enumerate}[$(i)$]
    \item If $u$ has a lower foliation, then it is one-sided minimizing for $\cA^{\theta}$ in $E_-(u)$.
    \item If $u$ has an upper foliation, then it is one-sided minimizing for $\cA^{\theta}$ in $E_+(u)$.
    \item If $(i)$ and $(ii)$ both hold, then $u$ is minimizing for $\cA^{\theta}$.
\end{enumerate}
\end{lemma}

\begin{lemma}\label{lemma:one-phase-foliations-imply-minimizing}
Let $u$ be a global solution of the problem~\eqref{eqn:one-phase-problem} with $\cJ(u, B) < + \infty$.
\begin{enumerate}[$(i)$]
    \item If $u$ has a lower foliation, then it is one-sided minimizing from below in $B$.
    \item If $u$ has an upper foliation, then it is one-sided minimizing from above in $B$.
    \item If both $(i)$ and $(ii)$ hold, then $u$ is minimizing in $B$.
\end{enumerate}
\end{lemma}

\begin{proof}
    The result of Lemma~\ref{lemma:one-phase-foliations-imply-minimizing} is well-known; see, for example, Savin-Yu~\cite{savin-yu-AP-cones}*{\S~2.2} and~\cite{desilva-jerison-cones}*{Theorems 2.7, 2.8}.
    The analogous property for the capillary problem~\eqref{eqn:capillary-fbp} follows by a direct adaptation of the arguments in Savin-Yu~\cite{savin-yu-AP-cones}*{Propositions 2.7, 2.9}.
    In both cases, a contradiction is obtained by considering the energy-minimizing competitor $u$ and sliding the leaves $V_{\lambda}$ of the lower foliation (resp.~the leaves $W_{\lambda}$ of the upper foliation) to produce a first instance of contact $\lambda_*$ where $V_{\lambda_*}$ touches $u$ from below (resp.~$W_{\lambda_*}$ touches $u$ from above).
    Both of these situations would lead to a contradiction, since a (super)solution cannot touch a (sub)solution from above.
    
    For the one-phase problem, the existence of a minimizing competitor with boundary data $u|_{\partial B}$ is obtained by the direct method.
    In the adaptation of the result to the capillary setting, we can likewise extract a minimizer of the capillary functional agreeing with $E_{\pm}(u)$, in the class of Caccioppoli sets.
    The existence of a minimizing competitor is guaranteed by the compactness theory for almost-minimizers of the capillary energy functional $\cA^{\theta}$, due to De~Philippis-Maggi~\cite{dephilippis-maggi}*{\S~2.6}.
    The maximum principle for sub- and supersolutions applies in the interior of the capillary minimizer, again producing a contradiction by the sliding-foliation argument described above.
\end{proof}

\begin{remark}\label{rmk:upper-lower-foliations}
    The lower and upper foliations of Definition~\ref{def:lowerLeaf} need not be graphical for Lemma~\ref{lemma:capillary-foliations-imply-minimizing} to apply.
    Indeed, the assumptions of Lemma~\ref{lemma:capillary-foliations-imply-minimizing} can be replaced by the existence of surfaces $S_- \subset E_-$ (for part $(i)$) and $S_+ \subset E_+$ (for part $(ii)$) such that, with respect to the upward-pointing normal vector, $S_-$ is mean-concave (resp.~$S_+$ is mean-convex) and forms a contact angle $\geq \theta$ (resp.~$\leq \theta$) with the plane $\Pi$.
    See also Lemma~\ref{lemma:get-a-foliation} for the strict version of this result.
\end{remark}

\begin{lemma}\label{lemma:minimizing-by-subsolutions}
    Let $u : \bR^n_{z'} \to \bR_{\geq 0 }$ be a $1$-homogeneous solution of the capillary free boundary problem~\eqref{eqn:capillary-fbp} and let $\mathbf{C} \subset \bR^{n+1}_+$ be the cone $\{ z = u(z') \}$ given by its graph.
    \begin{enumerate}[$(i)$]
        \item If there exists a weak subsolution $V$ of~\eqref{eqn:capillary-fbp} asymptotic to $u$, with $V \leq u$ on $\{ V > 0 \}$ and $\{ V > 0 \} \subset \{ u > 0 \}$, such that $\partial \{ V>0\}$ is disjoint from the origin, then ${\mathbf{C}}$ is minimizing for $\cA^{\theta}$ on $E_- (u)$.
        \item If there exists a weak supersolution $W$ of~\eqref{eqn:capillary-fbp} asymptotic to $u$, with $W \geq u$ on $\{ u > 0 \}$ and $\{ u > 0 \} \cap B^n_{\sigma}(0) \subset \{ W > 0 \}$, for some $\sigma>0$, then ${\mathbf{C}}$ is minimizing for $\cA^{\theta}$ on $E_+ (u)$.
    \end{enumerate}
\end{lemma}

\begin{lemma}\label{lemma:one-phase-minimizing-by-subsolutions}
    Let $u: \bR^n_{z'} \to \bR_{\geq 0}$ be a $1$-homogeneous solution of the one-phase problem.
    \begin{enumerate}[$(i)$]
        \item If there exists a weak subsolution $V$ of~\eqref{eqn:one-phase-problem} asymptotic to $u$, with $V \leq u$ on $\{ V > 0 \}$ and $\{ V > 0 \} \subset \{ u > 0 \}$, such that $\partial \{ V>0\}$ is disjoint from the origin, then $u$ is minimizing from below for $\cJ$.
        \item If there exists a weak supersolution $W$ of~\eqref{eqn:one-phase-problem} asymptotic to $u$, with $W \geq u$ on $\{ u > 0 \}$ and $\{ u > 0 \} \cap B^n_{\sigma}(0) \subset \{ W > 0 \}$, for some $\sigma > 0$, then $u$ is minimizing from above for $\cJ$.
    \end{enumerate}
\end{lemma}
\begin{proof}
    Both results~\ref{lemma:minimizing-by-subsolutions} and~\ref{lemma:one-phase-minimizing-by-subsolutions} follow from applying Lemmas~\ref{lemma:capillary-foliations-imply-minimizing} and~\ref{lemma:one-phase-foliations-imply-minimizing}, respectively using the homogeneity of $u$ and the respective energies $\cA^\theta$ and $\cJ$.
    Both arguments follow the same steps; for concreteness, let us focus on the situation of Lemma~\ref{lemma:minimizing-by-subsolutions}, and Lemma~\ref{lemma:one-phase-minimizing-by-subsolutions} follows mutatis mutandis.

    Using the fact that $\cA^\theta$ is scale invariant, we can produce a lower foliation from the single leaf $V$ (resp.~an upper foliation from the single leaf $W$) by rescaling.
    Indeed, the properties $(i)$ and $(ii)$ for $V, W$ imply that the rescaled functions
    \[
    V_\lambda(z) = \lambda^{-1}V(\lambda z')\qquad \text{and}\qquad W_\lambda = \lambda^{-1}W(\lambda z')
    \]
    are also (strict) weak sub- and supersolutions, respectively.
    Since $u$ is homogeneous, these rescalings will satisfy the same inclusion results for the free boundaries
    \[
    \{V_\lambda > 0\} \subset \{ u > 0\}\qquad \text{and}\qquad \{u > 0\} \cap B_{\lambda^{-1}\sigma}(0) \subset \{W_\lambda > 0\}
    \]
    for $\sigma$ as in $(ii)$.
    We claim that $V_\lambda$ (resp.~$W_\lambda$) gives a lower foliation according to Definition~\ref{def:lowerLeaf} (resp.~upper foliation with respect to Definition~\ref{def:upperLeaf}): the continuity $(\lambda,z') \mapsto V_{\lambda}(z')$ and $\lambda \mapsto \overline{ \{ V_{\lambda} > 0 \} }$ (resp.~$(\lambda,z') \mapsto W_{\lambda}(z')$) and the properties $(i)$~--~$(iii)$ of the respective definition are clear.
    Since $V$ is asymptotic to $u$, we have $V_{\lambda} \to u$ locally uniformly in $\{ u > 0\}$; the same holds for $W$.
    
    For $V$, $\partial \{ V>0\}$ being disjoint from the origin implies that $\sigma = \on{dist}( 0, \partial \{ V>0\}) > 0$, so $\{ V>0 \} \subset \{ u>0\}$ makes $B^n_{\sigma}(0) \subset \bR^n \setminus \{ V>0\}$.
    For any compact set $K \subset \bR^n$, we can find $\lambda_K > 0$ so that $K \subset B^n_{\lambda \sigma}(0)$ for all $\lambda>\lambda_K$, whereby $V_{\lambda}|_K = 0$ proves $(iv)$.
    For $W$, we use $B^n_{\sigma}(0) \subset \{ W > 0\}$ to obtain $\inf_{B^n_{\sigma/2}(0)} W = c > 0$, whereby $W_{\lambda} \geq c \lambda$ in $B^n_{\lambda\sigma/2}(0)$ yields $W_{\lambda} \to + \infty$ locally uniformly as $\lambda \to + \infty$, proving $(iv)$.
    Therefore, $V_{\lambda}$ is a lower foliation of $u$ (resp.~$W_{\lambda}$ is an upper foliation) and applying Lemma~\ref{lemma:capillary-foliations-imply-minimizing} proves that $u$ is (one-sided or globally) minimizing.
\end{proof}
When the subsolution $V$ (resp.~supersolution $W$) of Lemma~\ref{lemma:one-phase-foliations-imply-minimizing} is strict, it is well-known that the cone solution $u$ must be strictly stable; see, for example,~\cite{jerison-savin}*{Proposition~2.1} for the one-phase problem, and~\cite{singular-capillary}*{Proposition~2.4} for an analogous result on the capillary problem.

\subsection{Strictly minimizing cones}

In their foundational work on isolated singularities, Hardt and Simon~\cite{hardt-simon} introduced the notion of a \textit{strictly minimizing} cone $\mathbf{C} \subset \bR^{n+1}$.
A cone $\mathbf{C} \subset \bR^{n+1}$ with isolated singularities is said to be strictly minimizing if there exists a $\kappa > 0$ such that for any $\sigma \in (0,1)$, the inequality 
    \[
    \textbf{M}(\mathbf{C} \mres B_1) \leq \textbf{M}(S) - \kappa \cdot \sigma^n
    \]
    is satisfied for $S$ any integer multiplicity current with $\on{spt} S \Subset \bR^{n+1} \setminus B_{\sigma}$ and $\partial S = \partial (\mathbf{C} \cap B_1)$.
The class of strictly stable, strictly minimizing cones has important geometric properties and provides a natural setting for studying questions of uniqueness of tangent cones and their implications to the regularity theory of area-minimizing manifolds; see, for example,~\cites{SimonSolomon, uniqueness-cylindrical, uniqueness-gabor, uniqueness-ftw}.

This notion has a natural analogue for the capillary energy.
Denote by $E_{\pm}$ the two connected components of $\bR^{n+1}_+ \setminus \mathbf{C}$, whose reduced boundaries satisfy $\partial^* E_{\pm} \cap \bR^{n+1}_+ = \mathbf{C} \setminus \{ 0 \}$.

\begin{definition}\label{def:strictly-minimizing}
    We say that the capillary cone ${\mathbf{C}}$ is \textit{strictly minimizing} for $\cA^{\theta}$ if there exists a $\kappa>0$ such that for any $\sigma \in (0,1)$, we have
    \[
    \cA^{\theta}(E_+ ; B_1^+) \leq \cA^{\theta}(F; B_1^+) - \kappa \cdot \sigma^n
    \]
    where $B_1^+ := B_1 \cap \bR^{n+1}_+$ denotes the upper unit half-ball, and $F$ is any Caccioppoli set with $\cH^n(\partial^* F \cap B^+_{\sigma}) = 0$ and $F \triangle E_+ \Subset B_1^+$.
    \end{definition}
It will also be useful to formulate Definition~\ref{def:strictly-minimizing} in the language of currents.
    We say that the cone ${\mathbf{C}}$ is \textit{strictly minimizing} for $\cA^{\theta}$ if there exists a $\kappa > 0$ such that for any $\sigma \in (0,1)$, we have
    \[
    \cA^{\theta}( E_+ ; B_1^+) \leq \mathbf{M}( S \mres B_1^+) - \cos \theta \cdot \mathbf{M}(T) - \kappa \cdot \sigma^n
    \]
    for any $n$-dimensional currents $S,T$ satisfying $S \mres B_{\sigma} = T \mres (B_{\sigma} \cap \Pi) = 0$ and
    \[
    (S - \partial[\![ E_+]\!]) \mres (\bR^{n+1}_+ \setminus B_1) = 0 , \qquad ( T - \partial[\![ E_+]\!] ) \mres ( \Pi \setminus B_1) = 0, \qquad (\partial S + \partial T) \mres B_1 = 0.
    \]
    Taking $(S,T) = (\partial [\![ F ]\!] \mres B_1^+, \partial [\![ F]\!] \mres (B_1 \cap \Pi))$ in this setting shows that a strictly minimizing cone ${\mathbf{C}}$ in the sense of currents is also strictly minimizing in the Caccioppoli sense.

A well-established strategy for proving the mass-minimality of minimal surfaces is the technique of calibrations~\cite{bdgg}.
We refer the reader to~\cite{finn-capillary}*{Ch.~6 and Ch.~8} for the application of calibration techniques to capillary surfaces, as well as to~\cite{capillary-liouville}*{Appendix A} for a recent treatment.
De~Philippis and Paolini~\cite{dephilippis-paolini} introduced the more flexible notion of \textit{sub-calibrations} and used this technique to provide an elegant proof of the minimality of the Simons cone.
Later, De~Philippis-Maggi and Z.~Liu~\cites{sharp-stability-plateau, lawson-liu} strengthened this technique to obtain quantitative stability and strict minimality inequalities for the Lawson cones.
We introduce the analogue of this technique to the capillary setting and present important applications.

\begin{definition}\label{def:capillary-calibration}
    A \textit{capillary sub-calibration} of a cone ${\mathbf{C}}$ on the side $E_{\pm}$ is a $W^{1,1}_{\on{loc}}$ homogeneous degree-$0$ vector field $X_{\pm}$ defined on $\overline{E_{\pm}} \setminus K$, for $K$ a closed, homothety-invariant set with $\cH^n(K) = 0$, satisfying the following properties:
    \begin{enumerate}[$(i)$]
        \item $X_{\pm} = \nu_{\mathbf{C}}$ on $\mathbf{C} \setminus K$, where $\nu_{\mathbf{C}}$ denotes the normal vector to the cone pointing into $E_{\pm}$.
        \item The following inequalities are satisfied, for the sign $\pm$ consistent with the component $E_{\pm}$:
    \begin{align}
        |X_{\pm}| &\leq 1 & & \text{in } \; \overline{E_{\pm}} \setminus K, \tag{C1} \label{eqn:calibration-norm} \\
        \on{div} X_{\pm} &\geq 0 & & \text{in } \; E_{\pm}, \tag{C2} \label{eqn:calibration-div} \\
        \la X_{\pm} , e_{n+1} \rg &\geq \pm \cos \theta & & \text{on } \; ( \Pi \cap \overline{E_{\pm}}) \setminus K.\label{eqn:calibration-cos} \tag{C3} 
    \end{align}
\end{enumerate}
    The inequalities~\eqref{eqn:calibration-div} and~\eqref{eqn:calibration-cos} are understood in the pointwise sense, where $X_{\pm}$ is $C^1$, and in the trace (distributional) sense across $\Pi$.
    The sub-calibration $X_{\pm}$ of the component $E_{\pm}$ is \textit{strict}, if at least one of the inequalities~\eqref{eqn:calibration-norm}~--~\eqref{eqn:calibration-cos} holds strictly in the neighborhood of some point.
\end{definition}
When the sub-calibration is $C^{1,1}_{\on{loc}}$, the last condition can be replaced by~\eqref{eqn:calibration-norm}~--~\eqref{eqn:calibration-cos} holding strictly at some point.
If a cone ${\mathbf{C}}$ admits sub-calibrations $X_{\pm}$ of both sides $E_{\pm}$, we may define
\[
X := \begin{cases}
    X_+, & p \in \overline{E}_+ \cup \mathbf{C}, \\
    - X_-, & p \in E_-,
\end{cases}
\]
which is a $W^{1,1}_{\on{loc}}$ vector field on $\bR^{n+1} \setminus K$ that agrees with the normal vector to the cone pointing into $E_+$.
We call $X$ a \textit{capillary sub-calibration} of the cone, which is \textit{strict} if both $X_{\pm}$ are.

We will be especially interested in the sub-calibration obtained by a potential function $\varphi$ defined on $E_{\pm}$ and vanishing on ${\mathbf{C}}$.
Then, $X := \frac{\nabla \varphi}{|\nabla \varphi|}$ has $X = \nu_{\mathbf{C}}$, along ${\mathbf{C}}$, and $|X| = 1$.
Moreover, a classical divergence computation shows that
\begin{equation}\label{eqn:div-X-computation}
    \on{div} X = \frac{|\nabla \varphi|^2 \Delta \varphi - D^2 \varphi ( \nabla \varphi, \nabla \varphi)}{|\nabla \varphi|^3}.
\end{equation}

\begin{lemma}\label{lemma:strict-calibration}
    Let $\mathbf{C} \subset \bR^{n+1}_+$ be a regular capillary minimal cone with isolated singularities.
    If ${\mathbf{C}}$ admits a strict sub-calibration, then it is strictly minimizing for $\cA^{\theta}$ in the sense of currents.
\end{lemma}
\begin{proof}
    The proof is a straightforward adaptation of the argument in~\cite{hardt-simon}*{Theorem 3.2}; we provide it for completeness on $E_+$ and observe that the same argument transfers to $E_-$ using our conventions.
    Applying the reduction therein, it is sufficient to prove that
    \[
    \cA^{\theta}(E_+ ; B_1^+) \leq \mathbf{M}( S \mres B_1^+) - \cos \theta \cdot \mathbf{M}(T) - \kappa \cdot \sigma^n
    \]
    for competitor pairs $(S,T)$ of the form
    \[
    S = (\mathbf{C} \mres B_1^+) + \partial [\![ U]\!], \qquad T = \partial [\![ E_+ ]\!] \mres (B_1 \cap \Pi) + \partial [\![ U ]\!] \mres \Pi,
    \]
    where $U$ is some open set with $U \subset (E_+ \cap B_1^+) \setminus \bar{B}_{\sigma}$ and $\partial U = \on{spt} S \cup \overline{\mathbf{C} \cap B_1^+}$.
    Without loss of generality, we may assume that the condition~\eqref{eqn:calibration-norm} holds strictly on some homothetically invariant open set $W \subset E_+ \setminus K$, on which $|X| \leq 1 - \alpha$ for some $\alpha > 0$.
    Indeed, if either inequality~\eqref{eqn:calibration-div} or~\eqref{eqn:calibration-cos} held strictly, we could choose, for instance, a point $y \in E_+ \cap \bS^n_+ \setminus K$ where $\on{div} X(y) > 0$.
    We would then replace $X$ by $\tilde{X} := \zeta X$, for $\zeta \in C^{\infty}( \overline{\bR^{n+1}_+} \setminus \{ 0 \})$ a non-negative homogeneous function of degree $0$, to be specified below, with $\zeta \equiv 0$ on $K$, $\zeta(y) < 1$ and $\| \zeta - 1\|_{C^1(\bS^n_+)}$ small.
    We can furthermore arrange that $\on{spt} (\zeta - 1) \Subset E_+ \setminus K$ and $\zeta \equiv 1$ in a neighborhood of $(\mathbf{C} \cap B_1^+) \cup ( \Pi \cap B_1)$.
    This modification preserves the boundary inequality property~\eqref{eqn:calibration-cos} for $\tilde{X}$ and the agreement with $\nu_{\mathbf{C}}$ along ${\mathbf{C}}$.
    We seek to prove that the expression
    \[
    \Delta(U) := \left[ \mathbf{M}(S \mres B_1^+) - \cos \theta \cdot \mathbf{M}( T \mres (B_1 \cap \Pi)) \right] -  \cA^{\theta}( E_+ ; B_1^+)
    \]
    satisfies $\Delta (U) \geq \kappa \sigma^n$ for some $\kappa>0$ and any $U$.
   From Property~\eqref{eqn:calibration-norm}, since $|\tilde{X}| \leq 1$ and $\tilde{X} = X = \nu_{\mathbf{C}}$ on ${\mathbf{C}}$, we have
   \[
   \int_{\on{spt} R} \la \tilde{X} , \nu_R \rg \leq \mathbf{M}(R)
   \]
   for any rectifiable $n$-current $R$, with equality for $R=\mathbf{C}$.
   We now apply the divergence theorem on the region $U$ in the half-space $\bR^{n+1}_+$; see, for example,~\cite{maggi}*{Proposition~19.22 and Lemma~22.11},~\cite{sharp-stability-plateau}*{2.14}, and~\cite{dephilippis-maggi}*{Lemma~6.1} for the validity of the divergence theorem in our setting.
   Writing $\int_{\on{spt} S} - \int_{\mathbf{C} \cap B_1^+} = \int_{\partial U \cap \bR^{n+1}_+}$, we may express
   \begin{align*}
       \int_{\on{spt} S} \la \tilde{X} , \nu_S \rg - \int_{\mathbf{C} \cap B_1^+} \la \tilde{X} , \nu_{\mathbf{C}} \rg = \int_{\partial U \cap \bR^{n+1}_+} \la \tilde{X} , \nu_U \rg = \int_U \on{div} \tilde{X} + \int_{\partial U \cap \Pi} \la \tilde{X} , e_{n+1} \rg.
   \end{align*}
   For the term $\on{div} \tilde{X}$, we may bound
   \[
   - \on{div} \tilde{X} = - \on{div} ( \zeta X) = - ( \zeta \on{div} X + \la X, \nabla \zeta \rg) \leq |\nabla \zeta|, 
   \]
   using $\on{div} X \geq 0$ from~\eqref{eqn:calibration-div}.
   We therefore obtain
   \begin{equation}\label{eqn:Delta(U)-inequality}
   \Delta (U) \geq - \int_U |\nabla \zeta| + \left( \int_{\partial U \cap \Pi} \la \tilde{X} , e_{n+1} \rg - \cos \theta \, \mathbf{M} ( \partial U \cap \Pi) \right) + \left( \mathbf{M}(S) - \int_{\on{spt} S} \la \tilde{X} , \nu_S \rg \right) 
   \end{equation}
   where the second term is non-negative due to~\eqref{eqn:calibration-cos}.
   
    Shrinking $W \Subset E_+$ and $\alpha>0$, if needed, we can arrange to have $W \cap \Pi = \varnothing$ and $W \cap K = \varnothing$, while $\cH^n( W \cap \bS^n_+) > 0$ for the spherical slice of $W$.
    We then denote by $\pi_{\sigma}$ the radial retraction of $\bR^{n+1}_+$ onto $\bar{B}_{\sigma}$, so
    \[
    \mathbf{M}( S \mres W) \geq \mathbf{M} ( (\pi_{\sigma})_{\#}  (S \mres W)) \geq \cH^n( W \cap \bS^n_+) \, \sigma^n.
    \]
    The latter is obtained because $(\pi_{\sigma})_{\#}  (S \mres W)$ has multiplicity $\geq 1$ on $\pi_{\sigma}( W \cap \bS^n_+)$.
    Combining these inequalities, we have $|\tilde{X}| \leq 1-\alpha$ on $\on{spt} S \cap W$, so 
    \allowdisplaybreaks{
    \begin{align*}
        \Delta(U) &\geq - \int_U |\nabla \zeta| + \mathbf{M}( S \mres W) - \int_{\on{spt} S \cap W} \la \tilde{X} , \nu_S \rg \\
        &\geq - \int_U |\nabla \zeta| + \mathbf{M}(S \mres W) - (1-\alpha) \, \mathbf{M}( S \mres W) \\
        &\geq - \int_U |\nabla \zeta| + \alpha \cH^n( W \cap \bS^n_+) \, \sigma^n.
    \end{align*}}
    Since $\cH^n(K) = 0$, we may choose $\zeta$ in the above modification to satisfy $\zeta \equiv 0$ on $\{ p : \on{dist}(p,K) < \frac{\delta}{2} \}$ and $\zeta \equiv 1$ on $\{ p \in \bar{B}_1^+ : \on{dist}(p,K) \geq \delta \}$, with $\int_U |\nabla \zeta| \leq \ve(\delta)$, where $\ve(\delta) \downarrow 0$ and $\delta \downarrow 0$.
    Taking $\delta$ sufficiently small and working with the corresponding vector field $\tilde{X} = \zeta X$, we obtain 
    \[
    \Delta(U) > \kappa \sigma^n, \qquad \kappa := \tfrac{1}{2} \alpha \, \cH^n ( W \cap \bS^n_+),
    \]
    which shows that ${\mathbf{C}}$ is strictly minimizing for $\cA^{\theta}$.
\end{proof}
\begin{remark}\label{remark-1-proof}
    In view of the inequality~\eqref{eqn:Delta(U)-inequality}, the strict lower bound $\Delta(U) \geq \kappa \sigma^n$ would be obtained analogously if either of the conditions $\on{div} \tilde{X} \geq 0$ or $\la \tilde{X} , e_{n+1} \rg \geq \pm \cos \theta$ was assumed to be strict.
    Indeed, suppose that either expression (on the side $E_{\pm}$, with the corresponding choice of sign $\pm$) is 
    \[ 
    \on{div} \tilde{X}(p) \geq \alpha |p|^{-1} \qquad \text{or} \qquad \la \tilde{X}, e_{n+1} \rg\geq \alpha \pm \cos \theta \quad \text{for some } \; \alpha>0
    \]
    on some homothetically invariant set $W \Subset E_{\pm} \setminus K$, in the first case, or $W \Subset (\partial^* E_{\pm} \cap \Pi) \setminus K$, in the second case.
    Then, we would likewise obtain
    \begin{align*}
    \int_U \on{div} \tilde{X} &\geq \alpha \cH^n \bigl( U \cap W \cap (B_1^+ \setminus B_{\sigma}) \bigr) \gtrsim \sigma^n \qquad \text{or} \\
    \int_{\partial U \cap \Pi} ( \la \tilde{X} , e_{n+1} \rg - \cos \theta) &\geq \alpha  \cH^n ( (\partial^* U \cap \Pi) \cap W)\gtrsim \sigma^n,
    \end{align*}
    respectively, by the same radial projection argument involving $\pi_{\sigma}$.
\end{remark}

Hardt and Simon~\cite{hardt-simon}*{Theorem 3.2 and Remark 3.3} obtained equivalent conditions for a cone to be strictly minimizing.
Similarly, the property of Lemma~\ref{lemma:strict-calibration} is an equivalence: a minimal cone admits a strict sub-calibration if and only if it is strictly minimizing.
This result will follow from the construction of a foliation for minimizing capillary cones, analogous to the one considered by Hardt-Simon.
Moreover, global quadratic stability inequalities are linked to the uniform positivity of the second variation of the area for regular, uniquely $\cA^{\theta}$-minimizing hypersurfaces, as in~\cite{sharp-stability-plateau}.
We study these constructions and their implications on quantitative stability in upcoming work.

As a consequence of (the boundaryless version of) Lemma~\ref{lemma:strict-calibration} and Lawson's original construction~\cite{lawson}, Hardt-Simon deduced that the Lawson cones $C(\bS^{n-k-1} \times \bS^k)$ are all strictly minimizing, for $n \geq 8$.
Combined with the work of Sim\~oes, Lawlor, F.H.~Lin, and Morgan~\cites{lawlor, lin, morgan, simoes}, this result shows that the Lawson cones $C(\bS^{n-k-1} \times \bS^k) \subset \bR^{n+1}$ are strictly minimizing for all $n \geq 7$ except $(n,k) = (7,1)$ or $(7,5)$.

\begin{lemma}\label{lemma:one-sided}
    Let $\mathbf{C} \subset \bR^{n+1}_+$ be a regular capillary minimal cone with isolated singularities.
    If ${\mathbf{C}}$ admits a strict sub-calibration on one of the regions $E_{\pm}$, then it is strictly minimizing for $\cA^{\theta}$ in that region, in the sense of currents, and stable.
    If the calibration satisfies, in addition,
    \begin{equation}\label{eqn:one-sided-calibration}
    \on{div} X_{\pm} \geq c R^{-2} \on{dist}(p, \mathbf{C})
    \end{equation}
    on the side $E_{\pm}$ where it is defined, then the cone is strictly stable.
\end{lemma}
\begin{proof}
    The strictly minimizing property of ${\mathbf{C}}$ in the component supporting a strict sub-calibration follows from Lemma~\ref{lemma:strict-calibration}.
    This is well-known to imply the stability of ${\mathbf{C}}$, cf.~\cites{chs, lin}.
    If the one-sided calibration additionally satisfies~\eqref{eqn:one-sided-calibration} (without loss on generality, assumed to hold for $+ \on{div} X$ on $E_+$), then strict stability follows by a direct adaptation of~\cite{sharp-stability-plateau}*{\S~4, Proof of Theorem 3}, which we sketch here for the convenience of the reader.

    For any $\varphi \in C^1( \mathbf{C} \setminus \{ 0 \})$ with $\on{spt} \varphi \Subset B_R$, there exists a $t_0 > 0$ such that for every $t< t_0$, we may define an open set $F_t \subset \bR^{n+1}$ such that $E_+ \triangle F_t \Subset B_R$ and 
    \[
    \partial F_t \setminus \{ 0 \} = \left\{ p + t \, \varphi(p) \nu_{\mathbf{C}}(p) : p \in \mathbf{C} \setminus \{ 0 \} \right\}.
    \]
    Observe that it is sufficient to test the (strict) stability inequality on non-negative functions $\varphi$.
Indeed, Kato's inequality ensures that $|\nabla |\varphi|| \leq |\nabla \varphi|$ a.e., and since $|\varphi|^2 = \varphi^2$, we obtain $Q_{\theta}( |\varphi|) \leq Q_\theta(\varphi)$.
In fact, writing $\varphi = \varphi_+ - \varphi_-$ with $\varphi_{\pm} = \max \{ \pm \varphi, 0 \}$ makes $Q_{\theta}(\varphi) = Q_{\theta}(\varphi_+) + Q_{\theta}(\varphi_-)$, since the mixed terms vanish almost everywher.
Moreover, the function $|\varphi|$ may be approximated by the smooth functions $\varphi_{\ve} := \sqrt{\varphi^2 + \epsilon^2}$, which have $|\nabla \varphi_{\epsilon}| \leq |\nabla \varphi|$.
Using the $\varphi_{\ve}$ as test functions in the stability form $Q_{\theta}$ and sending $\epsilon \downarrow 0$ then yields the desired sufficiency.

    We may therefore work with positive functions $\varphi \in C^1( \mathbf{C} \setminus \{ 0 \})$ in the above construction, whereby $\partial F_t \setminus \{ 0 \} \subset E_+$ lies on the side of the cone calibrated by $X$.
    Computing as in Lemma~\ref{lemma:strict-calibration} and  Remark~\ref{remark-1-proof}, we find that
    \[
    \cA^{\theta} (F_t ; B_R) - \cA^{\theta}(E_+ ; B_R) \geq c \int_{ E_+ \triangle F_t } \frac{\on{dist}(p,\mathbf{C})}{R^2} \, d \cH^n.
    \]
    Let us write $\Phi(p,s) := p + s \nu_{\mathbf{C}}(p)$ for $0 \leq s \leq t \varphi(p)$ for $p \in \mathbf{C} \setminus \{ 0\}$, so $(p,s) \mapsto \Phi(p,s)$ is a normal coordinate system on the cube between ${\mathbf{C}}$ and the open set $F_t$ with $\partial F_t \setminus \{ 0 \} = \{ \Phi(p, t \varphi(p)) \}$.
    We now estimate the integrand as
    \begin{align*}
    \on{dist} (\Phi(p,s), \mathbf{C}) &= s + O ( R^{-1} s^2), \\
    |\Phi(p,s)|^{-2} &= R^{-2} + O (R^{-3} s), \qquad J(p,s) = 1 + O ( R^{-1} s),
    \end{align*}
    uniformly on compact subsets away from the apex.
    Consequently, the area Jacobian satisfies
    \[
    \frac{\on{dist}( \Phi(p,s), \mathbf{C})}{|\Phi(p,s)|^2} J(p,s) = s R^{-2}+ O( R^{-3} s^2).
    \]
    We may therefore use a Taylor expansion to compute
    \begin{align*}
        \cA^{\theta}(F_t ; B_R) - \cA^{\theta}(E_+ ; B_R) &= \tfrac{1}{2} t^2 Q_{\theta} (\varphi) + O(t^3), \\
        \int_{ E_+ \triangle F_t } \frac{\on{dist}(p,\mathbf{C})}{R^2} \, d \cH^n &= \frac{1}{2} t^2 \int_{\mathbf{C}} \frac{\varphi^2}{R^2} \, d\cH^n + O (t^3)
    \end{align*}
    Combining the two computations and taking $t \downarrow 0$ sufficiently small proves the strict stability~\eqref{eqn:strictly-stable} for every $\varphi \in C_c^1( \mathbf{C} \setminus \{ 0 \})$.
    
    To prove the property for an arbitrary $\varphi \in C^1(\mathbf{C})$, we let $\chi: \bR_{\geq 0} \to [0,1]$ be a smooth cutoff function with $\chi \equiv 0$ on $[0,1]$ and $\chi \equiv 1$ on $[2, \infty)$, with $|\chi'| \leq 2$ on $(1,2)$.
    Then, the functions $\psi_j (z) := \chi ( j R)$ satisfy $\psi_j \equiv 0$ on $B_{1/j}$, $\psi_j \equiv 1$ outside $B_{2/j}$, with $\| \psi_j \|_{\on{Lip}} \leq 2j$.
    Standard density estimates of~\cites{dephilippis-maggi} for the cone ${\mathbf{C}}$ yield $\cH^n ( \mathbf{C} \cap B_r) \leq c(n) r^n$ and $|A_{\mathbf{C}}| \leq c R^{-1}$, therefore we may apply the above inequality for each function $\psi_j \varphi$, where $\on{spt} (\psi_j \varphi) \cap \{ 0 \} = \varnothing$, to obtain
    \[
    Q_{\theta} (\psi_j \varphi) \geq c \int_{\mathbf{C}} \frac{\varphi^2}{R^2} \, d \cH^n.
    \]
    The desired inequality~\eqref{eqn:strictly-stable} now follows from passing to the limit as $j \to \infty$.
\end{proof}

\begin{remark}
    In general, it is not known whether strict minimality in the sense of Hardt-Simon (see Definition~\ref{def:strictly-minimizing}) suffices to deduce the strict stability of the cone, cf.~\cite{hardt-simon}*{Remark 3.3}.
    On the other hand, the cones considered here will satisfy a capillary analogue of Lawlor's criterion strictly, see~\cite{lawlor}*{\S~6.1 and Ch.~7(3)} as well as the discussion of Morgan in~\cite{morgan}*{\S~3.6}.
    When the cone is, in addition, uniquely minimizing, we can relate strict stability and strict minimality by a capillary analogue of the result of De~Philippis-Maggi~\cite{sharp-stability-plateau}, discussed in upcoming work~\cite{FTW_MinimizingII}.
\end{remark}

The following is a sufficient condition for the existence of a capillary sub-calibration.

\begin{lemma}\label{lemma:get-a-foliation}
    Let $\mathbf{C} \subset \bR^{n+1}_+$ be a capillary cone with contact angle $\theta$ and consider either connected component $E_{\pm}$ of $\bR^{n+1} \setminus \mathbf{C}$.
    Suppose that there exists a properly embedded, connected $C^1$ hypersurface $S_{\pm} \subset E_{\pm}$ with $\partial S_{\pm} \subset \Pi$ and $A_{S_{\pm}} \in L^1_{\on{loc}}$, satisfying:
\begin{enumerate}[$(i)$]
    \item $S_{\pm}$ is asymptotic to $\mathbf{C} = \partial E_{\pm}$.
    \item $S_{\pm}$ is star-shaped, meaning that $\pm \la p , \nu_{S_{\pm}}(p) \rg > 0$ for all $p \in S_{\pm}$, where $\nu_{S_{\pm}}$ denotes the normal vector of $S_{\pm}$ asymptotic to $\nu_{\mathbf{C}}$ in $E_{\pm}$.
    \item $S_{\pm}$ intersects the ray $\ell_q := \{ \lambda q : \lambda > 0 \}$ through some $q \in E_{\pm}$ exactly once.
    \item $H[S_{\pm}] \geq 0$, where $H[S_{\pm}]$ denotes the mean curvature of $S_{\pm}$, computed with respect to $\nu_{S_{\pm}}$.
    \item $\la \nu_{S_{\pm}} , e_{n+1} \rg \geq \pm \cos \theta$ holds along $\partial S_{\pm} \cap \Pi$.
\end{enumerate}
Then, ${\mathbf{C}}$ admits a capillary sub-calibration on $E_{\pm}$, which is strict if $(iv)$ or $(v)$ holds strictly in the neighborhood of some point.
\end{lemma}
\begin{proof}
We first show that the properties $(i)$~--~$(iii)$ are equivalent to the statement that the rescalings $\{ \lambda S_{\pm} \}_{\lambda > 0}$ of $S_{\pm}$ foliate $E_{\pm}$.
    Equivalently, we will prove that every ray $\ell_q := \{ \lambda q : \lambda > 0 \}$ through some $q \in E_{\pm}$ intersects $S_{\pm}$ exactly once.
    We consider the radial projection map
    \[
    \pi: S_{\pm} \to E_S := E_{\pm} \cap \bS^n_+, \qquad p \mapsto \frac{p}{|p|}.
    \]
    Since $\pm \la p , \nu_{S_{\pm}}(p) \rg > 0$, we find that $p$ is never tangent to $S_{\pm}$, which implies that $d \pi$ is everywhere injective, hence $\pi$ is a local diffeomorphism.
    Since $S_{\pm}$ is asymptotic to ${\mathbf{C}}$ away from a compact set, we deduce that the map $\pi$ is proper.
    It follows that $\pi$ is a proper local diffeomorphism, therefore a covering map.
    Since $S_{\pm}$ intersects some ray $\ell_q$ exactly once, meaning $\# \{ \pi^{-1} ( \frac{q}{|q|}) \} = 1$, we deduce that this covering map has degree $1$, hence it is a diffeomorphism of manifolds with boundary.
    
    Since $E_{\pm}$ is the complement of a cone, it is a conical region, meaning that $E_{\pm} = \{ \lambda \omega : \lambda>0, \omega \in E_S \}$ for $E_S := E_{\pm} \cap \bS^n_+$.
    The map $\pi: S_{\pm} \to E_S$ is a diffeomorphism, hence each $p \in E_{\pm}$ is contained in a unique leaf $\lambda S_{\pm}$ with $\lambda > 0$, so the positive homotheties $\{ \lambda S_{\pm} \}_{\lambda>0}$ of $S_{\pm}$ foliate $E_{\pm}$, as desired.
    We may therefore define a homogeneous degree zero vector field $X_{\pm}$ on $E_{\pm}$ by $X_{\pm}(q) = \nu_{S_{\pm}}( \lambda q)$ for any $q \in E_{\pm}$, where $\lambda = \lambda(q) > 0$ is the unique homothety such that $\lambda q \in S_{\pm}$.
    Since $A_{S_{\pm}} \in L^1_{\on{loc}}$, the normal vector satisfies $\nu_{S_{\pm}} \in W^{1,1}_{\on{loc}}(S_{\pm})$, so $X_{\pm} \in W^{1,1}_{\on{loc}}(E_{\pm} ; \bR^{n+1})$ with a $W^{1,1}_{\on{loc}}$ extension to $E_{\pm} \cup \mathbf{C}$ such that $X_{\pm}|_{\mathbf{C}} = \nu_{\mathbf{C}}$, where $\nu_{\mathbf{C}}$ denotes the normal vector of ${\mathbf{C}}$ pointing into $E_{\pm}$.
    The properties~\eqref{eqn:calibration-norm} and~\eqref{eqn:calibration-cos} follow from the definition of $X_{\pm}$ together with $(v)$.
    It is standard that $\on{div} X_{\pm} = H[S_{\pm}]$ is the mean curvature of $S_{\pm}$ computed with respect to $X_{\pm} = \nu_{S_{\pm}}$, proving~\eqref{eqn:calibration-div}.
    Therefore, $X_{\pm}$ defines a sub-calibration of ${\mathbf{C}}$ on $E_{\pm}$, which is strict if $(iv)$ or $(v)$ holds strictly in the neighborhood of some point.
\end{proof}

Extending Definition~\ref{def:strictly-minimizing} to the one-phase problem, we introduce the notion of \textit{strict minimality} for homogeneous critical points of the Alt-Caffarelli functional.

\begin{definition}\label{def:strictly-minimizing-one-phase}
    Let $u$ be a $1$-homogeneous critical point of the Alt-Caffarelli functional.
    We say that $u$ is \textit{strictly minimizing} for $\cJ$ if there exists a $\kappa > 0$ such that for any $\sigma \in (0,1)$, 
    \[
    \cJ( u ; B^n_1) \leq \cJ( v ; B^n_1) - \kappa \cdot \sigma^n
    \]
    holds for all $v \in H^1(B^n_1)$ with 
    \[
    v - u \in H^1_0(B^n_1), \qquad \cH^{n-1} ( \partial \{ v > 0 \} \cap B^n_{\sigma}) = 0.
    \]
    The latter condition means that $v \equiv 0$ on $B^n_{\sigma}$ or $v>0$ in $B^n_{\sigma}$.
\end{definition}

Chodosh-Edelen-Li~\cite{improved-regularity}*{Theorem 4.1} proved that given a sequence of contact angles $\theta_i \to 0$ and $\cA^{\theta_i}$-minimizing Cacciopoli sets $\Omega_i \subset B_1^+$ with an isolated singularity at $0$, the sets $M_i \cap B_{1/2}^+$ are contained in the graphs of Lipschitz functions $u_i$ over $B_{1/2}^n \subset \Pi$, for $i$ sufficiently large.
Moreover, the functions $\frac{1}{\tan \theta_i} u_i$ converge subsequentially in $(W^{1,2}_{\on{loc}} \cap C^{\alpha} )(B^n_{1/4})$ to a $1$-homogeneous minimizer $u_0$ of the Alt-Caffarelli functional $\cJ$, with $\partial \{ u_i > 0 \} \to \partial \{ v > 0 \}$ in the local Hausdorff distance.

We obtain a version of the above result for strictly minimizing solutions.
\begin{lemma}\label{lemma:strictly-minimizing-one-phase}
    Let $u_i$ be a sequence of graphical homogeneous solutions of the capillary problem with contact angle $\theta_i \to 0$.
    Suppose that the functions $\frac{1}{\tan \theta_i} u_i$ converge in $C^{\infty}_{\on{loc}}(\Pi \setminus \{ 0 \})$ to a solution $u_0$ of the one-phase problem, with
    \begin{equation}\label{eqn:ui-to-u0-converge}
        R^{-1} \left| u_i - \tan \theta_i \, u_0 \right| + |\nabla u_i - \tan \theta_i \, \nabla u_0| = C(n) o(\theta_i) \quad \text{on } \; B_R^n \subset \Pi 
    \end{equation}
    Suppose that each $u_i$ is one-sided strictly minimizing for $\cA^{\theta_i}$ in $E_-(u_i)$ with constant $\kappa_i \tan^2 \theta_i > 0$.
    Then, $u_0$ is strictly minimizing from below for $\cJ$ with constant $\kappa \geq \liminf_i \kappa_i$.
\end{lemma}
\begin{proof}
    If $\kappa=0$, then the result follows from the above discussion since $u_0$ is a minimizer for $\cJ$.
    We may therefore assume that $\kappa>0$ and consider a sequence of $\kappa_i \to \kappa$.
    For simplicity, we assume that $\kappa_i = \kappa$ for every $i$ and consider the corresponding homogeneous capillary minimizers $u_i$.
    Suppose, for contradiction, that $u_0$ is not $\kappa$-strictly minimizing in the ball $B_1^n$, so there exists a scale $\sigma>0$ and a $C^2$ competitor $v$ with $\on{spt} (u_0 - v) \Subset B_1^n$, such that 
    \begin{equation}\label{eqn:v-competitor-property}
    \partial \{ v >0 \} \cap B^n_{\sigma} = \varnothing, \qquad \cJ(v ; B^n_1) < \cJ(u_0;B^n_1) + (\kappa - \ve) \sigma^n
    \end{equation}
    for some $\ve>0$.
    For each $u_i$, we let $E_i := E_-(u_i)$.
    A standard computation (see, for example,~\cite{improved-regularity}*{(1.2) and Lemma~4.7}) shows that
    \begin{equation}\label{eqn:capillary-one-phase}
        \cA^{\theta_i}(E_i) = \frac{1}{2} \tan^2 \theta_i \int_{ \{ u_i >0 \} } \bigl( |\nabla \tilde{u}_i|^2 + \chi_{ \{ \tilde{u}_i > 0 \} } + O(\theta) \bigr) \, dz', \qquad \tilde{u}_i := \frac{1}{\tan \theta_i} u_i.
    \end{equation}
    Combined with the convergence~\eqref{eqn:ui-to-u0-converge}, this property implies that 
    \begin{align*}
    \cA^{\theta_i} (E_i, B_1^+) &= \tfrac{1}{2} \tan^2 \theta_i \, \cJ( \tilde{u}_i, B_1^n) + o(\theta_i^2) \\
    &= \tfrac{1}{2} \tan^2 \theta_i \, \cJ(u_0, B_1^n) + o(\theta_i^2)
    \end{align*}
    because $|\tilde{u}_i - u_0| = o(\theta_i)$ and $|\nabla \tilde{u}_i - \nabla u_0| = o(\theta_i)$.
    
    We may use $v$ to produce a competitor $v_i$ for $u_i$, when $\theta_i$ is sufficiently small, as follows.
    We choose a sequence $\delta_i \downarrow 0$ with $\delta_i < \tfrac{1}{10}(1-\sigma)$ and
\begin{equation}\label{eqn:delta-choice}
\sup_{B_1\setminus \bar{B}_{1-8\delta_i}} |\tilde u_i-u_0| \;+\;
\sup_{B_1\setminus \bar{B}_{1-8\delta_i}} |v-u_0|
\;\le\; \delta_i .
\end{equation}
Such a choice is possible due to $\tilde{u}_i \to u_0$ in $C_{\on{loc}}^0(B_1^n \setminus \{ 0 \})$ and $\on{spt}(u_0 -v) \Subset B_1^n$.
Let $\eta : [0,\infty) \to [0,1]$ be a smooth cutoff function with
\[
\eta(s) = 1 \quad \text{for } \; s \in [0,3], \qquad \eta(s) = 1 \quad \text{for } \; s \in [ 5, \infty), \qquad 0 \leq| \eta'(s)| \leq 10.
\]
We define $\eta_i \in C_c^{\infty}(B_1^n)$ to be a radial cutoff with $\eta_i(z') := \eta ( \frac{1- |z'|}{ \delta_i } )$, so that
\[
\eta_i \equiv 0 \ \text{on } B^n_{1-4 \delta_i}, \qquad
\eta_i \equiv 1 \ \text{on } B^n_1\setminus B^n_{1- 2\delta_i}, \qquad
|\nabla\eta_i| \leq C(n) \delta_i^{-1}.
\]
We may now define the competitor function $v_i$ by
\[
v_i := (1-\eta_i) \tan \theta_i v + \eta_i u_i
\]
extended by $u_i$ outside $B_1$.
We let $F_i := E_-(v_i)$, which therefore agrees with $E_i$ outside $B_1^+$.
By construction, $v_i = u_i$ in a neighborhood of $\partial B_1$, hence $F_i = E_-(v_i)$ satisfies $F_i \triangle E_i \Subset B_1^+$.
Moreover, using $\sigma < 1 - 4 \delta_i$ gives $\eta_i \equiv 0$ on $B^n_{\sigma}$.
Since also $v \equiv 0$, we deduce that $v_i \equiv 0$ on $B^n_{\sigma}$, so $\cH^n( \partial^* F_i \cap B^+_{\sigma}) = 0$, meaning that each $F_i$ avoids the ball $B^+_{\sigma}$.
We now claim that 
\begin{equation}\label{eqn:competitor-claim-before-tanthetai}
\cJ( ( 1 - \eta_i) v + \eta_i \tilde{u}_i ; B_1^n) = \cJ( v; B_1^n) + o(1) \qquad \text{as } \; i \to \infty.
\end{equation}
Indeed, $\eta_i \equiv 0$ on $B^n_{1 - 4 \delta_i}$, the two functions agree and the difference in energy comes from the annulus $A^n_i = B_1 \setminus \bar{B}_{1 - 4 \delta_i}$.
On this annulus, we may bound
\allowdisplaybreaks{
\begin{align*}
|\nabla \bigl( (1-\eta_i) v + \eta_i \tilde{u}_i \bigr)|^2 &\leq 2 |\nabla \bigl( (1-\eta_i) v \bigr)|^2 + 2 |\nabla ( \eta_i \tilde{u}_i)|^2 \\
&\leq 2 |\nabla v|^2 + 2 |\nabla \tilde{u}_i|^2 + 4 |\nabla \eta_i|^2 \cdot |\tilde{u}_i - v|^2 \\
&\leq 2 |\nabla v|^2 + 2 |\nabla \tilde{u}_i|^2 + 4 C(n).
\end{align*}}
To obtain the final inequality, we applied the bound $|\nabla \eta_i| \leq C(n) \delta_i^{-1}$ together with the triangle inequality~\eqref{eqn:delta-choice} for $|\tilde{u}_i - v|$.
We may therefore estimate
\[
\int_{A_i^n} |\nabla \bigl( (1-\eta_i) v + \eta_i \tilde{u}_i \bigr)|^2 \leq C(n) \int_{A_i^n} ( |\nabla v|^2 + |\nabla \tilde{u}_i|^2 + 1) \leq C(n) \cH^n( A^n_i) = O (\delta_i).
\]
In the last step, we used the fact that $\sup_{A^n_i} |\nabla \tilde{u}_i| \leq C(n)$ due to~\eqref{eqn:ui-to-u0-converge}.
We likewise bound
\[
\left| \{ v_i > 0 \} \triangle \{ v > 0 \} \right| \leq \cH^n(A_i^n) = O(\delta_i)
\]
and combining the two bounds implies estimate~\eqref{eqn:competitor-claim-before-tanthetai}.

Finally, we may use the expression~\eqref{eqn:competitor-claim-before-tanthetai}, the definition of $v_i$, and the estimate~\eqref{eqn:capillary-one-phase} to obtain
    \begin{align*}
    \cA^{\theta_i}(F_i ; B_1^+) &= \tfrac{1}{2} \tan^2 \theta_i \, \cJ(v, B_1^n) + o(\theta_i^2) \\
    &< \tfrac{1}{2} \tan^2 \theta_i \, \cJ(u_0, B_1^n) + (\kappa-\ve) \tan^2 \theta_i \, \sigma^n + o(\theta_i^2).
    \end{align*}
    for $F_i = E_-(v_i)$.
    Using the expression for $\cA^{\theta_i}(E_i, B_1^+)$, we conclude that
    \[
    \cA^{\theta_i}(F_i, B_1^+) - \cA^{\theta_i}(E_i, B_1^+) < (\kappa-\ve) \tan^2 \theta_i \sigma^n + o(\theta_i^2).
    \]
    where $F_i = E_-(v_i)$ and $E_i = E_-(u_i)$.
    Taking $\theta_i \to 0$ in this step produces a competitor for $E_i$ with 
    \[
    E_i \triangle F_i \Subset B_1, \qquad \cH^n(\partial^* F_i \cap B^+_{\sigma}) = 0, \qquad \cA^{\theta_i}(F_i;B_1^+) \leq \cA^{\theta_i}(E_i;B_1^+) + (\kappa - \tfrac{1}{2} \ve) \tan^2 \theta_i \, \sigma^n 
    \]
    which contradicts the assumption that $u_i$ is strictly minimizing with constant $\kappa \tan^2 \theta_i$.
    Consequently, no competitor $v$ satisfying~\eqref{eqn:v-competitor-property} exists, meaning that $\cJ(v;B_1^n) \geq \cJ(u;B_1^n) + \kappa \sigma^n$ whenever $\partial \{ v>0 \} \cap B^n_{\sigma} = \varnothing$.
    Therefore, $u$ is strictly minimizing, as desired.
\end{proof}

\section{\texorpdfstring{$O(n-k)\times O(k)$}{O(n-k)xO(k)}-invariant capillary cones}\label{section:o(n-k)-cones}

In this section, we construct non-flat $O(n-k) \times O(k)$-invariant capillary cones with an isolated singularity at the origin in $\bR^{n+1}_+$, with the properties described in Theorem~\ref{thm:capillary-cones}.
These cones have \textit{bi-orthogonal symmetry}, meaning that the factors of $O(n-k)\times O(k)$ act orthogonally on $\bR^n$.

Let us briefly outline the proof strategy.
The imposed symmetries reduce the existence of the cones to finding solutions to a free boundary ODE whose solutions must satisfy both Dirichlet and Neumann boundary data, with the specific Neumann data depending on the capillary angle $\theta$; cf.~Proposition~\ref{prop:FreeBdyODE=Cone}.
We produce such solutions by prescribing initial data $f(0) = a, f'(0) = 0$ and showing that for initial heights $a$ up to the one determined by the Lawson cone, the solutions $f$ exist until they attain $0$ and achieve every capillary angle.
The key ingredients are comparison and foliation results for solutions of the ODE with different initial heights, cf.~Lemma~\ref{lemma:comparison-of-solutions}, Proposition~\ref{prop:unique-crossing of solutions}, and Corollary~\ref{corollary:lambda-comparison}, as well as a monotone quantity that vanishes precisely on the Lawson cones and determines the zero-or-blow-up behavior of solutions to the ODE, cf.~Lemma~\ref{lemma:psi-detect-blow-up} and Propositions~\ref{prop:psi-at-turning-point} and~\ref{prop:general-a-behavior}.

In Sections~\ref{sec:minimizing} and~\ref{section:pi/2}, we will prove that these cones are area-minimizing in ambient dimension $8$ or higher (so $n \geq 7$), for the ranges of $1 \leq k \leq n-2$ and appropriate capillary angles as detailed in Theorems~\ref{thm:cones-are-minimizing-smalltheta} and~\ref{thm:cones-are-minimizing-nearpi/2}.

\subsection{ODE reduction} We work in Euclidean space $\bR^{n+1}_+ = \bR^{n-k} \times \bR^k \times \bR_+$ with coordinates $p = (x,y,z) \in \bR^{n-k} \times \bR^k \times \bR_+$, so $z := p_{n+1} = \la p , e_{n+1} \rg$.
We define
\[
R := |p| = \sqrt{|x|^2+|y|^2+z^2}, \qquad \rho := \sqrt{|x|^2 + |y|^2}, \qquad t := \frac{|y|}{\rho} \in [0,1].
\]
We denote the coordinates $x_i$ with Latin subscripts and $y_\alpha$ with Greek.
For a function $U(x,y) = U(\rho,t)$, we denote $Q(U) := \nabla^2 U( \nabla U, \nabla U)$ and study the minimal surface operator
\begin{align*}
\cM(U) &:= \on{div} \left( \frac{\nabla U}{\sqrt{1 + |\nabla U|^2}} \right), \qquad
\sqrt{1 + |\nabla U|^2} \, \cM(U) = \Delta U - \frac{Q(U)}{1 + |\nabla U|^2}.
\end{align*}
\begin{lemma}\label{lemma:general-computation-u-rho,t}
    For the function $U(x,y) = u(\rho,t)$, we have
    \begin{align}
    |\nabla U|^2 &= u_{\rho}^2 + \frac{1-t^2}{\rho^2} u_t^2 \label{eqn:grad-U-squared}, \\
    \Delta U &= u_{\rho \rho} + \frac{n-1}{\rho} u_{\rho} + \frac{1-t^2}{\rho^2} u_{tt} + \frac{(k-1) t^{-1} - (n-1) t}{\rho^2} u_t \label{eqn:laplace-general}, \\
    Q(U) &= u_{\rho}^2 u_{\rho \rho} + 2 \frac{1-t^2}{\rho^2} u_{\rho t} u_{\rho} u_t + \frac{(1-t^2)^2}{\rho^4} u_t^2 u_{tt} - \frac{1-t^2}{\rho^4} u_t^2 (\rho u_{\rho} + t u_t) \label{eqn:hessian-general}
\end{align}
The minimal surface operator is therefore
\begin{equation}\label{eqn:minimal-surface-urho,t}
    (1 + |\nabla U|^2 )^{\frac{3}{2}} \cM(U) = \left( 1 + u_{\rho}^2 + \frac{1-t^2}{\rho^2} u_t^2 \right) \Delta U - Q(U)
\end{equation}
for the expressions $\Delta U$ and $Q(U)$ given above.
\end{lemma}
\begin{proof}
In the variables $t, \rho$, we have
\begin{align*}
    \partial_i \rho &= \frac{x_i}{\rho}, &  \partial_{\alpha} \rho &= \frac{y_{\alpha}}{\rho}, \\
    \partial^2_{ij} \rho &= \frac{1}{\rho} \left( \delta_{ij} - \frac{x_i x_j}{\rho^2} \right), & \partial^2_{\alpha \beta} \rho &= \frac{1}{\rho} \left( \delta_{\alpha \beta} - \frac{y_{\alpha} y_{\beta}}{\rho^2} \right), & \partial^2_{i\alpha}\rho &= -\frac{x_iy_\alpha}{\rho^3}\\
    \partial_i t &= - \frac{x_i |y|}{\rho^3} = - \frac{x_i}{\rho^2} t, & \partial^2_{ij} t &= - \frac{\delta_{ij}}{\rho^2} t + 3\frac{x_i x_j}{\rho^4} t, & \partial^2_{i \alpha} t &= \frac{x_i y_{\alpha}}{\rho^4} ( 3t - t^{-1}) \\
    \partial_{\alpha} t &= \frac{y_\alpha}{\rho^2}(t^{-1}-t), &\quad \partial^2_{\alpha \beta} t &= \frac{1}{\rho^2} (t^{-1}-t) \left( \delta_{\alpha \beta} - (3 + t^{-2}) \frac{y_{\alpha} y_{\beta}}{\rho^2}  \right)
\end{align*}
We now consider a function $U(x,y) = u(\rho, t)$, so that 
\[
\nabla U = u_{\rho} \nabla \rho + u_t \nabla t = \left( u_{\rho} - \frac{t}{\rho} u_t \right) \frac{x}{\rho} + \left( u_{\rho} + \frac{t^{-1}(1-t^2)}{\rho} u_t \right) \frac{y}{\rho}  \, .
\]
We therefore obtain the expression~\eqref{eqn:grad-U-squared} by using $\frac{|x|^2}{\rho^2} = 1-t^2$ and $\frac{|y|^2}{\rho^2} = t^2$, whereby
\begin{align*}
    |\nabla U|^2 
    &= (1-t^2) \left( u_{\rho}^2 + \frac{t^2}{\rho^2} u_t^2 - 2 \frac{t}{\rho} u_t u_{\rho} \right) + t^2 \left( u_{\rho}^2 + \frac{t^{-2}(1-t^2)^2}{\rho^2} u_t^2 + 2 \frac{t^{-1}(1-t^2)}{\rho} u_t u_{\rho} \right) \\
    &= u_{\rho}^2 + \frac{(1-t^2)t^2}{\rho^2} u_t^2 + \frac{(1-t^2)^2}{\rho^2} u_t^2.
\end{align*}
Next, we compute the Laplacian $\Delta u$ and $Q(u) = \nabla^2 u(\nabla u, \nabla u)$.
The above computations imply
\begin{align*}
    |\nabla \rho|^2 &=1, &  \la \nabla \rho , \nabla t \rg &= 0, &|\nabla t|^2 = \frac{1-t^2}{\rho^2}, \\
    \Delta \rho &= \frac{n-1}{\rho}, &  \Delta t &= \frac{(k-1) t^{-1} - (n-1) t}{\rho^2}, \\
    \nabla^2 \rho (\nabla U, \nabla U) &= \frac{1-t^2}{\rho^3} u_t^2, &  \nabla^2 t(\nabla \rho, \nabla \rho) &= 0, \\
    \nabla^2 t(\nabla \rho, \nabla t) &= - \frac{1-t^2}{\rho^3}, &  \nabla^2 t (\nabla t, \nabla t) &= - \frac{t(1-t^2)}{\rho^4}.
\end{align*}
Then, we note the identities
{\allowdisplaybreaks
\begin{align*}
    \Delta U &= u_{\rho \rho} |\nabla \rho|^2 + 2 u_{\rho t} \la \nabla \rho, \nabla t \rg + u_{tt} |\nabla t|^2 + u_{\rho} \Delta {\rho} + u_t \Delta t \\
    &= u_{\rho \rho} + \frac{n-1}{\rho} u_{\rho} + \frac{1-t^2}{\rho^2} u_{tt} + \frac{(k-1) t^{-1} - (n-1) t}{\rho^2} u_t, \\
    Q(U) & = u_{\rho \rho} \la \nabla \rho, \nabla U \rg^2 + 2 u_{\rho t} \la \nabla \rho, \nabla U \rg \la \nabla t, \nabla U \rg + u_{tt} \la \nabla t, \nabla U \rg^2 \\
    & \quad \; + u_{\rho} \, \nabla^2 \rho (\nabla U, \nabla U) + u_t \, \nabla^2 t (\nabla U, \nabla U) \\
    &= u_{\rho}^2 u_{\rho \rho} + 2 \frac{1-t^2}{\rho^2} u_{\rho t} u_{\rho} u_t + \frac{(1-t^2)^2}{\rho^4} u_{tt} u_t^2 - \frac{1-t^2}{\rho^4} u_t^2 (\rho u_{\rho} + t u_t)
\end{align*}}
The equations~\eqref{eqn:laplace-general} and~\eqref{eqn:hessian-general} follow from these computations.
\end{proof}
To prove the existence of non-flat $O(n-k) \times O(k)$-invariant capillary cones with an isolated singularity at the origin, we study $1$-homogeneous solutions of equation~\eqref{eqn:minimal-surface-urho,t} having the form $u(\rho,t) = \rho f(t)$, for a function $f: [0,1] \to \bR$ with $f'(0) = 0$, and defined over the domain
\begin{equation}\label{eqn:capillary-domain}
\Gamma_{n,k,\theta} := \{ (x,y) \in \bR^{n-k} \times \bR^k : |y| \leq t_{n,k,\theta} \sqrt{|x|^2 + |y|^2} \}
\end{equation}
for a constant $t_{n,k,\theta} \leq 1$ to be determined.
\begin{proposition}\label{prop:FreeBdyODE=Cone}
    The graph of the function $U(x,y) = \rho f(t)$ defines an $O(n-k) \times O(k)$-invariant minimal cone with boundary in $\Pi$ if and only if $f$ satisfies $f'(0) = 0$ and
    \begin{equation}\label{eqn:cone-ODE}
    \begin{split}
        0 & =  (1+f^2) \left( (1-t^2) f'' + (n-1)(f - t f') \right) + (n-2)(1-t^2) f'^2 (f - t f') \\
        & \qquad + (k-1) t^{-1} f' \left( 1 + f^2 + (1-t^2) f'^2 \right).
    \end{split}
    \end{equation}
    The resulting cone has capillary boundary if and only if the constant $t_{n,k,\theta}$ of~\eqref{eqn:capillary-domain} is the first positive root of $f$, with $f$ positive on $[0,t_{n,k,\theta})$ and
    \begin{equation}\label{eqn:f'-at-theta}
    f'(t_{n,k,\theta}) = - \frac{\tan \theta}{\sqrt{1 - t_{n,k,\theta}^2}}.
\end{equation}
\end{proposition}
\begin{proof}
We use the computations of Lemma~\ref{lemma:general-computation-u-rho,t}, observing that
\[
u_{\rho} = f, \qquad u_t = \rho f', \qquad u_{\rho \rho} = 0, \qquad u_{\rho t} = f', \qquad u_{tt} = \rho f''.
\]
The computations~\eqref{eqn:grad-U-squared}~--~\eqref{eqn:minimal-surface-urho,t} therefore simplify to
\begin{align*}
    |\nabla U|^2 &= f^2 + (1-t^2) f'^2, \\
    \Delta U &= \frac{(n-1) (f - tf') + (1-t^2) f'' + (k-1) t^{-1} f'}{\rho}, \\
    Q(U) &= \frac{2 (1-t^2) f (f')^2 + (1-t^2)^2 (f')^2 f'' - (1-t^2) (f')^2 (f + tf')}{\rho} \\
    &= \frac{(1-t^2) (f')^2}{\rho} \left[ ( f- tf') + (1-t^2) f'' \right].
\end{align*}
Combining these computations yields the expression~\eqref{eqn:cone-ODE}.
The capillary condition then requires that $t_{n,k,\theta}$ be the first positive root of the even function $f$ solving the ODE~\eqref{eqn:cone-ODE}, and that $|\nabla U| = \tan \theta$ at this point.
The computation~\eqref{eqn:grad-U-squared} shows that $|\nabla U|^2 = (1-t_{n,k,\theta}^2) (f')^2$ at this point, with $f$ positive on $(0,t_{n,k,\theta})$.
The claimed equality~\eqref{eqn:f'-at-theta} follows.
\end{proof}
\begin{lemma}\label{lemma:lawson-cones}
For every $n \geq 3$ and $1 \leq k \leq n-2$, the function $\hat{f}_{n,k}$ given by
\[
\hat{f}_{n,k}(t) = \sqrt{\frac{k - (n-1) t^2}{n-k-1} }
\]
satisfies $\hat{f}_{n,k}' \to - \infty$ near its positive root, so it is an even solution of equation~\eqref{eqn:cone-ODE} for $\theta = \frac{\pi}{2}$.
The resulting cone in $\bR^{n+1}_+$, produced by the graph of $U_{n,k}(\rho,t) = \rho \hat{f}_{n,k}(t)$, is the half of the Lawson cone $C(\bS^{n-k-1} \times \bS^k)$ in the upper $z$-half-space, expressed in the form
\begin{align*}
C(\bS^{n-k-1} \times \bS^k) &= \left\{ (x,y,z) \in \bR^{n-k} \times \bR^k \times \bR : \tfrac{k}{n-k-1} |x|^2 = |y|^2 + z^2 \right\}, \\
\on{graph} (U_{n,k}) &= C \Bigl( \bS^{n-k-1} (\sqrt{\tfrac{k}{n-1}}) \times \bS^k (\sqrt{\tfrac{n-k-1}{n-1}}) \Bigr) \cap \{ z \geq 0 \}.
\end{align*}
\end{lemma}
\begin{proof}
    These properties are verified immediately by substituting the expression for $\hat{f}_{n,k}$ into equations~\eqref{eqn:cone-ODE} and~\eqref{eqn:f'-at-theta}.
    The expression of the graph of $U_{n,k}$ is likewise obtained from $\rho^2 t^2 = |y|^2$.
\end{proof}
We also observe that when $k=n-1$, equation~\eqref{eqn:cone-ODE} is solvable explicitly: the unique function $f$ solving~\eqref{eqn:cone-ODE} with $f(0) = \tan \theta$ and $f'(0) = 0$ is $f(t) = \tan \theta \sqrt{1-t^2}$.
This produces the function $U(x,y) = \tan \theta  |x|$ (where $x \in \bR$) corresponding to the union of two slanted half-spaces as the $\bZ_2 \times O(n-1)$-invariant capillary cone of this form.
The resulting cone is not minimizing for $\cA^{\theta}$; consequently, in what follows, we will only consider $1 \leq k \leq n-2$.
We denote by $\cL_{n,k}$ the linear Legendre operator with self-adjoint expression
\begin{equation}\label{eqn:legendre-form}
\begin{split}
\cL_{n,k} f &= (1-t^2) f'' + (n-1) (f - t f') + (k-1) t^{-1} f' \\
&= \frac{1-t^2}{p(t)} \left[ (p f')' + (n-1) \frac{p}{1-t^2} f \right], \qquad \text{where } \; p(t) := t^{k-1}(1-t^2)^{\frac{n-k}{2}}  .
\end{split}
\end{equation}
We also define the function $A_{n,k}(t) := t - \frac{k-1}{n-2} t^{-1} = t^{-1} ( t^2 - \frac{k-1}{n-2} )$.
We set $\alpha := \frac{k-1}{n-2}$, so $A(t) = t^{-1} (t^2-\alpha)$.
The ODE~\eqref{eqn:cone-ODE} for the even function $f$ assumes the form
\begin{equation}\label{eqn:odeStar}\tag{$\star$}
    (1 + f^2) \cL_{n,k} f + (n-2) (1-t^2) (f')^2 (f - A_{n,k} f') = 0.
\end{equation}
It will also be useful to consider a rescaled version of equation~\eqref{eqn:odeStar}.
For any $\lambda > 0$, let $f_{\lambda} := \frac{f}{\sqrt{\lambda}}$.
Expressing $f = \sqrt{\lambda} f_{\lambda}$ in~\eqref{eqn:odeStar} gives an ODE for $f_{\lambda}$, which assumes the form
\begin{equation}\label{eqn:rescaledODE}\tag{$\star_{\lambda}$}
    (1-t^2) f'' + (f - tf') + (n-2) \left( 1 + (1-t^2) \frac{\lambda (f')^2}{1 + \lambda f^2} \right) (f - A_{n,k} f') = 0.
\end{equation}
If $f(t)$ solves equation~\eqref{eqn:odeStar} for $\lambda = 1$ with initial data $f(0) = a$ and $f'(0) = 0$, then the function $f_{\lambda} := \frac{f}{\sqrt{\lambda}}$ solves~\eqref{eqn:rescaledODE} with initial data $f_{\lambda}(0) = \frac{a}{\sqrt{\lambda}}$ and $f'_{\lambda}(0) = 0$.
We may therefore study the dependence of the solutions to~\eqref{eqn:odeStar} on the initial height $a$ by equivalently studying the rescaled equation~\eqref{eqn:rescaledODE} for a fixed height $a$ and variable $\lambda$.

\subsection{Analysis and comparison of solutions}\label{section:analysis-of-solutions}
Throughout this section, we fix some $n \geq 3$ and $1 \leq k \leq n-2$.
We will construct capillary cones for every angle $\theta$ by proving the existence of initial data $f(0) = a, f'(0) = 0$ for which the graph of the function $f$ meets the $t$-axis at angle $\theta$.
This problem consists of two parts: first, showing that $f$ exists until it reaches 0; second, that we can sweep out every possible angle by varying the initial height of $f$.

We begin by describing the general and relative behaviors of solutions $f_a$ and $f_{\lambda}$ to the capillary ODEs~\eqref{eqn:odeStar} and~\eqref{eqn:rescaledODE} upon varying the initial height $a$ and the scaling parameter $\lambda$.
We prove that the solutions exhibit the behavior showcased in Figure~\ref{fig:fams-n7k1}: heuristically, the $f_a$ \textit{cross} and the $f_\lambda$ \textit{foliate}.
Moreover, the terminal contact angle increases with $a$ and with $\lambda$.
The solutions $f_a$ and $f_\lambda$ satisfy the following:
\begin{enumerate}[$(i)$]
    \item All solutions are strictly decreasing and strictly concave.
    \item All solutions of the capillary ODE~\eqref{eqn:odeStar} that reach zero cross, and subsequently stay below, any other solution that begins at a lower height.
    Consequently, the region under any solution reaching zero is foliated by the graphs of solutions starting from lower heights.
    \item The solutions of the rescaled equation~\eqref{eqn:rescaledODE} starting from the same height $a$ are ordered in strictly decreasing height by the parameter $\lambda$.
    For $a = a_{n,k} = \sqrt{\frac{k}{n-k-1}}$ the Lawson height, the solutions of the ODE~\eqref{eqn:rescaledODE} reach zero if and only if $\lambda \in [0,1]$, and foliate the region between the Lawson and one-phase solutions.
\end{enumerate}
\begin{figure}[t]
  \centering
  \includegraphics{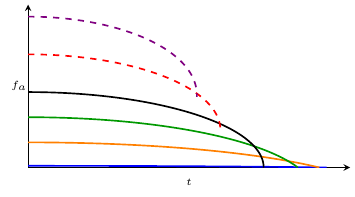}
  \includegraphics{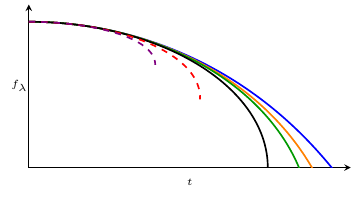}
  \caption{The left graph shows solutions to~\eqref{eqn:odeStar} and the right graph shows solutions to~\eqref{eqn:rescaledODE} numerically computed.
  The black curve illustrates the Lawson solution, the solid lines correspond to capillary cones with angle $\theta < \tfrac{\pi}{2}$, and the dashed lines are solutions to the ODE that start above the Lawson height and do not generate geometric solutions to the capillary problem.
  The blue curve is close, after rescaling, to the normalized solution of the linear one-phase problem.
  We see that for greater $a$, the root of $f_a$ moves inward and each graph $f_a$ crosses lower graphs once.
  On the $\lambda$ side, as $\lambda$ interpolates from 1 to 0, the functions $f_\lambda$ foliate the space between the solutions corresponding to the Lawson cone and one-phase limit.
  }
  \label{fig:fams-n7k1}
\end{figure}

We first derive a general comparison result for solutions $f_i = f_{\lambda_i}$ of the rescaled equation~\eqref{eqn:rescaledODE} with parameters $\lambda_i$.
Throughout this section we fix $n,k$ and follow the notation $\alpha := \frac{k-1}{n-2} \in [0,1)$, $A(t) := t - \alpha t^{-1}$, and $h_f := f - Af'$.
We also let $\psi(t) := \sqrt{|t^2 - \alpha|}$, so $\frac{\psi'}{\psi} = \frac{1}{A}$ for $t \in [0,1] \setminus \{ \sqrt{\frac{k-1}{n-2}} \}$.
Then, the function $h_f = f - \frac{\psi}{\psi'} f'$ satisfies
\begin{equation}\label{eqn:integrating-factor-f-psi}
\left( \frac{f}{\psi} \right)' = \frac{f' - \frac{\psi'}{\psi} f}{\psi} = - \frac{h_f}{A \psi}.
\end{equation}
Using the form $\cL f = (1-t^2) f'' + (f-tf') + (n-2) h_f$, we compute for any function $f$ that
\begin{equation}\label{eqn:h'-Lf-general}
    (1-t^2) h_f' = A ((n-1) h_f - \cL f) - \alpha(1-\alpha) t^{-2} f'.
\end{equation}
In particular, for a function $f$ satisfying~\eqref{eqn:rescaledODE}, this becomes
\begin{equation}\label{eqn:h'-Lf-particular}
    h'_f = A \left( \frac{n-1}{1-t^2} + (n-2) \frac{\lambda (f')^2}{1+\lambda f^2} \right) h_f - \frac{\alpha (1-\alpha)}{t^2 (1-t^2)} f'.
\end{equation}
We will suppress $h_f$ to $h$ when there is no ambiguity.

\begin{lemma}\label{lemma:comparison-of-solutions}
Given a point $t_0 \geq \sqrt{\frac{k-1}{n-2}}$, let $f_1, f_2 \in C^2( [t_0, T) )$ be non-negative functions such that $f_i$ satisfies~\eqref{eqn:rescaledODE} with parameter $\lambda_i$, where $\lambda_2 \geq \lambda_1$ and $T \leq 1$.
Consider the condition
\begin{equation}\label{eqn:the-condition}\tag{$\mathbf{C}$}
    h_{f_2} \geq h_{f_1}, \qquad 0 \leq f_2 \leq f_1, \qquad f'_2 \leq f'_1 \leq 0.
\end{equation}
If the condition~\eqref{eqn:the-condition} is satisfied at $t_0$, then either $\lambda_1 = \lambda_2$ and $f_1 = f_2$ identically, or we have
\[
h_{f_2} > h_{f_1}, \qquad 0 \leq f_2 \leq f_1, \qquad f'_2 < f'_1 \leq 0.
\]
In particular, $f_1 - f_2$ is a strictly positive, strictly increasing function on $(t_0, T]$.
If $f_2$ has maximal existence interval $[t_0, T]$ with $f_2(T) \geq 0$, then $f_1$ exists and remains positive on $[t_0, T]$.

Notably, any two distinct positive solutions of equation~\eqref{eqn:odeStar} intersect at most once over $[ \sqrt{\frac{k-1}{n-2}}, 1)$.
\end{lemma}
\begin{proof}
We abbreviate the functions $h_i := h_{f_i}$ and work in the maximal interval $[t_0, T]$ where $f_2$ exists and is non-negative, with $T \leq 1$.
First, the functions $f_i$ satisfy $f'_i < 0$ while positive; if there was a first instance $t_*$ where $f'_i(t_*) = 0$ while $f_i(t_*) > 0$, then $f''_i(t_*) \geq 0$.
On the other hand, the equation~\eqref{eqn:rescaledODE} would yield 
\[
(1-t_*^2) f''_i(t_*) = - (n-1) f_i(t_*) < 0,
\]
at $t_*$, a contradiction.
Therefore, $f_i'<0$ and $h_i = f_i - A f'_i > 0$ while the $f_i$ are positive.

At any point $t_*$ in the interval $[t_0,T)$ where both solutions are positive, we use $\lambda_2 \geq \lambda_1$ to see \begin{equation}\label{eqn:coefficient-ordering}
    0 \leq f_2 \leq f_1, \qquad f'_2 \leq f'_1 \leq 0 \qquad \implies \qquad  \frac{\lambda_2 (f'_2)^2}{1 + \lambda_2 f_2^2} \geq \frac{\lambda_2 (f'_1)^2}{1 + \lambda_2 f_1^2} \geq \frac{\lambda_1 (f'_1)^2}{1 + \lambda_1 f_1^2}
    \end{equation}
    and $- f'_2 \geq - f'_1 \geq 0$.
    At least one of the above inequalities must be strict, otherwise $\lambda_1 = \lambda_2$ and ODE uniqueness would force $f_1 \equiv f_2$.
    If $h_2 \geq h_1$ as well, then we use $A \geq 0$ for $t \geq \sqrt{\frac{k-1}{n-2}}$ to obtain
    \begin{equation}\label{eqn:h2'-prime}
        \begin{split}
        h'_2(t_*) &= A \left( \frac{n-1}{1-t_*^2} + (n-2) \frac{\lambda_2 (f'_2)^2}{1 + \lambda_2 f_2^2} \right) h_2 - \frac{\alpha (1-\alpha)}{t_*^2 (1-t_*^2)} f'_2(t_*) \\
        &> A \left( \frac{n-1}{1-t_*^2} + (n-2) \frac{\lambda_1 (f'_1)^2}{1 + \lambda_1 f_1^2} \right) h_1 - \frac{\alpha (1-\alpha)}{t_*^2 (1-t_*^2)} f'_1(t_*) 
        \end{split}
    \end{equation}
    due to~\eqref{eqn:h'-Lf-particular}, whereby $h'_2(t_*) > h'_1(t_*)$ and $h_2 > h_1$ on $(t_*, t_* + \delta)$ for some $\delta>0$.
    Next, using the relation~\eqref{eqn:integrating-factor-f-psi} for $i=1,2$, we have
    \[
    \left( \frac{f_i}{\psi} \right)' = - \frac{h_i}{A \psi}, \qquad \implies \qquad \left( \frac{f_2 - f_1}{\psi} \right)' = - \frac{h_2 - h_1}{A \psi} \qquad \text{on } \; [t_*, t_* + \delta),
    \]
    where $\psi = \sqrt{t^2 - \frac{k-1}{n-2}}$.
    Since $h_2 > h_1$ and $A > 0$, we deduce that $\left( \frac{f_2 - f_1}{\psi} \right)' < 0$, so $\frac{f_2 -f_1}{\psi}$ is a strictly decreasing function on $[t_*, t_* + \delta)$ with $\frac{f_2 - f_1}{\psi}(t_*) \leq 0$.
    On $(t_*, t_* + \delta)$, we now have that $f_2 - f_1 < 0$ while $\psi$ is strictly increasing and $\frac{f_2 - f_1}{\psi}$ is strictly decreasing, so $f_2 - f_1$ is a strictly decreasing negative function.
    Consequently, if the condition~\eqref{eqn:the-condition} is satisfied at a point $t_*$, then it holds strictly on some interval $(t_*, t_* +  \delta)$, meaning that
    \[
    h_2 > h_1, \qquad 0 < f_2 < f_1, \qquad f'_2 < f'_1 < 0 \qquad \text{on } \; (t_*, t_* + \delta).
    \]
    Applying this argument for $t_* = t_0$ shows that~\eqref{eqn:the-condition} holds strictly on $(t_0, t_0 + \delta)$.
    If it failed to hold with strict inequality on $(t_0, T]$, there would exist a first instance $t_1$ where at least one of the equalities is satisfied.
    These equalities would require, respectively,
    \[
    h'_2(t_1) \leq h'_1(t_1), \qquad f''_2(t_1) \geq f''_1(t_1), \qquad \text{or} \qquad f'_2(t_1) \geq f'_1(t_1).
    \]
    The computation~\eqref{eqn:h2'-prime} shows that $h'_2(t_1) > h'_1(t_1)$ at such a point, ruling out the first case.
    Therefore, $h_2(t_1) > h_1(t_1)$, meaning that
    \[
    (f_2(t_1) - f_1(t_1)) - A(t_1) \cdot (f '_2(t_1) - f'_1(t_1)) > 0.
    \]
    If $(f'_2 - f'_1)(t_1) = 0$, then $(h_2 - h_1)(t_1) = (f_2 - f_1)(t_1) \leq 0$ would give a contradiction.
    Similarly, if $(f_2 - f_1)(t_1) = 0$, then $(h_2 - h_1)(t_1) = - A(t_1) \cdot (f'_2 - f'_1)(t_1) \leq 0$ again produces a contradiction.
    We conclude that~\eqref{eqn:the-condition} holds with strict inequality on $(t_0, T]$, as claimed.
    If $f_2$ exists and is positive over $[t_0, T]$, then the above argument shows that $|f'_1| < |f'_2|$ for all $t \in (t_0, T]$, meaning that $f'_1$ cannot blow up before $f'_2$.
    Consequently, the function $f_1$ can be extended up to $t=T$, with $(f_1 - f_2)(T) > (f_1 - f_2)(t_0) = 0$.

    As a consequence of this result, we consider any two distinct positive solutions of equation~\eqref{eqn:odeStar} intersecting at some $t_0 \geq \sqrt{\frac{k-1}{n-2}}$.
    By symmetry, we suppose that $f_1(t_0) = f_2(t_0)$ and $f'_2(t_0) < f'_1(t_0) < 0$, so $h_2(t_0) > h_1(t_0)$.
    Thus, the condition~\eqref{eqn:the-condition} holds strictly, showing that $f_1 > f_2$ and $f'_1 > f'_2$ on $[t_0, T]$, meaning that the solutions remain disjoint.
\end{proof}

To prove the various comparison properties illustrated in Figure~\ref{fig:fams-n7k1}, we will need to understand the behavior of various quantities near $0$ to high order, depending on the parameters $\lambda, a$.
\begin{lemma}\label{lemma:values-at-zero}
For any $\lambda \in [0,\infty)$ and $a>0$, let $f_{\lambda, a}$ be the unique even solution of equation~\eqref{eqn:rescaledODE} with scaling parameter $\lambda$ and initial data $f_{\lambda, a}(0) = a$ and $f'_{\lambda,a}(0) = 0$.
Then, for $h_{\lambda,a} = f_{\lambda,a} - A_{n,k} f'_{\lambda, a}$ we have the following values at $0$:
\begin{align}
    f''_{\lambda, a}(0) &= - \frac{n-1}{k} a, \qquad \qquad h_{\lambda, a}(0) = \frac{n-k-1}{k(n-2)} a > 0, \\
    f_{\lambda, a}^{(4)}(0) &= -\frac{3(n-1)a}{k^3(k+2)}\big(k^2(n+1)+2(n-1)(n-k-1)\big)
   + \frac{6(n-1)^2(n-k-1)}{k^3(k+2)}\cdot\frac{a}{1+\lambda a^2}, \label{eqn:fourth-derivative}
\end{align}
\end{lemma}
\begin{proof}
    Since $f_{\lambda, a}'(0) = 0$, equation~\eqref{eqn:rescaledODE} together with $\lim_{t \to 0} (t^{-1} f_{\lambda, a}') = f_{\lambda, a}''(0)$ implies that
    \[
    f_{\lambda, a}''(0) + (n-1) f_{\lambda, a}(0) + (k-1) f_{\lambda, a}''(0) = 0 \implies f_{\lambda, a}''(0) = - \tfrac{n-1}{k} f_{\lambda, a}(0),
    \]
    whereby $f''_{\lambda, a}(0) = - \frac{n-1}{k}a$.
    For $h_{\lambda,a} =f_{\lambda, a} - A f'_{\lambda, a}$, we compute
     \begin{align*}
        h_{\lambda, a}(0) &= f_{\lambda, a}(0) - \lim_{t \to 0} (t^2-\alpha) (t^{-1} f_{\lambda, a}') = f_{\lambda, a}(0) + \alpha f_{\lambda, a}''(0) \\
        &= f_{\lambda, a}(0) - \frac{n-1}{k} \frac{k-1}{n-2} f_{\lambda, a}(0),
    \end{align*}
    so $h_{\lambda, a}(0) = \frac{n-k-1}{k(n-2)} a > 0$ for $k \leq n-2$.
    Also, $f^{(3)}_{\lambda, a}(0) = 0$ by evenness.
    To obtain $f^{(4)}_{\lambda,a}(0)$, we differentiate~\eqref{eqn:rescaledODE} twice, obtaining
    \[
    0 = f^{(4)} (0)- 3f''(0) + (n-2) \left[ 2 \frac{\lambda  f''(0)^2 }{1 + \lambda f(0)^2} \,  h(0) + h''(0) \right].
    \]
    Here, we used the fact that $f$ and $h$ are even functions, and computed
    \[
    \left( (1-t^2) \frac{\lambda (f')^2}{1 + \lambda f^2} \right)''(0) = \left(\frac{\lambda (f')^2}{1 + \lambda f^2} \right)''(0) = 2 \frac{\lambda  f''(0)^2 }{1 + \lambda f(0)^2}.
    \]
    Moreover, $h_{\lambda, a}'(0) = 0$ by evenness, and
    \[
    h''_{\lambda , a}(0) = f''(0) - 2 f''(0) + \tfrac{1}{3} \alpha f^{(4)}(0) = \tfrac{1}{3} \alpha f_{\lambda, a}^{(4)}(0) - f_{\lambda, a}''(0).
    \]
    Substituting the values of $f''_{\lambda, a}(0)$ and $h_{\lambda, a}(0)$ into these computations proves the formulas.
\end{proof}

\begin{lemma}\label{lemma:general-behavior-prep}
    The unique even solution $f_{\lambda, a}$ of equation~\eqref{eqn:rescaledODE} with scaling parameter $\lambda \in \bR_+$ and initial data $f_{\lambda, a}(0) = a > 0$ and $f'_{\lambda, a}(0) = 0$ exists for finite time $[0,b_{\lambda, a})$ with $b_{\lambda,a} \leq 1$.
    The function $f_{\lambda, a}$ remains strictly decreasing on its maximal existence interval $[0, b_{\lambda, a}]$ and has at most one positive root $t_{\lambda, a}$, in which case $t_{\lambda,a} > \sqrt{\frac{k-1}{n-2}}$.
    Moreover, $h = f_{\lambda,a} - A f'_{\lambda,a} > 0$ everywhere.
\end{lemma}
\begin{proof}
Working with the rescaled function $f = \sqrt{\lambda} f_{\lambda,a}$, which solve~\eqref{eqn:odeStar}, we may set $\lambda=1$.
We suppress the dependence on $a$ in what follows and denote by $f, t, b$ the function $f_{\lambda,a}$ and points $b_{\lambda,a}, t_{\lambda, a}$ introduced above.
Since $f(0) = a > 0$ and $f'(0) = 0$, we have $h>0$ for small $t>0$.
We will make repeated use of the fact that $f \equiv 0$ is a solution of the ODE~\eqref{eqn:odeStar}, therefore ODE uniqueness forces any non-trivial solution $f$ to have $f'(t_*) \neq 0$ at any point where $f(t_*) = 0$.
    Using the definition of $f$ and equation~\eqref{eqn:h'-Lf-general} likewise shows that $f'(t_*) \neq 0$ whenever $f(t_*) = 0$, otherwise $f'(t_*) = f(t_*) = 0$ would force $f \equiv 0$.
    Combining these steps shows that the same is true for the function $\frac{f}{\psi}$ from~\eqref{eqn:integrating-factor-f-psi}.
    As a result, ODE uniqueness implies that $h(t_*) = h'(t_*) = 0$ cannot occur at some point $t_*$, unless $f \equiv 0$.

    We will show that $h>0$ in $[0,b]$.
    Let $t_0$ be the first zero of $f$, if it exists, so $f(t_0) = 0$ and $f'(t_0) < 0$.
    Lemma~\ref{lemma:values-at-zero} gives $f''(0) = - \frac{n-1}{k} < 0$, so $f'' < 0$ and $f' <0$ for small $t>0$.
    By the argument of Lemma~\ref{lemma:comparison-of-solutions}, we obtain $f'<0$ on $(0,t_0]$.
    We use this to prove that $h>0$ on $[0, t_0]$.
    Since $h(0) = \frac{n-k-1}{k(n-2)} a > 0$ from Lemma~\ref{lemma:values-at-zero}, we know that $h>0$ for small $t>0$.
    Therefore, if $h>0$ fails in $[0,b)$, there is a first zero $t_1$ with $h(t_1) = 0$ and $h'(t_1) < 0$.
    Using equation~\eqref{eqn:h'-Lf-general} and the definition of $h$, we obtain
    \begin{equation}\label{eqn:equalities-at-t1}
    f(t_1) = A(t_1) f'(t_1), \qquad h'(t_1) = - \frac{\alpha(1-\alpha)}{t_1^2(1-t_1^2)} f'(t_1).
    \end{equation}
    Note that $\alpha = \frac{k-1}{n-2} \in [0,1)$ for $1 \leq k \leq n-2$, so $\alpha(1-\alpha) \geq 0$.
    If $\alpha = 0$, then $h'(t_1) = 0$ would force $f \equiv 0$, as discussed above; therefore, $h'(t_1) < 0$, $\alpha \in (0,1)$ and $f'(t_1) = - \frac{t_1^2(1-t_1^2)}{\alpha(1-\alpha)} h'(t_1) > 0$.
    Also, $\frac{f(t_1)}{A(t_1)} = f'(t_1) > 0$.
    Since $f'<0$ on $(0,t_0]$ was proved above, we must have $t_1 > t_0$, so $h>0$ on $[0,t_0]$ unless $t_1 = t_0 = \sqrt{\alpha}$.
    Since $\frac{f}{\psi} > 0$ on $(0,t_0)$ and $\frac{f}{\psi}(t_0) = 0$, we have $(\frac{f}{\psi})'(t_0) \leq 0$, so $\frac{h}{A \psi}(t_0) = - (\frac{f}{\psi})'(t_0) \geq 0$.
    For $t_0 \neq \sqrt{\alpha}$, the above argument gives $\frac{h}{\psi}(t_0) > 0$, so $\frac{h}{A \psi}(t_0) \geq 0$ implies $A(t_0) \geq 0$ and $t_0 > \sqrt{\alpha}$.
    For $t_0 > \sqrt{\alpha}$, we can argue as in Lemma~\ref{lemma:comparison-of-solutions} to show that the properties 
    \[
    h > 0 \qquad \text{and} \qquad f' < 0
    \]
    are preserved on $[t_0, b)$.
    It follows that $(\frac{f}{\psi})' = - \frac{h}{A \psi} < 0$ and $\frac{f}{\psi}$ is a strictly decreasing function on $(t_0, b)$.
    In particular, $f(t_0) = 0$ makes $f<0$ on $(t_0, b)$, so $f$ has at most one zero in $(0, b)$.
    On the interval $(t_0, b)$, we now have that $f<0$, $\psi$ is strictly increasing, and $\frac{f}{\psi}$ is strictly decreasing, so $f$ is a strictly decreasing negative function.
    Using $f'<0$ on $(0, t_0]$, we conclude that $f'<0$ on $(0,b)$, so $f$ remains strictly decreasing on $[0,b]$.
    Note that the conclusions $f'<0$ and $h>0$ on $(0,b)$ do not require the existence of a positive root $t_0$, so they are valid for any solution $f$.
\end{proof}

For any function $f_{\lambda}$ satisfying the rescaled equation~\eqref{eqn:rescaledODE}, the function $f := \sqrt{\lambda} f_{\lambda}$ is a solution of the original equation~\eqref{eqn:odeStar}, so the results of Lemma~\ref{lemma:general-behavior-prep} imply that $f_{\lambda} - A f'_{\lambda} > 0$ on the domain of definition $(0,b_{\lambda})$ of $f_{\lambda}$.
This also implies that
\begin{equation}\label{eqn:non-positive}
    \cL f_{\lambda} = - \lambda (n-2) (1-t^2) \frac{(f'_{\lambda})^2}{1 + \lambda f_{\lambda}^2} (f_{\lambda} - A f'_{\lambda}) < 0 \qquad \text{on } \; (0,b_{\lambda}).
\end{equation}
We will use this fact repeatedly to show that the solutions of the rescaled equation~\eqref{eqn:rescaledODE} as $\lambda \uparrow \infty$ and as $\lambda \downarrow 0$ provide lower and upper barriers for the solution of~\eqref{eqn:rescaledODE} with the same initial data at $0$.
As $\lambda \downarrow 0$ in~\eqref{eqn:rescaledODE}, the limit equation becomes $\cL_{n,k} f = 0$, so the resulting function $u_0 = \rho f_0(t)$ is harmonic.
As $\lambda \to \infty$, the limit equation becomes $\on{div} ( \frac{\nabla u_{\infty}}{|\nabla u_{\infty}|}) = 0$, which is the $1$-Laplace equation.
Then, $u_{\infty} = \rho f_{\infty} (t)$ is a function of least gradient ($1$-harmonic), with the significance that all its regular level sets in $\bR^n$ are minimal hypersurfaces.
Here, $f_\infty$ solves with 
\begin{equation}\label{eqn:f_infty}
f_{\infty}^2 \Bigl( (1-t^2) f_{\infty}'' + (n-1) (f_{\infty}-t f_{\infty}') \Bigr) + (n-2)(1-t^2) (f_{\infty}')^2 (f_{\infty} - t f_{\infty}') = 0 
\end{equation}
such that $f_\infty(0) = 1$ and $f_\infty'(0) = 0$, which is the limit of solutions to equation~\eqref{eqn:rescaledODE} as $\lambda \uparrow \infty$.

We will now establish comparison results for the solutions of equation~\eqref{eqn:rescaledODE} by recasting the problem as a first-order ODE system.
For any function $f$, we introduce the Riccati variable $q := \frac{f'}{f}$ and let $w := \frac{\lambda f^2}{1 + \lambda f^2}$.
Then, equation~\eqref{eqn:rescaledODE} assumes the form
\begin{empheq}[left=\empheqlbrace]{align}
q' &= - q^2 - \dfrac{1-tq}{1-t^2} - \dfrac{n-2}{1-t^2}\bigl(1+(1-t^2)q^2 w\bigr)\bigl(1-A(t)q\bigr)
\label{eqn:system-ode-1} \\
w' &= 2w(1-w)q
\label{eqn:system-ode-2}
\end{empheq}
The properties $f(0) = a > 0$ and $f'(0) = 0$ become $q(0) = 0$ and $w(0) = w_0 \in (0,1)$.
Moreover, $q$ is an odd function while $w$ is even, so $q^{(2m)}(0) = w^{(2m+1)}(0) = 0$.
Applying the computations of Lemma~\ref{lemma:values-at-zero}, we obtain $q'(0) = - \frac{n-1}{k}$ and $w''(0) = - \frac{2(n-1)}{k} w_0(1-w_0)$, while
\begin{equation}\label{eqn:q-w-values-at-zero}
    \begin{aligned}
        q^{(3)}(0) &= - \frac{6(n-1)}{k^3 (k+2)} \bigl( k^2 n + (n-1) ( (n-k-1) w_0 + k) \bigr), \\
        w^{(4)}(0) &= \frac{12 (n-1)^2  w_0(1 - w_0)}{k^3 (k+2)} \Bigl(  k(k+2)(1 - 2 w_0) - \Bigl[ k^2 \frac{n}{n-1} + (n-k-1) w_0 + k \Bigr] \Bigr).
    \end{aligned}
\end{equation}
We recall from Lemma~\ref{lemma:general-behavior-prep} that any solution of the general equation~\eqref{eqn:rescaledODE} has $q < 0$ and $h = f_{\lambda} - Af_{\lambda}' > 0$; consequently, $1 - A(t)q = \frac{h}{f_{\lambda}} > 0$ while $f_{\lambda}$ remains positive.
We therefore obtain the following simple comparison property.
\begin{lemma}\label{lemma:simple-system-comparison}
    Let $(q_1, w_1)$ and $(q_2, w_2)$ be analytic solutions of equation~\eqref{eqn:system-ode-1} on an interval $[t_1,t_2]$, with $q_1(t_1) = q_2(t_1) < 0$ and $w_i(t_1) \in (0,1)$.
    If $w_2 - w_1$ has a constant sign on $(t_1, t_2]$, then
    \begin{equation}\label{eqn:constant-sign-q1-q2}
        \on{sgn} ( q_1 - q_2) = \on{sgn} ( w_2 - w_1) \qquad \text{on } \; (t_1,t_2].
    \end{equation}
\end{lemma}
\begin{proof}
    By symmetry, we may assume that $w_2 - w_1 > 0$ on $(t_1, t_2]$.
    We consider a point $t_*$ where $q_1(t_*) = q_2(t_*)$.
    Subtracting equations~\eqref{eqn:system-ode-1} for the $q_i$ term-wise, we obtain
    \[
    (q_1 - q_2)' (t_*) = \tfrac{n-2}{1- t_*^2} (1 - A(t_*) q) (1 + (1-t_*^2) q^2 (w_2 - w_1) ) > 0
    \]
    by the positivity of $w_2 - w_1$.
    Since $q_1(t_1) = q_2(t_1) < 0$, we may apply this discussion at $t_* = t_1$ to deduce that $(q_1 - q_2)'(t_1) \geq 0$, with equality only if $w_2 = w_1$.
    In that case, we may use $w_2 - w_1 > 0$ on $(t_1, t_2]$ to find the first $m>0$ such that $(w_2 - w_1)^{(i)}(t_1) = 0$ for $0 \leq i \leq m-1$ and $(w_2 - w_1)^{(m)}(t_1) > 0$.
    The above relation implies that $(q_1 - q_2)^{(i+1)}(t_1)= 0$ for $0 \leq i \leq m-1$, and 
    \[
    (q_1 - q_2)^{(m+1)}(t_1) = \tfrac{n-2}{1 - t_1^2} ( 1 - A(t_1) q) (1 + (1- t_1^2) q^2 (w_2 - w_1)^{(m)}(t_1)) > 0
    \]
    whereby $(q_1 - q_2)' > 0$, and $q_1 -q_2 > 0$ for small $t - t_1 > 0$.
    If there was a first instance $t_* > t_1$ where $(q_1 - q_2)(t_*) = 0$, we would encounter $(q_1 - q_2)' (t_*) \leq 0$.
    On the other hand, the above discussion forces $(q_1 - q_2)'(t_*) > 0$, producing a contradiction.
    Therefore, $q_1 - q_2$ must be strictly positive on $(t_1, t_2]$, completing the proof.
\end{proof}

\begin{corollary}\label{cor:comparison-before-sqrtalpha}
    Consider functions $f_1, f_2$ defined on $[0,T]$ such that $f_i$ satisfies~\eqref{eqn:rescaledODE} with parameter $\lambda_i$, where $\lambda_2 \geq \lambda_1$, and $f_i(0) = a, f'_i(0) = 0$.
    If $\frac{f_2}{f_1} \geq \sqrt{\frac{\lambda_1}{\lambda_2}}$ holds on $[0,T]$, then either $\sqrt{\lambda_1} f_1 = \sqrt{\lambda_2} f_2$ or the function $\frac{f_2}{f_1}$ is strictly decreasing on $[0,T]$, with
    \begin{equation}\label{eqn:hf-comparison-before-alpha}
    f_2 - A f'_2 < f_1 - A f'_1 \qquad \text{on } \; [0, \min \{ \sqrt{\alpha}, T \} ].
    \end{equation}
    Moreover, $|f'_1| < + \infty$ holds as long as $|f'_2|<+\infty$ and $\frac{f_2}{f_1} \geq \sqrt{\frac{\lambda_1}{\lambda_2}}$.
\end{corollary}
\begin{proof}
    The functions $f_i$ are solutions of the ODE system~\eqref{eqn:system-ode-1}~--~\eqref{eqn:system-ode-2} with $w_i = \frac{\lambda_i f_i^2}{1 + \lambda_i f_i^2} = 1 - \frac{1}{1 + \lambda_i f_i^2}$.
    The assumption that $\lambda_2 f_2^2 \geq \lambda_1 f_1^2$ implies $w_2 \geq w_1$ on $[0,T]$, whereby Lemma~\ref{lemma:simple-system-comparison} shows that $( \frac{f_2}{f_1})' = \frac{f_2}{f_1} (q_2 - q_1)$ is decreasing on this interval, where $q_i := \frac{f'_i}{f_i}$.
    The functions $\sqrt{\lambda_i} f_i$ both satisfy the equation~\eqref{eqn:odeStar}, whereby ODE uniqueness forces $( \frac{f_2}{f_1})' < 0$ unless $f_1 \equiv f_2$.
    
    When the $\sqrt{\lambda_i} f_i$ are distinct, we conclude that $\frac{f_2}{f_1}$ is strictly decreasing on $[0,T]$; since $\frac{f_2}{f_1}(0) = 1$ and $f'_i < 0$, the function $f_1 - f_2$ is positive and $f_2 > \sqrt{\frac{\lambda_1}{\lambda_2}}f_1 > 0$, whereby $f'_2 < \sqrt{\frac{\lambda_1}{\lambda_2}} f'_1 < 0$.
    Thus, $|f'_1|$ remains bounded while $|f'_2|$ does, and $f_1$ has no derivative blow-up before $f_2$.

    Finally, the property~\eqref{eqn:hf-comparison-before-alpha} follows from writing
    \[
    f_2 - A f'_2 = f_2 ( 1 - A q_2) < f_1 (1 - A q_1) = f_1 - A f'_1
    \]
    due to $1 - A q_i > 0$ and $A < 0$ on $[0, \sqrt{\alpha}]$, whereby $q_2 < q_1$ implies $1 - A q_2 < 1 - A q_1$.
\end{proof}

\begin{lemma}\label{lemma:derivative-blow-up}
    The even solution $f_{\lambda,a}$ of equation~\eqref{eqn:rescaledODE} with $f_{\lambda,a}(0) = a > 0$ has derivative blow-up at a point $b_{\lambda,a} < 1$, where $|f_{\lambda,a}(b_{\lambda,a})| < + \infty$ and $f'_{\lambda,a}(b_{\lambda,a}) = - \infty$.
\end{lemma}
\begin{proof}
    As in Lemma~\ref{lemma:general-behavior-prep}, we may reduce the result to $\lambda=1$ and suppress the dependence on $\lambda,a$.
    Consider the hypergeometric function $f_0(t) := {}_2F_1( \frac{n-1}{2}, - \frac{1}{2} ; \frac{k}{2} ; t^2)$, which satisfies $\cL_{n,k} f_0 = 0$ for $f_0$ as in~\eqref{eqn:legendre-form}.
    Let $t_{n,k}$ be the zero of $f_0$, which has $t_{n,k} > \sqrt{\frac{k-1}{n-2}}$ by Lemma~\ref{lemma:general-behavior-prep}.
    Since $f_0$ is the solution of~\eqref{eqn:rescaledODE} with $\lambda=0$, Corollary~\ref{cor:comparison-before-sqrtalpha} implies that $\frac{f}{f_0}$ is strictly decreasing while $f$ is positive.
    Therefore, either $f$ has derivative blow-up while non-negative, or it must reach zero at some point $t_0 \leq t_{n,k}$, with $t_0 > \sqrt{\frac{k-1}{n-2}}$ and $f'(t_0) \in (-\infty, 0)$.
    The desired result is automatic in the first case, so it remains to prove the desired property for a solution $f$ of~\eqref{eqn:odeStar} with a zero at $t_0 > \sqrt{\frac{k-1}{n-2}}$.

    \smallskip \noindent \textbf{Step 1:}
    We claim that there exists a $c(n,k) \in (0,1)$ such that every such solution satisfies
    \begin{equation}\label{eqn:nonlinear-term-bound}
        f - Af' \geq c |Af'| \quad \text{for } \; t \geq t_0.
    \end{equation}
    We equivalently seek to prove that the modified quantity $Q_c := f - (1-c) A f'$ satisfies $Q_c \geq 0$.
    First, when $k \geq 2$, we observe that
    \begin{align*}
        Q'_c = f' - (1-c) A' f' - (1-c) A f'' > (-f') \bigl( (1-c) A' -1 \bigr) 
    \end{align*}
    since $A>0$ and $f$ is strictly concave by Lemma~\ref{lemma:general-behavior-prep}.
    Since $A' = 1 + \frac{k-1}{n-2} t^{-2} \geq \frac{n+k-3}{n-2}$, we may choose $c(n,k) = \frac{1}{2} \frac{k-1}{n+k-3}$ to ensure that
    \[
    (1-c) A' -1 \geq \tfrac{2n+k-5}{2(n+k-3)} \cdot \tfrac{n+k-3}{n-2} - 1 = \tfrac{k-1}{2(n-2)}, \qquad Q'_c > \tfrac{k-1}{2(n-2)} |f'| > 0,
    \]
    which ensures that $Q_c$ is strictly increasing.
    Since $Q_c(t_0) = (1-c) |Af'|(t_0)> 0$, we conclude that $Q_c > 0$ and obtain~\eqref{eqn:nonlinear-term-bound}.

    It remains to consider $k=1$.
    Recall $f_{\infty}$ as defined in equation~\eqref{eqn:f_infty} to be the solution of equation~\eqref{eqn:rescaledODE} as $\lambda \uparrow \infty$.
    The function $f_{\infty}$ remains defined and positive on some interval $[0, \tilde{b}_{n}]$, whereby Corollary~\ref{cor:comparison-before-sqrtalpha} shows that $\frac{f_{\infty}}{f}$ is strictly decreasing on $[0, \tilde{b}_n]$.
    In particular, the zero of $f$ satisfies $t_0 > \tilde{b}_n$.
    For $t \geq t_0$, we consider $r := \frac{-f}{t (-f)' } \geq 0$, for which we compute
    \[
    r' = \frac{1-r}{t} - r \frac{f''}{f'} \leq \frac{1-r}{t}, \qquad \text{due to } \; \frac{f''}{f'} > 0.
    \]
    Equivalently, $ \bigl( t (1-r) \bigr)' \geq 0$, and since $r(t_0) =0 $, we obtain
    \[
    t(1-r(t)) \geq t_0(1-r(t_0)) = t_0 \implies r(t) \leq 1 - \tfrac{t_0}{t} \leq 1 - t_0 < 1 - \tilde{b}_n.
    \]
    Since $A_{n,1}(t) = t$, we conclude that~\eqref{eqn:nonlinear-term-bound} is satisfied with $c(n,1) = \tilde{b}_n > 0$, since
    \[
    f > (1- \tilde{b}_n) Af' \implies f - Af' > \tilde{b}_n |Af'|.
    \]
    This establishes the desired property~\eqref{eqn:nonlinear-term-bound} for all $n,k$.

    \smallskip \noindent \textbf{Step 2:}
    Using the lower bound~\eqref{eqn:nonlinear-term-bound}, we now prove the desired blow-up behavior of $f$.
    Since $f$ is concave, the secant line bound yields
    \[
    f(t) \leq f(t_0) + f'(t_0) (t - t_0) = f'(t_0) (t -t_0) \implies |f'(t_0)| \geq \tfrac{a}{t_0}.
    \]
    Using $t_0 \leq t_{n,k}$ and $f'(t) < - \frac{a}{t_0}$ for $t>t_0$, we may therefore take 
    \[
    \hat{t}_1 := \max \Bigl\{ 1 - \frac{1}{5n} , \frac{1 + t_{n,k}}{2} \Bigr\} \in (t_{n,k}, 1)
    \]
    to see that either $f$ is undefined due to derivative blow-up while remaining positive, or $f(\hat{t}_1) < - C(n,k) a$.
    Consequently, $\frac{f^2}{1+f^2} > C(n,k,a)$ for $t > \hat{t}_1$, and $A(t)> c_2(n,k)$ on $[\hat{t}_1 ,1]$, since $\hat{t}_1 > t_{n,k} > \sqrt{\frac{k-1}{n-2}}$.
    Finally, the inequality~\eqref{eqn:nonlinear-term-bound} allows us to bound $Aq-1 \geq c Aq \geq Cq$.
    Combining these bounds in~\eqref{eqn:system-ode-1}, together with $w = \frac{f^2}{1+f^2} < 1$ and $(n-2) A = (n-2) t - (k-1) t^{-1}$, we obtain
    \begin{equation}\label{eqn:q-ODE-bound}
    \begin{split}
    q' &= -q^2+\frac{tq-1}{1-t^2}
    +\frac{n-2}{1-t^2}(Aq-1) + (n-2)q^2w(Aq-1) \\
    &\geq - (n-1) \left( q^2 + \frac{1}{1-t^2} \right) + \frac{(n-1) t - (k-1) t^{-1}}{1-t^2} q + C_a q^3.
    \end{split}
    \end{equation}
    Next, we prove that $q \to \infty$ as $t \uparrow 1$: since $Aq-1 > 0$ from Lemma~\ref{lemma:general-behavior-prep}, the inequality~\eqref{eqn:q-ODE-bound} implies that $q \geq \tilde{q}$, where $\tilde{q} = \frac{\tilde{f}'}{\tilde{f}}$ is defined by
    \[
    \tilde{q}' = - \tilde{q}^2 + \frac{t \tilde{q} - 1}{1-t^2} + \frac{n-2}{1-t^2} (A \tilde{q} - 1), \qquad \tilde{q} = \frac{\tilde{f}'}{\tilde{f}}, \qquad \cL_{n,k} \tilde{f} = 0,
    \]
    for $\tilde{f}$ the solution of the linear operator with $\tilde{f}(t_0) = 0$ and $\tilde{f}(t_0) = -1$.
    Expanding near $t=1$ shows that the kernel of the operator $\cL_{n,k}$ has two linearly independent elements with Frobenius series expansions corresponding to the indicial roots $r_{\text{reg}} = 0$ and $r_{\text{sing}} = \frac{k}{2} - \frac{n-1}{2} - (- \frac{1}{2}) = - \frac{n-k-2}{2}$.
    The first root corresponds to a solution $\tilde{f}_{\text{reg}}(t)$ that remains regular and bounded at $1$, while the second produces $\tilde{f}_{\text{sing}}(t) \sim (1-t^2)^{- \frac{n-k-2}{2}}$ near $t$ for $k<n-2$, or $\tilde{f}_{\text{sing}}(t) \sim \log(1-t^2)$ for $k=n-2$.
    Let us write $\tilde{f} = c_1 \tilde{f}_{\text{reg}}(t) + c_2 \tilde{f}_{\text{sing}}(t)$; we will show that $c_2 \neq 0$.
    Using the self-adjoint form~\eqref{eqn:legendre-form} of $\cL_{n,k}$, we may multiply the expression $\cL_{n,k} f_{\text{reg}}(t) = 0$ and integrate on any interval $[\tau,1]$ to obtain
    \[
    \int_{\tau}^1 p(t) (\tilde{f}'_{\text{reg}})^2 \, dt + (n-1) \int_{\tau}^1 \frac{p(t)}{1-t^2} \tilde{f}_{\text{reg}}^2 \, dt = - p(\tau) \tilde{f}_{\text{reg}}(\tau) \tilde{f}'_{\text{reg}}(\tau) = - \tfrac{1}{2} p(\tau) ( \tilde{f}_{\text{reg}}^2)'(\tau)
    \]
    where $p(t) = t^{k-1}(1-t^2)^{\frac{n-k}{2}}$, so the boundary term at $t=1$ vanishes due to $p(1) =0$ and $|\tilde{f}_{\text{reg}}| < \infty$.
    The left-hand side is a sum of two strictly positive terms, so $( \tilde{f}_{\text{reg}}^2)' < 0$ forces $\tilde{f}_{\text{reg}} > 0$ and $\tilde{f}'_{\text{reg}} < 0$ on $[0,1)$.
    If $c_2 = 0$, then $\tilde{f} = c_1 \tilde{f}_{\text{reg}}(t) \neq 0$ on $[0,1)$, contradicting $\tilde{f}(t_0) = 0$.
    Therefore, $c_2 \neq 0$ yields $\tilde{q} = \frac{\tilde{f}'(t)}{\tilde{f}(t)} \sim \frac{n-k-2}{2(1-t)}$, if $k<n-2$, or $\tilde{q} \sim \frac{1}{(1-t) |\log(1-t)|}$ if $k=n-2$.
    In either case, we have $\tilde{q} \to \infty$, hence also $q \to \infty$ for $t \uparrow 1$ as claimed.

    \smallskip \noindent \textbf{Step 3.} By the above discussion, we may find some point $t_1 \in ( \hat{t}_1 , 1)$ where $q ( t_1) \geq q_0 (t_1) \geq 10(C_a^{-1}+1) n$.
    Furthermore, for $t \geq \hat{t}_1 \geq 1 - \frac{1}{5n}$ and $k \leq n-2$ we may bound
    \[
    (n-1) t - (k-1) t^{-1} \geq (n-1) ( 1 - \tfrac{1}{5n}) - (n-3) (1 - \tfrac{1}{5n})^{-1} = \tfrac{14}{5(5n-1)} + \tfrac{1}{5n} + \tfrac{8}{5},
    \]
    so the expression~\eqref{eqn:q-ODE-bound} implies the bound
    \begin{align*}
    q' \; \geq \; - (n-1) \bigl( q^2 + (1-t^2)^{-1} \bigr) + \tfrac{8}{5} (1-t^2)^{-1} q + C_a q^3  \; \geq \; (1-t^2)^{-1} q +  \tfrac{1}2{ C_a q^3} \, .
    \end{align*}
    This inequality holds at $t_1$ and implies that $q'(t_1) > 0$; therefore, $q$ is strictly increasing for $t \geq t_1$, hence the above inequality remains valid for all $t \geq t_1$.
    We now claim that any function $q$ satisfying
    \begin{equation}\label{eqn:q-differential-inequality}
        q' \geq (1-t^2)^{-1} q + C q^3, \qquad \text{for some } \; C>0,
    \end{equation}
    with initial condition $q(t_1) > 0$ at some $t_1 < 1$ blows up at a point $b<1$.
    This property follows directly by rearranging the inequality~\eqref{eqn:q-differential-inequality} into
    \[
    (q^{-2})' \leq (1-t^2)^{-1} q^{-2} + C \implies q(t)^{-2} \leq \frac{1-t}{1+t} \left( \frac{1 + t_1}{1 - t_1} q(t_1)^{-2} - 2 C \int_{t_1}^t \frac{1+s}{1-s} \, ds \right)
    \]
    by using the integrating factor $\mu(t) = \exp( \int \frac{2}{1-t^2} \, dt) = \frac{1+t}{1-t}$ for the ODE.
    Using the equality $\int \frac{1+x}{1-x} \, dx = - x - 2 \log (1-x)$ to evaluate the above integral, we see that the right-hand side becomes zero at the point $b_* < 1$ that solves the equation
    \[
    t_1 - b_* - 2 \log \Bigl( \frac{1-b_*}{1 - t_1} \Bigr) = \frac{1 + t_1}{1 - t_1} \cdot \frac{1}{2 C q(t_1)^2}.
    \]
    Consequently, $q = \frac{f'}{f} \to \infty$ as $t \uparrow b < 1$ blows up strictly before $1$, as claimed.
    Computing near that point, we see that equation~\eqref{eqn:q-ODE-bound} produces dominant behavior $q' \asymp C_1 q^3$, meaning that $q$ has square-root blow-up $q \sim C_2 (t-b)^{- \frac{1}{2}}$ near $b$.
    Finally, writing $f(b) = f(t_1) \exp( \int_{t_1}^b q)$ with $\int_{t_1}^b q < + \infty$ shows that $|f(b)| < + \infty$.
    Thus, $f$ remains finite until blow-up, completing the proof.
\end{proof}

We now study the variation of solutions to the general equation~\eqref{eqn:rescaledODE} with respect to the initial height and scaling parameters.
We express the equation as
    \begin{equation}\label{eqn:express-in-form-a-lambda}
    \cL f = - (n-2) (1-t^2) \Phi(t,f,f', \lambda), \qquad \Phi(t,f,p,\lambda) := \frac{\lambda p^2}{1+ \lambda f^2}( f - Ap)
    \end{equation}
and let $f_{\lambda, a}$ be the unique locally defined even solution of equation~\eqref{eqn:rescaledODE} with parameter $\lambda$ and initial data $f_{\lambda, a}(0) = a$ and $f'_{\lambda, a}(0) = 0$, having maximal interval of definition $[0, b_{\lambda,a}]$ where $b_{\lambda,a} \leq 1$ and $f'_{\lambda,a}(b_{\lambda,a}) = - \infty$.
Expressed in the form $f'' = F(t,f,f',\lambda)$, the ODE for $f$ has $F \in C^{\infty}( (0 , 1) \times \bR \times \bR \times \bR )$ with a removable singularity at $0$ within the class of even functions.
Applying the continuous dependence theorem for the solutions of ODEs in terms of the initial conditions and parameters shows that $(\lambda,a) \mapsto b_{\lambda,a}$ is a $C^2$ map $[0,\infty)^2 \to (0,1]$.
Moreover, $(\lambda, a) \mapsto f_{\lambda, a}$ defines a $C^2$ mapping from $[0,\infty)^2$ to the space of even $C^2$ functions on $(0, \inf_{\lambda,a} b_{\lambda,a})$.

Suppressing either variable $\lambda$ or $a$ when it is kept fixed, we denote by $v_{\lambda} := \partial_{\lambda} f_{\lambda}(t)$ and $v_a := \partial_a f_a(t)$, respectively, the variational derivatives of $f_{\lambda, a}$ along the parameter $\lambda$ or the initial height variable $a$.
The functions $v_{\lambda}$ and $v_a$ are even and solve the linearized equations
\begin{equation}\label{eqn:linearized-va-vlambda}
    P_{f_{\lambda}} v_{\lambda} = - (n-2) (1-t^2) \partial_{\lambda} \Phi_{\lambda} \qquad \text{and} \qquad  P_{f_a} v_a = 0,
\end{equation}
where $P_f$ denotes the linearization of equation~\eqref{eqn:rescaledODE} at $f$, given by
\begin{equation}\label{eqn:linearization-at-f}
\begin{split}
    P_f v &:= \cL_{n,k} v + (n-2) (1-t^2) \left[ \partial_f \Phi (t,f,f', \lambda) \, v + \partial_p \Phi(t, f, f',\lambda) \, v' \right], \\
    \frac{P_f v}{1-t^2} \!&\;= v'' - \frac{t}{1-t^2} v' + (n-2) \left( \partial_p \Phi - \frac{t - \alpha t^{-1}}{1- t^2} \right) v' + \left( \frac{n-1}{1-t^2} + (n-2) \partial_f \Phi \right) v .
\end{split}
\end{equation}
The coefficients of the ODE~\eqref{eqn:linearized-va-vlambda} are therefore smooth up to $t=0$ because the singularity introduced by $A_{n,k}(t)$ is removable due to the evenness of the functions $f_{\lambda, a}$ and $v_a, v_{\lambda}$.
We differentiate the equalities of Lemma~\ref{lemma:values-at-zero} to obtain the values of $v_a, v_{\lambda}$ and their derivatives at $0$,
\begin{equation}\label{eqn:variation-field-values-at-zero}
\begin{aligned}
v_a(0) &= 1, \qquad v'_a(0) = 0, \qquad & v''_a(0) &= - \frac{n-1}{k}, \\
v_{\lambda}(0) &= v'_{\lambda}(0) = v''_{\lambda}(0) = v^{(3)}_{\lambda}(0) = 0, \qquad & v^{(4)}_{\lambda}(0) &= - \frac{ 6a^3 (n-1)^2 (n-k-1)}{k^3 (k+2) ( 1+ \lambda a^2)^2}.
\end{aligned}
\end{equation}
These computations show that for small $t>0$, we have $v'_a, v''_a < 0$, and $v^{(j)}_{\lambda} < 0$ for all $0 \leq j \leq 4$.
In fact, it will be true that for any $t>0$, we have $v_a^{(i)} < 0$ for all $1 \leq i \leq 4$ and $v_{\lambda}^{(j)} < 0$ for all $0 \leq j \leq 4$, but we will not require such a strong conclusion.
\begin{lemma}\label{lemma:relating-va-v-lambda}
The variation fields $v_a, v_{\lambda}$ at a given function $f_{\lambda,a}$ are related by
\begin{equation}\label{eqn:relating-the-variations}
2 \lambda v_{\lambda} = a v_a - f_{\lambda,a}.
\end{equation}
\end{lemma}
\begin{proof}
Since $f_{\lambda,a}$ is a solution of the rescaled equation~\eqref{eqn:rescaledODE} with parameter $\lambda$ and $f_{\lambda,a}(0) = a$, then $\sqrt{\lambda} f_{\lambda,a}$ solves the original equation~\eqref{eqn:odeStar} with initial value $\sqrt{\lambda} a$.
Consequently, $f_{\lambda,a}(t) = a f_{\lambda a^2,1}(t)$; let $\tilde{f}_{\mu}(t) := f_{\mu,1}(t)$.
Setting $\mu = \lambda a^2$, we compute $\tilde{f}_{\mu} = \frac{1}{a} f_{\lambda,a}$ and
\begin{align*}
    v_a &= \partial_a f_{\lambda,a} = \partial_a ( a f_{\lambda a^2,1}(t)) = \tilde{f}_{\mu} + a (\partial_{\mu} \tilde{f}_{\mu}) \partial_a ( \lambda a^2) =  \tfrac{1}{a} f_{\lambda,  a} + 2 \lambda a^2 (\partial_{\mu} \tilde{f}_{\mu}), \\
    v_{\lambda} &= \partial_{\lambda} f_{\lambda,a} = \partial_{\lambda} (a f_{\lambda a^2,1}(t)) = a (\partial_{\mu} \tilde{f}_{\mu}) \partial_{\lambda} ( \lambda a^2) = a^3 ( \partial_{\mu} \tilde{f}_{\mu}).
\end{align*}
Using $a (\partial_{\mu} \tilde{f}_{\mu}) = \frac{1}{a^2} v_{\lambda}$ from the second equation into the first one, we obtain $v_a = \frac{1}{a} f_{\lambda,a} + \frac{2\lambda}{a} v_{\lambda}$.
Rearranging this expression establishes the property~\eqref{lemma:relating-va-v-lambda}.
\end{proof}

The coefficients of the linearized operator $P_f$ in~\eqref{eqn:linearization-at-f} are given by
\begin{equation}\label{eqn:pg-Phi-coefficients}
\partial_f \Phi = - \frac{2 \lambda^2 f p^2}{(1+ \lambda f^2)^2} (f - Ap) + \frac{\lambda p^2}{1+ \lambda f^2}, \qquad \partial_p \Phi = \frac{2 \lambda p}{1+ \lambda f^2} (f - Ap) - \lambda A \frac{p^2}{1+\lambda f^2}.
\end{equation}
Using these computations in~\eqref{eqn:linearization-at-f}, we record the resulting expression as
\begin{equation}\label{eqn:Pf-v-linearized}
    P_f v := \cL v + (n-2) (1-t^2) \left( \frac{\lambda (2ff' - 3 A (f')^2)}{1+ \lambda f^2} v' + \frac{\lambda (f')^2 (1- \lambda f^2 + 2 \lambda A f f') }{(1+ \lambda f^2)^2} v  \right)
\end{equation}
Observe that $v_a$ is the even principal solution of the homogeneous equation $P_f v = 0$, while $v_{\lambda}$ is the even solution of the inhomogeneous equation $P_fv = - (n-2) (1-t^2) \partial_{\lambda} \Phi_{\lambda}$.
We compute
\[
\partial_{\lambda} \Phi = \frac{p^2}{(1 + \lambda f^2)^2} (f - Ap), \qquad P_{f_{\lambda}} v_{\lambda} = - (n-2)(1-t^2) \frac{(f_{\lambda}')^2}{(1+\lambda f_{\lambda}^2)^2} (f_{\lambda} - A f'_{\lambda}) < 0
\]
with strictly negative source term for $t>0$, in view of~\eqref{eqn:non-positive}.

A direct computation gives the homogeneity identity
\begin{equation}\label{eqn:phi-homogeneity-defect}
   p (\partial_p \Phi) + f ( \partial_f \Phi) - \Phi = 2 \lambda \frac{p^2}{(1 + \lambda f^2)^2} ( f - A p) = 2 \lambda \partial_{\lambda} \Phi.
\end{equation}
Therefore, using $\cL f = - (n-2) (1-t^2) \Phi$ from the nonlinear equation~\eqref{eqn:express-in-form-a-lambda}, we find that
\begin{align*}
    P_f f &= \cL f + (n-2)(1-t^2) ( p \,\partial_p \Phi + f \, \partial_f \Phi) = (n-2)(1-t^2) ( p \, \partial_p \Phi + f \, \partial_f \Phi - \Phi) \\
    &= 2 \lambda (n-2)(1-t^2) \partial_{\lambda} \Phi.
\end{align*}
Combined with the property~\eqref{eqn:linearized-va-vlambda}, this implies that $P_f ( v_{\lambda} + \frac{1}{2 \lambda} f) = 0 = P_f v_a$, consistent with the identity $v_{\lambda} + \frac{1}{2 \lambda} f = a v_a$ of Lemma~\ref{lemma:relating-va-v-lambda}.

Next, we will obtain a perturbative estimate for the profile function of the capillary cone $\mathbf{C}_{n,k,\theta}$ with sufficiently small contact angle $\theta$, in terms of the hypergeometric function $f_{n,k}(t) := {}_2 F_1( \frac{n-1}{2} , - \frac{1}{2} ; \frac{k}{2} ; t^2)$.
This function satisfies $\cL_{n,k} f_{n,k} = 0$ with $f_{n,k}(0) = 1$ and zeroes at $t = \pm t_{n,k}$.
We refer the reader to~\cites{gauss} and~\cite{dtmf}*{Ch.~15} for standard properties of hypergeometric functions.
A finer analysis of the functions $f_{n,k}(t)$ and their properties related to the one-phase problem is performed in our companion paper~\cite{FTW-stability-one-phase}.

\begin{corollary}\label{corollary:linearized-deviation}
    There exists a $\lambda_{n,k} > 0$ such that for all $\lambda \in (0,\lambda_{n,k})$, the solution $f_{\lambda}$ of equation~\eqref{eqn:rescaledODE} with $f'_{\lambda}(0) =0$ and $f_{\lambda}(0) = 1$ satisfies
    \begin{equation}\label{eqn:f-lambda-f0}
      \left| D^{\ell}(f_{\lambda} - f_0 - \lambda v_0) \right| < C(n,k,\ell) \lambda^2 \qquad \text{on } \; [-t_{n,k}, t_{n,k}]  
    \end{equation}
    where $v_0$ is the unique non-positive even solution of equation
    \begin{equation}\label{eqn:w0-solution}
        \cL_{n,k} v_0 = - (n-2) (1-t^2) (f'_0)^2 ( f_0 - A f_0'), \qquad v_0(0) = v'_0(0) = v_0''(0) = 0.
    \end{equation}
    Moreover, there exists a $\theta_{n,k} > 0$ such that for all $\theta \in (0,\theta_{n,k})$, the even solution $f_{\theta}$ of equation~\eqref{eqn:odeStar} with a zero at $t_{\theta}$, producing a contact angle $\theta = \on{arctan} ( \sqrt{1 - t_{\theta}^2} \, |f'_{\theta}(t_{\theta})|) \in (0, \theta_{n,k})$, satisfies
    \begin{equation}\label{eqn:theta-expansion}
        |D^{\ell} (f_{\theta} - c_{n,k} \theta f_0)| < C(n,k,\ell) \theta^3 , \qquad |t_{\theta} - t_{n,k}| < C(n,k) \theta^2
    \end{equation}
\end{corollary}
\begin{proof}
    This estimate follows from the linearized framework of equation~\eqref{eqn:linearized-va-vlambda}, where $v_0 := \partial_{\lambda} f_{\lambda}|_{\lambda = 0}$.
    Our computation therein implies that
    \[
    f_{\lambda} = f_0 + \lambda v_0 + O(\lambda^2) \qquad \text{on } \; [ - t_{n,k}, t_{n,k}],
    \]
    with the same expansion being valid for all higher derivatives.
    This produces the expansion~\eqref{eqn:f-lambda-f0}.
    The linearized equation for $v_0$ assumes the form
    \[
\cL v_0 = - (n-2)(1-t^2) (f'_0)^2 (f_0 - A f'_0).
    \]
    The initial condition $v_0(0) = v'_0(0) = v_0''(0) = 0$ follows from~\eqref{eqn:variation-field-values-at-zero}.
    Moreover, letting $t_{\lambda}$ denote the zero of $f_{\lambda}$, we may differentiate the relation $f_{\lambda}(t_{\lambda}) = 0$ in $\lambda$ to obtain $v_{\lambda}(t_{\lambda}) + f'_{\lambda}(t_{\lambda}) ( \partial_{\lambda} t_{\lambda}) = 0$.
    Evaluating at $\lambda=0$ implies $\partial_{\lambda} t_{\lambda}|_{\lambda=0} = - \frac{v_0(t_0)}{f'_0(t_0)}$, whereby
    \begin{equation}\label{eqn:t-lambda-expansion}
        \bigl| t_{\lambda} - t_0 + \tfrac{v_0(t_0)}{f'_0(t_0)} \lambda \bigr| = O (\lambda^2), \qquad \text{for small } \; \lambda \in (0,\lambda_0).
    \end{equation}
    Next, we define $c_{n,k} := \frac{1}{\sqrt{1 -t_0^2} \, |f'_0(t_0)|}$.
    For the solution $f_a$ of~\eqref{eqn:odeStar} with $f_a(0) = a$, rescaling to $\frac{1}{a} f_a$ produces a solution of~\eqref{eqn:rescaledODE} with parameter $\lambda = a^2$.
    Consequently,~\eqref{eqn:f-lambda-f0} and~\eqref{eqn:t-lambda-expansion} imply
    \begin{equation}\label{eqn:a-f0-fa-expansion}
    |D^{\ell}( \tfrac{1}{a} f_a - f_0)| = O(a^2), \qquad |t_a - t_0| = O(a^2),
    \end{equation}
    from which we find $| \tfrac{1}{a}\sqrt{1 - t_a^2} (-f'_a(t_a)) - \sqrt{1 - t_0^2} \, (- f'_0(t_0)) | = O(a^2)$.
    We conclude that
    \[
    \sqrt{1 - t_a^2} \, ( - f'_a(t_a)) = c_{n,k}^{-1} a + O(a^3).
    \]
    For $\theta \in (0,\theta_{n,k})$ sufficiently small, the condition of contact angle $\theta$ then implies $\theta = c_{n,k}^{-1} a + O (a^3)$.
    Using this property in~\eqref{eqn:a-f0-fa-expansion} completes the proof.
\end{proof}

We now establish the various properties leading to the result of Figure~\ref{fig:fams-n7k1}, starting with the fact that the solutions $f_a$ cross uniquely.

\begin{proposition}\label{prop:unique-crossing of solutions}
    For a fixed parameter $\lambda>0$, let $f_a$ denote the solution of equation~\eqref{eqn:rescaledODE} with $f_a(0) = a$ and $f'_a(0) = 0$, and maximal interval of definition $[0,b_a]$.
    For any $0 < a_1 < a_2$, we write $f_i := f_{a_i}$.
    Then, the following properties are satisfied.
    \begin{enumerate}[$(i)$]
        \item The solutions $f_i$ intersect at most once over their common interval of definition $[0, \min_{i=1,2} b_i ]$.
        \item If $f_2$ reaches zero at some point $t_2$, then for any $a \in (0,a_2)$, the solutions $f_a$ and $f_2$ intersect at some point $t_* < t_2$.
        \item If $f_a$ reaches zero at a point $t_a$ for every $a \in [a_1, a_2]$, then $a \mapsto t_a$ is a strictly decreasing map.
        \item For any $a>0$, the variation field $v_a := \partial_a f_a$ has at most one zero and remains strictly decreasing in its positive phase.
        If $f_a$ has a zero at some point $t_a$, then $v_a$ has its unique sign change at $t_0 \in (0,t_a)$.
    \end{enumerate}
\end{proposition}
\begin{proof}
    We introduce the setup of the proof and then treat each item separately.
    For any parameter $\lambda$ and solutions $f_a$ of~\eqref{eqn:rescaledODE}, we recall that the functions $\sqrt{\lambda} f_a$ are solutions of~\eqref{eqn:odeStar}, so we may work with $\lambda=1$ and equation~\eqref{eqn:odeStar} in what follows.

\smallskip \noindent \textbf{Starting point.}
    Let us define the ``intersection instance''
    \[
    t_* := \sup \{ t \in [0,b_2] : f_1, f_2 \; \text{are defined and} \; f_1 < f_2 \; \text{ on } \; [0,t] \}.
    \]
    We apply Lemma~\ref{lemma:simple-system-comparison} and Corollary~\ref{cor:comparison-before-sqrtalpha}, letting $q_i := \frac{f'_i}{f_i}$ and $w_i := \frac{f_i^2}{1 + f_i^2}$.
    Since $q_i(0) = 0$ and $w_2 > w_1$ if and only if $f_2 > f_1$, we obtain $w_2 > w_1$ on $[0,t_*]$, whereby $q_1 > q_2$ on $(0,t_*]$.
    As in Corollary~\ref{cor:comparison-before-sqrtalpha}, this implies that on the interval $[0,t_*)$ we have $f_1>0$, the functions $f_i$ are both defined, and $\frac{f_2}{f_1}$ is strictly decreasing.
    In particular, the function $f_1$ cannot have derivative blow-up before intersecting, and therefore exceeding, $f_2$.

\smallskip \noindent \textbf{Property $(i)$.}
Suppose, for contradiction, that $f_2$ has at least two points in common with $f_1$.
    We consider the family of solutions $\{ f_a \}_{a \in (0,a_1]}$ of equation~\eqref{eqn:odeStar} with initial heights $f_a(0) = a \leq a_1$ and let $b_2 := b_{a_2}$ denote the blow-up instance of $f'_2$, where $f(b_{2}) = c_2 \in (-\infty,\infty)$.
    Our starting point shows that for any $a \leq a_1$, the function $f_a$ cannot have derivative blow-up before intersecting, and therefore exceeding, $f_2$.
    Since $f_a \to f_1$ uniformly as $a \uparrow a_1$, the family intersects $f_2$ at least twice for $a \uparrow a_1$.
    On the other hand, Corollary~\ref{corollary:linearized-deviation} gives $|f_a - a f_0| < C(n,k) a^2$, whereby $f_2$ intersects the members of this family at most once, for $a>0$ sufficiently small.
    By the compactness of the intervals, we may therefore find a largest value $a_* \in (0,a_1)$ for which the curves $f_{a_*}$ and $f_2$ have contact of order $\geq 2$, or have at least two points of intersection, one of which occurs over the boundary $t \in \{0, b_2 \}$ of the domain of $f_2$.
    In the latter case, exactly two points of intersection must occur, otherwise we could decrease $a_*$ while maintaining at least two points of intersection.

    It remains to rule out the two above situations.
    In the first case, $f_{a_*}$ and $f_2$ have contact of order $\geq 2$ at some point $t_* \in (0,b_2)$, whereby ODE uniqueness forces $f_{a_*} \equiv f_2$; this produces a contradiction, since $a_* \leq a_1 < a_2$.
    In the second case, $f_{a_*}$ and $f_2$ cannot intersect at $\{ t=0\}$, due to $a_* \leq a_1 < a_2$, so they have exactly two points of intersection, occurring at $b_2$ and at some $t_* \in (0,b_2)$.
    Since the function $f_2$ has derivative blow-up at $b_2$, with $f(b_2) = c_2$, we have $f''_2 \to - \infty$ as well, for $t \uparrow b_2$.
    By Lemma~\ref{lemma:derivative-blow-up}, we know that $b_2 < 1$, so the factor $(1-t^2)$ remains regular and non-degenerate as $t \uparrow b_2$ in equation~\eqref{eqn:odeStar}.
    We may therefore combine these properties to balance the dominant terms $(f')^3 \asymp f''$, obtaining $f' \sim (b_2 - t)^{- \frac{1}{2}}$.
    For general $\lambda$, this leads to the expansion
    \[
    f_2(t) = c_2 + \sqrt{\frac{2(1+ \lambda c_2^2)}{\lambda (n-2) A(b_2) }} \, \sqrt{b_2 - t} + O ( |b_2 - t|), \quad f'_2(t) = - \frac{1}{2} \sqrt{\frac{2(1+ \lambda c_2^2)}{\lambda (n-2) A(b_2) }} \frac{1}{\sqrt{b_2 - t}} + O (1)
    \]
    and similarly, $f''_2(t) = - \frac{1}{4} \sqrt{\frac{2(1+ \lambda c_2^2)}{\lambda (n-2) A(b_2) }} (b_2 - t)^{- \frac{3}{2}} + O ( \frac{1}{\sqrt{b_2 - t}}  ) $.
    Recall that $b_2 > \sqrt{\frac{k-1}{n-2}}$ holds by Lemma~\ref{lemma:general-behavior-prep}, so $A(b_2) > 0$ and the above expressions are well-defined.
    Near $b_2$, we may then expand $f_2$ as a series in $\sqrt{b_2 - t}$ to transform
    \[
    f_2(t) = c_2 + \sqrt{b_2 - t} \, g( \sqrt{b_2 - t}), \qquad \text{where} \quad g(0)>0.
    \]
    Using the fact that $f_2$ is a solution of~\eqref{eqn:odeStar}, we derive a uniquely determined power series expansion for $g_i$ (with $\lambda=1$) of the form
    \[
    g(s) = \sqrt{\tfrac{2 ( 1 + c_2^2)}{(n-2) A(b_2)}} + d_2 s^2 + d_4 s^4 + \cdots
    \]
    with the coefficients $d_k$ obtained recursively.
    This means that the function $g$, and therefore $f_2$, is uniquely determined by the properties $f_2(b_2) = c_2$ and $f'_2(b_2) = - \infty$ at $b_2$.

    Returning to the above situation, we have assumed that $f_{a_*}$ and $f_2$ intersect at exactly two points, $t_* \in (0,b_2)$ and $b_2$.
    Since $f_{a_*}(0) \leq a_1 < a_2 = f_2(0)$, we have $f'_{a_*}(t_*) \geq f'_2(t_*)$ at the first point of intersection, and equality cannot hold by the ODE uniqueness argument.
    Therefore, $f_{a_*}'(t_*) > f_2'(t_*)$ and $f_{a_*} > f_2$ for small $t - t_* > 0$, hence for all $t \in (t_* , b_2)$ until the next point of intersection, at $f_{a_*}(b_2) = f_2(b_2) = c_2$.
    Therefore, $\lim_{t \uparrow b_2} (f'_{a_*} - f'_2) \leq 0$, and since $f'_2(b_2) = - \infty$, this forces $f'_{a_*}(b_2) = - \infty$ as well.
    However, the above discussion shows that $f_{a_*}(b_2) = c_2$ and $f'_{a_*}(b_2) = - \infty$ would force $f_{a_*} \equiv f_2$, leading to a contradiction.
    This rules out the second situation, proving that $f_1$ and $f_2$ have at most one point of intersection, as desired.

\smallskip \noindent \textbf{Property $(ii)$.}
In our starting point, we showed that $f_1$ remains defined on $[0,t_*]$ while $f_1 < f_2$.
Since $f(t_2) = 0$, we conclude that either $t_* < t_2$ or $t_* = t_2$ must hold, meaning that $f_1, f_2$ must intersect at some point $t_* \leq t_2$.
We will prove that the solutions are strictly separated, meaning that $t_* < t_2$.
    We form the quantity $W := f_1 f'_2 - f'_1 f_2$ and suppose for contradiction that $t_2 = t_*$ is the first crossing of the solutions, meaning that $f_1(t_*) = f_2(t_*) = 0$.
    Moreover, $W' = f_1 f''_2 - f''_1 f_2$ and the above comparison implies $W = f_1 f_2 (q_2 - q_1) < 0$ on $(0,t_*)$.
    Comparing the identities~\eqref{eqn:odeStar} for the $f_i$ in the form~\eqref{eqn:express-in-form-a-lambda}, we consider the expression   
    \[
    f_1 \left( \cL f_2 + (n-2) (1-t^2) \Phi(t,f_2, f'_2) \right) - f_2 \left( \cL f_1 + (n-2) (1-t^2) \Phi( t, f_1, f'_1) \right) = 0.
    \]
    We collect terms using $(n-2) A + t = (n-1) t - (k-1) t^{-1}$ to obtain
    \begin{equation}\label{eqn:useful-expression-near-t*-root}
    \begin{split}
0 &= (1-t^2) W' - \bigl( (n-1) t - (k-1) t^{-1} \bigr) W \\
&\quad+ (n-2)(1-t^2) \left[ f_1 f_2 \left( \frac{(f'_2)^2}{1 + f_2^2} - \frac{(f'_1)^2}{1 + f_1^2} \right) + A\left( \frac{(f'_1)^2}{1 + f_1^2}f'_1 f_2 - \frac{(f'_2)^2}{1 + f_2^2} f_1 f'_2  \right) \right].
    \end{split}
        \end{equation}
    First, suppose that $|f'_2 (t_*)| < + \infty$, whereby $|f'_1| < + \infty$ on $[0,t_*]$ as in Corollary~\ref{cor:comparison-before-sqrtalpha}.
    Since the functions $f_i$ are analytic, we may write $f_i(t) = (t_* - t) g_i(t)$ with $g_i > 0$ on $[0,t_*]$ and $g_i(t_*) = - f'_i(t_*) \in (0,\infty)$.
    In particular, $f_2 > f_1$ on $(0,t_*)$ implies $g_2 > g_1$, and $g_2(t_*) > g_1 (t_*)$ since otherwise $(f_1(t_*), f'_1(t_*)) = (f_2 (t_*), f'_2(t_*))$ would force $f_1 \equiv f_2$, by ODE uniqueness.
    Moreover,
    \[
    f'_i(t) = - g_i(t) + O (t_* - t), \qquad f''_i(t) = - 2 g'_i(t) + O(t_* - t).
    \]
    We now write $W = (t_* -t) ( g_1 f'_2 - f'_1 g_2)$ and $W' = (t_* - t) ( g_1 f''_2 - f''_1 g_2)$, whereby 
    \begin{align*}
        W &= f_1 f'_2 - f'_1 f_2 = (t_* - t) \left( g_1 ( - g_2 + (t_* - t) g'_2 ) - ( - g_1 + (t_* - t) g'_1 ) g_2 \right) \\
    &= (t_* - t)^2 ( g_1 g'_2 - g'_1 g_2).
    \end{align*}
    Consequently, the fact that $W<0$ on $(0,t_*)$ requires $g'_1 g_2 - g_1 g_2' > 0$ there.
    Similarly, 
    \[
    W' = (t_* - t) (g_1 f''_2 - f''_1 g_2) = 2 (t_* - t) ( g'_1 g_2 - g_1 g'_2)+ O ( ( t_* - t)^2 ) .
    \]
    Using this computations in the above equality and taking $t \uparrow t_*$, we obtain
    \begin{align*}
        & 2 (1-t_*^2) ( g'_1 g_2 - g_1 g'_2) (t_*) + O (t_* - t) \\
        &\quad=  - (n-2)(1-t^2) \left[ (t_* - t) g_1 g_2 \left( \frac{(f'_2)^2}{1 + f_2^2} - \frac{(f'_1)^2}{1 + f_1^2} \right) + A \left( \frac{(f'_1)^2}{1 + f_1^2} f'_1 g_2 - \frac{(f'_2)^2}{1 + f_2^2} g_1 f'_2 \right)  \right] \\
        &\quad= (n-2) (1-t_*^2) A(t_*) g_1 g_2 \left( g_1^2 - g_2^2 \right)  + O (t_* - t).
    \end{align*}
    Since the zero of $f_i$ satisfies $t_i > \sqrt{\frac{k-1}{n-2}}$ by Lemma~\ref{lemma:general-behavior-prep}, we have $A(t_*) > 0$, which implies that
    \[
    \on{sgn} ( g'_1 g_2 - g_1 g'_2)(t_*) = \on{sgn} (g_1 - g_2)(t_*) < 0,
    \]
    where the negativity is strict by the above ODE uniqueness discussion.
    However, $W<0$ forces $g'_1 g_2 - g_1 g'_2 > 0$ on $(0,t_*)$, whereby $(g'_1 g_2 - g_1g'_2)(t_*) \geq 0$; this produces a contradiction.

    We turn to the case $|f'_2(t_*)| \to + \infty$.
    Since $\frac{f_2}{f_1}$ is strictly decreasing on $(0,t_*]$, we have
    \[
    \frac{a_2}{a_1} = \frac{f_2}{f_1} (0) > \lim_{t \uparrow t_*} \frac{f_2(t)}{f_1(t)} = \frac{f'_2 (t_*)}{f'_1(t_*)}
    \]
    by applying L'Hospital's rule.
    Therefore, $|f'_2(t_*)| \to + \infty$ forces $|f'_1(t_*)| \to + \infty$ as well.
    However, the discussion of Property $(i)$ shows that $f_1(t_*) = 0 $ and $f'_1(t_*) = - \infty$ determines the solution uniquely, showing that $f_1 \equiv f_2$.
    We again reach a contradiction, whereby the intersection point $t_* < t_2$ must occur strictly before the zero of $f_2$, as claimed.

\smallskip \noindent \textbf{Property $(iii)$.}
The above discussion shows that for any two solutions $f_i$ reaching zero at $t_i$ with $f_1(0) < f_2(0)$, the unique crossing occurs before $\min \{ t_1, t_2\}$, hence $t_1>t_2$, so we conclude that $a \mapsto t_a$ is a strictly decreasing map.

\smallskip \noindent \textbf{Property $(iv)$.}
This fact is a direct consequence of properties $(i)$~--~$(iii)$.
Indeed, if $v_a$ had at least two zeroes for some $f_a$, then $f_a$ and $f_{a \pm \ve}$ would have at least two points of intersection, lying $\ve$-close to those zeroes, for $\ve \in (0,\ve_0)$ sufficiently small; this contradicts $(i)$.
Similarly, the positivity $v_a \geq 0$ on $[0,t_a]$ would imply that $f_a > f_{a-\ve}$ on $[0,t_a)$, for small $\ve \in (0,\ve_0)$.
This would contradict $(ii)$, since $f_a$ and $f_{a-\ve}$ must intersect at some point $t_* < t_a$.
Therefore, $v_a$ has positive phase $[0,t_0)$, where $b_a \geq t_0$ is the maximal time of definition of $f_a$ and $t_0 < t_a$, meaning that $f_a > 0$ while $v_a > 0$.
Our initial argument shows that for small $\ve \in (0,\ve_0)$, the function $\frac{f_{a+\ve}}{f_a}$ is strictly decreasing while larger than $1$ and both functions are positive.
Therefore, $\frac{f'_{a+\ve} - f'_{a}}{\ve} < \frac{(f_{a+\ve} - f_a) f'_{\ve}}{\ve f_a}$, and sending $\ve \downarrow 0$ shows that $v_a' \leq \frac{f_a'}{f_a}v_a < 0$, so $v_a$ is strictly decreasing on $[0,t_0]$ as claimed.
\end{proof}
Proposition~\ref{prop:unique-crossing of solutions} establishes the first part of the foliation-and-ordering property illustrated in Figure~\ref{fig:fams-n7k1}, showing that \textit{the solutions $f_a$ cross.}
We now complete the picture, proving that \textit{the solutions $f_{\lambda}$ foliate.}

\begin{corollary}\label{corollary:lambda-comparison}
    For a fixed parameter $a>0$ and $\lambda \in (0,\infty)$, let $f_{\lambda}$ denote the solution of equation~\eqref{eqn:rescaledODE} with $f_{\lambda}(0) = a$ and $f'_{\lambda}(0) = 0$, and maximal interval of definition $[0,b_{\lambda}]$.
    The solutions $f_{\lambda}$ lie in strictly decreasing order of $\lambda$ over their positive phase, and foliate a region between $f_{\infty}$ and $f_0$ in the first quadrant.
\end{corollary}
\begin{proof}
As discussed in Proposition~\ref{prop:unique-crossing of solutions}, the strictly decreasing foliation is equivalent to proving that for any $\lambda>0$, the variation field $v_{\lambda} := \partial_{\lambda} f_{\lambda}$ produced by varying the scaling parameter satisfies $v_{\lambda} < 0$ on $\{ f_{\lambda} > 0 \}$.
    As in~\eqref{eqn:Pf-v-linearized}, we denote by $P_f$ the linearized operator at the solution $f$ and let $v_a$ be the principal even solution of the homogeneous equation $P_f v_a = 0$ with $v_a(0) = 1$ and $v'_a(0) = 0$.
    Recall that $v_a$ has at most one zero, which we denote by $t_0$, and $t_0 \in \{ f_{\lambda} > 0 \}$.
 
    On the interval $[0, t_0]$, the fact that $v_a \geq 0$ together with $P_f v_{\lambda} < 0$ from~\eqref{eqn:linearized-va-vlambda}, and the Green's kernel representation, implies that $v_{\lambda} < 0$ and $v'_{\lambda} < 0$ on $(0,t_0]$.
    Equivalently, since $P_f(- v_{\lambda}) >0$ and $P_f v_a = 0$, the Sturm-Picone theorem guarantees that $v_{\lambda}$ has no sign changes on $(0,t_0)$, since $v_a$ is strictly positive.
    If $t_0 = \min \{ t_{\lambda}, b_{\lambda} \}$, then we are done; otherwise, on the interval $(t_0, \min \{ t_{\lambda}, b_{\lambda} \})$, we have $-v_a, f_{\lambda,a} > 0$, so the relation~\eqref{eqn:relating-the-variations} implies
    \[
    2 \lambda v_{\lambda} = a v_a - f_{\lambda,a} < 0
    \]
    as desired.
    The inequality is strict up to the endpoint, since $v_{\lambda}(t_{\lambda}) = 0$ would force $v_a(t_{\lambda}) = f_{\lambda,a}(t_{\lambda}) = 0$, contradicting the property $v_a(t_{\lambda}) < 0$ obtained above.
    We conclude that $v_{\lambda} <0$ on $[0, \min \{ t_{\lambda}, b_{\lambda} \} ]$ for every $f_{\lambda}$, hence the functions $\{ f_{\lambda} \}_{\lambda \in (0,\infty)}$ obey a strictly decreasing ordering in $\lambda \in (0,\infty)$ and foliate the region between $f_0$ and $f_{\infty}$.
\end{proof}

\begin{corollary}\label{corollary:terminal-slope-monotonicity}
Let $f_a$ be the solution of equation~\eqref{eqn:odeStar} with initial height $f_a(0) = a$ and define
\[
a_* := \sup \{ a > 0 : f_a(t) \; \text{reaches zero within its interval of existence} \}.
\]
For $a \in (0,a_*]$, let $t_a$ denote the zero of $f_a$.
Then, the terminal slope map 
\[
a \mapsto |f'_a(t_a)|
\]
is strictly increasing in $a \in (0, a_*]$.
Moreover, letting $f_{\lambda}$ be the solution of the rescaled equation~\eqref{eqn:rescaledODE} with fixed initial height and scaling parameter $\lambda$, which reaches zero at $t_{\lambda}$, the terminal slope map $\lambda \mapsto |f'_{\lambda}(t_{\lambda})|$ is strictly increasing in $\lambda \in [0,1]$.
This means that the map
\[
a \mapsto a^{-1} |f'_a(t_a)|
\]
is also strictly increasing in $a \in (0, a_*]$.
\end{corollary}

\begin{proof}
We prove the result for the map $a \mapsto |f'_a(t_a)|$, using Proposition~\ref{prop:unique-crossing of solutions}.
We first show that $a \mapsto |f'_a(t_a)|$ is strictly increasing for small $a>0$.

Applying the continuous dependence theorem for solutions of the ODE~\eqref{eqn:odeStar} in terms of the initial conditions shows that the map $a \mapsto f_a$ is $C^1$; by the implicit function theorem, the map $a \mapsto t_a$ is also $C^1$.
Differentiating the relation $f_a(t_a) = 0$ and working with the variation field $v_ a:= \partial_a f_a$ along solutions produced by varying the initial height, as in~\eqref{eqn:linearization-at-f},  therefore gives
\begin{align*}
    \partial_a t_a &= - \frac{v_a(t_a)}{f'_a(t_a)}, \\ 
    \partial_a f'_a(t_a) &= v'_a( t_a) + f''_a( t_a) \, \partial_a t_a =  v'_a(t_a) - f''_a(t_a) \frac{v_a(t_a)}{f'_a(t_a)} \\
    &= f'_a(t_a) \cdot \left( \frac{v_a}{f'_a} \right)' (t_a).
    \end{align*}
For small $a \in (0, \ve_{n,k})$, we may write $|D^{\ell}(f_a - a f_0)| < C(n,k,\ell) a^2$, where $f_0$ is the solution of the linear equation $\cL_{n,k} f_0 = 0$, given by the hypergeometric function ${}_2 F_1(\frac{n-1}{2} , - \frac{1}{2} ; \frac{k}{2} ; t^2)$.
Since $\cL_{n,k}$ is a linear operator, the variation field satisfies the same linearized equation, so $|D^{\ell}(v_a - f_0)| < C(n,k,\ell) a^2$ and $|t_a - t_0| < C(n,k) a^2$.
We therefore obtain
\begin{align*}
|v_a(t_a)| &= |v_a(t_a) - a f_0(t_0)| < C(n,k) a^2, \qquad |v'_a(t_a) - a f'_0(t_0)| < C(n,k) a^2, \\
\left| \partial_a f'_a(t_a) - a f'_0(t_0) \right| &\leq C(n,k) a^2 + 2 \left| \frac{f''_0(t_0)}{f'_0(t_0)} \right| a^2 < C'(n,k) a^2,
\end{align*}
so $\partial_a f'_a(t_a) < a ( f'_0(t_0) - C'(n,k) a)$.
Since $f'_0(t_0) < 0$, we conclude that $\partial_a f'_a(t_a) < 0$ for $a \in (0,\ve_{n,k})$, meaning that $a \mapsto |f'_a(t_a)|$ is strictly increasing for small $a > 0$.

For the blow-up time $b_a$ of the solution, we know that $b_a \to 1$ as $a \downarrow 0$, by comparison with the hypergeometric function $f_0$; consequently, $t_a < b_a$ as $a \downarrow 0$.
The definition of $a_*$ implies that $f_{a_* + \ve}$ has derivative blow-up while positive, for any $\ve > 0$ sufficiently small; taking the limit as $\ve \downarrow 0$, we deduce that $t_{a_*} = b_{a_*}$.
The computation of Proposition~\ref{prop:unique-crossing of solutions} shows that $f_{a_*}(t) = \sqrt{\frac{2}{(n-2) A(t_*)}} (t_* - t)^{\frac{1}{2}} + o (|t_* - t|)$, whereby $f_{a_*}(t)$ produces a solution with contact angle $\frac{\pi}{2}$ (free-boundary) in the notation of Proposition~\ref{prop:FreeBdyODE=Cone}.
The continuity in the parameter $a$ implies that as $a \in (0,a_*]$, the solutions $f_a$ produce all capillary angles $\theta \in (0, \frac{\pi}{2}]$.

The above discussion also implies that the map $a \mapsto |f'_a(t_a)|$ is strictly increasing on $(a_* - \ve, a_*]$, for some $\ve > 0$.
If the strictly increasing property failed somewhere in $(0, a_*)$, there would exist a first instance $a_0 > 0$ where $\frac{d}{da}\big\rvert_{a = a_0} |f'_a(t_a)| = 0$, and some $\ve > 0$ such that $\frac{d}{da}|f'_a(t_a)| < 0$ for $a \in (a_0, a_0 + \ve)$.
This means that the map $a \mapsto |f'_a(t_a)|$ is strictly decreasing on $[a_0, a_0 + \ve]$, so we may find some $a_1 < a_2 \in (a_0, a_0 + \ve)$ such that $|f'_{a_2}(t_{a_2})| \leq |f'_{a_1} (t_{a_2})|$.
However, Proposition~\ref{prop:unique-crossing of solutions} ensures that the functions $f_{a_1}$ and $f_{a_2}$ intersect at a unique point $t_* < t_{a_2}$, and $f'_{a_2}(t) < f'_{a_1}(t) < 0$ for all $t > t_*$; applying this at $t = t_2$ produces a contradiction, hence the map $a \mapsto |f'_a(t_a)|$ must be strictly increasing.

The argument for $\lambda \mapsto |f'_{\lambda}(t_{\lambda})|$ is identical, where we instead apply Corollary~\ref{corollary:lambda-comparison}; this guarantees that any $\lambda_1 < \lambda_2$ satisfy $f'_{\lambda_2} < f'_{\lambda_1}$ on the common domain of non-negativity of the two functions.
To see the final monotonicity in $a$, we use the correspondence between $f_{\lambda}$ (with fixed initial height, equal to $a_*$) and $f_a$, given by $f_a = \sqrt{\lambda} f_{\lambda}$ for $\lambda = ( \frac{a}{a_*} )^2 \in (0,1)$.
\end{proof}
In Proposition~\ref{prop:general-a-behavior} below, we will prove that $a_* = a_{n,k} = \sqrt{\frac{k}{n-k-1}}$ is precisely the initial height of the Lawson cone solution from Lemma~\ref{lemma:lawson-cones}.
This means that the admissible solutions $f_a$ (i.e., the ones reaching zero within their interval of existence) are precisely those with $a \in (0 , a_{n,k}]$.

\subsection{Existence and capillary angle}
Completing the proof of Theorem~\ref{thm:capillary-cones}, we now show that as we vary the initial height $a$ between the Lawson height and 0, the solutions $f_a$ of~\eqref{eqn:odeStar} exist until they attain 0 and sweep out all possible capillary angles.
The key ingredient is a monotone quantity that measures the deviation of a solution of the capillary ODE from the Lawson cone and determines its blow-up-or-zero behavior.

\begin{lemma}\label{lemma:psi-detect-blow-up}
    Fix a point $t_0 > \sqrt{\frac{k-1}{n-2}}$ and let $f \in C^2([t_0, T))$ be a maximal solution of equation~\eqref{eqn:odeStar} with initial data $f(t_0) > 0$ and $f'(t_0) \leq 0$.
    We form the quantity $\Psi := f \cdot h_f - \frac{1}{n-2} = f (f - A f') - \frac{1}{n-2}$.
    \begin{enumerate}[$(i)$]
        \item if $\Psi(t_0) \leq 0$ and $\Psi'(t_0) \leq 0$, then $\Psi$ is strictly decreasing on $(t_0, T)$, or $\Psi \equiv 0$ and $f = \hat{f}_{n,k}$ coincides with the Lawson solution of Lemma~\ref{lemma:lawson-cones}.
        \item if $\Psi(t_0) \geq 0$ and $\Psi'(t_0) \geq 0$, then $\Psi$ is strictly increasing on $(t_0, T)$, or $\Psi \equiv 0$ and $f = \hat{f}_{n,k}$ coincides with the Lawson solution of Lemma~\ref{lemma:lawson-cones}.
    \end{enumerate}
    If $f$ is not the Lawson solution, then in case $(i)$, it reaches zero with contact angle $< \frac{\pi}{2}$ and finite derivative at the zero.
    In case $(ii)$, the derivative $f'$ blows up while $f$ remains positive.

    The same result holds, with $t_0 = \sqrt{\frac{k-1}{n-2}}$, upon replacing the conditions $(i)$, $(ii)$, respectively, with
    \[
    \Bigl\{ \Psi(\sqrt{\tfrac{k-1}{n-2}}) \leq 0 \; \text{ and } \; \Psi''(\sqrt{\tfrac{k-1}{n-2}}) \leq 0 \Bigr\} \quad \text{or} \quad \Bigl\{ \Psi(\sqrt{\tfrac{k-1}{n-2}}) \geq 0 \; \text{ and } \; \Psi''(\sqrt{\tfrac{k-1}{n-2}}) \geq 0\Bigr\}.
    \]
\end{lemma}
The quantity $\Psi$ has the significance that $\Psi \equiv 0$ if and only if $f \equiv \hat{f}_{n,k}$ is the Lawson solution.
\begin{proof}
Suppose that $\Psi(t_0) = 0$ and $\Psi'(t_0) = 0$, then solving~\eqref{eqn:odeStar} for $f''$ at $t_0$ and substituting into $\Psi'(t_0) = 0$ produces a $2 \times 2$ system for $(f(t_0), f'(t_0))$.
Direct computation shows that $t_0 \leq \sqrt{\frac{k}{n-1}}$ and $f(t_0) = \sqrt{\frac{k - (n-1) t_0^2}{n-k-1}}$ with $f'(t_0) = \hat{f}'_{n,k}(t_0)$ is the unique admissible solution, whereby ODE uniqueness forces $f \equiv \hat{f}_{n,k}$ to be the Lawson solution.

Otherwise, at least one of the inequalities is strict at $t_0$; we will show that $\Psi$ is strictly monotone.
Suppose that this is not the case, so there is a first instance $t_* \neq \sqrt{\frac{k-1}{n-2}}$ where $\Psi'(t_*) = 0$ and $\pm \Psi''(t_*) \geq 0$.
The same ODE uniqueness computation shows that $\Psi''(t_*) \neq 0$ for $f \not\equiv \hat{f}_{n,k}$.
At this point, $\pm \Psi(t_*) < 0$ by the assumptions.
On the other hand, we will show that
\begin{equation}\label{eqn:sign-Psi''-Psi}
\on{sgn} \Psi''(t_*) = \on{sgn} A(t_*) \cdot \on{sgn} \Psi(t_*) 
\end{equation}
at any point $t_* \neq \sqrt{\frac{k-1}{n-2}}$ where $\Psi'(t_*) = 0$.
This will produce a contradiction, whereby $\Psi'$ and $\Psi$ maintain the same fixed sign on $(t_0, T)$.
The full expression at the critical point is
\begin{equation}\label{eqn:Psi''(t*)-expression}
\Psi''(t_*) = \frac{2 (n-2)^2 (-f') (f - A f')}{(1-t_*^2) (1+f^2)} \left( 1 + \frac{(1-t_*^2) (f')^2}{1 + f^2} \right) A (t_*)  \Psi(t_*).
\end{equation}
The equality~\eqref{eqn:sign-Psi''-Psi} will follow from this fact due to $(-f') > 0$ and $h_f = f - Af' > 0$ for any solution, the latter by Lemma~\ref{lemma:general-behavior-prep}.
Using this fact, we will deduce the monotonicity of $\Psi$ for $t_* > \sqrt{\alpha}$.

To prove the expression~\eqref{eqn:Psi''(t*)-expression} for $\Psi''(t_*)$, we introduce the Riccati function $q = \frac{f'}{f}$, so $q' =  \frac{f''}{f} - q^2$ and $f'' = f(q^2 + q')$ transforms equation~\eqref{eqn:odeStar} (in the form~\eqref{eqn:system-ode-1}) into
\begin{equation}\label{eqn:v-Riccati-ODE}
0 = (1-t^2) (q^2 +q') + (1-tq) + (n-2) (1-Aq) + ( (n-2) \Psi + 1) \frac{ (1-t^2)q^2}{1+f^2}.
\end{equation}
Differentiating this equation at a point where $\Psi'(t_*) = 0$, we solve for $q''$ to obtain
\begin{equation}\label{eqn:v''-equation}
    \begin{split}
        (1-t^2) q'' &=  2t q^2 + 3 tq' - 2(1-t^2) qq'  + q + (n-2) (q + \alpha t^{-2} q + A q') \\
     & \quad - 2 ((n-2) \Psi + 1) \left( \frac{ -t q^2 + (1-t^2) qq '}{1+f^2} - \frac{(1-t^2) q^3 f^2}{(1+f^2)^2}\right).
    \end{split}
\end{equation}
Moreover, $\Psi = f^2 (1 - Aq) - \frac{1}{n-2}$ and $A(t) = t - \alpha t^{-1}$ has $A' = 1 + \alpha t^{-2}$ and $A'' = - 2 \alpha t^{-3}$, where $\alpha = \frac{k-1}{n-2}$.
Consequently, we compute
\begin{align*}
    \Psi' &= A f^2 \left(- q' - 2q^2 + t^{-1} q \right), \\
    \Psi'' &= f^2 \left(  2 \alpha t^{-3}  q - 2 \alpha t^{-2} q' - 4 \alpha t^{-2} q^2 - A q'' - 6 A qq' - 4 A q^3 \right)
\end{align*}
where $ 1 - \alpha t^{-2} = t^{-1} A \geq 0$ for $t \geq \sqrt{\alpha}$.
At a point $t_* \neq \sqrt{\frac{k-1}{n-2}}$ where $\Psi'(t_*) = 0$, we have $A(t_*) \neq 0$, so we obtain the relations
\[
q' = -2 q^2 + t^{-1} q \qquad \text{and} \qquad \Psi'' = A f^2( -q'' + 8 q^3 - 6 t^{-1} q^2).
\]
In equations~\eqref{eqn:v-Riccati-ODE} and~\eqref{eqn:v''-equation}, this produces the simplifications 
\[
q^2 + q' = - q^2 + t^{-1} q, \qquad qq' = -2q^3 + t^{-1} q^2, \qquad q + \alpha t^{-2} q + Aq' = 2q - 2 A q^2.
\]
Substituting these expressions for $q''$, we arrive at
\begin{align*} (1-t^2) \frac{\Psi''}{2 f^2 Aq} &=  (n-2) \frac{f^2 (1 - Aq) q}{(1+f^2)^2 t} \left[ t (1-t^2) q + (1+f^2) \left(2 t^3 q - 2 t^2-2 t q +1\right)\right] \\ 
& \quad - (1-tq) - (1-t^2) \frac{q (1 - tq)}{t} .
\end{align*}
Here, we also used the relations for $q'$ to transform the equality~\eqref{eqn:v-Riccati-ODE} into
\begin{equation}\label{eqn:1+tv-t-v-negative}
(1-tq) ( 1 - tq + t^{-1}q) + (n-2) (1-Aq) \left( 1 +  \frac{ (1-t^2) f^2q^2}{1+f^2} \right) = 0.
\end{equation}
Combining these computations, we rearrange the above expression into
\[
\Psi'' =  \frac{2 (n-2) f^2 Aq (1-tq) (1 - tq + t^{-1} q)}{(1-t^2) (1+f^2)} \Psi.
\]
Combining this equality with~\eqref{eqn:1+tv-t-v-negative} and using the definition $q =  \frac{f'}{f}$ finally proves the identity~\eqref{eqn:Psi''(t*)-expression}.
The sign~\eqref{eqn:sign-Psi''-Psi} of $\Psi''$ at the critical point $t_*$ follows, proving the monotonicity of $\Psi$.

Next, we discuss the case of $t_0 = t_{\alpha} := \sqrt{\frac{k-1}{n-2}}$.
We recall the identity
\begin{align*}
\Psi' &= A f^2 (-q' - 2q^2 + t^{-1} q ) = A(t) ( t^{-1} f f' - (f')^2 - ff''), \\
\Psi''( t_{\alpha}) &= 2 ( t^{-1} f f' - (f')^2 - ff'') \big\rvert_{t = t_{\alpha}}
\end{align*}
If the inequality $\pm \Psi''(t_{\alpha}) \leq 0$ is strict at this point, then we can find some $\ve > 0$ sufficiently small such that $\pm \Psi''(t_{\alpha} + \ve) < 0$.
Moreover, $\Psi'(t_{\alpha}) = 0$ due to $A(t_{\alpha}) = 0$, so $\pm \Psi'(t_{\alpha} + \ve) < 0$ as well.
Since $\pm \Psi'(t_{\alpha}) \leq 0$ was assumed, we conclude that 
\[
\pm \Psi(t_{\alpha} + \ve) < 0, \qquad \pm \Psi'(t_{\alpha} + \ve) < 0
\]
for $\ve > 0$ sufficiently small.
Therefore, our previous argument can be applied, starting from $t_0 = t_{\alpha} + \ve$, to conclude the result.
Therefore, it suffices to consider the case $\Psi''(t_{\alpha}) = 0$, meaning that $q' =- 2q^2 + t^{-1} q$ at this point.
We therefore compute, as before,
\[
\Psi'''(t_{\alpha}) = 4 (t^{-1} f f' - (f')^2 - ff'')' = 4 f^2 (-q'' + 8 q^3 - 6 t^{-1} q^2).
\]
In terms of our previous computation, this assumes the form
\[
\frac{\Psi'''(t_{\alpha})}{4} = (t^{-1} f f' - (f')^2 - ff'')'|_{t = t_{\alpha}} = \frac{\Psi''}{A} \bigg\rvert_{t = t_{\alpha}}
\]
Consequently, our earlier algebraic manipulations again become applicable at this point, which allows us to obtain an analogue of~\eqref{eqn:Psi''(t*)-expression} in the form
\[
\Psi'''(t_{\alpha}) = \frac{8(n-2)^3 f(-f')}{ (n-k-1) (1+f^2)} \left( 1 + \frac{\frac{n-k-1}{n-2} (f')^2}{1 + f^2} \right)  \Psi(t_{\alpha}).
\]
We conclude that $\on{sgn} \Psi'''(t_{\alpha}) = \on{sgn} \Psi(t_{\alpha})$ at that point; since $\Psi'(t_{\alpha}) = \Psi''(t_{\alpha}) = 0$, this means that $\pm \Psi(t_{\alpha} + \ve), \pm \Psi'(t_{\alpha} + \ve), \pm \Psi''(t_{\alpha}+\ve) < 0$ for $\ve>0$ sufficiently small.
Therefore, we are again in the case where the previous argument may be applied.
Observe that $\Psi(t_{\alpha}) = 0$ cannot occur, unless $f$ coincides with the Lawson solution $\hat{f}_{n,k}(t)$: $\Psi(t_{\alpha}) = 0$ forces $f(t_{\alpha})^2 = \frac{1}{n-2} = \hat{f}_{n,k}(t_{\alpha})^2$, and the relation $\Psi''(t_{\alpha}) = 0$ requires
\[
(t^{-1} f f' - (f')^2 - ff'')|_{t = t_{\alpha}} = 0.
\]
Since $f''(t_{\alpha})$ can be expressed in terms of $f(t_{\alpha}), f'(t_{\alpha})$ via the ODE~\eqref{eqn:cone-ODE}, this condition produces a quartic equation for $f'(t_{\alpha})$ whose only admissible solution is $f'(t_{\alpha}) = - \frac{(n-1)\sqrt{k-1}}{n-k-1} = \hat{f}'_{n,k}(t_{\alpha})$.
We conclude that $( f(t_{\alpha}) , f'(t_{\alpha}) ) = ( \hat{f}_{n,k} (t_{\alpha}), f'_{n,k}(t_{\alpha}))$, and since both functions $f, \hat{f}_{n,k}$ solve the same equation~\eqref{eqn:cone-ODE}, the uniqueness theorem for solutions of ODE would force $f \equiv \hat{f}_{n,k}$.
This proves the first part of the result.

We now observe that if the solution $f$ has a zero at $t_1 < T$, where $|f'(t_1)| < + \infty$, then $\Psi(t_1) = - \frac{1}{n-2}$.
    On the other hand, if $f$ has finite-time blow up at $T$ where $\lim_{t \uparrow T} f(t) > 0$, then $\Psi(t) \to \infty$ as $t \to T$.
    This is guaranteed by $T > \sqrt{\frac{k-1}{n-2}}$, so $A(T) > 0$ and $(- A f') \to + \infty$ as $t \to T$ and $f' \to - \infty$.
    Finally, if $t_1 = T$ and $f(T) = 0$ while $|f'| \to - \infty$ as $t \to T$, the computation of Proposition~\ref{prop:unique-crossing of solutions} shows that $f(t) = \sqrt{\frac{2}{(n-2) A(T)}} (T-t)^{\frac{1}{2}} + o (|T-t|)$ as $t \to T$, whereby
\[
f'(t) = - \frac{1}{\sqrt{2(n-2) A(T)}} (T-t)^{-\frac{1}{2}} + o( \sqrt{T-t}).
\]
Consequently, $ff' = -\frac{1}{(n-2) A} + o( \sqrt{T-t})$ and $\Psi = f(f-Af') - \frac{1}{n-2} = o(\sqrt{T-t})$, meaning that $\lim_{t \uparrow T} \Psi(t) = 0$ in this case.
Since $\Psi$ has a strict sign and monotonicity on $(t_0, T)$, we cannot have $|\Psi| \to 0$ as $t \uparrow T$ in either case $(i)$ or $(ii)$, which means that $t_1 = T$ is impossible.
Using the strict monotonicity of $\Psi$ when $f$ does not coincide with the Lawson solution, we conclude that $t_1<T$ and $f$ attains zero in case $(i)$, while $f$ remains positive until its derivative blows up in case $(ii)$.
\end{proof}
\begin{remark}
A more precise property is in fact true: for any solution $f$, the quantity
\[
\tilde{\Psi}(t) := t \cdot \Psi(t) =  t \cdot \bigl( f (f - Af') - \tfrac{1}{n-2} \bigr)
\]
has fixed sign and fixed monotonicity on the entire interval $[0,b]$ where $f$ is defined.
Specifically, $\tilde{\Psi}$ is positive and strictly increasing if $\tilde{\Psi}'(0) > 0$, and negative and strictly decreasing if $\tilde{\Psi}'(0) < 0$.
We compute that $\tilde{\Psi}'(0) = \frac{n-k-1}{k(n-2)} (a^2 - a_{n,k}^2)$, confirming that the monotonicity of $\tilde{\Psi}$ and the blow-up-or-zero dichotomy for $f$ are determined by the sign of $a - a_{n,k}$.
\end{remark}

\begin{proposition}\label{prop:psi-at-turning-point}
Let $a_{n,k} := \sqrt{\frac{k}{n-k-1}}$ and consider the solution $f_a$ of equation~\eqref{eqn:odeStar} with initial $f_a(0) = a>0$ and $f_a'(0) = 0$.
    The quantity $\Psi_a := f_a ( f_a - Af'_a) - \frac{1}{n-2}$ satisfies
    \[
    \on{sgn} \Psi_a ( \sqrt{\tfrac{k-1}{n-2}} + \ve ) = \on{sgn} \Psi'_a ( \sqrt{\tfrac{k-1}{n-2}} + \ve ) = \on{sgn} ( a - a_{n,k})
    \]
    for $\ve>0$ sufficiently small, depending on $n,k,a$.
\end{proposition}
\begin{proof}
We prove the result in three steps.
First, we obtain a stronger property, similar in spirit to Lemma~\ref{lemma:psi-detect-blow-up}.
Then, we show that the solutions $f_a$ satisfy this property, producing the claimed sign.
This discussion is specific to the case $k \geq 2$, which is more challenging.
In the last step, we prove the property when $k=1$.

\smallskip \noindent \textbf{Step 1:}
Let $k \geq 2$.
We claim that for any solution $f$ of equation~\eqref{eqn:odeStar} on some interval $[t_0, \sqrt{\frac{k-1}{n-2}}]$, the regions
\begin{equation}\label{eqn:s-plus-s-minus}
    S_+ := \{ f > \hat{f}, \; \Psi > 0, \; \Psi' < 0 \} \qquad \text{and} \qquad S_- := \{ f < \hat{f}, \; \Psi<0, \; \Psi' > 0 \}
\end{equation}
are invariant, i.e., any solution of~\eqref{eqn:odeStar} starting in these regions must remain in $S_{\pm}$.

The key ingredient of~\eqref{eqn:s-plus-s-minus} is an identity analogous to~\eqref{eqn:Psi''(t*)-expression}.
Let $f$ be any solution of equation~\eqref{eqn:odeStar}, with $\Psi := f (f - Af') - \frac{1}{n-2}$.
We claim that at a point where $\Psi(t_*) = 0$, it holds that
\begin{equation}\label{eqn:Psi-prime-computation}
    \Psi'(t_*) = - \frac{ (n-k-1) \bigl( (k-1) (1+f^2) - (n-1) t^2_*) \bigr) (f^2 - \hat{f}_{n,k}^2)  }{(n-2)^2 t_*^2 (1-t_*^2) (1+f^2) A(t_*)  }  \, .
\end{equation}
Here, $\hat{f} := \hat{f}_{n,k}(t) = \sqrt{\frac{k - (n-1) t^2}{n-k-1}}$ denotes the Lawson solution of Lemma~\ref{lemma:lawson-cones}.
For brevity, we denote $D := f^2 - \hat{f}^2$ and $\alpha := \frac{k-1}{n-2}$ in what follows.
Therefore,
\begin{equation}\label{eqn:1+f2-simplifications}
    1+f^2 = D + \tfrac{n-1}{n-k-1}(1-t^2), \qquad f^2 - \tfrac{1}{n-2} = D - \tfrac{n-1}{n-k-1} (t^2-\alpha).
\end{equation}
At a zero of $\Psi$, we use~\eqref{eqn:1+f2-simplifications} to write
\begin{align*}
ff' &= \frac{1}{A} \Bigl(f^2 - \frac{1}{n-2} \Big) = t \left( \frac{1}{t^2-\alpha} D - \frac{n-1}{n-k-1} \right), \\
0&= (1-t^2) f f'' + f (f - tf') + 1 + (1-t^2) \frac{(f')^2}{1+f^2}.
\end{align*}
Using the expression $\frac{\Psi'}{A} = t^{-1} ff' - (f')^2 - ff''$ along with equation~\eqref{eqn:odeStar}, we therefore obtain
\begin{align*}
    \frac{\Psi'}{A} &= t^{-1} f f' - (f')^2 + \frac{(f')^2}{1+f^2} + \frac{1 + f (f - tf')}{1-t^2} \\
    &= \frac{(1-2t^2)(1+f^2) - t (1-t^2) f f' }{t(1-t^2)(1+f^2)}  f f' + \frac{1}{(1-t^2)(1+f^2)} \Bigl( D + \frac{n-1}{n-k-1}(1-t^2) \Bigr)^2.
\end{align*}
Expanding the first numerator via~\eqref{eqn:1+f2-simplifications}, we obtain
\[
(1-2t^2) (1+f^2) - t (1-t^2) f f' = \frac{-\alpha-t^4 + 2 \alpha t^2 }{t^2-\alpha} D + \frac{n-1}{n-k-1}(1-t^2)^2.
\]
Consequently,
\begin{align*}
    (1-t^2)(1+f^2)\frac{\Psi'}{A} &= \left( \frac{-\alpha-t^4 + 2 \alpha t^2 }{t^2-\alpha} D + \frac{n-1}{n-k-1}(1-t^2)^2 \right) \left( \frac{1}{t^2-\alpha} D - \frac{n-1}{n-k-1} \right) \\
    & \quad + \Bigl( D + \frac{n-1}{n-k-1}(1-t^2) \Bigr)^2 \\
    &=  \frac{\alpha(\alpha-1)}{(t^2-\alpha)^2} D^2 + \frac{n-1}{n-k-1}\frac{ 1-\alpha}{t^2-\alpha} D.
\end{align*}
Writing $\alpha = \frac{k-1}{n-2}$ and $1-\alpha = \frac{n-k-1}{n-2}$, we observe the simplification
\[
(n-1)(t^2-\alpha) - \alpha(n-k-1) D = - \tfrac{n-k-1}{n-2} \bigl( (k-1) (1+f^2) - (n-1) t^2 \bigr)
\]
by using $f^2 = D + \frac{n-1}{n-k-1}(1-t^2)$.
Substituting these properties into the above computation produces~\eqref{eqn:Psi-prime-computation}.
Examining this expression, we notice that $A<0$ and $f'<0$ make 
\[
0 = \Psi(t_*) = f(f-Af') - \tfrac{1}{n-2} < f(t_*)^2 - \tfrac{1}{n-2}.
\]
Therefore, the denominator has negative sign, due to $A(t_*) < 0$, while the numerator has
\[
(k-1) (1+f^2) - (n-1) t_*^2 > (k-1) (1 + \tfrac{1}{n-2}) - (n-1) \tfrac{k-1}{n-2} = 0,
\]
which is strictly positive for $t_* < \sqrt{\frac{k-1}{n-2}}$.
We therefore obtain
\begin{equation}\label{eqn:psi-prime-sign-at-zero}
\Psi(t_*) = 0 \quad \implies \quad \on{sgn} \Psi'(t_*) = \on{sgn} \bigl( f(t_*) - \hat{f} (t_*) \bigr).
\end{equation}
By the symmetry of the conditions $\{ \Psi>0, \Psi' < 0 \}$ and $\{ \Psi < 0, \Psi' > 0 \}$, thanks to the result of Lemma~\ref{lemma:psi-detect-blow-up}, it suffices to consider the region $S_+$.
We will show that any two properties imply that the third one holds strictly; by continuity, this implies the desired condition.
Observe that if any two of the inequalities in~\eqref{eqn:s-plus-s-minus} held as equalities at some point $t_0$, then automatically $(f(t_0), f'(t_0)) = (\hat{f}(t_0), \hat{f}'(t_0))$ by the discussion of Lemma~\ref{lemma:psi-detect-blow-up}.
By ODE uniqueness, this would force $f = \hat{f}$, leading to a contradiction; therefore, at most one condition can produce equality.
We will also apply the result~\eqref{eqn:sign-Psi''-Psi}, whereby at a critical point $t_* \neq \sqrt{\frac{k-1}{n-2}}$ of $\Psi$, we have
\[
\on{sgn} \Psi''(t_*) = \on{sgn} A(t_*) \cdot \on{sgn} \Psi(t_*).
\]
Consider the first exit time from $S_+$, i.e., a first instance $t_*$ where the trajectory $f$ reaches the boundary of $S_+$ and one of the three inequalities reaches equality at $t_*$; we rule out each possibility.
\begin{enumerate}[$(i)$]
    \item If $\Psi'(t_*) = 0$, then $\Psi' < 0$ for $t<t_*$ would force $\Psi''(t_*) \geq 0$; on the other hand, the property~\eqref{eqn:sign-Psi''-Psi} with $\Psi>0$ implies that $\Psi''(t_*) < 0$, producing a contradiction.
    \item If $\Psi(t_*)=0$, with the assumptions $\Psi'<0$ and $f > \hat{f}$, then the property~\eqref{eqn:psi-prime-sign-at-zero} implies that $\Psi'(t_*) > 0$, resulting in a contradiction.
    \item If $f(t_*) = \hat{f}(t_*)$, with the assumptions $\Psi>0$ and $f>\hat{f}$ for $t<t_*$, then $f'(t_*) < \hat{f}'(t_*) < 0$ with strict inequality by ODE uniqueness.
    Since $A<0$ on $[0,\sqrt{\frac{k-1}{n-2}})$, this implies that
    \[
    - A f' < - A \hat{f}' \implies \Psi(t_*) = f (f -Af') - \tfrac{1}{n-2} < \hat{f}(\hat{f} - A \hat{f}') - \tfrac{1}{n-2} = 0.
    \]
    We again arrive at a contradiction, therefore the regions $S_{\pm}$ are preserved as desired.
\end{enumerate}
The above argument proves the claim~\eqref{eqn:s-plus-s-minus}, completing our first step.

\smallskip \noindent \textbf{Step 2:}
We now show that the solutions $f_a$ belong in the regions $S_{\pm}$ of~\eqref{eqn:s-plus-s-minus} corresponding to $\on{sgn}(a - a_{n,k})$, meaning that $f \in S_+$ for $a > a_{n,k}$ and $f \in S_-$ for $a < a_{n,k}$.
Since the regions $S_{\pm}$ are preserved along the ODE, it suffices to prove these properties are satisfied for small $t>0$.
We claim that for any $a>0$ and $k \geq 2$, we have
\begin{equation}\label{eqn:k>2-sign-property}
    \on{sgn} \Psi_a(t) = \on{sgn} ( a - a_{n,k}), \qquad \on{sgn} \Psi_a'(t) = \on{sgn} \Psi''_a(t) = \on{sgn} (a_{n,k} - a)
\end{equation}
for $t>0$ sufficiently small.
Suppressing the dependence of $\Psi, f$ on $a$, we write
\[
\Psi = f ( f - Af') - \tfrac{1}{n-2} = f \bigl(f-tf' + \tfrac{k-1}{n-2} ( t^{-1} f') \bigr) - \tfrac{1}{n-2}.
\]
We may therefore compute
\begin{equation}\label{eqn:psi'-psi''-computations}
\begin{split}
    \Psi'(t) &= ff' - t (f')^2 - t f f'' + \tfrac{k-1}{n-2} f' (t^{-1}f') + \tfrac{k-1}{n-2} f (t^{-1} f')', \\
    \Psi''(t) &= - 3 t f' f'' - t f f^{(3)} + \tfrac{k-1}{n-2} f'' (t^{-1} f') + 2 \tfrac{k-1}{n-2} f' (t^{-1} f')' + \tfrac{k-1}{n-2} f (t^{-1} f')''.
\end{split}
\end{equation}
Recall that $f''(0) = - \frac{n-1}{k} a$, as obtained in Lemma~\ref{lemma:values-at-zero}.
For small $t>0$, we may expand $f_a(t)= a - \frac{n-1}{2k} a t^2 + \frac{1}{24} f^{(4)}_a(0) t^4 + O(t^6)$ in order to express
\[
(t^{-1} f')(0) = - \tfrac{n-1}{k} a, \qquad (t^{-1} f')'(0) = 0, \qquad (t^{-1} f')''(0)  = \tfrac{1}{3} f^{(4)}_a(0).
\]
We therefore obtain $\Psi'(0) = 0$ and
\begin{align*}
    \Psi(0) &= a \Bigl(a - \frac{k-1}{n-2} \frac{n-1}{k} a \Bigr) - \frac{1}{n-2} \\
    &= a^2 \frac{(n-2)k - (k-1)(n-1)}{k (n-2) } - \frac{1}{n-2} \\
    &= \frac{n-k-1}{(n-2) k} (a^2 - a^2_{n,k}), \\
    \Psi''(0) &=  \frac{k-1}{n-2} \cdot (- \frac{n-1}{k} a)^2 + \frac{k-1}{n-2} a \cdot \frac{1}{3} f_a^{(4)}(0) \\
    &= \frac{2(k-1)(n-1)(n-k-1)^2}{(n-2) k^3 (k+2)} \frac{a^2}{1+a^2} \Bigl( \frac{k}{n-k-1} - a^2  \Bigr) \\
    &=\frac{2(k-1)(n-1)(n-k-1)^2}{(n-2) k^3 (k+2)} \frac{a^2}{1+a^2} ( a_{n,k}^2 - a^2).
\end{align*}
If $k \geq 2$, then the prefactor of $\Psi''(0)$ is non-negative, leading to $\on{sgn} \Psi(0)  = -\on{sgn} \Psi''(0) = \on{sgn} (a - a_{n,k})$; this proves the desired sign for $t>0$ small.
This proves~\eqref{eqn:k>2-sign-property}, and combined with the above discussion, shows that the property $f \in S_{\pm}$ is preserved on $[0, \sqrt{\frac{k-1}{n-2}})$.
We conclude that $\on{sgn} \Psi' = - \on{sgn} (a - a_{n,k})$ on $[0 , \sqrt{\frac{k-1}{n-2}})$.
The above discussion shows that
\[
\on{sgn} ( f - \hat{f}) ( \sqrt{\tfrac{k-1}{n-2}}) = \on{sgn}(a - a_{n,k}).
\]
Since $\Psi_a( \sqrt{\frac{k-1}{n-2}}) = f_a( \sqrt{\frac{k-1}{n-2}})^2 - \frac{1}{n-2}$, this proves the first part of the Proposition on $\on{sgn} \Psi_a ( \sqrt{\frac{k-1}{n-2}})$.
Finally, we recall that
\[
\Psi'(t) = A(t) ( t^{-1} f f' - (f')^2 - ff''), \qquad \Psi''( \sqrt{\tfrac{k-1}{n-2}} ) = 2 ( t^{-1} f f' - (f')^2 - ff'')|_{t = \sqrt{\frac{k-1}{n-2}}}.
\]
Since $A(t) < 0$ for $t < \sqrt{\frac{k-1}{n-2}}$, we obtain
\[
\on{sgn} ( t^{-1} f f' - (f')^2 - ff'') = - \on{sgn} \Psi'(t) = \on{sgn} (a - a_{n,k}),
\]
whereby sending $t \uparrow \sqrt{\frac{k-1}{n-2}}$ shows that
\[
\Psi''(\sqrt{\tfrac{k-1}{n-2}} ) \geq 0 \quad \text{if } \; a > a_{n,k}, \qquad \Psi''(\sqrt{\tfrac{k-1}{n-2}}) \leq 0 \quad \text{if } \; a < a_{n,k}.
\]
If either inequality holds strictly, then the desired property follows from $\Psi'(\sqrt{\frac{k-1}{n-2}}) = 0$.
If $\Psi''( \sqrt{\frac{k-1}{n-2}}) = 0$, then the second part of Lemma~\ref{lemma:psi-detect-blow-up} shows that $\on{sgn} \Psi'''( \sqrt{\frac{k-1}{n-2}}) = \on{sgn} \Psi( \sqrt{\frac{k-1}{n-2}})$.
We therefore obtain
\[
\on{sgn} \Psi'( \sqrt{\tfrac{k-1}{n-2}} + \ve) = \on{sgn} \Psi''( \sqrt{\tfrac{k-1}{n-2}} + \ve) = \on{sgn} \Psi( \sqrt{\tfrac{k-1}{n-2}})
\]
for $\ve>0$ sufficiently small.
This completes the proof for $k \geq 2$.

\smallskip \noindent \textbf{Step 3.}
Finally, we let $k=1$, so the desired property becomes
\[
\on{sgn} \Psi(\ve) = \on{sgn} \Psi'(\ve) = \on{sgn} (a - a_{n,k}).
\]
Specializing the computation of $\Psi''$ from~\eqref{eqn:psi'-psi''-computations} to $k=1$ and differentiating twice, we obtain
\begin{align*}
    \Psi^{(3)}(t) &= - 3 f' f'' - 3 t (f'')^2 - 4 t f' f^{(3)} - f f^{(3)} - t f f^{(4)}, \\
    \Psi^{(4)}(t) &= - ( 2 f f^{(4)} + 8 f' f^{(3)} + 6 (f'')^2 ) - t ( f f^{(5)} + 5 f' f^{(4)} + 10 f'' f^{(3)} ), \\
    \Psi^{(4)}(0) &= - 6 f''(0)^2 - 2f(0)f^{(4)}(0) \\
&= \frac{4 (n-1) (n-2)^2 a^2}{1+a^2} \Bigl( a^2 - \frac{1}{n-2} \Bigr).
\end{align*}
In the last step, we used the computations of $f''(0)= - (n-1) a$ and $f^{(4)}(0)$ obtained in Lemma~\ref{lemma:values-at-zero}.
Since $a^2_{n,1} = \frac{1}{n-2}$, it follows that the solution $f_a$ with initial height $f_a(0) = a$ satisfies
\[
\on{sgn} \Psi(t) = \on{sgn} ( a^2 - a_{n,1}^2), \qquad \on{sgn} \Psi'(t) = \on{sgn} \Psi''(t) = \on{sgn} \Psi^{(4)}(t) = \on{sgn} ( a^2 - a^2_{n,1} ) 
\]
for sufficiently small $t > 0$ in this case.
For small $\ve > 0$, this implies that 
\[
\on{sgn} \Psi(\ve) = \on{sgn} \Psi'(\ve) = \on{sgn} \Psi''(\ve) = \on{sgn} \Psi^{(3)}(\ve) = \on{sgn} \Psi^{(4)}(\ve) = \on{sgn} ( a - a_{n,1}),
\]
establishing the desired result.

\end{proof}

\begin{proposition}\label{prop:general-a-behavior}
    Let $a_{n,k} := \sqrt{\frac{k}{n-k-1}}$ and consider the solution $f_a$ of equation~\eqref{eqn:odeStar} with initial $f_a(0) = a>0$ and $f_a'(0) = 0$, having derivative blow-up at $b_a <1$.
    \begin{enumerate}[$(i)$]
        \item For $a < a_{n,k}$, we have $f_a(b_a) < 0$, meaning that the function $f_a$ has a positive root $t_a < b_a$, reaching zero with contact angle $< \frac{\pi}{2}$ and finite derivative.
        \item For $a = a_{n,k}$, we have $f = \hat{f}_{n,k}$, the Lawson solution of Lemma~\ref{lemma:lawson-cones}.
        \item For $a > a_{n,k}$, we have $f_a(b_a) > 0$, meaning that the function does not reach zero: it stays positive until the derivative blows up.
    \end{enumerate}
\end{proposition}
\begin{proof}
    This result follows by combining Lemma~\ref{lemma:psi-detect-blow-up} with Proposition~\ref{prop:psi-at-turning-point}.
    Recall that
    \[
    \on{sgn} \Psi_a ( \sqrt{\tfrac{k-1}{n-2}} + \ve ) = \on{sgn} \Psi'_a ( \sqrt{\tfrac{k-1}{n-2}} + \ve ) = \on{sgn} ( a - a_{n,k})
    \]
    for $\ve>0$ sufficiently small, depending on $n,k,a$.
    For $a < a_{n,k}$, Lemma~\hyperref[lemma:psi-detect-blow-up]{\ref{lemma:psi-detect-blow-up} $(i)$} shows that $f$ reaches zero with contact angle $< \frac{\pi}{2}$ and finite derivative.
    For $a>a_{n,k}$, Lemma~\hyperref[lemma:psi-detect-blow-up]{\ref{lemma:psi-detect-blow-up} $(ii)$} shows that $f$ has derivative blow-up while positive.
\end{proof}

Using the results of this section on the properties of $f_a$, we now prove Theorem~\ref{thm:capillary-cones}.
\begin{proof}[Proof of Theorem~\ref{thm:capillary-cones}]
We now combine the above results to establish the existence and uniqueness of the graphical capillary cones $\mathbf{C}_{n,k,\theta}$ with $O(n-k) \times O(k)$ symmetry.
We express such cones as the graphs of functions $U_{n,k,\theta}(x,y) := \rho f_{n,k,\theta} ( \frac{|y|}{\rho}) $ over the domain $\Gamma_{n,k,\theta} := \{ (x,y) : |y| \leq t_{n,k,\theta} \rho \}$, as in~\eqref{eqn:capillary-domain}.
The function $f_{n,k,\theta}$ produces a capillary cone of angle $\theta$ if and only if it reaches zero at $t_{n,k,\theta}$ and satisfies~\eqref{eqn:f'-at-theta}.
By Proposition~\ref{prop:general-a-behavior}, the function $f_a$ solving~\eqref{eqn:odeStar} reaches zero if and only if $a \leq a_{n,k} = \sqrt{\frac{k}{n-k-1}}$, forming a capillary angle $\theta_a = \arctan ( \sqrt{1-t_a^2} \, |f'_a(t_a)| )$, where $t_a$ denotes the unique positive root of the function $f_a$.

Proposition~\ref{prop:unique-crossing of solutions} and Corollary~\ref{corollary:lambda-comparison} show that $\hat{t}_{n,k} = \sqrt{\frac{k}{n-1}} < t_a < t_{n,k}=: t_0$, so $\sqrt{1-t_a^2} > \sqrt{1 - t_{n,k}^2} > 0$.
The terminal slope function $a \mapsto |f'_a(t_a)|$ has $|f'_a(t_a)| \to 0$ as $a \downarrow 0$ and $|f'_a(t_a)| \to \infty$ as $a \uparrow a_{n,k}$, so it surjects onto $\bR_+$ for $a \in ( 0 , a_{n,k})$.
Therefore, the map $\theta_a$ is surjective.
Moreover, Proposition~\ref{prop:unique-crossing of solutions} shows that the map $a \mapsto t_a$ (initial height to zero) is strictly decreasing, so $\partial_a t_a < 0$ and the function $a \mapsto \sqrt{1 - t_a^2}$ is strictly increasing.
It follows from Corollary~\ref{corollary:terminal-slope-monotonicity} that $a \mapsto |f'_a(t_a)|$ is also strictly increasing, therefore $a \mapsto \theta_a$ is strictly increasing.
We conclude that $a \mapsto \theta_a$ is a bijection, meaning that the cones $\mathbf{C}_{n,k,\theta}$ exist and are unique for every contact angle $\theta \in (0, \frac{\pi}{2}]$.
\end{proof}

As a Corollary of Theorem~\ref{thm:capillary-cones}, we obtain the uniqueness of the half-Lawson cone as $\mathbf{C}_{n,k,\frac{\pi}{2}}$.
\begin{corollary}
The Lawson cone $C(\bS^{n-k-1} \times \bS^k)$ is the unique cohomogeneity two minimal cone in $\bR^{n+1}$, equivariant under $O(n-k) \times O(k) \times \bZ_2$, that is expressible as a bigraph.
\end{corollary}

\begin{remark}\label{rmk:distinguished-jacobi-field}
    Following Theorem~\ref{thm:capillary-cones} and the subsequent discussion in Remark~\ref{rmk:remark-after-existence}, we observe that the cones $\mathbf{C}_{n,k,\theta}$ carry a distinguished Jacobi field $J_{n,k,\theta}$ produced by varying the capillary angle $\theta$.
    The monotonicity of $a \mapsto t_a$ and $a \mapsto \theta_a$ (cf.~Proposition~\ref{prop:unique-crossing of solutions} and Corollary~\ref{corollary:terminal-slope-monotonicity}) allows us to work in the $a$-framework, letting $U_a(x,y) := U_{n,k,a}(x,y) = \rho f_a(t)$, for brevity.
    Then, the infinitesimal variation through the family is given by the vertical field $\partial_a U_a = \rho v_a(t)$, for $v_a := \partial_a f_a$ as in~\eqref{eqn:linearized-va-vlambda}.
    The associated Jacobi field (normal speed) on the hypersurface $\mathbf{C}_a := \text{graph}(U_a)$ is then
    \[
    J_{n,k,a} := \frac{\partial_a U_a}{\sqrt{1 + |\nabla U_a|^2}} = \frac{\rho v_a(t)}{\sqrt{1 + f_a(t)^2 + (1-t^2) f'_a(t)^2}}
    \]
    which produces a $O(n-k) \times O(k)$-invariant conical ($1-$homogeneous) Jacobi field that restricts to a function on the link.
    The property $P_{f_a} v_a = 0$ for $v_a$ is equivalent to $(\Delta_{\mathbf{C}_a} + |A_{\mathbf{C}_a}|^2) J_a = 0$, meaning that $J_a$ satisfies the Jacobi equation on the interior.
    The properties $\partial_a t_a = - \frac{v_a(t_a)}{f'_a(t_a)} < 0$, and $v_a(0) = 1, v_a(t_a) < 0$ imply that $J_a$ changes sign, from positive near the cone axis to negative at the free boundary, hence $J_a$ simultaneously changes the cone profile, the free-boundary aperture $t_a$, and the contact angle $\theta(a)$.
    Denoting by $\Sigma_{\theta} := \mathbf{C}_{\theta} \cap \bS^n_+$ the link of the cone, we see that $J_a$ satisfies the forced Robin condition
    \[
    \bigl( \partial_{\nu} - \cot \theta \, A_{\mathbf{C}_\theta} (\nu,\nu) \bigr) J_{\theta} = - 1, \qquad (\partial_{\nu} - \cot \theta(a) \, A_{\mathbf{C}_a} (\nu, \nu) \bigr) J_a = - \theta'(a)
    \]
    along the free boundary $\partial \Sigma_{\theta}$ (resp.~$\partial \Sigma_a$) in the $\theta$-variation (resp.~$a$-variation) of the cones.
\end{remark}

\section{Minimality for small contact angle}\label{sec:minimizing}

We now study the area-minimizing properties of the cones constructed in Section~\ref{section:o(n-k)-cones}.
In this Section, we prove Theorems~\ref{thm:cones-are-minimizing-smalltheta} and~\ref{thm:one-phase-cones}, establishing the minimality of the cones $\mathbf{C}_{n,k,\theta}$ for small contact angle in every dimension $n \geq 7$.
Our argument utilizes a foliation by sub- and supersolutions, which will enable us to apply the techniques developed in Section~\ref{subsection:foliations-minimality}.

The main result of this section can be summarized as follows:
\begin{theorem}
For every $n \geq 7$ and $1 \leq k \leq n-2$, there exists a $\theta_{n,k}$ such that for all $\theta \in (0,\theta_{n,k})$, the cone $\mathbf{C}_{n,k,\theta}$ is strictly stable and one-sided minimizing for the $\cA^{\theta}$-capillary energy.

Suppose, moreover, that there exists a $\beta_{n,k} \in (2-n,-1)$ satisfying the conditions~\eqref{eqn:S'1-new}~--~\eqref{eqn:S'3-new} of Proposition~\ref{prop:supersolution}.
Then, there exists a $\theta_{n,k}$ such that for all $\theta \in (0,\theta_{n,k})$, the cone $\mathbf{C}_{n,k,\theta}$ is minimizing for the $\cA^{\theta}$-capillary energy.
\end{theorem}
The conditions~\eqref{eqn:S'1-new}~--~\eqref{eqn:S'3-new} are inequalities for hypergeometric functions.
In the present paper, we verify these properties by direct computation for $7 \leq n \leq 100$ and the $k$ indicated in Theorem~\ref{thm:cones-are-minimizing-smalltheta}.
As explained in~\eqref{eqn:betaInductive}, for larger $n$, the range of admissible $\beta_{n,k}$ becomes larger and easier to verify; the strictest cases are $(n,k) = (7,1)$ and $(7,2)$.

Our construction of sub- and supersolutions which foliate both sides of the cone enable us to prove that the constructed capillary cones are minimizing, using the techniques developed in Section~\ref{subsection:foliations-minimality}.
The subsolution will be an interior perturbation of the cone solution $U_{n,k,\theta}$, defined over a sub-domain of $\{ (x,y) : |y| \leq t_{n,k,\theta} \rho \}$ disjoint from the origin.
The supersolution is defined over a domain containing a full neighborhood of the origin.
Rescaling these solutions produces lower and upper foliations of the capillary cone, which imply its minimality by Lemma~\ref{lemma:capillary-foliations-imply-minimizing}.
Our approach is motivated by the energy-minimizing De~Silva-Jerison cone for the one-phase problem~\cite{desilva-jerison-cones}.

For $n \geq 7, 1 \leq k \leq n-2$, and capillary angles satisfying $\theta \in (0, \theta_{n,k})$, we will construct a subsolution $V_{\theta}$ and a supersolution $W_{\theta}$ for~\eqref{eqn:capillary-fbp}, which satisfy
\[
V_{\theta} < U_{n,k,\theta} = \rho f_{n,k,\theta}(t) < W_{\theta}, \qquad \{ V_{\theta} > 0 \} \subset \{ U_{n,k,\theta} > 0 \} \subset \{ W_{\theta} > 0 \}.
\]
We will therefore prove that the cones $\mathbf{C}_{n,k,\theta}$ are minimizing for $\cA^{\theta}$.
The positive phases of the corresponding functions are depicted in Figure~\ref{fig:subsupersolutiondomains} below.
\begin{figure}
    \centering
    \includegraphics[scale =0.7]{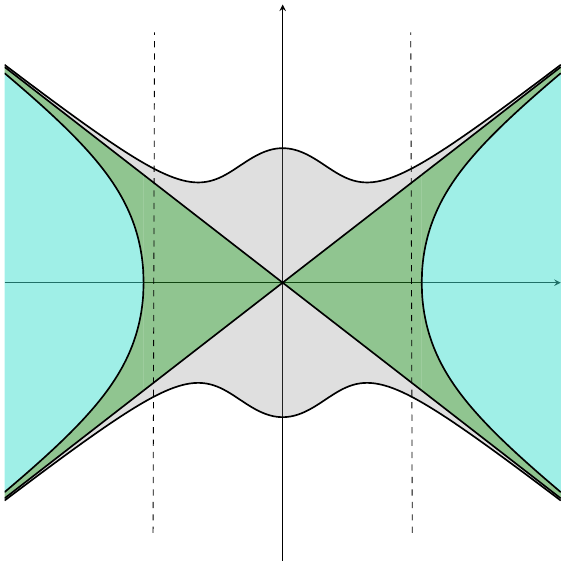}
    \caption{The nested positive phases of the subsolution (blue), the solution (green), and the supersolution (gray).
    This picture corresponds to $(n,k) = (7,2)$, small angle $\theta$, and parameters $\alpha = -3, \beta = -2.5$, for which $(t_0, \tau, \bar{r}, A) = (0.688, 0.731, 0.932, 1.514)$.
    The subsolution, as a perturbation at infinity of the capillary solution, has its boundary asymptotic to the cone.
    The supersolution is defined piecewise, on a compact set near the singularity, and at infinity, where it is again asymptotic to the cone.
    Along the compact piece, the free boundary is given by $s = v(r)$, which corresponds to the curve $t = \frac{1-\frac{\bar{r}^2-r^2}{2A}}{\sqrt{\left(1-\frac{\bar{r}^{2}-r^{2}}{2A}\right)^{2}+r^{2}}}$ in the $(\rho,t)$ coordinates.
    The dashed lines illustrate the interface ${ |x| = \bar{r} }$ along which the two pieces of the supersolution are glued.}
    \label{fig:subsupersolutiondomains}
\end{figure}

When working in the perturbative regime near a given solution, we will estimate the mean curvature of graphs with small gradient as follows.
Note that for any function $u$, we may bound
\begin{align*}
    - |\nabla U|^2 |D^2 U| \leq Q(U) \leq |\nabla U|^2 |D^2 U| \qquad \text{and} \qquad - n |D^2 U| \leq \Delta U \leq n |D^2 U |,
    \end{align*}
where $|D^2 U|$ denotes the operator norm (norm of the largest eigenvalue) of the Hessian.
Therefore,
\begin{align*}
    \left| (1 + |\nabla U|^2)^{\frac{3}{2}} \cM(U) - \Delta U \right| &= \left| |\nabla U|^2 \Delta U - Q(U) \right| \leq n |D^2 U| \cdot |\nabla U|^2 + |\nabla U|^2 |D^2 U|,
\end{align*}
which allows us to bound 
\begin{equation}\label{eqn:MC-bound}
     \left| (1 + |\nabla U|^2)^{\frac{3}{2}} \cM(U) - \Delta U \right| \leq (n + 1) |D^2 U| \cdot |\nabla U|^2.
\end{equation}

\subsection{Construction of a subsolution}\label{subsection:subsolution}

We will construct a subsolution $V$ to~\eqref{eqn:capillary-fbp} as a lower-order perturbation of the cone-defining function $\rho f(t)$.
For any $\Lambda>0$, we define
\begin{equation}\label{eqn:V-subsolution}
    V_{\theta} = \rho f_{\theta}(t) - \Lambda \tan \theta \rho^{\alpha} g(t)
\end{equation}
where $f_{\theta} = f_{n,k,\theta}$ is the unique even solution of equation~\eqref{eqn:odeStar} with initial height $a = a_{n,k,\theta}$ such that the resulting capillary angle is equal to $\theta$.
Moreover, $g$ is the even solution of $\Delta (\rho^{\alpha'} g) = 0$, normalized by $g(0) = 1$; $\alpha = \alpha(n,k,\theta) < 1$ will be determined later, with $\alpha \in (2-n, 0)$ and $|\alpha_{n,k,\theta} - \alpha_{n,k,0}| = O (\theta)$; and the exponent $\alpha'_{n,k,\theta}$ is such that $|\alpha'_{n,k,\theta} - \alpha_{n,k,\theta}| = O (\theta)$, and
\[
c(n,\alpha, \alpha') = \alpha'(\alpha' + n-2) - \alpha(\alpha + n-2) > c(n,\alpha) \theta.
\]
\begin{proposition}\label{prop:subsolution}
    Let $f_0 := f_{n,k}$ denote the solution of the linear problem $\cL_{n,k} f_0 = 0$, given by the hypergeometric function ${}_2F_1(\frac{n-1}{2} , - \frac{1}{2} ; \frac{k}{2} ; t^2)$ with a zero at $t_0$.
    Suppose that $\alpha \in (2-n, 0)$ is such that the function
\begin{equation}\label{eqn:Gn,k-alpha,t}
    \begin{split}
    G_{n,k} (\alpha, t) := (1-\alpha)^2 f_0^2 + (1-t^2) (f_0' - f_0 g'/g)^2, \qquad t \in [0, t_0]
\end{split}
\end{equation}
    is minimized over $[0,t_0]$ at the endpoint $t_0$.
    Then, there exists a $\theta_{n,k} > 0$ such that for all $\theta \in (0, \theta_{n,k})$, the function $V_{\theta}$ defined in~\eqref{eqn:V-subsolution} is a strict weak subsolution of~\eqref{eqn:capillary-fbp}, meaning that
    \[
    \cM(V_{\theta}) > 0 \quad \text{in } \; \{ V_{\theta} > 0 \}, \qquad |\nabla V_{\theta}|^2 > \tan^2 \theta \quad \text{on } \; \partial \{ V_{\theta}>0\}
    \]
    and has positive phase $\{ V_{\theta} > 0 \} \subset \Gamma_{n,k,\theta}$.
\end{proposition}
\begin{proof}
Denoting by $t_{n,k,\theta}$ the unique positive zero of the cone solution $f = f_{n,k,\theta}$, the capillary condition~\eqref{eqn:f'-at-theta} gives $f'_{n,k,\theta} (t_{n,k,\theta}) = - \frac{\tan \theta}{\sqrt{1 - t_{n,k,\theta}^2}}$.
We suppress the dependence on $n,k$ in what follows and denote $t_{\theta} := t_{n,k,\theta}$; when there is no ambiguity, we will use $f$ and $f_{\theta}$ interchangeably (for $\theta \in (0,\theta_{n,k})$ sufficiently small).
Note that $|t_{n,k,\theta} - t_{n,k}| = O(\theta^2)$ and $|f_{n,k,\theta} - c_{n,k} \theta f_{n,k}| = O(\theta^3)$, as obtained in Corollary~\ref{corollary:linearized-deviation}; here, $c_{n,k} := \frac{1}{\sqrt{1 - t_{n,k}^2} \, |f'_{n,k}(t_{n,k})|}$.
We also have $\sup |D^{\ell} f| = O(\theta)$ for $\theta$ small.
Our subsequent computations are invariant under multiplying the functions by a constant, so may normalize by $c_{n,k}$ in what follows to write $f_{\theta} = \theta f_0 + O(\theta^3)$.

Using the expression~\eqref{eqn:laplace-general}, we compute
\begin{equation}\label{eqn:laplace-rho-a-g}
    \Delta (\rho^{\alpha} g) = \rho^{\alpha-2} \left( (1-t^2) g'' + \bigl( (k-1) t^{-1} - (n-1) t \bigr) g' + \alpha (\alpha + n-2) g \right).
\end{equation}
The even solution of the ODE $\rho^{2-\alpha} \Delta(\rho^{\alpha} g) = 0$ from~\eqref{eqn:laplace-rho-a-g} is given by
\begin{equation}\label{eqn:general-hypergeometric-expression}
g(t) = {}_2 F_1 \left( \frac{n + \alpha-2}{2} ,  - \frac{\alpha}{2}; \frac{k}{2} ; t^2 \right)
\end{equation}
and has a zero $t_{\alpha} \in (0,1)$ if and only if $\alpha \not\in [2-n,0]$, otherwise remains strictly positive on $[-1,1]$.
In fact, $\alpha \in (1-n,1)$ makes $g(t) > 0$ in the interval $[-t_0, t_0]$ where $f$ is positive.
In the limit, $|t_{\alpha}| \uparrow 1$ as $\alpha \to 0^{\pm}$ or $\alpha \to (2-n)^{\pm}$, and $|t_{\alpha}| \downarrow 0$ as $|\alpha| \to \infty$.
Moreover, choosing $\alpha \in (2-n, 0)$ and $|\alpha'_{\theta} - \alpha| = O(\theta)$ in the manner prescribed above ensures that $g > c(n,k)$ and $\sup |D^{\ell} g| = C(n,k,\ell)$ on $[-t_{\theta}, t_{\theta}]$.
Since $g$ satisfies $\rho^{2-\alpha'} \Delta(\rho^{\alpha'} g) = 0$, we deduce that
\[
    \rho^{2-\alpha} \Delta (\rho^{\alpha} g) = (1-t^2) g'' + ( (k-1) t^{-1} - (n-1) t ) g' + \alpha(\alpha + n-2) g \leq - c(n, \alpha) \theta \, g
\]
holds on $[-t_{\theta}, t_{\theta}]$, for a small constant $c(n,\alpha)$ obtained above.

Using $g > c(n,k)$ on $[- t_{\theta}, t_{\theta}]$, we deduce that $V_{\theta} < \rho f_{\theta}(t)$ for all $t \in [-t_{\theta}, t_{\theta}]$.
Therefore, $V_{\theta}(\rho,t) < 0$ for all $|t| > t_{\theta}$, meaning that $\{ V_{\theta} > 0 \} \subset \Gamma_{\theta}$ and we may only consider $t \in [-t_{\theta}, t_{\theta}]$.
On the domain $\{ V_{\theta} > 0 \}$ we have $f > \Lambda \tan \theta \rho^{\alpha - 1} g$, so $\rho^{\alpha-1} < \frac{f}{\Lambda \tan \theta g}$.
Noting that $g>c(n,k)$ and $|f| < C(n,k) \theta$ on $[-t_{\theta}, t_{\theta}]$, for $\theta$ sufficiently small, we obtain $\rho^{\alpha-1} < C(n,k)$ in $\{ V_{\theta} > 0 \}$.
We may therefore apply the notation and the computations of Lemma~\ref{lemma:general-computation-u-rho,t} to bound
\begin{align*}
    |\nabla V_{\theta}|^2 &= (f - \Lambda \tan \theta \alpha \rho^{\alpha-1} g)^2 + (1-t^2) (f' - \Lambda \tan \theta \rho^{\alpha-1} g')^2 \\
    &= |\nabla ( \rho f)|^2 + \rho^{\alpha-1} O(\theta^2), \\
    Q(V_{\theta}) &= Q(\rho f) + \rho^{-1} O(\theta^3).
\end{align*}
The mean curvature computation~\eqref{eqn:minimal-surface-urho,t} combined with the Laplacian~\eqref{eqn:laplace-rho-a-g} therefore becomes
\begin{equation}\label{eqn:MC-computation+error}
\begin{split}
    (1 + |\nabla V_{\theta}|^2)^{\frac{3}{2}} \cM(V_{\theta}) &= - \Lambda \tan \theta \, \Delta (\rho^{\alpha} g) + \rho^{\alpha-2} O (\theta^3) \\
    &= \rho^{\alpha-2} ( - \Lambda \tan \theta \rho^{2-\alpha} \Delta (\rho^{\alpha} g) + O(\theta^3) ) \\
    &\geq \rho^{\alpha-2} \theta ( c(n,k) \Lambda \tan \theta - c(n,k,\alpha) \theta^2),
\end{split}
\end{equation}
since $g>c(n,k)>0$ on $[-t_{\theta}, t_{\theta}]$.
We conclude that $\cM(V_{\theta}) > 0$ on $\{ V_{\theta} > 0 \}$, for $\theta \in (0,\theta_{n,k})$.

The free boundary of the function $V_{\theta}$ is the domain $\partial \{ V_{\theta} > 0 \} = \{ (\rho, t) : \rho^{1-\alpha} = \frac{\Lambda \tan \theta g}{f} \}$.
In this region, the computation from Lemma~\ref{lemma:general-computation-u-rho,t} simplifies to
\begin{equation}\label{eqn:nabla-V-squared}
    |\nabla V_{\theta}|^2 = G_{\theta}(\alpha, t) := (1-\alpha)^2 f^2 + (1-t^2) (f' - fg'/g)^2, \qquad t \in [0, t_{\theta}]
\end{equation}
Writing the condition~\eqref{eqn:capillary-condition-graph} as $f'_{\theta}(t_{\theta}) = - \frac{\tan \theta}{\sqrt{1-t_{\theta}^2}}$ and using $f_{\theta}(t_{\theta}) = 0$, we obtain
\[
G_{\theta}(\alpha, t_{\theta}) = (1-t_{\theta}^2) f_{\theta}'(t_{\theta})^2 = \tan^2 \theta
\]
so the subsolution condition $|\nabla V_{\theta}| \geq \tan \theta$ from~\eqref{eqn:capillary-fbp} is equivalent to $G_{\theta}(\alpha, t)$ attaining its absolute minimum on $[0, t_{\theta}]$ at the endpoint $t_{\theta}$.
Finally, recall that $t_{\theta} < t_0$ as well as $|t_{\theta} - t_0| = O(\theta^2)$ and $|f_{\theta} - \theta f_0| = O(\theta^3)$, by Corollary~\ref{corollary:linearized-deviation}.
If $G_{n,k}(\alpha, t)$ attains its absolute minimum on $[0, t_0]$ at the endpoint $t_0$, for small $\theta \in (0,\theta_{n,k})$ we can choose $|\alpha_{\theta} - \alpha| = O(\theta)$ to ensure that $G_{\theta}(\alpha_{\theta}, t)$ also attains its absolute minimum at $t_{\theta}$.
We conclude that $|\nabla V_{\theta}| \geq \tan \theta$ along $\partial \{ V_{\theta} > 0 \}$, so $V_{\theta}$ is a subsolution to~\eqref{eqn:capillary-fbp} in its positive phase, for $\theta \in (0,\theta_{n,k})$ small.
\end{proof}

Therefore, to produce a subsolution $V_{\theta}$ of~\eqref{eqn:capillary-fbp} for small $\theta \in (0, \theta_{n,k})$ it remains to exhibit some $\alpha_{n,k}$ such that the function $G_{n,k}(\alpha,-)$ of~\eqref{eqn:Gn,k-alpha,t} attains its minimum in $[0, t_0]$ at the endpoint $t_0$.
The computation of Lemma~\ref{lemma:values-at-zero} shows that $f''(0) = - \frac{n-1}{k} f(0)$, and using~\eqref{eqn:laplace-rho-a-g} for $g$ satisfying $\Delta(\rho^{\alpha'} g) = 0$ implies that $g''(0) = - \frac{\alpha' (\alpha' + n-2)}{k} g(0)$.
In the sharp threshold $\alpha' = \alpha$, we obtain
\begin{align*}
    G''_{n,k}(\alpha, 0) &= 2 (1-\alpha)^2 f f'' + 2 (f'' - f g''/g)^2 \\
    &= \frac{2}{k^2}(1-\alpha)^2 f(0)^2 \left( (\alpha + n-1)^2 - (n-1) k \right).
\end{align*}
Since $G_{n,k}(\alpha, t)$ is an even function of $t$, this computation shows that $t=0$ is a local maximum of $G_{n,k}(\alpha, t)$ if $(\alpha + n-1)^2 < (n-1) k$, and is a local minimum if $(\alpha + n-1)^2 > (n-1) k$.

The boundary minimum condition for $G_{n,k}(\alpha, t)$ is the first instance where the dimensional restriction $n \geq 7$ appears: a subsolution alone proves that these cones are strictly stable.
In an upcoming paper, we prove that for dimension $n\leq 6$ and any $k$ and contact angle, the cones $\mathbf{C}_{n,k,\theta}$ are indeed unstable.
We also show that for dimension $n \geq 7$ and any $k, \theta$, the cones $\mathbf{C}_{n,k,\theta}$ are all stable.
In particular, for $n \geq 7$, we now prove that such an $\alpha$ can always be produced, for any $1 \leq k \leq n-2$.
For $t_0$ to be the minimum point of $G_{n,k}$ on $[0,t_0]$, we require $G'_{n,k}(\alpha, t_0) \leq 0$.
Using the expression~\eqref{eqn:Gn,k-alpha,t} along with $f_0(t_0) = 0$, we compute that $G'_{n,k}(\alpha, t_0) < 0$ is equivalent to 
\begin{equation}\label{eqn:strict-stability-condition}
\frac{g'(t_{n,k})}{g(t_{n,k})} > \frac{(n-2) t_{n,k} - (k-1) t_{n,k}^{-1}}{1 - t_{n,k}^2}.
\end{equation}
The existence of some $\alpha \in (2-n,0)$ satisfying~\eqref{eqn:strict-stability-condition} is equivalent to the strict stability of the corresponding solution $U_{n,k}$ of the one-phase problem.
In our companion paper~\cite{FTW-stability-one-phase}*{Theorem 1.1}, we prove that for any $1 \leq k \leq n-2$, these solutions are strictly stable when $n \geq 7$ and unstable when $n \leq 6$.
See~\cite{FTW-stability-one-phase}*{Proposition~2.4} for the derivation of the stability condition~\eqref{eqn:strict-stability-condition} for the solutions $U_{n,k}$.
We also refer the reader to~\cites{caffarelli-jerison-kenig, jerison-savin, hong-singular}.

Let us summarize the results from~\cite{FTW-stability-one-phase} that are needed for verifying the global subsolution property in our setting.
\begin{lemma}\label{lemma:one-shot-reference-FTW}
    For some $n \geq 7$ and $1 \leq k \leq n-2$, let $g(t)$ be as in~\eqref{eqn:general-hypergeometric-expression}, with $\alpha \in (2-n,0)$.
    Let $t_{n,k} > 0$ be the zero of $f_{n,k}(t) = {}_2F_1 ( \frac{n-1}{2} , - \frac{1}{2} ; \frac{k}{2} ; t^2)$, and consider the condition
    \begin{equation}\label{g'-g-condition}
        \frac{g'(t)}{g(t)} - \frac{(n-2) t - (k-1) t^{-1}}{1-t^2} > 0 \qquad \text{on } \; [0,t_k].
    \end{equation}
    For any given $\alpha \in (2-n,0)$, the property~\eqref{g'-g-condition} holds provided that it holds at $t_k$, meaning that the condition~\eqref{eqn:strict-stability-condition} is satisfied.
    Moreover,~\eqref{eqn:strict-stability-condition} and~\eqref{g'-g-condition} both hold for $\alpha = 4-n$.
\end{lemma}
\begin{proof}
    The fact that~\eqref{eqn:strict-stability-condition} implies~\eqref{g'-g-condition} is obtained as the ``propagation of positivity'' result~\cite{FTW-stability-one-phase}*{Lemma~3.1}.
    The fact that~\eqref{eqn:strict-stability-condition} holds for $\alpha = 4-n$ is established as~\cite{FTW-stability-one-phase}*{Proposition~1.2}.
\end{proof}

With those auxiliary properties, we are ready to verify that our construction indeed produces a subsolution for all $(n,k)$.
The proof of this result is split into two cases: $n \geq 8$, which admits a simple general argument, and $n=7$, which requires additional care.
\begin{proposition}\label{prop:subsolution-construction}
For every $n \geq 8, 1 \leq k \leq n-2$, and $\alpha \in (2-n, 0)$, we define the functions
\[
g(t) := {}_2 F_1 \left( \frac{n+\alpha-2}{2} , - \frac{\alpha}{2} ; \frac{k}{2} ; t^2 \right), \qquad G_{n,k}(\alpha, t) := (1-\alpha)^2 f_0(t)^2 + (1-t^2) (f'_0 - f_0 g'/g)^2.
\]
Let $t_0 \in (0,1)$ denote the positive zero of $f_0(t)$.
We define $\alpha_{n,k}$ for $n \geq 7$ and $1 \leq k \leq n-2$ by
    \begin{equation}\label{eqn:alpha-values}
   \alpha_{8,k} = - \tfrac{9}{2},  \qquad \alpha_{9,k} = - \tfrac{11}{2}, 
    \qquad \alpha_{n,k}= 4-n \; \text{ for }  \;n \geq 10.
    \end{equation}
    Then, $G_{n,k}(t) := G_{n,k}(\alpha_{n,k},t)$ attains its absolute minimum on $[0,t_0]$ at the endpoint $t_0$.
\end{proposition}

\begin{proof}
    We will prove that the function $G_{n,k}$ is strictly decreasing on $[0,t_0]$ for the chosen $\alpha_{n,k}$, from which the desired property will follow.
    This claim is established in three steps.

\smallskip \noindent \textbf{Step 1:} 
We prove that the $\alpha_{n,k}$ from~\eqref{eqn:alpha-values} satisfy~\eqref{eqn:strict-stability-condition}.
For $n \in \{ 8, 9\}$, we compute that
\begin{align*}
    \frac{g'_{8,k,-\frac{9}{2}}(t_{8,k})}{g_{8,k,-\frac{9}{2}}(t_{8,k})} - \frac{6 t_{8,k} - (k-1) t_{8,k}^{-1}}{ 1 - t_{8,k}^2} &> 10^{-2} & \qquad &\text{for } \; 1 \leq k \leq 6, \\
    \frac{g'_{9,k,-\frac{11}{2}}(t_{9,k})}{g_{9,k,-\frac{11}{2}}(t_{9,k})} - \frac{7 t_{9,k} - (k-1) t_{9,k}^{-1}}{ 1 - t_{9,k}^2} &> 10^{-1} & \qquad &\text{for } \; 1 \leq k \leq 7,
\end{align*}
which shows that the property~\eqref{eqn:strict-stability-condition} is satisfied for $n \in \{ 8, 9 \}$.
For $n \geq 10$ and $\alpha_{n,k} = 4-n$, the property~\eqref{eqn:strict-stability-condition} is obtained as~\cite{FTW-stability-one-phase}*{Proposition~1.2} of our companion paper.
Thus, the $\alpha_{n,k}$ satisfy~\eqref{eqn:strict-stability-condition} for all $n \geq 8$.

\smallskip \noindent \textbf{Step 2:} Next, we reduce the analysis of $G_{n,k}$ to the derivative $\Phi = \frac{1}{2 f_0^2} G'_{n,k}$.
    Let us denote $v_f := \frac{f'}{f}$ and $v_g := \frac{g'}{g}$, since $f>0$ on $[0,t_0)$ and $g>0$ on $[0,1]$ for $\alpha \in (2-n,0)$.
    Since $g$ satisfies~\eqref{eqn:laplace-rho-a-g}, writing $\frac{g''}{g} = v'_g + v_g^2$ shows that the function $v_g$ solves the Riccati ODE
    \begin{equation}\label{eqn:riccati-equation}
    (1-t^2) (v'_g + v_g^2) + ( (k-1) t^{-1} - (n-1) t) \cdot v_g + \alpha (\alpha + n - 2) = 0
    \end{equation}
    and likewise for $f_0$, with $\alpha = 1$.
    We may therefore write
    \begin{align*}
        G_{n,k} &= f_0^2 \left[ (1-\alpha)^2 + (1-t^2) (v_f - v_g)^2 \right], \qquad G'_{n,k} = 2f_0^2 \Phi, \qquad \text{where} \\
        \Phi(t) &= - t (v_f - v_g)^2 + (1-t^2) (v_f - v_g) (v'_f - v'_g) + v_f \left[ (1-\alpha)^2 + (1-t^2) (v_f - v_g)^2 \right].
    \end{align*}
    Combining the Riccati equations for $v_f$ and $v_g$, we obtain
    \[
    (1-t^2) (v_f - v_g)' + \left[ (1-t^2)(v_f + v_g) + (k-1) t^{-1} - (n-1) t \right] (v_f - v_g) + (1-\alpha)(n+\alpha - 1) = 0.
    \]
    We therefore seek to prove that $\Phi< 0$ on $(0,t_0)$, where
    \begin{equation}\label{eqn:Phi-computation}
    \Phi = (1-\alpha)^2 v_f - \left[ (1-t^2) v_g - (n-2) t + (k-1) t^{-1} \right] (v_f - v_g)^2 - (1-\alpha)(n+\alpha-1)(v_f - v_g).
    \end{equation}
    Using $v_g(0) = 0$ and $v'_g(0) = - \frac{\alpha(\alpha+n-2)}{k}$, we obtain $\Phi(0) = 0$ and
    \[
    \Phi'(0) = \frac{(1-\alpha)^2 }{k^2} \left[ (n+\alpha - 1)^2 - k (n-1) \right],
    \]
    meaning that $G''_{n,k}(\alpha,0)$ has the sign of $|n+\alpha - 1| - \sqrt{k(n-1)}$.

    \smallskip
    \noindent \textbf{Step 3:} We now prove that $G_{n,k}$ is strictly decreasing on $[0,t_0]$ for $n \geq 8$.
    Since $\alpha_{n,k}$ was checked above to have the properties~\eqref{eqn:two-proof-properties}, we may apply Lemma~\ref{lemma:one-shot-reference-FTW} to obtain the positivity property~\eqref{g'-g-condition}.
    In terms of $v_g = \frac{g'}{g}$, this result is equivalent to
\[
(1-t^2) v_g - (n-2) t + (k-1) t^{-1} > 0 \qquad \text{on } \; [0,t_0].
\]
    We may therefore bound
    \[
    \Phi (t) \leq (1-\alpha) \left[ (1-\alpha) v_f - (n+\alpha - 1) (v_f - v_g) \right] = (1-\alpha) E,
    \]
    where $E$ denotes the quantity $E := (2 - 2 \alpha - n) v_f + (n + \alpha - 1) v_g$.
    Using the Riccati equations satisfied by $v_f$ and $v_g$, we compute that $E$ satisfies the evolution equation
    \begin{equation}\label{eqn:e-evolution-equation}
    \begin{split}
E' &= - (2 - 2 \alpha -n) v_f^2 - (n+\alpha-1) v_g^2 - \frac{(k-1) t^{-1} - (n-1) t}{1 - t^2} E + k\frac{E'(0)}{1-t^2}.
    \end{split}
\end{equation}
Moreover, using $v'_f(0) = - \frac{n-1}{k}$ and $v'_g(0) = - \frac{\alpha(\alpha + n-2)}{k}$, we find
\[
E'(0) = - \frac{(n-1)(2-2\alpha-n) + \alpha(n+\alpha-1)(n+\alpha-2)}{k}.
\]
For $n \geq 8$, we obtain the derivatives $E'_{n,k}(0)$ as 
\[
E'_{8,k}(0) = - \frac{33}{8k}, \qquad E'_{9,k}(0) = -\frac{91}{8k} < 0, \qquad E'_{n,k}(0) = - \frac{(n-3)(n-10)}{k} \leq 0 \quad \text{for } \; n \geq 10,
\]
and $E'''_{10,k} = - \frac{4536}{k^2 (k+2)} < 0$, using the computations of Lemma~\ref{lemma:values-at-zero}.
This proves that $\alpha_{n,k} \in  (2-n, \frac{2-n}{2})$ produces $E'(0) \leq 0$ and $E'''(0) < 0$ in all cases, so $E(t) < 0$ for small $t>0$.
If there was a first instance $t_* \in (0, t_0)$ where $E(t_*) = 0$, then the uniqueness of solutions to ODE would force $E'(t_*) > 0$.
On the other hand, the equation~\eqref{eqn:e-evolution-equation} yields
\[
E'(t_*) = - (2 - 2 \alpha -n) v_f^2 - (n+\alpha-1) v_g^2 + k\frac{E'(0)}{1-t_*^2} < 0
\]
at such a point, where $2 - 2 \alpha - n > 0$ for $\alpha < \frac{2-n}{2}$; this produces a contradiction.
Consequently, $\Phi \leq (1-\alpha) E < 0$ implies that $G_{n,k}$ is strictly decreasing on $[0,t_0]$ for all such $n,k,\alpha$.
Therefore, it attains its absolute minimum on $[0,t_0]$ at the endpoint $t_0$, completing the proof.
\end{proof}

\begin{proposition}\label{prop:subsolution-construction-n=7}
Let $n=7$ and $1 \leq k \leq 5$.
We define $\alpha_1 = - 3.23$ and $\alpha_k = -3$ for $2 \leq k \leq 5$.
Let $f_k(t) := {}_2F_1( 3, - \frac{1}{2} ; \frac{k}{2} ; t^2)$ and $g_k(t) := {}_2 F_1 ( \frac{5 + \alpha_k}{2} , - \frac{\alpha_k}{2} ; \frac{k}{2} ; t^2)$.
Then, the function 
\[
G_k(t) = (1-\alpha_k)^2 f_k(t)^2 + (1-t^2) (f'_k - f_k g_k'/g_k)^2
\]
attains its absolute minimum on $[0,t_k]$ at $t_k$, the positive zero of $f_k(t)$.
\end{proposition}
\begin{proof}
    The argument for this case follows the same strategy as Proposition~\ref{prop:subsolution-construction}, but requires some more refined estimates.
    For the $\alpha_k$ given above, we consider the two properties
    \begin{equation}\label{eqn:two-proof-properties}
    \frac{g'_k(t_k)}{g_k(t_k)} - \frac{5 t_k - (k-1) t_k^{-1}}{1 - t_k^2} > 0, \qquad \sqrt{1 - t_k^2} \, |f'_k(t_k)| < 1- \alpha_k.
\end{equation}
    We verify that these hold by direct computation: for $1 \leq k \leq 5$ and the corresponding $\alpha_k$, we have
    \begin{align*}
        1 - \sqrt{1 - t_k^2} \, |f'_k(t_k)| - \alpha_k > 10^{-2}, \qquad
        \frac{g'_{k}(t_k)}{g_{k}(t_k)} - \frac{5 t_k - (k-1) t_k^{-1}}{ 1 - t_k^2} > 10^{-2}.
\end{align*}
    The proof of the claim proceeds in two steps.
    We first treat the case of $2 \leq k \leq 5$, where the argument of Proposition~\ref{prop:subsolution-construction} is applicable after more careful estimates.
    Then, we address the case of $k=1$, which is more delicate.

\smallskip \noindent \textbf{Step 1:} 
For $2 \leq k \leq 5$, we will prove that the function $G_k$ is strictly decreasing on $[0,t_k]$.
In view of the computation $G'_k = 2 f_0^2 \Phi_k$ in Proposition~\ref{prop:subsolution-construction}, we may use the identities~\eqref{eqn:Phi-computation} and~\eqref{eqn:e-evolution-equation} with $\alpha = -3$ and $E = v_f + 3 v_g$, whereby
\begin{align*}
\Phi_k &= 4 E -  \left[ (1-t^2) v_g - 5 t + (k-1) t^{-1} \right] (v_f - v_g)^2, \\
E' &= - v_f^2 -3 v_g^2 - \frac{(k-1) t^{-1} - 6 t}{1-t^2} E + \frac{12}{1-t^2}.
\end{align*}
Our earlier computation of $\Phi'_k(0)$ in Proposition~\ref{prop:subsolution-construction} gives $\Phi'_k(0) = \frac{48}{k^2} (3 - 2k) < 0$ for $2 \leq k \leq 5$, so $\Phi<0$ for small $t>0$.
Moreover, $v_f \to \infty$ as $t \uparrow t_k$ makes $\Phi \to - \infty$, due to
\[
(1-t^2) v_g - 5t + (k-1) t^{-1} > 0 \qquad \text{on } \; [0,t_k]
\]
by the property~\eqref{eqn:two-proof-properties}.
Recall that $g_k(t) = {}_2 F_1( 1, \frac{3}{2} ; \frac{k}{2} ; t^2)$ satisfies the identities
\[
{}_2 F_1(a,b;a;s) = (1-s)^{-b}, \qquad {}_2F_1(a,b;b;s) = (1-s)^{-a}
\]
whereby $g_{7,2}(t) = (1-t^2)^{- \frac{3}{2}}$ and $g_{7,3}(t) = (1-t^2)^{-1}$.
In these cases, we may simplify
\[
(1-t^2) v_g - 5 t + (k-1) t^{-1} = (k-1) t^{-1} - kt
\]
and we obtain $\Phi_k < 0$ on $[0,t_k]$ for $k \in \{2,3\}$ by direct computation.
Finally, for $k \in \{ 4, 5\}$, we first apply a lower bound for the function
\[
(1-t^2) v_g - 5t + (k-1) t^{-1}
\]
by the technique of~\cite{FTW-stability-one-phase}*{Lemma~3.1}.
Using the subsolution construction therein, we obtain quadratic barriers of the form $(k-1) - \beta_k t^2$ for the function
\[
t (1-t^2) \tfrac{g'_{\alpha}(t)}{g_{\alpha}(t)} - (n-2) t^2 + (k-1).
\]
Computing in the interval $[0,t_k]$, we find that $\beta_4 = \frac{37}{10}$ and $\beta_5 = \frac{43}{10}$ satisfy the subsolution conditions, whereby
\begin{align*}
    (1-t^2)v_g - 5t+ (k-1) t^{-1} &\geq (k-1) t^{-1} - \beta_k t \qquad \text{on } \; [0, t_k], \\
    \Phi_k(t) &\leq 4 E(t) - \left( (k-1) t^{-1} - \beta_k t \right) (v_f - v_g)^2.
\end{align*}
We compute that for this choice of constants $\beta_k$, 
\[
(k-1) t_k^{-1} - \beta_k t_k > 10^{-2},
\]
whereby $(k-1) t^{-1} - \beta_k t > 0$ on $[0,t_k]$.
Moreover, we may bound the term $4 E(t)$ by
\[
4 E(t) = 4 v_f + 12 v_g \leq 4 (v_g - v_f)
\]
which is equivalent to $ v_f + v_g \leq 0$.
The negativity of this expression follows from the general computation~\eqref{eqn:e-evolution-equation}, which now specializes to
\[
(v_f + v_g)' = - (v_f^2 + v_g^2) - \tfrac{(k-1) t^{-1} - 6t}{1-t^2} (v_f + v_g).
\]
This equation has the structure considered in~\eqref{eqn:e-evolution-equation}.
Using the power series expansion for the hypergeometric function and the computation of the derivatives $v^{(i)}_f(0), v^{(i)}_g(0)$, we find that
\[
v_f(t) = - \tfrac{6}{k} t - \tfrac{18}{k^2} t^3 + O(t^5), \qquad v_g(t) = \tfrac{6}{k} t - \tfrac{18}{k^2} t^3 + O(t^5)
\]
for small $t>0$.
Consequently, $(v_f + v_g)(t) = - \tfrac{36}{k^2} t^3+ O(t^5)$ is negative for small $t>0$.
At a first zero of $(v_f + v_g)$, we would obtain $(v_f + v_g)'(t_*) < 0$, which is impossible; therefore, $v_f + v_g < 0$ on $[0,t_k]$.
Combining these bounds, we arrive at
\begin{align*}
\Phi_k(t) &\leq 4(v_g - v_f) - \left( (k-1) t^{-1} - \beta_k t \right) (v_g - v_f)^2 \\
&= (v_g - v_f) \left[ 4 - \left( (k-1) t^{-1} - \beta_k t \right) (v_g - v_f) \right]
\end{align*}
and since $v_g - v_f > 0$, the property $\Phi_k(t) \leq 0$ is reduced to
\[
\tilde{\Phi}_k(t) := 4 - \left( (k-1) t^{-1} - \beta_k t \right) (v_g - v_f) \leq 0.
\]
We compute that $\tilde{\Phi}_k(0) = 4 - (k-1) \cdot \frac{12}{k} < 0$ and $\tilde{\Phi}_k(t) \to - \infty$ as $t \uparrow t_k$ due to $(k-1) t^{-1} - \beta_k t > 0$ on $[0,t_k]$.
The functions $f_k, g_k$ for $k \in \{ 4, 5 \}$ have explicit expressions (up to a constant factor)
\begin{align*}
    f_4(t) &= (1-t^2)^{- \tfrac{1}{2}} ( 1 - \tfrac{5}{4}t^2), & \qquad f_5(t) &= 15 - t^{-2} - (15 t - 6 t^{-1} - t^{-3}) \on{atanh} t, \\
    g_4(t) &= t^{-2} \bigl( (1-t^2)^{- \frac{1}{2}} - 1 \bigr), & \qquad g_5(t) &= t^{-3} ( \on{atanh} t - t).
\end{align*}
We therefore obtain
\begin{align*}
    k&=4:  &  v_f(t) &= \frac{t(6 - 5t^2)}{(1-t^2)(5t^2-4)}, \qquad \qquad  v_g(t) = \frac{2t^2 + \sqrt{1-t^2}-1}{t(1-t^2)}, \\
    k&=5: & v_f(t) &= \frac{15 t^6 \on{atanh} t - 15 t^5 - 9 t^4 \on{atanh} t + 4 t^3 - 3t^2 \on{atanh} t + 3(t -  \on{atanh} t)}{t (15 t^6 \on{atanh} t - 15 t^5 - 21 t^4 \on{atanh} t + 16 t^3 + 5 t^2 \on{atanh} t  + \on{atanh} t - t)}, \\
    & &v_g(t) &= \frac{2t(1-t^2) + t + (3t^2-1) \on{atanh} t}{t (\on{atanh} t-t)(1-t^2)}.
\end{align*}
Using these explicit expressions, we find that $\tilde{\Phi}_k(t) < -4 < 0$ on $[0, t_k]$ in both cases; in particular, $\tilde{\Phi}_k$ remains negative, hence $\Phi_k < 0$ as well.
We conclude that $G'_k = 2 f_k^2 \Phi_k < 0$ and $G_k$ is strictly decreasing on $[0,t_k]$.
Therefore, it attains its absolute minimum at $t_k$, as claimed.

\smallskip \noindent \textbf{Step 2:}
For $k=1$, we claim that $G_1$ has exactly one critical point, meaning that the function $\Phi_1 := \frac{1}{2f_1^2} G'_1$ has exactly one zero.
\[
f_{7,1}(t) = {}_2 F_1 \left( 3 , - \tfrac{1}{2} ; \tfrac{1}{2} ; t^2 \right) = - \tfrac{15}{8} t \cdot \on{atanh} t + \frac{ \frac{15}{8} t^4 - \frac{25}{8} t^2 + 1}{(1-t^2)^2}, \qquad t_1 \approx 0.517331.
\]
Taking $\alpha_1 = - 3.23$, we compute as in~\cite{desilva-jerison-cones}*{\S 3} (where $\alpha_{7,1} = - 3.21122$) to verify that the resulting function $\Phi_1$ from~\eqref{eqn:Phi-computation} has a unique zero, at $\hat{t}_1 \approx 0.304097$, with $\Phi_1' < 0$ for $t \in (\hat{t}_1, t_1)$.
Therefore, $G_1(t)$ has a unique critical point.
Next, using
\[
G_1(0) = (1-\alpha_1)^2, \qquad G_1(t_1) = (1 - t_1^2) f'_1(t_1)^2,
\]
due to $f_1(t_1) = 0$, we compute that the conditions~\eqref{eqn:two-proof-properties} are equivalent to
\[
G_1(0) > G_1(t_1), \qquad G'_1(t_1) < 0.
\]
We define the function $\tilde{G}(t) := G_1(t) - G_1(t_1)$, so the conditions~\eqref{eqn:two-proof-properties} are equivalent to $\tilde{G}(0) > 0$ and $\tilde{G}'(t_1) < 0$, where $\tilde{G}(t_1) = 0$.
    Therefore, if $G_1$ has a lower minimum than $G_1(t_1)$, then $\tilde{G}$ must become negative in $(0,t_1)$, with at least two interior zeros.
By Rolle's theorem, this would produce at least two points $0 < \tilde{t}_1 < \tilde{t}_2 < t_1$ where $G_1'(\tilde{t}_i) = \tilde{G}'(\tilde{t}_i) = 0$, which contradicts $G_1$ having at most one interior critical point.
We conclude that $G_1$ attains its minimum on the interval $[0,t_1]$ at $t_1$, completing the proof.
\end{proof}

\begin{proposition}\label{prop:small-theta-one-sided}
    For every $n \geq 7$ and $1 \leq k \leq n-2$, there exists a $\theta_{n,k} > 0$ such that for all $\theta \in ( 0 , \theta_{n,k})$, there exists a strict subsolution for the problem~\eqref{eqn:capillary-fbp} lying strictly below the cone $\mathbf{C}_{n,k,\theta}$.
    In particular, for all $\theta \in (0, \theta_{n,k})$, the capillary cone $\mathbf{C}_{n,k,\theta}$ is one-sided minimizing for $\cA^{\theta}$, and strictly stable.
\end{proposition}
\begin{proof}
    Combining Propositions~\ref{prop:subsolution},~\ref{prop:subsolution-construction}, and~\ref{prop:subsolution-construction-n=7} shows that for $\theta \in (0,\theta_{n,k})$ sufficiently small, we obtain a strict subsolution of the capillary problem~\eqref{eqn:capillary-fbp} by defining $V_{\theta}$ as in~\eqref{eqn:V-subsolution},
    \[
    V_{\Lambda, \theta}(z') := \rho f_{\theta}(t) - \Lambda \tan \theta \rho^{\alpha} g(t),
    \]
    for arbitrary $\Lambda > 0$.
    Indeed, given a subsolution $V_{\Lambda_0, \theta}$ corresponding to some $\Lambda_0>0$ as in Proposition~\ref{prop:subsolution}, its rescalings $\lambda^{-1} V_{\Lambda_0, \theta}(\lambda z')$ are also strict subsolutions, since the conditions~\eqref{eqn:capillary-fbp} are preserved under rescaling.
    We then observe that 
    \[
    \lambda^{-1} V_{\Lambda_0, \theta}(\lambda z') = \rho f_{\theta}(t) - \lambda^{\alpha-1} \Lambda_0 \tan \theta  \rho^{\alpha} g(t) = V_{\lambda^{\alpha-1} \Lambda_0 , \theta}(z')
    \]
    and varying $\lambda > 0$ recovers every $V_{\Lambda, \theta}$.
    The existence of a strict subsolution implies the strict stability of the cone $\mathbf{C}_{n,k,\theta}$, as discussed in Section~\ref{subsection:definitions}.
    Moreover, we observe that the graph of $V_{\Lambda, \theta}$ is the $\Lambda \tan \theta$-level set (for $\Lambda >0$) of the potential function
    \[
    \varphi_{\theta}(x,y,z) := \tfrac{1}{g(t)} \rho^{- \alpha} \bigl( \rho f_{\theta}(t) - z \bigr).
    \]
    Consequently, the result of this section can be restated as proving that for $\theta \in (0, \theta_{n,k})$ sufficiently small, $\varphi_{\theta}$ is the potential function for a sub-calibration $X_{\theta} := \frac{\nabla \varphi_\theta}{|\nabla \varphi_{\theta}|}$ of the side $E_-$ of the cone $\mathbf{C}_{n,k,\theta}$, satisfying the strict inequality~\eqref{eqn:one-sided-calibration}.
    Therefore, Lemma~\ref{lemma:strict-calibration} implies that the cone $\mathbf{C}_{n,k,\theta}$ is strictly minimizing for $\cA^{\theta}$ on the side $E_-$, as desired.
\end{proof}

\begin{corollary}\label{cor:one-phase-strictly-minimizing}
    For every $n \geq 7$ and $1 \leq k \leq n-2$, the one-homogeneous solution $U_{n,k} := c_{n,k} \rho f_{n,k}(t)$ of the one-phase problem is strictly minimizing from below for $\cJ$, and strictly stable.
\end{corollary}
\begin{proof}
    The strict stability of the solutions $U_{n,k}$ for all $n \geq 7$ and $1 \leq k \leq n-2$ is obtained as~\cite{FTW-stability-one-phase}*{Theorem 1.1} in our companion paper.
    In the current setting, the result is equivalent to the existence of a strict subsolution of the one-phase problem~\ref{eqn:one-phase-problem}, as discussed above.
    To prove the strict minimality of the solution $U_{n,k}$, we will apply Lemma~\ref{lemma:strictly-minimizing-one-phase} for $u_{\theta} := c_{n,k} \rho f_{\theta}$ and $u_0 := U_{n,k} = c_{n,k} \rho f_{n,k}$.
    The results of Corollary~\ref{corollary:linearized-deviation} imply that
    \[
    |f_{\theta} - \tan \theta \, f_{n,k}| + |f'_{\theta} - \tan \theta \, f'_{n,k}| = C(n) O(\theta^2)
    \]
    for $\theta \in (0,\theta_{n,k})$, meaning that the $u_i$ and $u_0$ satisfy the condition~\eqref{eqn:ui-to-u0-converge}.
    Moreover, Lemma~\ref{lemma:strict-calibration} and Remark~\ref{remark-1-proof} imply that if a capillary cone $\mathbf{C}$ admits a sub-calibration $X_-$ on the side $E_-$ satisfying $\on{div} X_- \geq \alpha R^{-1}$ on some homothetically invariant set $W \subset E_-$, then it is strictly minimizing for $\cA^{\theta}$ with constant $\kappa \geq \tfrac{1}{2} \alpha \cH^n(W \cap \bS^n_+)$.
    Let $\delta_{n,k}>0$ be sufficiently small so that $f_{\theta} > c_1(n,k) \tan \theta > 0$ for $t \in [0, \delta_{n,k}]$ and all $\theta \in ( 0 , \theta_{n,k})$, and take
    \[
    W_{\theta} := \{ (\rho, t, z) \in E_- : |t| \leq \delta_{n,k}, \; z < \tfrac{1}{2} \rho f_{\theta}(t) \}.
    \]
    It follows from Corollary~\ref{corollary:linearized-deviation} that $\cH^n(W_{\theta} \cap \bS^n_+) \geq \tfrac{1}{2} \cH^n(W_0 \cap \bS^n_+)$ holds for sufficiently small $\theta$.
    The computation of Proposition~\ref{prop:subsolution} (see~\ref{eqn:MC-computation+error}) implies that $\varphi_{\theta}$ satisfies
    \[
    \on{div} X_{\theta} = \on{div} \frac{\nabla \varphi_{\theta}}{|\nabla \varphi_{\theta}|} \geq c(n,k) \tan^2 \theta \on{dist}(p, \mathbf{C}_{n,k,\theta}) \qquad \text{on } \; \mathbf{C}_{n,k,\theta} \cap \bS^n_+
    \]
    on the link of the cone, therefore the homogeneity of $X_{\theta}$ implies that $\on{div} X_{\theta} \geq c \tan^2 \theta \frac{\on{dist}(p, \mathbf{C}_{n,k,\theta})}{R^2}$ on $E_-(\mathbf{C}_{n,k,\theta})$.
    The set $W_{\theta}$ is homothetically invariant and $\text{dist}(p, \mathbf{C}_{n,k,\theta}) > \tfrac{1}{2} \rho f_{\theta}(t) \geq c_1(n,k) \rho$ holds on $W_{\theta}$ due to $z< \tfrac{1}{2} \rho f_{\theta}(t)$ and $f_{\theta} > c_1(n,k) \tan \theta > 0$.
    Finally, $R^2 = \rho^2 (1 + f_{\theta}(t)^2) \leq 2 \rho^2$ allows us to bound
    \[
    \text{div} X_{\theta} \geq c \tan^2 \theta \frac{\on{dist}(p, \mathbf{C}_{n,k,\theta})}{R^2} \geq c \tan^2 \theta \frac{c_1(n,k) \rho}{2 \rho^2} \geq C(n,k) \tan^2 \theta R^{-1} .
    \]
    The above discussion now implies that $\mathbf{C}_{n,k,\theta}$ is strictly minimizing for $\cA^{\theta}$ on $E_-$, with constant
    \[
    \kappa_{\theta} \geq \tfrac{1}{2} C(n,k) \tan^2 \theta \cH^n(W_{\theta} \cap \bS^n_+) \geq \tfrac{1}{4} C(n,k) \cH^n(W_0 \cap \bS^n_+) \tan^2 \theta,
    \]
    whereby $\kappa_\theta \geq C(n,k) \tan^2 \theta$.
    We may therefore apply Lemma~\ref{lemma:strictly-minimizing-one-phase} to conclude that the one-phase solution $U_{n,k}$ is strictly minimizing from below for $\cJ$ with constant $\kappa \geq C(n,k) >0$, as desired.
\end{proof}

\subsection{Construction of a supersolution}\label{subsection:supersolution}

We will construct a weak, strict supersolution $W$ to the capillary problem~\eqref{eqn:capillary-fbp} over $\bR^n$, defined over a domain containing 
\[
\Gamma_{n,k,\theta} = \{ (x,y) : |y| \leq t_{n,k,\theta} \rho \}  = \Bigl\{ (x,y) : |y| < \tfrac{t_{n,k,\theta}}{\sqrt{1- t_{n,k,\theta}^2}} |x| \Bigr\}
\]
together with a neighborhood of the origin.
We construct $W_{\theta}$ in two pieces $W_{1,\theta}$ and $W_{2,\theta}$, where $W_{1,\theta}$ is defined in a region near infinity and $W_{2,\theta}$ is defined in a region near the origin.
The two pieces interface along the hypersurface $\{ |x| = \bar{r} \}$, with $\bar{r} < \frac{\sqrt{1-t^2_{n,k,\theta}}}{t_{n,k,\theta}}$ determined by the construction.
As in the case of the subsolution, the construction of the two pieces, together with their matching along the interface, depends on a single auxiliary constant: an exponent $\beta \in (2-n,-1)$.
For the convenience of the reader, we first introduce all the ingredients of the construction, and then verify each part of the supersolution property.
Throughout, we will choose $\theta \in ( 0, \theta_{n,k})$ appropriately small to absorb terms that are uniformly higher-order in $\theta$.
In what follows, we fix a pair $(n,k)$ as in Theorem~\ref{thm:cones-are-minimizing-smalltheta} and suppress the dependence of the various quantities, domains, and functions on $(n,k)$.
We write $\beta := \beta_{n,k,\theta}, \beta_0 := \beta_{n,k,0}$, and use similar notation for all the auxiliary parameters that we will introduce, with the understanding that $|\beta - \beta_0| = O(\theta)$ for $\theta$ sufficiently small.

Let $\hat{g}$ be the even solution of equation $\Delta( \rho^{\beta} \hat{g} ) = 0$, normalized by $\hat{g}(0) = 1$, hence given by the hypergeometric expression~\eqref{eqn:general-hypergeometric-expression}.
We define the function
\[
\hat{G}_{\theta}(\beta, t) := (1-\beta)^2 f_{\theta}^2 + (1-t^2) (f'_{\theta} - f_{\theta} \hat{g}'/\hat{g})^2.
\]
Suppose that $\hat{G}_{\theta}$ is strictly decreasing in a right neighborhood of the root $t_{\theta}$ of $f_{\theta}$; as computed in~\eqref{eqn:strict-stability-condition}, this property is equivalent to $\hat{G}'_{\theta}(\beta, t_{\theta}) < 0$, meaning
\[
    (n-2) t_{\theta} - (k-1) t_{\theta}^{-1} - (1 - t_{\theta}^2) \frac{\hat{g}'(t_{\theta})}{\hat{g}(t_{\theta})} < 0.
\]
Using $|f_{\theta}| \to \infty$ (from Lemma~\ref{lemma:general-behavior-prep}) and $|\hat{g}| \to \infty$ as $|t| \uparrow b_{\theta} \leq 1$ we obtain the existence of a unique $\tau = \tau(n,k,\beta) > t_{\theta}$ such that $\hat{G}_{\theta}(\beta, \tau) = \hat{G}_{\theta}(\beta, t_{\theta})$.
We choose some exponent $\beta'$ such that $|\beta' - \beta| = O(\theta)$, for which 
\[
c(n,\beta,\beta') = \beta(\beta+n-2) - \beta'(\beta'+n-2) > c(n,\beta) \theta.
\]
For a point $(x,y) \in \bR^{n-k} \times \bR^k$ we denote $r := |x|, s := |y|$.
These variables are related to $\rho, t$ by
\[
\rho = \sqrt{|x|^2 + |y|^2}, \qquad t = \tfrac{|y|}{\sqrt{|x|^2 + |y|^2}}, \qquad |x| = \rho \sqrt{1-t^2}, \qquad |y| = \rho t.
\]
We then define $W_{1,\theta}$ by 
\begin{equation}\label{eqn:W1-supersolution}
    W_{1,\theta}(r,s): = \rho f_{\theta}(t) + \frac{- f_{\theta}(\tau'_{\theta})}{(\tau'_{\theta})^{1-\beta'} \hat{g}(\tau'_{\theta})} \rho^{\beta'} \hat{g}(t) \qquad \text{over} \quad \Gamma'_{\theta} := \Bigl\{ (x,y) : |y| < \tfrac{\tau'_{\theta}}{\sqrt{1-\tau'_{\theta}{}^2}} |x| \Bigr\}
\end{equation}
where $\tau'_{\theta} = \tau_{\theta} - \ve$, for $0 < \ve < \ve(n,k,\theta,\beta)$ sufficiently small.
For brevity, we denote
\[
\Lambda = \frac{- f_{\theta}(\tau'_{\theta})}{\tan \theta (\tau'_{\theta})^{1-\beta'} \hat{g}(\tau'_{\theta})}, \qquad W_{1,\theta} = \rho f_{\theta}(t) + \Lambda \tan \theta \rho^{\beta'} \hat{g}(t).
\]
The level set $\{ W_{1,\theta} = 0 \}$ intersects $\partial \Gamma' = \{ t = \tau'_{\theta} \}$ at the points $( |x|, |y|) = ( \bar{r}, 1)$, where $\bar{r} = \frac{\sqrt{1 - (\tau'_{\theta})^2}}{\tau'_{\theta}}$ and $\rho = (\tau'_{\theta})^{-1}$.
Let us also define $H_{\theta}(\xi) := W_{1,\theta}(\bar{r}, \xi)$ for $\xi \in [0,1]$.
To construct the function $W_{2,\theta}$, we denote by $A := \left| \frac{dr}{ds} \right|_{(\bar{r}, 1)}$ the slope of the level set of $W_1(\rho, t)$ (viewed as a function of $(r,s)$) at $(\bar{r}, 1)$, computed in the $(s,r)$-plane for a graph $r = r(s)$.
We then introduce the functions
\[
v(r) := 1 - \frac{\bar{r}^2 - r^2}{2 \bar{r} A}, \qquad \varphi(r) := \sqrt{1 + A^{-2}} \frac{v(r)}{\sqrt{1 + v'(r)^2}} = \sqrt{1 + A^{-2}} \frac{v(r)}{\sqrt{1 + (\bar{r} A)^{-2}  r^2}}.
\]
With this notation, we denote $(r,s) := (|x|, |y|)$ and define the supersolution $W$ by
\begin{equation}\label{eqn:W-defined-supersolution}
W_{\theta}( x, y) := \begin{cases}
    W_{1,\theta} (r,s) & \text{in } \; \overline{\Gamma'_{\theta} \cap \{r > \bar{r} \} }, \\
    \varphi ( r ) \, H_{\theta} \left( \frac{s}{v ( r )} \right) & \text{in } \; \{ r \leq \bar{r}, \, s \leq v(r) \}.
\end{cases}
\end{equation}
The function $W_{\theta}$ has free boundary consisting of
\[
\partial \{ W_{\theta} > 0 \} = \bigl( \, \partial \{ W_1 > 0 \} \cap \overline{\Gamma'_{\theta} \cap \{ r > \bar{r} \} } \, \bigr) \cup \{ r \leq \bar{r}, \; s = v(r) \},
\]
which ensures that $\Gamma_{\theta} \subset \{ W_{\theta} > 0\} $.
We denote by $W_{2,\theta}$ the restriction $W_{\theta}|_{\{ r \leq \bar{r}, s \leq v(r) \}}$.

We will produce an exponent $\beta$, with associated constants $\tau, \tau',\beta,\Lambda, \bar{r}$ (depending on $n,k,\beta$ and within $O(\theta)$ at $\theta = 0$, for $\theta \in (0, \theta_{n,k})$ sufficiently small) so that the function $W_{\theta}$ is a weak strict supersolution for~\eqref{eqn:capillary-fbp} in its positive phase.
We will show that $W_{\theta}$ is a strict supersolution by establishing the following properties:
\begin{align}
    A &> \bar{r} =  \frac{\sqrt{1-(\tau'_{\theta})^2}}{\tau'_{\theta}} = \text{slope}(\partial \Gamma'_{\theta}), \tag{S1} \label{eqn:slope-at-match} \\
    \cM(W_{1,\theta}) &<0 & & \text{in } \; \{ W_1 > 0\} \cap \overline{\Gamma'_{\theta} \cap \{ r > \bar{r} \} }, & \tag{S2} \label{eqn:MW1<0} \\
    |\nabla W_{1,\theta}| &< \tan \theta & & \text{on } \; \partial \{W_1 > 0 \} \cap \overline{\Gamma'_{\theta} \cap \{ r > \bar{r} \} }, \label{eqn:nablaW1<1} \tag{S3} \\
    |\nabla W_{2,\theta}|^2 \big\rvert_{s = v(r)} &\leq |\nabla W_{1,\theta}|^2 \big\rvert_{(\bar{r}, 1)} & & \text{for all } \; r \leq \bar{r}, & \tag{S4} \label{eqn:match-W1-W2} \\
    \frac{\partial}{\partial r} W_{1,\theta}(r,s) \big\rvert_{\bar{r}} &< \frac{\partial}{\partial r} W_{2,\theta}(r,s) \big\rvert_{\bar{r}} & & \text{for } \; s \in (- 1, 1), \tag{S5} \label{eqn:ddrW-1ddrW2} \\
    \cM(W_{2,\theta}) &<0, & & \text{in } \; \{ r \leq \bar{r}, \; s \leq v(r) \}, \tag{S6} \label{eqn:MW2<0} \\
    |\nabla W_{2,\theta}| &< \tan \theta & & \text{on } \; \{ r \leq \bar{r}, \; s = v(r) \}.\tag{S7} \label{eqn:nablaW2<1}
\end{align}
The requirements~\eqref{eqn:MW1<0}~--~\eqref{eqn:nablaW1<1} and~\eqref{eqn:MW2<0}~--~\eqref{eqn:nablaW2<1} are the standard supersolution properties for each component.
The condition~\eqref{eqn:slope-at-match} will ensure that $\partial \{ W_{1,\theta} > 0 \} \subset \Gamma'_{\theta}$ for $|x| = r > \bar{r}$.
The property $v'(\bar{r}) = \frac{1}{A}$ guarantees that the level curves $W_{1,\theta}(\rho,t) = 0$ and $W_{2,\theta}(r,s) = 0$ have equal slopes at the points $(\bar{r}, \pm 1)$, and~\eqref{eqn:match-W1-W2} ensures that $|\nabla W_{2,\theta}|$ is bounded along the free boundary $\{ s = v(r) \}$ of $W_{2,\theta}$ by the value $|\nabla W_{1,\theta}|_{(\bar{r}, 1)}$.
Combining this property with~\eqref{eqn:nablaW1<1} will yield $|\nabla W_{2,\theta}| < \tan \theta$ along $\partial \{ W_{2,\theta} > 0 \}$, which is~\eqref{eqn:nablaW2<1}.
Finally, the property~\eqref{eqn:ddrW-1ddrW2} will imply that the piecewise function $W$ is the minimum of $W_{1,\theta}$ and $W_{2,\theta}$ in $\{ W_{\theta} > 0 \}$ across $\{ r = \bar{r} \}$, so it is a weak strict supersolution in its positive phase $\{ W_{\theta} > 0 \}$.
We show that for $\theta \in (0, \theta_{n,k})$ sufficiently small, the verification of these properties can be reduced to the case of the one-phase problem, for $\theta=0$.

The construction of the supersolution $W_{\theta}$ in~\eqref{eqn:W-defined-supersolution} depends entirely on the parameter $\beta \in (2-n , -1)$.
For small contact angle $\theta$, we may reduce the computation of the conditions~\eqref{eqn:slope-at-match}~--~\eqref{eqn:nablaW2<1} to the model case of the corresponding solution to the one-phase problem.
The auxiliary quantities needed for that construction are presented in Definition~\ref{def:supermodel-solution}.
The following Proposition~\ref{prop:supersolution} shows that given a supersolution parameter $\beta$ satisfying the properties of Lemmas~\ref{lemma:W-computation} and~\ref{lemma:K-computation}, we can construct a supersolution and deduce that $\mathbf{C}_{n,k,\theta}$ is $\cA^{\theta}$-minimizing for small angle $\theta \in (0, \theta_{n,k})$.
Following this, we exhibit explicit examples in Table~\ref{table:ExplicitSuperQuadratics} of parameters $\beta$ that produce supersolutions.
\begin{definition}\label{def:supermodel-solution}
     Let $f_0 := {}_2 F_1 \Bigl( \frac{n-1}{2} , - \frac{1}{2} ; \frac{k}{2} ; t^2 \Bigr)$ be the hypergeometric function solving the linear problem $\cL_{n,k} f_0 = 0$ with a zero at $t_0$.
     Consider a parameter $\beta \in (2 - n,-1)$ and define
     \[
     \hat{g}(t) := {}_2 F_1 \left( \frac{n+\beta-2}{2} , - \frac{\beta}{2} ; \frac{k}{2} ; t^2 \right).
     \]
     Suppose that $\beta$ is such that $\hat{g}$ satisfies 
     \begin{equation}\label{eqn:n-2-k-1-condition}
   \frac{\hat{g}'(t_0)}{\hat{g}(t_0)} > \frac{ (n-2) t_0 - (k-1) t_0^{-1}}{1 - t_0^2}
\end{equation}
    and define $\tau = \tau(n,k,\beta) > t_0$ to satisfy 
    \begin{equation}\label{eqn:tau-definition}
        \tau := \tau(n,k,\beta), \qquad (1-\beta)^2 f_0(\tau)^2 + (1-\tau^2) (f_0' - f_0 \, \hat{g}'/\hat{g})(\tau)^2 = (1 - t_0^2) f_0'(t_0)^2.
    \end{equation}
    We define the constant $\bar{r} := \frac{\sqrt{1 - \tau^2}}{\tau}$ and set  
    \begin{equation}\label{eqn:define-A}
    A := - \frac{1}{\sqrt{1-\tau^2}} \frac{(1-\beta) \tau f_0(\tau) + (1-\tau^2) (f'_0 - f_0 \hat{g}'/\hat{g})(\tau)}{(1-\beta) f_0(\tau) - \tau (f'_0 - f_0 \hat{g}'/ \hat{g}) (\tau)}.
    \end{equation}
    We denote $p(t) := (1-t^2)^{\frac{1-\beta}{2}}$ and define the function $\cW(t)$ by 
\begin{equation}\label{eqn:W-function-simplified}
\begin{split}
    \cW(t) &:= \left( \left(1 - \tfrac{\bar{r}}{A} \right) (1-t^2) + \tfrac{1}{A^2+1} \right)f_0(t) + t(1-t^2)\left(\tfrac{\bar{r}}{A} - 1\right) f'_0(t) \\
    &- \tfrac{f_0(\tau)}{\hat{g}(\tau)} \tfrac{p(t)}{p(\tau)} \left[ \left( \beta \left(1 - \tfrac{\bar{r}}{A} \right) (1-t^2) - (1-\beta) \tfrac{\bar{r}}{A} + \tfrac{1}{A^2+1} \right)\hat{g}(t) + t(1-t^2)\left(\tfrac{\bar{r}}{A} - 1\right) \hat{g}'(t) \right].
\end{split}
\end{equation}
We also introduce auxiliary constants $a_1, a_0$ and the function $u(t)$ by
\begin{align}
    a_1 &:= (n-k) \bigl( \tfrac{3}{2} A^2 \bar{r} - A + \tfrac{1}{2} \bar{r} \bigr) + \bigl( - 2 A^2 \bar{r} + 3 A - \tfrac{3}{2} \bar{r} \bigr), \label{eqn:auxiliary-a1} \\ 
    a_0 &:= (n-k) A^2 (A^2 \bar{r} - A + \tfrac{1}{2} \bar{r}), \label{eqn:auxiliary-a0} \\
    u(t) &:= f_0(t) - \frac{f_0(\tau)}{\hat{g}(\tau)} \Bigl( \frac{1-t^2}{1-\tau^2} \Bigr)^{\frac{1-\beta}{2}} \hat{g}(t) . \label{eqn:u(t)-function-auxiliar}
% \\    H(\xi) &:= \frac{\bar{r}}{\sqrt{1-t^2}} u(t), \qquad \text{where } \; t(\xi) := \frac{\tau \xi}{\sqrt{1 -\tau^2 + \tau^2 \xi^2}} . \label{eqn:h-function}
    %\frac{( 1-\tau^2 + \tau^2 \xi^2)^{\frac{1}{2}}}{\tau} \left( f_0(t) + \frac{-f_0(\tau)}{\hat{g}(\tau)} ( 1 - \tau^2 + \tau^2 \xi^2)^{\frac{\beta - 1}{2}} \hat{g}(t) \right)_{t = \frac{\tau \xi}{\sqrt{ 1 - \tau^2 + \tau^2 \xi^2 }}}.
\end{align}
Finally, let $\tilde{v}(x) := 1 - \frac{\bar{r}}{2A}(1-x)$ and define the function $K(x,\xi)$, for $(x,\xi) \in [0,1] \times [0,1]$, by
\begin{equation}\label{eqn:K-tilde-x,xi}
        \begin{split}
            K(x,\xi) &= (1 + A^{-2} x \xi^2) H''(\xi) - \left( \left( n-k - \frac{2x}{A^2 + x} \right) \frac{\tilde{v}(x)}{\bar{r} A} - \frac{k-1}{\xi^2} \right) \xi H'(\xi) \\
            & \quad+ A\frac{ \frac{\bar{r}}{2}(n-k-1) x^2 + a_1 x + a_0}{(\bar{r} A)^2 (A^2 + x)^2} \tilde{v}(x) H(\xi),
        \end{split}
\end{equation}
where $H(\xi) := \frac{\bar{r}}{\sqrt{1-t^2}} u(t)$ from~\eqref{eqn:u(t)-function-auxiliar}, for $t(\xi) = \frac{\tau \xi}{\sqrt{1 - \tau^2 + \tau^2 \xi^2}}$.
\end{definition}
In the following Proposition we reduce the properties~\eqref{eqn:slope-at-match}~--~\eqref{eqn:nablaW2<1}, for cones $\mathbf{C}_{n,k,\theta}$ with small contact angle $\theta$, to three computable conditions for the model one-phase solution.
\begin{proposition}\label{prop:supersolution}
Suppose that for the construction of Definition~\ref{def:supermodel-solution}, we have $\cW'(\tau) > 0$ and
\begin{align}
 \tau \left( f'_0(\tau)/ f_0(\tau) - \hat{g}'(\tau) / \hat{g}(\tau) \right)& > 1 - \beta, \label{eqn:S'1-new} \tag{S${}^\prime$1} \\
    \cW(t) &<0 \qquad & & \text{for }  0 \leq t < \tau, \label{eqn:S'2-new} \tag{S${}^\prime$2} \\
    K(x,\xi) &<-c(n,k) \qquad & & \text{for } 0 \leq x,\xi \leq 1.\label{eqn:S'3-new} \tag{S${}^\prime$3}
\end{align}
    Then, there is a $\theta_{n,k} > 0$ such that for all $\theta \in (0, \theta_{n,k})$, we can produce $|\beta_{\theta} - \beta_0| = O(\theta)$ for which the function $W_{\theta}$ defined in~\eqref{eqn:W-defined-supersolution} is a strict weak supersolution of the capillary problem~\eqref{eqn:capillary-fbp} with contact angle $\theta$, whose positive phase $\{ W_{\theta} > 0 \}$ encloses $\Gamma'_{n,k,\theta}$ and a neighborhood of the origin.
\end{proposition}
\begin{proof}
    Our argument proceeds in three steps: first, we compute the mean curvature conditions~\eqref{eqn:MW1<0} and \eqref{eqn:MW2<0} for the function $W_{\theta}$, and relate them to $K(x,\xi)$.
    Second, we examine the slope matching conditions~\eqref{eqn:match-W1-W2} and~\eqref{eqn:ddrW-1ddrW2}, and relate them to the function $\cW(t)$.
    Finally, we compute these quantities in the model case of the one-phase reduction and obtain~\eqref{eqn:W-function-simplified} and~\eqref{eqn:K-tilde-x,xi}.
    In what follows, we fix the pair $(n,k)$ in the range of Theorem~\ref{thm:cones-are-minimizing-smalltheta} and suppress the dependence of the domains $\Gamma_{n,k,\theta}$, the function $f_{n,k,\theta}$, and the points $t_{n,k,\theta}, \tau_{n,k,\theta}$ on $\theta$; instead, we write $\Gamma_{\theta}, \Gamma'_{\theta}, f_{\theta}, t_{\theta}, \tau_{\theta}$.

    Let us recall the ingredients of the construction of the supersolution $W_{\theta}$ from~\eqref{eqn:W-defined-supersolution} in the model case of the one-phase problem: we first introduce
    \begin{align*}
     W_1(r,s) &:= (r^2 + s^2)^{\frac{1}{2}} f_0 \left( \frac{s}{\sqrt{r^2 + s^2}} \right) + \Lambda (r^2 + s^2)^{\frac{\beta}{2}} \hat{g} \left( \frac{s}{\sqrt{r^2 + s^2}} \right), \\ 
     \Lambda &:= \frac{- f_0(\tau)}{\tau^{1-\beta} \hat{g}(\tau)}, \qquad 
     A := \left| \frac{\partial_s W_1}{\partial_r W_1} \right|_{(\bar{r}, 1)} = - \frac{\partial_s W_1}{\partial_r W_1} \bigg\rvert_{(\bar{r}, 1)}.
     \end{align*}
     The definition of $H(\xi)$ via~\eqref{eqn:u(t)-function-auxiliar} corresponds to $H(\xi) := W_1(\bar{r}, \xi) = \frac{\bar{r}}{\sqrt{1-t^2}} u(t)$, for $|\xi| \leq 1$.
     Indeed,
     \[
     \rho (\bar{r},\xi) = \sqrt{ \tfrac{1-\tau^2}{\tau^2} + \xi^2} = \tau^{-1} \sqrt{1-\tau^2 + \tau^2 \xi^2}, \qquad t(\bar{r}, \xi) = \frac{\xi}{\tau^{-1} \sqrt{1-\tau^2 + \tau^2 \xi^2}} = \frac{\tau \xi}{\sqrt{1-\tau^2 + \tau^2 \xi^2}},
     \]
     which sends $[0,1] \ni \xi \mapsto [0,\tau] \in t$.
     The inverse map of $\xi \mapsto t(\xi)$ along $\{ r = \bar{r} \}$ is given by
     \begin{equation}\label{eqn:inverse-map-xi-t}
     \frac{\xi^2}{\bar{r}^2 + \xi^2 } = t^2 \iff \xi = \frac{\sqrt{1-\tau^2}}{\tau} \cdot \frac{t}{\sqrt
     {1-t^2}} = \frac{\bar{r} t}{\sqrt{1-t^2}}.
     \end{equation}
     We also define the functions $v(r), \varphi(r)$, for $|r| \leq \bar{r}$, by
     \begin{align*}
    v(r) &:= 1 - \frac{\bar{r}^2-r^2}{2 \bar{r} A}, \qquad v'(r) = \frac{r}{\bar{r}A}, \qquad v''(r) = \frac{1}{\bar{r} A}, \qquad v'(\bar{r}) = \frac{1}{A}, \qquad v(\bar{r}) = 1, \\
    \varphi(r) &:= \sqrt{1 + A^{-2}} \frac{v(r)}{\sqrt{1 + v'(r)^2}} = \sqrt{1 + A^{-2}} \frac{v(r)}{\sqrt{1 + (\bar{r} A)^{-2}  r^2}}.
\end{align*}
    Using $|D^{\ell}(f_{\theta} - \theta f_0)| < C(n,k,\ell) \theta^3$ shows that $\left|\frac{f_{\theta}(\tau)}{\tan \theta} \right| = C(n,k, \tau) + O(\theta^2)$ for $\theta \in (0, \theta_{n,k})$, so the computation of $\tau_{\theta}, \Lambda, \bar{r}, A$ in~\eqref{eqn:W1-supersolution} (for general $\theta$) and in~\eqref{eqn:tau-definition} (for $\theta = 0$) implies that
\[
|\tau_{\theta} - \tau_0| = O(\theta), \qquad |\Lambda_{\theta} - \Lambda_0| = O(\theta), \qquad  |\bar{r}_{\theta} - \bar{r}_0| = O(\theta), \qquad |A_{\theta} - A_0| = O(\theta)
\]
holds for small $\theta \in (0, \theta_{n,k})$.
The choice $\beta \in (2-n, -1)$ ensures that $\hat{g} > c(n,k)$ on $[ - \tau_{\theta}, \tau_{\theta}] $ and $\Lambda \rho^{\beta-1} \tan \theta \leq 2 \Lambda \bar{r}^{\beta-1} \theta = O(\theta)$ holds on $\overline{\Gamma'_{\theta} \cap \{ r > \bar{r} \} }$.

\smallskip
\noindent\textbf{Step 1: Estimates for the mean curvature operator.}
Applying the computations~\eqref{eqn:laplace-rho-a-g} and~\eqref{eqn:MC-computation+error}, we estimate the operator $\cM(W_{1,\theta})$ in the region $\{ W_{1,\theta} > 0 \} \cap \overline{\Gamma'_{\theta} \cap \{ r > \bar{r} \} }$ by
\begin{align*}
    ( 1 + |\nabla W_{1,\theta}|^2)^{\frac{3}{2}} \cM(W_{1,\theta}) &\leq \rho^{\beta'-2} ( \Lambda \tan \theta \rho^{2-\beta'} \Delta(\rho^{\beta'} \hat{g}) + O(\theta^3)) \\
    &\leq \rho^{\beta'-2} \theta ( - c(n,k,\Lambda) \theta + c(n,k,\beta) \theta^2),
\end{align*}
which means that $\cM(W_{1,\theta}) < 0$ on $\{ W_{1,\theta} > 0 \} \cap \overline{\Gamma'_{\theta} \cap \{ r > \bar{r} \} }$, for $\theta \in (0,\theta_{n,k})$.
The second step in this inequality comes from estimating
\begin{align*}
\rho^{2-\beta'} \Delta (\rho^{\beta'} \hat{g}) &= \rho^{2-\beta} \Delta (\rho^{\beta} \hat{g}) - \left[ \beta(\beta+ n-2) - \beta'(\beta'+n-2) \right]  \hat{g} < - c(n,k,\beta) \theta,
\end{align*}
since $c(n,\beta,\beta') \asymp c(n,k,\beta) \theta$ and the function $\hat{g}$ is constructed to satisfy $\hat{g} > c(n,k)$ on $[-\tau_{\theta}, \tau_{\theta}]$.
Next, we verify the boundary gradient condition~\eqref{eqn:nablaW1<1} for the supersolution along the free boundary $\partial \{ W_{1,\theta} > 0 \} \cap \overline{\Gamma'_{\theta} \cap \{ r > \bar{r} \}}$.
The function $\hat{g}$ is constructed to satisfy $\hat{g} > c(n,k)$ on $[ - \tau_{\theta}, \tau_{\theta}]$, so
\[
\partial \{ W_{1,\theta} > 0 \} \cap \overline{\Gamma'_{\theta} \cap \{ r > \bar{r} \} } \subset \{ t: t_{\theta} \leq t \leq \tau_{\theta} \}.
\]
The computation from~\eqref{eqn:nabla-V-squared} becomes
\[
\hat{G}_{\theta}(\beta, t) := |\nabla W_{1,\theta}|^2 = (1-\beta)^2 f_{\theta}^2 + (1-t^2) (f'_{\theta} - f_{\theta} \hat{g}'/\hat{g})^2.
\]
Since $\hat{G}_{\theta}(\beta, t_{\theta}) = \tan^2 \theta$, the condition~\eqref{eqn:nablaW1<1} is equivalent to $\hat{G}_{\theta}(\beta, -)$ attaining its maximum on $[t_{\theta}, \tau_{\theta}]$ at $t_{\theta}$.
The construction of $\beta$, with $\hat{G}'(\beta, t_0) < 0$, implies that $\hat{G}'_{\theta}(\beta, t_{\theta}) < 0$ for $\theta \in (0, \theta_{n,k})$.
We may therefore produce a $\tau_{\theta} > t_{\theta}$ with $\hat{G}_{\theta}(\beta, \tau_{\theta}) = \hat{G}_{\theta}(\beta, t_{\theta})$ and take $\tau'_{\theta} := \tau_{\theta} - \ve$, for $0 < \ve < \ve(n,k,\theta, \beta)$ sufficiently small, whereby~\eqref{eqn:nablaW1<1} is satisfied on $[t_{\theta}, \tau'_{\theta}]$.
We conclude that $W_{1,\theta}$ satisfies the supersolution conditions for~\eqref{eqn:capillary-fbp} in $\{ W_{1,\theta} > 0 \} \cap \overline{\Gamma' \cap \{ r > \bar{r} \} } $, for $\theta \in (0, \theta_{n,k})$.

To study $W_{2, \theta} = W_{\theta}|_{ \{ r \leq \bar{r}, s \leq v(r) \} }$, notice that the variables $r = |x|$ and $s = |y|$ satisfy $\partial_r x= \frac{x}{r}$ and $\partial_s y= \frac{y}{s}$, so $W = W(r,s)$ has
\begin{align*}
\nabla W &= \partial_r W \, \tfrac{x}{r} \oplus \partial_s W \tfrac{y}{s}, \\
D^2 W &= \begin{pmatrix}\partial^2_{rr} W \frac{x \otimes x}{r^2} + \frac{1}{r} \partial_r W (I_x - \frac{x \otimes x}{r^2} ) & \partial^2_{rs} W \frac{x \otimes y}{rs}  \\
\partial^2_{rs} W \frac{x \otimes y}{rs} & \partial^2_{ss} W \frac{y \otimes y}{s^2} + \frac{1}{s} \partial_s W (I_y - \frac{y \otimes y}{s^2})
\end{pmatrix}, \\
\Delta W &= \partial^2_{rr} W + \tfrac{n-k-1}{r} \partial_r W + \partial^2_{ss} W + \tfrac{k-1}{s} \partial_s W.
\end{align*}
For the function $W_{2,\theta}(r,s) = \varphi(r) H_{\theta} \bigl( \frac{s}{v(r)} \bigr)$, where $H_{\theta}(\xi) := W_{1,\theta}(\bar{r}, \xi)$, we compute
\begin{align*}
    |\nabla W_{2,\theta}|^2 &= \left( \varphi' H_{\theta} - \varphi H'_{\theta} v' \frac{s}{v^2} \right)^2 + \frac{(\varphi H'_{\theta})^2}{v^2}, \\
    \Delta W_{2,\theta} &= \frac{\varphi}{v^2} \left( 1 + s^2 \frac{v'^2}{v^2} \right) H_{\theta}'' - \frac{s \varphi}{v^2} \left( 2 \frac{\varphi'}{\varphi} v' + v'' - 2 \frac{(v')^2}{v} + \frac{n-k-1}{r} v' - (k-1) \frac{v}{s^2} \right) H'_{\theta} \\
    & \quad + \left( \varphi'' + \frac{n-k-1}{r} \varphi' \right) H_{\theta}
\end{align*}
as functions of $r$ and $\xi = \frac{s}{ v(r)}$.
Let us denote
\begin{equation}\label{eqn:model-K-tilde}
\begin{split}
    \tilde{K}(r,\xi) &:= \left( 1 + s^2 \frac{v'^2}{v^2} \right) H_{\theta}'' - s \left( 2 \frac{\varphi'}{\varphi} v' + v'' - 2 \frac{(v')^2}{v} + \frac{n-k-1}{r} v' - (k-1) \frac{v}{s^2} \right) H'_{\theta} \\
    &\; \quad + \left( \frac{\varphi''}{\varphi} + \frac{n-k-1}{r} \frac{\varphi'}{\varphi} \right) v^2 H_{\theta}
\end{split}
\end{equation}
coming from the model case of the one-phase problem, as $\theta \downarrow 0$.
    In the next step, we will prove that the function $\tilde{K}(r,\xi)$ is equivalent to the definition of $K(x,\xi)$ in~\eqref{eqn:K-tilde-x,xi}, after the change of variable $x := (\frac{r}{\bar{r}})^2$.
    Notably, this implies $\tilde{K}(r,\xi) < - c(n,k)$ due to~\eqref{eqn:S'3-new}.
    Using these computations in the expression for $\Delta W_{2, \theta}$, we will show that as $\theta \downarrow 0$, we have $\frac{v^2}{\varphi} \Delta W_{2,\theta} \xrightarrow{C^2} \tilde{K}(r,\xi)$.
    
    We recall the definition of $H_{\theta}$ for general $\theta$, together with the expression for $H(\xi)$ and the bounds $|D^{\ell} g| = C(n,k,\beta, \ell)$ and $|D^{\ell}(f_{\theta} - \theta f_0)| = C(n,k,\ell) \theta^3$ from Corollary~\ref{corollary:lambda-comparison}.
This lets us bound
\begin{equation}\label{eqn:h-theta-bounds}
|D^{\ell} (H_{\theta}(\xi) - \theta H_0 (\xi)) | < C(n,k,\ell, \Lambda, \bar{r}) \theta^2 = C(n,k,\ell, \beta) \theta^2
\end{equation}
for $r< \bar{r}$ with $\bar{r} > 0$ and $|\xi| \leq 1$ finite.
Since $A, \bar{r}$ (and therefore $v(r)$) satisfy uniform bounds in $\theta$, we obtain $\| W_{2,\theta} \|_{C^4} < C(n,k,\beta) \theta$.
Observe that the $(ss)$-Hessian term in $D^2 W_{2,\theta}$ is also uniformly bounded near $s=0$, due to $|\frac{1}{s} \partial_s W_{2,\theta}| < 2 |\partial^2_{ss}W_{2,\theta}| = O(\theta)$.
On the compact domain $\{ r \leq \bar{r}, \xi \leq 1 \}$, we may therefore use~\eqref{eqn:h-theta-bounds} and the computation of $\Delta W_{2,\theta}$ and $\tilde{K}(r,\xi) = \frac{v^2}{\varphi} \Delta W_{2,0}$ to bound
\[
\left|\Delta W_{2,\theta} - \tan \theta \frac{\varphi}{v^2} \tilde{K}(r,\xi) \right| < C(n,k,\beta, \Lambda , \tau) \theta^2 = C(n,k,\beta) \theta^2.
\]
Combining this bound with the estimate~\eqref{eqn:MC-bound} and $\tilde{K}(r,\xi) < - c(n,k)$ shows that
\begin{align*}
(1 + |\nabla W_{2,\theta}|^2)^{\frac{3}{2}} \cM(W_{2,\theta}) &\leq \Delta W_{2,\theta} + C(n,k,\beta) \theta^3 \\
&\leq - \theta \left( c(n,k) - C_1(n,k,\beta) \theta - C_2(n,k,\beta) \theta^2 \right)
\end{align*}
meaning that $\cM(W_{2,\theta}) < 0$ provided that $\theta \in ( 0 , \theta_{n,k})$ is chosen sufficiently small.

\smallskip
\noindent\textbf{Step 2: Estimates for the slope.}
It remains to verify the slope matching conditions~\eqref{eqn:match-W1-W2} and~\eqref{eqn:ddrW-1ddrW2} along the interface $\{ r = \bar{r} \}$.
Using the expression of $H_{\theta}(\xi) = W_{1,\theta}(\bar{r}, \xi)$ in $(r,s)$-coordinates, we can write $H'_{\theta}(\xi) = \partial_s W_1(\bar{r}, \xi)$.
For $\xi = 1$ we obtain $H_{\theta}(1) = W_{1,\theta}(\bar{r}, 1) = 0$ and $\frac{\tau_{\theta} \xi}{\sqrt{1 - \tau_{\theta}^2 + \tau_{\theta}^2 \xi^2}} = \tau_{\theta}$, whereby $\partial_s W_1(\bar{r},1) = H'_{\theta}(1)$.
The slope expression $A(\tau_{\theta}) = \left| \frac{dr}{ds} \right|_{(\bar{r},1)} = - \frac{\partial_s W_1}{\partial_r W_1} \Big|_{(\bar{r},1)}$ becomes
\begin{align*}
    |\nabla W_{1,\theta}|^2|_{(\bar{r},1)} &= (\partial_s W_{1,\theta})^2 + (\partial_r W_{1,\theta})^2 =  H'_{\theta}(1)^2 (1 + A(\tau_{\theta})^{-2}), \\
    |\nabla W_{2,\theta}|^2|_{s = v(r)} &= H'_{\theta}(1)^2 \frac{\varphi^2}{v^2} (1 + v'(r)^2).
\end{align*}
Therefore, the construction of $\varphi(r) = \frac{\sqrt{1+A^{-2}} v(r)}{\sqrt{1 + v'(r)^2}}$ implies that
\[
|\nabla W_{2,\theta}|^2|_{s = v(r)} = H'_{\theta}(1)^2 \frac{\varphi^2}{v^2} (1 + v'(r)^2) = H'_{\theta}(1)^2 (1 + A^{-2}) = |\nabla W_{1,\theta}|^2|_{(\bar{r}, 1)}.
\]
Consequently, the level curves $\{ W_1 = 0\}$ and $\{ W_2 = 0 \}$ have equal slopes at the points $(\bar{r}, \pm 1)$, and satisfy~\eqref{eqn:match-W1-W2} for all $r$.
To verify the slope condition~\eqref{eqn:slope-at-match}, note that 
\begin{align*}
    \partial_r W_{1,\theta}(\bar{r}, 1) &= (1 - \beta)\sqrt{1-\tau_{\theta}^2} f_{\theta}(\tau_{\theta}) - \tau_{\theta} \sqrt{1-\tau_{\theta}^2} (f'_{\theta} - f_{\theta} \hat{g}'/\hat{g})(\tau_{\theta}), \\
    \partial_s W_1(\bar{r}, 1) &= (1-\beta) \tau_{\theta} f_{\theta}(\tau_{\theta}) + (1-\tau^2_{\theta}) (f'_{\theta} - f_{\theta} \hat{g}'/\hat{g})(\tau_{\theta})
\end{align*}
whereby $|A_{\theta} - A| = O(\theta)$, where $A$ is defined in~\eqref{eqn:define-A}.
Using 
\[
|f'_{\theta}/f_{\theta} - f'_0/f_0| = O(\theta), \qquad  |\tau_{\theta} - \tau| = O(\theta), \qquad |\bar{r}_{\theta} - \bar{r}| = O(\theta),
\]
as above, we find that $A_{\theta} > \bar{r}_{\theta}$, for $\theta \in (0, \theta_{n,k})$ sufficiently small, is reduced to $A > \bar{r}$.
Using the expression~\eqref{eqn:define-A}, we compute that
\begin{equation}\label{eqn:A-rbar}
    A - \bar{r} = \frac{1 - \beta}{\tau \sqrt{1-\tau^2} \left[ \tau \bigl( \frac{f'_0(\tau)}{f_0(\tau)} - \frac{\hat{g}'(\tau)}{\hat{g}(\tau)} \bigr) - (1 - \beta)  \right] }
\end{equation}
whereby $A_{\theta} > \bar{r}_{\theta}$ is equivalent to the positivity of the denominator, which is~\eqref{eqn:S'1-new}.

We finally examine the condition~\eqref{eqn:ddrW-1ddrW2} that $\frac{\partial}{\partial r} W_{1,\theta}(s,r)|_{\bar{r}} < \frac{\partial}{\partial r} W_{2,\theta}(s,r)|_{\bar{r}}$ for $s \in (-1, 1)$.
We compute $\partial_r W_2 = \varphi' H - \varphi H' v' \frac{s}{v^2}$, and using $v(\bar{r}) = 1$ and $v'(\bar{r}) = \frac{1}{A}$ makes~\eqref{eqn:ddrW-1ddrW2} equivalent to
\[
\cW_{\theta}(\xi) := \partial_r W_{1,\theta}(\bar{r}, \xi) - \left( \varphi'(\bar{r}) H_{\theta}(\xi) - \frac{\xi H'_{\theta}(\xi)}{A} \right) < 0, \qquad \text{for } \; 0 \leq \xi < 1.
\]
Observe that $\cW_{\theta} (1) = 0$ by construction.
Indeed, the definition of $A = - \frac{\partial_s W_1}{\partial_r W_1} \big\rvert_{(\bar{r}, 1)}$ gives
\begin{align*}
    \partial_r W_{1,\theta}(\bar{r}, 1) &= - \frac{1}{A} \partial_s W_1|_{(\bar{r}, 1)} = - \frac{\partial_s W_1(\bar{r}, s)|_{(\bar{r},1)}}{A} = - \frac{H'_{\theta}(1)}{A}, \\
    \partial_r W_{2,\theta} (\bar{r}, 1) &= \varphi'(\bar{r}) H_{\theta}(1) - \frac{H'_{\theta}(1)}{A} = - \frac{H'_{\theta}(1)}{A},
\end{align*}
meaning that $\cW_{\theta}(1) = \partial_r W_{1,\theta}(\bar{r},1) - \partial_r W_{2,\theta}(\bar{r}, 1) = 0$.
Let us consider
\begin{equation}\label{eqn:W-tilde-intermediate}
    \tilde{\cW}(\xi) := \partial_r W_1(\bar{r}, \xi) - \left( \varphi'(\bar{r}) H(\xi) - \frac{1}{A} \xi H'(\xi) \right) \qquad \text{for } \; 0 \leq \xi < 1.
\end{equation}
If the function $\tilde{\cW}(\xi)$ additionally has $\tilde{\cW}(\xi)<0$ for $0 \leq \xi < 1$ and $\tilde{\cW}'(1)>0$, then in fact $\tilde{\cW}'(\xi) > 0$ in an interval $[1-\ve,1+\ve]$.
At the last step of the proof, we will show that $\tilde{\cW}(\xi) < 0$ and $\tilde{\cW}'(1) > 0$ is equivalent to $\cW(t) < 0$ and $\cW'(\tau) > 0$ for the function $\cW(t)$ from~\eqref{eqn:W-function-simplified}.

Given these properties, the continuity in $\theta$ and the property~\eqref{eqn:h-theta-bounds} ensure that $\cW'_{\theta} (1) > 0$ and $\cW'_{\theta}(\xi) > 0$ on $[ 1 - \frac{\ve}{2}, 1 + \frac{\ve}{2}]$, for small $\theta \in (0, \theta_{n,k})$, as well as $\cW_{\theta} < 0$ on $[0, 1 - \frac{\ve}{2}]$.
Using $\cW'_{\theta}>0$ on $[1 - \frac{\ve}{2} , 1 + \frac{\ve}{2}]$ and $\cW_{\theta}(1) = 0$ then yields~\eqref{eqn:ddrW-1ddrW2}.
We conclude that the function $W_{\theta}$ defined in~\eqref{eqn:W-defined-supersolution} is a strict weak supersolution of the capillary problem~\eqref{eqn:capillary-fbp} at contact angle $\theta$ with the desired properties, for $\theta \in (0, \theta_{n,k})$ sufficiently small.

\smallskip
\noindent\textbf{Step 3a: Reduction of $\tilde{K}$.}
We now simplify the function $\tilde{K}(r,\xi)$ from~\eqref{eqn:model-K-tilde} into the expression $K(x,\xi)$ from~\eqref{eqn:K-tilde-x,xi}.
Using the expression $\varphi(r) := \frac{\sqrt{1+A^{-2}} v(r)}{\sqrt{1 + v'(r)^2}}$, we compute
\begin{align*}
&(\log \varphi)' = \frac{v'}{v} - \frac{v' v''}{1+ v'^2}, \qquad \qquad 2 \frac{\varphi'}{\varphi} v' = 2 \frac{(v')^2}{v} - 2 v'' + \frac{2v''}{1 + v'^2}, \\
& 2 \frac{\varphi'}{\varphi} v' + v'' - 2 \frac{(v')^2}{v} + \frac{n-k-1}{r} v' - (k-1) \frac{v}{s^2} = \frac{1 - v'^2}{1 + v'^2} v'' + \frac{n-k-1}{r} v' - \frac{k-1}{\xi^2 v}.
\end{align*}
Consequently, the expression~\eqref{eqn:model-K-tilde} for $\tilde{K}(r,\xi)$ becomes
\begin{equation}\label{eqn:K-tilde-intermediate}
    \begin{split}
    \tilde{K}(r,\xi) &:= (1 + v'^2 \xi^2) H'' - \left( \frac{1 - v'^2}{1 + v'^2} vv'' + \frac{n-k-1}{r} vv' - \frac{k-1}{\xi^2} \right) \xi H' \\
    & \; \quad+ \left( \frac{\varphi''}{\varphi} + \frac{n-k-1}{r} \frac{\varphi'}{\varphi} \right) v^2 H.
    \end{split}
    \end{equation}
    We introduce the variable $x := (\frac{r}{\bar{r}})^2 \in [0,1]$, so $v'(r)^2 = A^{-2} x$ and $v(r) = \tilde{v}(x)$, where $\tilde{v}(x) := (1 - \frac{\bar{r}}{2A}) + \frac{\bar{r}}{2 A} x$.
    The computations of Proposition~\ref{prop:supersolution} together with the explicit expressions for $\varphi(r), v(r)$ allow us to write
    \allowdisplaybreaks{
    \begin{align*}
          \frac{\varphi''}{\varphi} +& \frac{n-k-1}{r} \frac{\varphi'}{\varphi} \\
         &= \frac{(n-k) (\bar{r}A) (1 + A^{-2} x)^2 - 2 \frac{\bar{r}}{A} x(1 + A^{-2} x) + v \left[ ( 3 A^{-2} x - (n-k) (1 + A^{-2}x) \right]}{(\bar{r} A)^2 v (1 + A^{-2}x)^2} \\
        &= \frac{ \frac{\bar{r}}{2}(n-k-1) x^2 + a_1 x + a_0}{A(\bar{r} A)^2 v (1 + A^{-2} x)^2},
    \end{align*}}
    for $a_1, a_0$ the coefficients of~\eqref{eqn:auxiliary-a1}, \eqref{eqn:auxiliary-a0}.
    Substituting in~\eqref{eqn:K-tilde-intermediate} produces the expression~\eqref{eqn:K-tilde-x,xi}.

\smallskip
\noindent\textbf{Step 3b: Reduction of $\tilde{\cW}$.}
Finally, we reduce the function $\tilde{\cW}(\xi)$ of~\eqref{eqn:W-tilde-intermediate} to the expression $\cW(t)$ from~\eqref{eqn:W-function-simplified}.
Using the expression $\varphi(r) := \frac{\sqrt{1+A^{-2}} v(r)}{\sqrt{1 + v'(r)^2}}$, we first compute that
\[
\varphi'(\bar{r}) = \frac{1}{A} \left( 1 - \frac{  v''(\bar{r})}{1+A^{-2}} \right) = \frac{1}{A} \left( 1 - \frac{1}{\bar{r} A (1+A^{-2})} \right).
\]
The expression~\eqref{eqn:W-tilde-intermediate} for $\tilde{\cW}(\xi)$ therefore becomes
\[
\tilde{\cW}(\xi) = \partial_r W_1 (\bar{r},\xi) - \frac{1}{A} \left( \left( 1 - \frac{1}{\bar{r} A (1+A^{-2})} \right) H(\xi) - \xi H'(\xi) \right)
\]
where $H(\xi)$ is defined via~\eqref{eqn:u(t)-function-auxiliar} as $H(\xi) := W_1(\bar{r},s)_{s = \xi} = \frac{\bar{r}}{\sqrt{1-t^2}}u(t)$, where
\[
W_1 (r,s) := (r^2 + s^2)^{\frac{1}{2}} f_0 \left( \frac{s}{\sqrt{r^2 + s^2}} \right) + \Lambda (r^2 + s^2)^{\frac{\beta}{2}} \hat{g} \left( \frac{s}{\sqrt{r^2+s^2}} \right).
\]
Consequently, $H'(\xi) = \partial_s W_1(\bar{r}, s)|_{s = \xi}$.
We will express $H(\xi)$ and its derivatives in the variable $t$, which is related to $\xi$ via $t(\xi) = \frac{\tau \xi}{\sqrt{1-\tau^2 + \tau^2 \xi^2}}$; note that the map $\xi \mapsto t(\xi)$ sends $[0,1]$ monotonically onto $[0,\tau]$.
The inverse map is given by $\xi = \frac{\sqrt{1-\tau^2}}{\tau} \cdot \frac{t}{\sqrt{1-t^2}} = \frac{\bar{r} t}{\sqrt{1-t^2}}$, as computed in~\eqref{eqn:inverse-map-xi-t}.

We denote $p(t) := (1-t^2)^{\frac{1-\beta}{2}}$ and compute the resulting expressions as
\begin{align*}
    \partial_r W_1(\bar{r},\xi) &= \sqrt{1-t^2}\Big(f_0(t)-t f_0'(t)- \frac{f_0(\tau)}{\hat{g}(\tau)} \frac{p(t)}{p(\tau)}\big(\beta\,\hat g(t)-t\,\hat g'(t)\big)\Big),\\
    H(\xi)&= W_1(\bar r,\xi)=\frac{\bar r}{\sqrt{1-t^2}}\Big(f_0(t)- \frac{f_0(\tau)}{\hat{g}(\tau)} \frac{p(t)}{p(\tau)}\,\hat g(t)\Big),
\end{align*}
For $t(\xi) = \frac{\tau \xi}{\sqrt{1 - \tau^2 + \tau^2 \xi^2}}$ and $u(t) = f_0(t) - \frac{f_0(\tau)}{\hat{g}(\tau)} \bigl( \frac{1-t^2}{1-\tau^2} \bigr)^{\frac{1-\beta}{2}} \hat{g}(t)$, this produces~\eqref{eqn:u(t)-function-auxiliar}-\eqref{eqn:H(xi)-expression} via
\begin{equation}\label{eqn:H(xi)-expression}
    H(\xi) = \frac{\bar{r}}{\sqrt{1-t^2}} u(t), \qquad \frac{dt}{d\xi} = \frac{(1-t^2)^{\frac{3}{2}}}{\bar{r}}, \qquad H'(\xi) = t u(t) + (1-t^2) u'(t).
\end{equation}
We now prove that $\tilde{\cW}(\xi) = \frac{\cW(t)}{\sqrt{1-t^2}}$ as in~\eqref{eqn:W-function-simplified}.
Using the above relations, we compute $\cW(t) := \sqrt{1-t^2} \tilde{\cW}(\xi)$ as
\allowdisplaybreaks{
\begin{align*}
\cW(t)
&=\sqrt{1-t^2}\,\partial_r W_1(\bar r,\xi)
-\frac{\sqrt{1-t^2}}{A}\left(\left(1-\frac{1}{\bar r A(1+A^{-2})}\right)H(\xi)-\xi H'(\xi)\right)\\
&=(1-t^2)\left(f_0-t f_0' - \frac{f_0(\tau)}{\hat{g}(\tau)} \frac{p(t)}{p(\tau)}(\beta \hat g-t \hat g') \right) - \frac{1}{A} \left(1-\frac{1}{\bar r A(1+A^{-2})}\right)\bar r \left(f_0 - \frac{f_0(\tau)}{\hat{g}(\tau)} \frac{p(t)}{p(\tau)} \hat g \right)  \\
%&\quad + \frac{1}{A} \bar{r} t \left(t f_0+(1-t^2)f_0' - \frac{f_0(\tau)}{\hat{g}(\tau)} \frac{p(t)}{p(\tau)} (\beta t \hat g+(1-t^2)\hat g') \right) \\
%&= \left[ \left(1 - \frac{\bar{r}}{A} \right) (1-t^2) + \frac{1}{A^2+1} \right] f_0(t) + t(1-t^2)\left(\frac{\bar{r}}{A} - 1\right) f'_0(t) \\
    &\;\quad- \frac{f_0(\tau)}{\hat{g}(\tau)} \frac{p(t)}{p(\tau)} \left[ \left( \beta \left(1 - \frac{\bar{r}}{A} \right) (1-t^2) - (1-\beta) \frac{\bar{r}}{A} + \frac{1}{A^2+1} \right)\hat{g}(t) + t(1-t^2)\left(\frac{\bar{r}}{A} - 1\right) \hat{g}'(t) \right]
\end{align*}}
which rearranges to the claimed expression~\eqref{eqn:W-function-simplified}.
This completes our computation.
\end{proof}
\begin{remark}
The proof strategy used in Theorem~\ref{thm:cones-are-minimizing-smalltheta} relies on a perturbative construction of sub- and supersolutions, allowing us to reduce the barrier conditions to verifying the corresponding inequalities for the linearized problem.
The same argument can be applied to other cones for the one-phase problem in $\bR^n$ in order to obtain the following general result:
\begin{proposition}
    Let $u = \rho f(\omega)$ be a one-homogeneous solution to the one-phase problem in $\bR^n$, for $\rho := |z'|$ and $\omega = \frac{z'}{|z'|} \in \bS^{n-1}$.
    Suppose that $u$ is strictly minimizing and strictly stable,  and the positive phase $\{ f > 0 \}$ has link $\Omega \subset \bS^{n-1}$, such that $u$ admits a strict weak supersolution $W$ with $|\nabla W|^2 < 1$ away from $\partial \Omega$, and $W = \rho f + \rho^{\beta} \hat{g}$ away from a compact set, where $\inf_{\Omega} \hat{g} > 0$.
Then, there is a $\theta_0 = \theta_0(f) > 0$ such that for all $\theta \in (0,\theta_0)$, there exists a minimizing graphical capillary cone $\mathbf{C}_{\theta} \subset \bR^{n+1}_+$ for $\cA^{\theta}$, which is close in $C^{k,\alpha}_{\on{loc}}$ to the graph of $\rho f\tan \theta$.
\end{proposition}
The strict stability assumption for $u$ ensures the existence of a strict subsolution of the form $\rho f(\omega) - \rho^{\beta} g$ away from a compact set.
Therefore, only the more delicate construction of a supersolution requires this particular presentation near infinity.
This general argument, building on minimizing cones for the one-phase problem, provides a partial converse to Theorems~1.3 and 1.4 of Chodosh-Edelen-Li~\cite{improved-regularity}, which prove that minimizing capillary cones converge, under rescaling by $\tan \theta$ as $\theta \downarrow 0$, to a minimizing solution of the one-phase problem; in particular, they are graphical for small contact angles.
The relevant construction would rely on perturbing the foliations for one-phase minimizers constructed by De~Silva-Jerison-Shahgolian and Edelen-Spolaor-Velichkov~\cites{hardt-simon-DSJ, SMP-one-phase} to obtain barriers with the desired asymptotic behavior, therefore showing that the critical regularity threshold for the one-phase problem is equivalent to that for capillary cones with small angle.
We develop these ideas in upcoming work~\cite{FTW_MinimizingII}.
\end{remark}

\subsection{Explicit supersolutions}\label{subsection:explicit-supersolutions}

The construction of a supersolution in Proposition~\ref{prop:supersolution} relied on a single parameter $\beta$, in terms of which the auxiliary constants $A, \bar{r}, \Lambda$ and functions $h, \cW, K$ were defined.
Our computations therein reduce the verification of the supersolution property for the function $W(x,y)$ in~\eqref{eqn:W-defined-supersolution} to the inequality $A>\bar{r}$ together with certain inequalities for these functions,~\eqref{eqn:S'2-new} and~\eqref{eqn:S'3-new}.

As a first step, we reduce the functional inequality conditions~\eqref{eqn:S'2-new} and~\eqref{eqn:S'3-new} to pointwise inequalities that are explicitly computable in terms of $\beta$, in Lemmas~\ref{lemma:W-computation} and~\ref{lemma:K-computation}.
Therefore, verifying the minimizing property of the cones $\mathbf{C}_{n,k,\theta}$, for $\theta \in (0,\theta_{n,k})$ (as well as for the corresponding homogeneous solution of the one-phase problem) for some $(n,k)$ is reduced to finding a $\beta_{n,k} \in (2-n,0)$ that satisfies the above inequalities.
For any given pair $(n,k)$ in the range of Theorem~\ref{thm:cones-are-minimizing-smalltheta}, the existence of such a $\beta_{n,k}$ can be easily obtained by direct computation.
In fact, for $n \geq 20$ there are large intervals of admissible $\beta_{n,k}$.
Our setup shows that producing such $\beta_{n,k}$ for all $(n,k)$ is reduced to analyzing the behavior of hypergeometric functions, requiring additional results beyond what is in \cite{FTW-stability-one-phase}.
Therefore, we delay this analysis of hypergeometric functions to an upcoming note.
The proof of Theorem~\ref{thm:cones-are-minimizing-smalltheta} follows a simple inductive picture: letting $\cB_{n,k}$ denote the interval of $\beta \in (2-n,0)$ satisfying the requirements of Lemmas~\ref{lemma:W-computation} and~\ref{lemma:K-computation}, we find that
\begin{equation}\label{eqn:betaInductive}
\inf \cB_{n,k} \in \cB_{n+1,k} \cap \cB_{n+1,k+1}.
\end{equation}
Observe that $\cB_{n,k} \neq \varnothing$ is equivalent to the cone $\mathbf{C}_{n,k,\theta}$ being $\cA^{\theta}-$minimizing.
Consequently, $\cB_{n,k} \neq \varnothing$ implies that $\cB_{n+1,k} , \cB_{n+1,k+1} \neq \varnothing$, so that
\begin{equation}\label{eqn:inductive-picture}
\mathbf{C}_{n,k,\theta} \; \text{ minimizing for } \theta \in (0,\theta_{n,k}) \implies \mathbf{C}_{n+1,k,\theta}, \mathbf{C}_{n+1,k+1,\theta} \; \text{ minimizing for } \theta \in (0, \tilde{\theta}_{n,k}).
\end{equation}
Since $\cB_{n,k} \neq \varnothing$ for all $1 \leq k \leq n-2$ when $n = 20$, this will imply the area-minimizing result for all $n \geq 20$.
On the other hand, this construction is insufficient for the twenty pairs $(n,k)$ not covered by Theorem~\ref{thm:cones-are-minimizing-nearpi/2}, which occur in dimension $n \leq 19$.
Strong numerical evidence suggests that certain exceptional pairs are not strictly minimizing; we study them in our upcoming work~\cite{FTW_MinimizingII}.

In Table~\ref{table:ExplicitSuperQuadratics}, we present key examples of pairs $(n,k)$ and a viable corresponding supersolution parameter $\beta$.
Consistent with the inductive picture~\eqref{eqn:inductive-picture}, we list the pairs $(n,k)$ with $n-k = d_0$,
\[
(7,2): \; - \tfrac{5}{2} \in \cB_{7,2}, \qquad (9,5): \; -3 \in \cB_{9,5}, \qquad (12,9): \; -4 \in \cB_{12,9}, \qquad (20,18): \; -7 \in \cB_{20,18}
\]
from which the subsequent values for $(n,k) = (n, n-d_0)$ for $n \geq n_0$ may be obtained.

\begin{table}
\begin{tabular}{|c||c|c|c|c|c|c|c|c|c|c|c|c|}
\hline
 $n$    &$7$&$7$&$8$&$9$&$9$&$10$&$12$&$13$&$20$&$21$&$21$&$21$  \\ \hline 
  $k$   &$1$&$2$&$3$&$4$&$5$&$6$&$9$&$10$&$18$&$17$&$18$&$19$    \\ \hline \hline 
 $\beta$ & $-2$ & $ - 2.5 $ & $-3$ & $-4$ & $-3 $ & $-3$ & $-4$ & $-4$ & $-10$ & $-10$ & $-10$ & $-10$ \\ \hline   
\end{tabular}
\caption{Examples of $n,k$ and the corresponding supersolution exponent $\beta$ satisfying Proposition~\ref{prop:supersolution}.
See Table~\ref{table:bigTableOfBetas} for a more complete analysis illustrating that these exponents yield valid supersolutions.}
\label{table:ExplicitSuperQuadratics}
\end{table}

\begin{lemma}\label{lemma:W-computation}
Let $f_0, \hat{g}, u, \tau, A , \bar{r}$ be as in Definition~\ref{def:supermodel-solution}, with $A > \bar{r}>0$.
Consider the function
\[
\cQ(t) := (1-t^2) \left( 1 - t \frac{u'(t)}{u(t)} \right) + \tau (1-\tau^2) \frac{u'(\tau)}{f_0(\tau)} \Bigl( \frac{f_0(t)}{u(t)} - 1 \Bigr).
\]
Then, $u(t)>0$ on $[0, \tau)$ and the condition~\eqref{eqn:S'2-new} is equivalent to $\cQ(t) + \frac{1}{(A^2+1)\bigl(1-\frac{\bar r}{A}\bigr)} < 0$.
\end{lemma}
\begin{proof}
For brevity, we denote $\tilde{g}(t):=(1-t^2)^{\frac{1-\beta}{2}}\hat g(t)$, so $u(t) = f_0(t) - \frac{f_0(\tau)}{\tilde{g}(\tau)} \tilde{g}(t)$.
Then, $u(\tau) = 0$ and $\frac{f_0(\tau)}{\hat g(\tau)} \bigl( \tfrac{1-t^2}{1-\tau^2} \bigr)^{\frac{1-\beta}{2}} \hat{g}(t)=f_0(t)-u(t)$.
Differentiating this relation, we obtain
\[
\tfrac{f_0(\tau)}{\hat g(\tau)} \bigl( \tfrac{1-t^2}{1-\tau^2} \bigr)^{\frac{1-\beta}{2}} \hat{g}'(t)
= f_0'(t)-u(t)+\tfrac{1-\beta}{1-t^2} t \bigl(f_0(t)-u(t)\bigr).
\]
The function $\cW(t)$ was defined in~\eqref{eqn:W-function-simplified} as
\begin{align*}
    \cW(t) &:= \left( \left(1 - \tfrac{\bar{r}}{A} \right) (1-t^2) + \tfrac{1}{A^2+1} \right)f_0(t) + t(1-t^2)\left(\tfrac{\bar{r}}{A} - 1\right) f'_0(t) \\
    &\;\quad- \tfrac{f_0(\tau)}{\hat{g}(\tau)} \tfrac{p(t)}{p(\tau)} \left[ \left( \beta \left(1 - \tfrac{\bar{r}}{A} \right) (1-t^2) - (1-\beta) \tfrac{\bar{r}}{A} + \tfrac{1}{A^2+1} \right)\hat{g}(t) + t(1-t^2)\left(\tfrac{\bar{r}}{A} - 1\right) \hat{g}'(t) \right]
\end{align*}
Substituting $\frac{f_0(\tau)}{\hat g(\tau)} \bigl( \tfrac{1-t^2}{1-\tau^2} \bigr)^{\frac{1-\beta}{2}} \hat{g} = f_0 - u$ into this expression and collecting terms, we obtain
\begin{align*}
    \cW(t) &=\bigl(1-\tfrac{\bar r}{A}\bigr)\bigl[(1-\beta)f_0+(\beta-t^2) u -t(1-t^2) u' \bigr] +\tfrac{1}{A^2+1} u +(1-\beta)\tfrac{\bar r}{A} (f_0 - u).
\end{align*}
The definition of $u(t)$ implies that $u(\tau) = 0$ and $\frac{1 - \beta}{1 - \frac{\bar{r}}{A}} = \tau (1-\tau^2) \frac{u(\tau)}{f_0(\tau)} $.
Using the definition of $\cQ(t)$, we may therefore rearrange $\cW(t)$ as
\begin{equation}\label{eqn:w-expression}
\cW(t) = \Bigl( 1 - \frac{\bar{r}}{A} \Bigr) u(t) \Bigl( \cQ(t) + \frac{1}{(A^2+1)\bigl(1-\frac{\bar r}{A}\bigr)} \Bigr).
\end{equation}
Since $A > \bar{r} > 0$, we have $1 - \frac{\bar{r}}{A} \in (0,1)$.
Letting $t_0$ denote the zero of $f_0$, we have $f_0>0$ on $[0,t_0)$ and $f_0<0$ on $(t_0, \tau]$, while $\hat{g}>0$ on $[0,1]$ as discussed in Section~\ref{subsection:subsolution}.
Consequently,
\[
u(t) = f_0(t) + \tfrac{|f_0(\tau)|}{\hat{g}(\tau)}\bigl( \tfrac{1-t^2}{1-\tau^2} \bigr)^{\frac{1-\beta}{2}} \hat{g}(t) > 0 \qquad \text{for } \; t \in [0,t_0].
\]
Next, we recall that $f_0, \hat{g}$ satisfy the equations
\[
\bigl( t^{k-1} (1-t^2)^{\frac{n-k}{2}} F' \bigr)' + \lambda_F t^{k-1} (1-t^2)^{\frac{n-k-2}{2}} F = 0
\]
with $\lambda_f = n-1$ and $\lambda_{\hat{g}} = \beta(\beta+n-2) < 0$ for $\beta \in (2-n,-1)$.
The Sturm-Picone identity yields
\[
- \Bigl( t^{k-1} (1-t^2)^{\frac{n-k}{2}} \hat{g}^2 ( f/\hat{g})' \Bigr)' = - (\lambda_{\hat{g}} - \lambda_f) t^{k-1} (1-t^2)^{\frac{n-k-2}{2}} f(t) \hat{g}(t) < 0 \qquad \text{on } \; (t_0, \tau],
\]
due to $\lambda_{\hat{g}} -\lambda_f < 0$ and $f_0 < 0$.
Since $\frac{f_0}{\hat{g}}>0$ for $t<t_0$, we have $-(\frac{f_0}{\hat{g}})' (t_0) > 0$.
Moreover, the condition~\eqref{eqn:S'1-new} implies
\[
\tau \bigl( f'_0(\tau)/f_0(\tau) - \hat{g}'(\tau)/\hat{g}(\tau) \bigr) > 1 - \beta \implies \tfrac{f'_0}{f_0}(\tau) - \tfrac{\hat{g}'}{\hat{g}}(\tau) > 0.
\]
Since $(f_0 \hat{g})(\tau) < 0$, this implies that $-( \frac{f}{\hat{g}})'(\tau) > 0$.
We therefore obtain
\[
\frac{d}{dt}\left((1-t^2)^{-\frac{1-\beta}{2}}\frac{f_0}{\hat{g}}\right)=(1-t^2)^{-\frac{1-\beta}{2}}\left(\left(\frac{f_0}{\hat{g}}\right)'+\frac{(1-\beta)t}{1-t^2}\frac{f_0}{\hat{g}}\right) < 0,
\]
whereby $\frac{f_0(t)}{(1-t^2)^{\frac{1-\beta}{2}} \hat{g}(t) }$ is strictly decreasing on $[t_0, \tau]$, due to $f_0<0$ and $\hat{g}>0$.
Therefore, writing
\[
u(t) = (1-t^2)^{\frac{1-\beta}{2}} \hat{g}(t) \Bigl( \frac{f_0(t)}{(1-t^2)^{\frac{1-\beta}{2}} \hat{g}(t) } - \frac{f_0(\tau)}{(1-\tau^2)^{\frac{1-\beta}{2}} \hat{g}(\tau) } \Bigr)
\]
shows that $u>0$ on $[0,\tau)$, as desired.
Finally, the expression~\eqref{eqn:w-expression} completes the proof.
\end{proof}

    In all tested cases, the maximum of the quantity $\cQ(t)$ occurs at one of the endpoints of the interval $[0,\tau]$, so the desired negativity is reduced to $\max \{ \cQ(0), \cQ(\tau) \} < - \frac{1}{(A^2+1)(1 - \frac{\bar{r}}{A})}$.
    In terms of the function $\cW(t)$ of Definition~\ref{def:supermodel-solution}, these two properties are equivalent to
    \[
    \cW(0) < 0 \qquad \text{and} \qquad \cW'(\tau) > 0,
    \]
    respectively.
    Both conditions are explicitly computable in terms of $\beta, \tau,A$.

\begin{lemma}\label{lemma:K-computation}
    Consider the function $K(x,\xi)$ defined in~\eqref{eqn:K-tilde-x,xi} by
    \begin{align*}
            K(x,\xi) &= (1 + A^{-2} x \xi^2) H''(\xi) - \left( \left( n-k - \frac{2x}{A^2 + x} \right) \frac{\tilde{v}(x)}{\bar{r} A} - \frac{k-1}{\xi^2} \right) \xi H'(\xi) \\
            & \quad+ A\frac{ \frac{\bar{r}}{2}(n-k-1) x^2 + a_1 x + a_0}{(\bar{r} A)^2 (A^2 + x)^2} \tilde{v}(x) H(\xi).
    \end{align*}
    where $\tilde{v}(x) := 1 - \frac{\bar{r}}{2A}(1-x)$, in the notation of Definition~\ref{def:supermodel-solution}.
    Suppose that the conditions
    \[
    A > \bar{r}, \qquad \max_{ \xi \in [0,1] } K(0,\xi) < 0, \qquad \max_{\xi \in [0,1] } K(1,\xi) < 0, \qquad \min_{\xi \in [0,1] } \cP(\xi) > 0
    \]
    are satisfied, where the quantity $\cP$ is defined in~\eqref{eqn:final-quantity-P}.
    Then, we have $\max_{(x,\xi) \in [0,1] \times [0,1]}K(x,\xi) < - c(n,k) < 0$ and the condition~\eqref{eqn:S'3-new} holds.
\end{lemma}
\begin{proof}
We reduce the negativity of $K(x,\xi)$ to the negativity of the two ``boundary quantities'' and the positivity of the ``concavity quantity'' $\cP$ in terms of $\xi$ by forming a cubic in $x$.
Using the computation~\eqref{eqn:H(xi)-expression} from Proposition~\ref{prop:supersolution}, for $u(t) = f_0(t) - \frac{f_0(\tau)}{\hat{g}(\tau)} \bigl( \frac{1-t^2}{1-\tau^2} \bigr)^{\frac{1-\beta}{2}} \hat{g}(t)$, we may express $H(\xi) = \frac{\bar{r}}{\sqrt{1-t^2}} u(t)$ with $t = \frac{\tau \xi}{\sqrt{1 - \tau^2 + \tau^2 \xi^2}}$.
We therefore obtain
\begin{equation}\label{eqn:H''(xi)}
H'(\xi) = t u(t) + (1-t^2) u'(t), \qquad H''(\xi) = \tfrac{1}{\bar{r}} (1-t^2)^{\frac{3}{2}} \bigl( u(t) -t u'(t) + (1-t^2) u''(t) \bigr),
\end{equation}
    with $H$ strictly decreasing in $\xi \in [0,1]$ (equivalently, $t \in [0,\tau]$) to
\[
H(0) = \tfrac{\sqrt{1-\tau^2}}{\tau} \Bigl( 1 -\tfrac{f_0(\tau)}{\hat{g}(\tau)} ( 1 - \tau^2)^{\frac{\beta-1}{2}} \Bigr), \qquad H(1) = 0.
\]
Suppose that $A > \bar{r}$ and $K(0,\xi), K(1,\xi) < 0$ for $\xi \in [0,1]$.
    We are interested in determining whether $K(x,\xi) < -c(n,k)$, so we may clear denominators by $(A^2 + x)^2 \in [A^4, (A^2+1)^2]$ and equivalently prove that $P_{\xi}(x) < 0$, where we define
    \begin{equation}\label{eqn:cubic-in-x,xi}
    P_{\xi}(x) := (A^2 + x)^2 K(x,\xi) = P_3(\xi) x^3 + P_2(\xi) x^2 + P_1(\xi) x + P_0(\xi).
    \end{equation}
    Here, $P_{\xi}(x)$ is a cubic function of $x$ with coefficients obtained from~\eqref{eqn:K-tilde-x,xi}, as
    \begin{align*}
        P_3(\xi) &= A^{-2} \left( \xi^2 H'' - \tfrac{n-k-2}{2} \xi H' + \tfrac{n-k-1}{4} H \right), \\
        P_2(\xi) &=(1+ 2 \xi^2) H'' - \bigl( 1 + (n-k-2) \tfrac{A^2 \bar{r} + A - \frac{1}{2} \bar{r}}{\bar{r} A^2} + \tfrac{k-1}{\xi^2} \bigr) \xi H' +  \tfrac{(n-k-1) (A - \frac{1}{2} \bar{r}) + a_1}{2 \bar{r} A^2} H, \\
        P_1(\xi) &= A^2 (2 + \xi^2) H'' + \left( (n-k-1) \Bigl( 1 - \tfrac{2A}{\bar{r}} - \tfrac{1}{2} A^2 \Bigr) - \tfrac{1}{2} A^2 + \tfrac{2 A^2 (k-1)}{\xi^2} \right) \xi H' + \tfrac{2 a_1 A + (a_0 - a_1) \bar{r}}{2 (\bar{r} A)^2} H,
    \end{align*}
    where $a_0, a_1$ are as in~\eqref{eqn:auxiliary-a1}~--~\eqref{eqn:auxiliary-a0}.
    For $x=0$, we compute $P_{\xi}(0) = P_0(\xi) = A^4 K(0,\xi)$, where
    \begin{align*}
    K(0,\xi) &= H''- \left(  \frac{n-k}{\bar{r} A} \Bigl(1 - \frac{\bar{r}}{2A} \Bigr) - \frac{k-1}{\xi^2} \right) \xi H' + \frac{(n-k) (A^2 \bar{r} - A + \tfrac{1}{2} \bar{r})}{(\bar{r} A)^2 A} \Bigl(1 - \frac{\bar{r}}{2A} \Bigr) H, \\
    K(1,\xi) &= (1 + A^{-2} \xi^2) H'' - \left( \frac{1}{\bar{r} A} \left( n-k - \frac{2}{1+A^2} \right) - \frac{k-1}{\xi^2}  \right) \xi H' \\
    & \quad + \frac{(n-k) (A^2 + 1) \left( \bar{r} (A^2 + 1) - A \right)  + (3 A - 2 \bar{r} (A^2 + 1) )}{\bar{r}^2 A (A^2 + 1)^2} H.
    \end{align*}
    Consequently, the assumptions $K(0,\xi) < 0$ and $K(1,\xi) < 0$ imply that
    \[
    P_0(\xi) = P_{\xi}(0) < 0\quad \text{ and } \quad P_0(\xi) + P_1(\xi) + P_2(\xi) + P_3(\xi) = P_{\xi}(1) < 0.
    \]
    We define the function
\begin{equation}\label{eqn:final-quantity-P}
    \cP(\xi) := P_2(\xi) + P_3 (\xi) + \min \{ P_3(\xi), 0 \}.
\end{equation}
    and observe that the cubic $P_{\xi}(x)$ satisfies the identity
    \begin{equation}\label{eqn:general-expansion}
    P_{\xi}(x) = (1-x) P_{\xi}(0) +x P_{\xi}(1) - x(1-x) \bigl( P_2(\xi) + (1+x) P_3(\xi) \bigr) \, .
\end{equation}
    For $x \in [0,1]$, we have $P_2(\xi) + (1+x) P_3(\xi) \geq \cP(\xi)$ by inspecting the sign of $P_3$ and evaluating at the endpoints.
    Using $P_{\xi}(0)< 0$ and $P_{\xi}(1) < 0$ in ~\eqref{eqn:general-expansion} together with $\cP(\xi) > 0$, we obtain
    \[
    P_{\xi}(x) \leq (1-x) P_{\xi}(0)  +x P_{\xi}(1) - x(1-x) \cP(\xi) < 0.
    \]
    Finally, compactness implies that $K(x,\xi) = \frac{1}{(A^2+x)^2} P_{\xi}(x) < - c(n,k) < 0$, completing the proof.
\end{proof}

\begin{proof}[Proof of Theorem~\ref{thm:cones-are-minimizing-smalltheta}]
    Combining Proposition~\ref{prop:supersolution} with Lemmas~\ref{lemma:W-computation} and~\ref{lemma:K-computation} shows that for a given pair $(n,k)$, the existence of a capillary supersolution lying above the graphing function of the $\mathbf{C}_{n,k,\theta}$ is reduced to verifying the inequalities~\eqref{eqn:S'1-new}~--~\eqref{eqn:S'3-new}.
    For the pairs $(n,k)$ in the range specified by Theorem~\ref{thm:cones-are-minimizing-smalltheta}, we produce $\beta_{n,k}$ such that all the above inequalities are satisfied.
    Examples of admissible parameters $\beta_{n,k}$ are presented in Table~\ref{table:ExplicitSuperQuadratics}; we refer the reader to Appendix~\ref{subsection:exponents} for a detailed presentation of valid supersolution exponents and the verification of the conditions of Lemmas~\ref{lemma:W-computation} and~\ref{lemma:K-computation} for each triple $(n,k,\beta)$.

    We may therefore combine Propositions~\ref{prop:small-theta-one-sided} and~\ref{prop:supersolution} to see that for some $\theta_{n,k} > 0$, the cones $\mathbf{C}_{n,k,\theta}$ satisfy the conditions of Lemma~\ref{lemma:minimizing-by-subsolutions} for every $\theta \in (0,\theta_{n,k})$.
    We conclude that for every $\theta \in (0, \theta_{n,k})$, the cones $\mathbf{C}_{n,k,\theta}$ are minimizing for $\cA^{\theta}$.
\end{proof}

\begin{proof}[Proof of Theorem~\ref{thm:one-phase-cones}]
Corollary~\ref{cor:one-phase-strictly-minimizing} shows that the one-phase solution $U_{n,k} = c_{n,k} \rho f_{n,k}(t)$ is strictly minimizing from below for $\cJ$ and strictly stable for every $n \geq 7$ and $1 \leq k \leq n-2$.
    Moreover, for all the pairs $(n,k)$ specified in Theorem~\ref{thm:cones-are-minimizing-smalltheta}, we can produce supersolution parameters $\beta_{n,k}$ satisfying the conditions of Proposition~\ref{prop:supersolution}.
    Combining the two Propositions, we conclude that the homogeneous solution $U_{n,k}$ satisfies the conditions of Lemma~\ref{lemma:one-phase-minimizing-by-subsolutions}, so it is minimizing for $\cJ$.
\end{proof}

\section{Minimality for large contact angle}\label{section:pi/2}

In this section, we prove Theorem~\ref{thm:cones-are-minimizing-nearpi/2} showing that the capillary cones for angles near $\tfrac{\pi}{2}$ are minimizing for $n \geq 7$ as long as $(n,k) \not\in \{ (7,1), (7,5) \}$.
As a consequence, we show that $7$ is the optimal dimension for the capillary problem with contact angle close to $\frac{\pi}{2}$, which proves Theorem~\hyperref[thm:regularity]{\ref{thm:regularity} $(ii)$}.
Recall from Lemma~\ref{lemma:lawson-cones} that the cone $\mathbf{C}_{n,k,\frac{\pi}{2}}$ corresponds to the halved Lawson cone $C(\bS^{n-k-1} \times \bS^k_+)$, cut along a plane of symmetry $\{ z = 0 \}$.
Conversely, doubling the cones $\mathbf{C}_{n,k,\frac{\pi}{2}}$ along their free boundary recovers all Lawson cones.
Therefore, the capillary cones $\mathbf{C}_{n,k,\frac{\pi}{2} - \ve}$ are close to halved Lawson cones, which are strictly minimizing, in the sense of Definition~\ref{def:strictly-minimizing}, for all $n \geq 7$ except $C( \bS^5 \times \bS^1)$ and $C(\bS^1 \times \bS^5)$, corresponding to all $1 \leq k \leq n-2$ except $(n,k) \in \{ (7,1), (7,5) \}$ in our notation.
We refer the reader to the works~\cites{hardt-simon, lawlor, lin, morgan, simoes}.

In what follows, we will study the two sides $E_{\pm}$ of the graphical cones $\mathbf{C} := \{ (x,y,z) : z = \rho f(t) \}$.
These two regions are expressible as
\[
E_- := \{ (\rho, t, z) : z < \rho f(t) \}, \qquad E_+ := \{ (\rho, t, z) : z > \rho f(t) \}.
\]
We will prove that the cones $\mathbf{C}_{n,k,\frac{\pi}{2}-\ve}$ are strictly minimizing by constructing barrier surfaces $S_{\pm} \subset E_{\pm}$ asymptotic to $\mathbf{C}$ from either side, with a signed mean curvature and smaller (resp.~larger) contact angle along their free boundary on $\Pi$.
Unlike the constructions for small contact angles in Section~\ref{sec:minimizing}, our proof of strict minimality is essentially the same on the two sides $E_{\pm}$ of the cone $\mathbf{C}_{n,k,\frac{\pi}{2}-\ve}$, following the symmetry of the standard calibrations of Lawson's cones (see, for example,~\cites{bdgg, lawson, lawlor, dephilippis-maggi}).
Therefore, in our arguments below we shall mainly focus on producing the surface $S_+ \subset E_+$, whose construction is more involved because of the condition of smaller contact angle along the free boundary.
After describing the properties of $S_+$ in great detail, we will highlight the adaptations needed at each step to carry out the corresponding construction of the barrier surface $S_- \subset E_-$.

We now outline the main steps of the construction.
In Section~\ref{section:geometry-near-pi/2}, we obtain quantitative estimates for how close the cones $\mathbf{C}_{n,k,\frac{\pi}{2}-\ve}$ are to $C(\bS^{n-k-1} \times \bS^k_+)$.
These estimates control the behavior of the cones at infinity and at large radii, when $\ve$ is small compared to the scale.
In Section~\ref{section:linearized-operator}, we construct the piece of the surface $S_{\pm}$ ``at infinity,'' given by an explicit normal graph decaying to the cone with the appropriate power (contained in the indicial interval of the Jacobi operator) and forming a smaller (resp.~larger) contact angle than the cone along the free boundary; see Lemmas~\ref{lemma:infinity-piece-small-epsilon} and~\ref{lemma:infinity-piece-on-E-minus}.
In Section~\ref{section:compact-caps}, we construct compact ``cap'' surfaces with the desired mean curvature sign and contact angle by perturbing explicit sub-calibrations of Lawson's cones~\cites{sharp-stability-plateau, lawson-liu}, which are described in equations~\eqref{eqn:record-varphi-n,k} and~\eqref{eqn:record-varphi-n,k-on-E-minus}.
Importantly, these compact caps can be extended to any prescribed (large but finite) radius $R_2$ and are expressible as normal graphs in an annular region over $\mathbf{C}_{n,k,\frac{\pi}{2}-\ve}$, provided that $\ve$ is sufficiently small compared to $R_2$, cf.~Proposition~\ref{prop:mean-convex-cap}.
Finally, in Section~\ref{section:construct-foliation}, we show that the compact cap can be glued to the piece at infinity, as normal graphs over the cone, over a sufficiently large annular region.
By matching the leading-order behavior of the two graphs, we ensure that the resulting glued surface $S_{\pm}$ is star-shaped, so its rescalings foliate $E_{\pm}$ and prove the strict minimality of the cone; see Propositions~\ref{prop:mean-convex-surface} and~\ref{prop:mean-convex-surface-on-E-minus}.

The above strategy shares some similarities with the arguments of Savin-Yu~\cite{savin-yu-AP-cones}*{\S~7} for proving the minimality of axisymmetric cones of the Alt-Phillips problem with exponent $\gamma \to 1$.
On the other hand, our strategy is fairly general and can be applied to any capillary cone sufficiently close to a given strictly stable, strictly minimizing one.
Moreover, our arguments only require one-sided strict minimality of a given cone; in particular, we prove that the cones $\mathbf{C}_{7,1,\frac{\pi}{2}-\ve}$ and $\mathbf{C}_{7,5,\frac{\pi}{2}-\ve}$ are one-sided strictly area-minimizing (on the $E_+$ side and on the $E_-$ side, respectively) by modifying strict sub-calibrations of the cones $C(\bS^5 \times \bS^1_+)$ and $C(\bS^1 \times \bS^5_+)$ on their strictly minimizing side.
We then use the construction of Lawlor~\cite{lawlor} to produce an $\cA^{\theta}$-decreasing competitor for the cones $\mathbf{C}_{7,1,\theta}$ and $\mathbf{C}_{7,5,\theta}$, showing that they are not globally minimizing; see Section~\ref{section:non-minimizing}.

\subsection{Geometry near free-boundary cones}\label{section:geometry-near-pi/2}

We will first reinterpret the $C^{\infty}_{\on{loc}}$-closeness of the cones $\mathbf{C}_{n,k,\theta}$ to $\mathbf{C}_{n,k,\frac{\pi}{2}}$, for $\theta \in ( \frac{\pi}{2} - \ve_0, \frac{\pi}{2})$, discussed in Theorem~\hyperref[thm:capillary-cones]{\ref{thm:capillary-cones} $(ii)$}, in a more quantitative geometric manner.
The corresponding profile functions $f_{\ve}$ are $C^0$-close to $\hat{f}_{n,k}$ on their positive phase; however, they are not uniformly $C^m$-close due to their blow-up behavior near the endpoint.
Therefore, a more canonical gauge to compare the profile curves is by normal graphs, or that the profile curves are $C^\infty$ close not as graphs, but geometrically in $\bR^2$ as illustrated in Figure~\ref{fig:fams-n7k1}.

To obtain the needed $C^{\infty}$-closeness, we first resolve the singularity of $f_{\ve}$ near the free boundary.
Consider a solution $f$ of the ODE~\eqref{eqn:cone-ODE} producing a capillary angle $\theta \in ( \frac{\pi}{2}- \ve_{n,k}, \frac{\pi}{2}]$, for small $\ve_{n,k}>0$.
By Theorem~\ref{thm:capillary-cones}, the function $f$ has $f(0) = a_{n,k} - O ( \frac{\pi}{2} - \theta)$, where $a_{n,k} = \sqrt{\frac{k}{n-k-1}} = \hat{f}_{n,k}(0)$, and reaches zero at some point $t_*$ with $t_* - \sqrt{\frac{k}{n-1}} = O( \frac{\pi}{2} - \theta )$, where $\tan \theta = - \sqrt{1-t_*^2} f'(t_*)$.
Moreover, $f$ has derivative blow-up at some point $\hat{t}_* > t_*$, where $f(\hat{t}_*) = - \delta$ and $f'(\hat{t}_*) = - \infty$.
\begin{lemma}\label{lemma:equivalent-parametrization-delta-tan-theta}
    For $f$ as above and $\frac{\pi}{2} - \theta \in [0,\ve_{n,k})$ sufficiently small, we have the relations
    \begin{align*}
    \delta \tan \theta &= \sqrt{\frac{k}{n-k-1}} + O \left( \frac{\pi}{2} - \theta \right), \qquad \delta = \sqrt{\frac{k}{n-k-1}} \left( \frac{\pi}{2} - \theta \right) +  O \left( \bigl( \frac{\pi}{2} - \theta \bigr)^2 \right), \\
     \hat{t}_* - t_* &= \frac{n-k-1}{2 \sqrt{k(n-1)}} \delta^2 + O \left( \bigl( \frac{\pi}{2} - \theta \bigr)^3 \right).
    \end{align*}
\end{lemma}
\begin{proof}
    The discussion of Proposition~\ref{prop:unique-crossing of solutions} and Lemmas~\ref{lemma:derivative-blow-up} and~\ref{lemma:psi-detect-blow-up} shows that the function $f$ has square-root blow-up at the point $\hat{t}_*$, meaning that $f(t) \sim - \delta + C (\hat{t}_* - t)^{\frac{1}{2}}$ near $\hat{t}_*$.
    We may therefore produce a non-negative function $g(s)$, defined for $s \in [0, \hat{t}_*^2]$ with a simple zero at $s = \hat{t}_*^2$, such that $f(t) + \delta = \sqrt{g(t^2)}$.
    For $\theta = \frac{\pi}{2}$, the Lawson solution satisfies $\hat{f}_{n,k}(t)= \sqrt{\frac{k - (n-1) t^2}{n-k-1}}$, so the expression of Lemma~\ref{lemma:lawson-cones} gives $g_0(s) = \frac{k - (n-1) s}{n-k-1}$.
The results of Theorem~\ref{thm:capillary-cones} imply that for $\frac{\pi}{2} - \theta \in (0,\ve_{n,k})$ sufficiently small, we may estimate $\| D^{\ell}( g - g_0 ) \|_{C^2([0,\hat{t}_*^2])} \leq C(n,k,\ell) | \frac{\pi}{2} - \theta|$ upon extending $g_0(s)$ to a function on $[0,1]$.
Differentiating the relation $g(t^2) = (f(t) + \delta)^2$ and setting $t = t_*$, we obtain
\[
2 t g'(t^2) = 2 f'(t) ( f(t) + \delta) , \quad \implies \quad 2 t_* \sqrt{1-t_*^2} \, g'(t_*^2) = - 2 \delta \tan \theta.
\]
Recalling that $g_0(s) = \frac{k - (n-1) s}{n-k-1}$ and $g'_0(s) = - \frac{n-1}{n-k-1}$, we use $t_* - \sqrt{\frac{k}{n-1}} = O ( \frac{\pi}{2} - \theta)$ to find
\[
g'(t_*^2) = - \frac{n-1}{n-k-1} + O \bigl( \frac{\pi}{2} - \theta \bigr), \qquad t_* \sqrt{1 - t_*^2} = \frac{\sqrt{k(n-k-1)}}{n-1} + O \bigl( \frac{\pi}{2} - \theta \bigr).
\]
Multiplying the two equalities implies that $\delta \tan \theta = \sqrt{\frac{k}{n-k-1}} + O ( \frac{\pi}{2} - \theta)$.
Let us write $a_{n,k} = \sqrt{\frac{k}{n-k-1}}$ and $\cot \theta = ( \frac{\pi}{2} - \theta) + O ( ( \frac{\pi}{2} - \theta)^2)$ to rearrange
\[
\delta \tan \theta = a_{n,k} + O \Bigl( \frac{\pi}{2} - \theta \Bigr) \implies \delta = a_{n,k} \cot \theta + O \Bigl( \bigl( \frac{\pi}{2} - \theta \bigr)^2 \Bigr) = a_{n,k} \Bigl( \frac{\pi}{2} - \theta \Bigr) + O \Bigl( \bigl( \frac{\pi}{2} - \theta \bigr)^2 \Bigr),
\]
proving the first expression.
Next, using the mean value theorem, we can find point $s_* \in ( t_*^2, \hat{t}_*^2)$ with the property that
\[
- \delta^2 = g(\hat{t}_*^2) - g(t_*^2) = g'(s_*) ( \hat{t}_*^2 - t_*^2).
\]
Using the property $|-  g'(s_*) - \frac{n-1}{n-k-1} | = | g'_0(s_*) - g'(s_*) | = O ( \frac{\pi}{2} - \theta)$, we conclude that $\hat{t}_*^2 - t_*^2 = \frac{n-k-1}{n-1} \delta^2 + O ( ( \frac{\pi}{2} - \theta)^3)$.
Writing $\hat{t}_* - t_* = \frac{\hat{t}_*^2 - t_*^2}{\hat{t}_* + t_*}$ and $\hat{t}_* > t_* > \sqrt{\frac{k}{n-1}}$ makes $\hat{t}_* - t_* = O (\delta^2)$, whereby $\hat{t}_* + t_* = 2 \sqrt{\frac{k}{n-1}} + O ( \frac{\pi}{2} - \theta)$.
The expression for $\hat{t}_* - t_*$ follows, completing the proof.
\end{proof}

The results of Lemma~\ref{lemma:equivalent-parametrization-delta-tan-theta} imply that the profile functions of the cones $\mathbf{C}_{n,k,\theta}$ with $\frac{\pi}{2} - \theta \in [0, \ve_{n,k})$ sufficiently small may equivalently be studied in terms of their terminal value at blow-up.
\begin{definition}\label{definition:f-eps-near-pi-over-2}
    For $\ve \in [0,\ve_{n,k})$, let $f_{\ve}$ be a solution of the ODE~\eqref{eqn:odeStar} blowing up at a point $\hat{t}_{\ve}$, where
    \[
    f_{\ve} ( \hat{t}_{\ve}) = - \ve, \qquad f'_{\ve}(\hat{t}_{\ve}) = - \infty, \qquad \hat{t}_{\ve} - \sqrt{\tfrac{k}{n-1}} = O(\ve).
    \]
\end{definition}
The results of Section~\ref{section:o(n-k)-cones} imply that the zero map $f(0) \mapsto f^{-1}(0)$ is strictly decreasing, while the angle map $f(0) \mapsto \theta(f)$ is strictly increasing, whereby $t_{\ve} \in (\sqrt{\frac{k}{n-1}}, \hat{t}_{\ve})$ and $f_{\ve}(0) = a_{n,k} - O(\ve)$ for any such solution, where $a_{n,k} = \sqrt{\frac{k}{n-k-1}} = \hat{f}_{n,k}(0)$.
The resulting angle is then $\theta = \frac{\pi}{2} - O(\ve)$, therefore Lemma~\ref{lemma:equivalent-parametrization-delta-tan-theta} applies, for $\ve \in (0,\ve_{n,k})$ sufficiently small, to write
\begin{align}
    g_{\ve}(t^2) &:= ( f_{\ve}(t) + \ve)^2, & \qquad \| D^{\ell} (g_{\ve} - g_0) \|_{C^2(0, \hat{t}_{\ve}^2) } &\leq C(n,k,\ell) \ve, \label{eqn:g-eps-g0-closeness} \\
    \ve \tan \theta &= \sqrt{\frac{k}{n-k-1}} + O(\ve), & \qquad \theta &= \frac{\pi}{2} - \sqrt{\frac{n-k-1}{k}} \, \ve + O(\ve^2), \label{eqn:small-epsilon-relation} \\
     \hat{t}_{\ve}- t_{\ve} &= \frac{n-k-1}{2 \sqrt{k(n-1)}} \ve^2 + O(\ve^3), & \qquad t_{\ve} &= \sqrt{\frac{k}{n-1}} + O(\ve).\label{eqn:t-hat-t-eps}
    \end{align}

Unlike the cones $\mathbf{C}_{n,k,\theta}$, which are unique for given $\theta$, we are not claiming that the terminal-value parametrization $\ve \mapsto f_{\ve}$ is unique (although it likely is).
The only information required is that for $\frac{\pi}{2} - \theta \in [0,\ve_{n,k})$ sufficiently small, the profile function of the cone has the form $f_{\ve}$ of Definition~\ref{definition:f-eps-near-pi-over-2} for some $\ve \in (0,\ve_{n,k})$, which satisfies the $\ve$-closeness properties~\eqref{eqn:g-eps-g0-closeness}~--~\eqref{eqn:t-hat-t-eps}.
Since $\ve$ and $\theta$ are linearly related via~\eqref{eqn:small-epsilon-relation}, we may equivalently study the cones $\mathbf{C}_{n,k,\theta}$ with $\frac{\pi}{2} - \theta = O (\ve)$ via the expression $\{ (\rho, t,z ) : z = \rho f_{\ve}(t) \}$ corresponding to such a profile function.

Writing $a_{n,k} = \sqrt{\frac{k}{n-k-1}}$, the properties~\eqref{eqn:g-eps-g0-closeness}~--~\eqref{eqn:t-hat-t-eps} allow us to express $\ve \tan \theta = a_{n,k} + O (\ve)$ and $t_{\ve}= \sqrt{\frac{k}{n-1}} + c(n,k) \ve + O(\ve^2)$, for some constant $c(n,k)>0$ and $\ve \in (0,\ve_{n,k})$.
Therefore, the free boundary of the cone $\mathbf{C}_{n,k,\frac{\pi}{2}-\ve}$ along the plane $\Pi = \{ z = 0 \}$ is given by
\begin{equation}\label{eqn:kappa-n-k}
\left\{ |y| = \frac{t_{\ve}}{\sqrt{1- t_{\ve}^2}} |x| \right\}, \qquad \frac{t_{\ve}}{\sqrt{1 - t_{\ve}^2}} = \sqrt{\frac{k}{n-k-1}} + \left( \frac{n-1}{n-k-1} \right)^{\frac{3}{2}} c(n,k) \ve + O(\ve^2)
\end{equation}
In particular, we may write $\{ |y| = a_{n,k}(1 + \kappa_{n,k}\ve + O(\ve^2)) |x|\}$ where $\kappa_{n,k} = \frac{n-1}{n-k-1} \sqrt{\frac{n-1}{k}} c(n,k)$.

\smallskip

For the cones $\mathbf{C}_{n,k,\frac{\pi}{2} - \ve}$, the normal vector pointing into $E_+$ has everywhere non-zero vertical component $e_{n+1}$.
As a result, the normal graph of an everywhere-positive function over this cone would be disjoint from $\Pi = \{ z = 0 \}$.
To produce hypersurfaces with free boundary along $\Pi$, we will consider normal graphs over the extended cone with extended link in $\bS^n$, which continues below the equator until the blow-up instance of $f_{\ve}$.
We shall denote the extended link of the cone by $\hat{\Sigma}_{\ve}$.
Using the spherical parametrization
\[
F_{\ve}(\rho, t , \xi , \eta) = ( \rho \sqrt{1-t^2}  \,\xi, \rho t  \eta , \rho f_{\ve}(t)) , \qquad \xi \in \bS^{n-k-1}, \; \eta \in \bS^{k-1},
\]
we compute that $|F_{\ve}(\rho, t , \xi, \eta)|^2 = \rho^2 ( 1+ f_{\ve}(t)^2)$.
Consequently,
\begin{equation}\label{eqn:s-hat-eps-extended-link}
    \hat{\Sigma}_{\ve} := \left\{ \left( \frac{\sqrt{1-t^2}}{\sqrt{1+f_{\ve}^2}} \xi, \frac{t}{\sqrt{1+f_{\ve}^2}} \eta, \frac{f_{\ve}}{\sqrt{1+f_{\ve}^2}} \right) : (t, \xi, \eta) \in [0, \hat{t}_{\ve}] \times \bS^{n-k-1} \times \bS^{k-1} \right\} \subset \bS^n
\end{equation}
where $\hat{t}_{\ve} > t_{\ve}$ is the blow-up instance of the solution $f_{\ve}$ with $f_{\ve}(\hat{t}_{\ve}) = -\ve$.
We denote
\begin{equation}\label{eqn:extended-cone}
\hat{\mathbf{C}}_{n,k,\frac{\pi}{2} - \ve} := \{ p = R \omega :  ( R, \omega) \in (0,\infty) \times \hat{\Sigma}_{\ve} \} = \left\{ z = \rho f_{\ve}(t) : t \in [0,\hat{t}_{\ve}] \right\}
\end{equation}
which we call the \textit{extended cone} of $\mathbf{C}_{n,k,\frac{\pi}{2}-\ve}$.
We observe that 
\[
\partial \hat{\mathbf{C}}_{n,k,\frac{\pi}{2}-\ve} = \{ ( \rho, t, z) :  t = \hat{t}_{\ve}, z = - \ve \rho \} \subset \{ z < 0 \} \cup \{ (0,0,0)\}.
\]
We previously remarked that the cones $\mathbf{C}_{n,k,\frac{\pi}{2}-\ve}$, for $\ve$ sufficiently small, are expressible as small normal graphs over the cone $\mathbf{C}_{n,k,\frac{\pi}{2}}$ inside annuli $A(R_1, R_2) := B_{R_2} \setminus \bar{B}_{R_1}$.
We observe that the reverse property also holds, in the sense of normal graphs over the extended cone:
\begin{lemma}\label{lemma:graph-over-extended-link-pi/2}
    For any $n \geq 3$ and $1 \leq k \leq n-2$, there exists an $\ve_{n,k}>0$ such that for all $\ve \in (0,\ve_{n,k})$, we may express the link $\Sigma_0$ as the normal graph of a function $v^\Sigma_{\ve}$ over $\hat{\Sigma}_{\ve}$, meaning that $\normalfont{\Sigma_0 = \text{graph}_{\hat{\Sigma}_{\ve}} v^\Sigma_{\ve} \cap \bS^n_+}$, where $\| \nabla^{\ell}_S v^\Sigma_{\ve} \|_{C^1(\hat{\Sigma}_{\ve})} \leq C(n,k,\ell) \ve$.
    Equivalently, over any annulus $A(R_1, R_2)$, the cone $\mathbf{C}_{n,k,\frac{\pi}{2}} \cap A(R_1, R_2)$ is expressible as the normal graph of a function $v^{\mathbf{C}}_{\ve}$ over $\hat{\mathbf{C}}_{n,k,\frac{\pi}{2}-\ve}$, given by the homogeneous extension of $v^\Sigma_\ve$.
    This means that
    \[
    \mathbf{C}_{n,k,\frac{\pi}{2}} \cap A(R_1, R_2) = \normalfont{\text{graph}}_{\hat{\mathbf{C}}_{n,k,\frac{\pi}{2}-\ve} \cap A(R_1, R_2)} v_{\ve}^\mathbf{C} \cap \bR^{n+1}_+ ,
    \] 
    where the function $v_{\ve}^\mathbf{C}$ satisfies
    \begin{equation}\label{eqn:v-eps-C-estimate}
        |v^{\mathbf{C}}_{\ve}| \leq C(n,k) \ve R_2, \qquad |\nabla_{\hat{\mathbf{C}}_{\ve}} v^{\mathbf{C}}_{\ve}| \leq C(n,k) \ve,\qquad |\nabla_{\hat{\mathbf{C}}_{\ve}}^{\ell} v^{\mathbf{C}}_{\ve}| \leq C(n,k) \ve R_1^{1-\ell}.
    \end{equation}
\end{lemma}

\begin{proof}
We use the $O(n-k)\times O(k)$-symmetry and parametrization $\omega_{\ve}(t,\xi,\eta)$ of $\hat\Sigma_\ve$ from~\eqref{eqn:s-hat-eps-extended-link}
for $t \in [0,\hat{t}_\ve]$ to reduce the claim to finding an invariant function $v := v^{\Sigma}_{\ve}$ on $\hat\Sigma_\ve$ depending only on $t$.
This graphing profile will satisfy $|\nabla_{\hat\Sigma_\ve} v| = |v'(t)|/|\partial_t\omega_\ve|$.
Let $\Phi(\omega)=|x(\omega)|^2-\frac{n-k-1}{n-1}$, so that the Clifford hypersurface is given by $\Sigma_0=\{\Phi=0\}\cap\mathbb S^n_+$.
Let $\nu_\ve$ denote the unit normal of $\hat\Sigma_\ve\subset\mathbb S^n$ induced from $\nu_{\mathbf{C}}$, with $\langle\nu_\ve,e_{n+1}\rangle>0$ and set
\[
G_\ve(t,s):=\Phi\!\left(\exp_{\omega_\ve(t)}(s\nu_\ve(t))\right).
\]
Since $\Phi(\omega_0(t))\equiv0$ for the Lawson profile and $(f_{\ve}(t) + \ve)^2 \to \hat{f}_{n,k}(t)^2$ in $C^{\infty}_{\text{loc}}$, with $g_{\ve}(t^2) := (f_{\ve}(t) + \ve)^2$ and $g_0(s) = \frac{k - (n-1) s}{n-k-1}$ satisfying~\eqref{eqn:g-eps-g0-closeness}, we have
$\sup_{t\in[0,\hat t_\ve]}|G_\ve(t,0)|\le C\ve$.
Moreover, $\partial_s G_\ve(t,0)=\langle\nabla\Phi(\omega_\ve(t)),\nu_\ve(t)\rangle$
is uniformly bounded away from $0$ on $[0,\hat t_\ve]$ for $\ve$ small.
Indeed, denoting by $\widetilde\nu_\ve(t,\xi,\eta)=\bigl(-\sqrt{1-t^2}\,(f_\ve-tf_\ve')\,\xi,\,-(t f_\ve+(1-t^2)f_\ve')\,\eta,\,1\bigr)$ the parallel vector to $\nu_{\ve}$, a direct computation shows
\[
\partial_sG_\ve(t,0)=2\,x(\omega_\ve(t))\cdot x(\nu_\ve(t))
=-\frac{2(1-t^2)\bigl(f_\ve(t)-t f_\ve'(t)\bigr)}{\sqrt{1+f_\ve(t)^2}\,|\widetilde\nu_\ve(t)|},
\]
and strict concavity implies $f_\ve-tf_\ve'\ge f_\ve(0)\ge c>0$ when $|f_\ve'|$ is bounded.
In the large-slope regime, $(f_\ve-tf_\ve')/|\widetilde\nu_\ve|\sim \frac{t}{\sqrt{1-t^2}}$, since $1 + f_{\ve}^2 = \frac{(n-1) (1-t^2)}{n-k-1} + O(\ve)$.
Combining these two estimates shows $|\partial_sG_\ve(t,0)|\ge c_0(n,k)$ uniformly on $[0,\hat t_\ve]$.
Therefore, the implicit function theorem provides a unique smooth
$v_\ve^\Sigma:[0,\hat t_\ve]\to(-\delta,\delta)$ such that $G_\ve(t,v_\ve^\Sigma(t))=0$ and
$\|v_\ve^\Sigma\|_{C^0}\le C\ve$.
This exhibits $\Sigma_0=\mr{graph}_{\hat\Sigma_\ve}(v_\ve^\Sigma)\cap\mathbb S^n_+$.

Differentiating $G_\ve(t,v_\ve^\Sigma(t))=0$ gives
$(v_\ve^\Sigma)'=-(\partial_t G_\ve)/(\partial_s G_\ve)$, and since $\partial_t G_0(t,0)=0$
and $(\omega_\ve,\nu_\ve)$ depend smoothly on $\ve$, we obtain
$|\partial_tG_\ve(t,s)|\le C\ve\,|\partial_t\omega_\ve(t)|$ for $|s|\le C\ve$.
Together with $|\partial_s G_\ve|\ge c_0$ this yields
$|(v_\ve^\Sigma)'(t)|\le C\ve\,|\partial_t\omega_\ve(t)|$ and hence
$\|\nabla_{\hat\Sigma_\ve}v_\ve^\Sigma\|_{C^0}\le C\ve$.
Higher covariant derivatives follow by differentiating the identity repeatedly and using uniform bounds on mixed derivatives of $G_\ve$.
Finally, define the homogeneous extension $v_\ve^{\mathbf C}(R\omega)=Rv_\ve^\Sigma(\omega)$.
Since $v_\ve^{\mathbf C}$ is $1$-homogeneous, $|\nabla^\ell v_\ve^{\mathbf C}|$ is
$(1-\ell)$-homogeneous, giving~\eqref{eqn:v-eps-C-estimate} on $A(R_1,R_2)$.
\end{proof}

\subsection{Linearized operator}\label{section:linearized-operator}

We begin the construction of the barrier surfaces $S_{\pm} \subset E_{\pm}$ by constructing a piece outside a sufficiently large ball.
We call this the \textit{piece at infinity}, which will be a hypersurface $S^{\phi}_{\pm} \subset E_{\pm}$ produced as a small positive (resp.~negative) normal graph of a function $\phi$ over the capillary cone ${\mathbf{C}}_{n,k,\frac{\pi}{2}-\ve}$.
In the graphical parametrization $\{ z = u(x,y) \}$ of the cones $\mathbf{C}_{n,k,\theta}$, we work with the upward normal vector $\nu_{\mathbf{C}} = \frac{(- \nabla u, 1)}{\sqrt{1 + |\nabla u|^2}}$, which points into the region $E_+$ and produces positive normal graphs lying on the $E_+$ side of ${\mathbf{C}}$.
This means that
\[
S^{\pm}_{\phi} := \text{graph}_{\mathbf{C}} \phi =  \{ p \pm \phi(p) \nu_{\mathbf{C}}(p) : p \in \mathbf{C} \} \subset E_{\pm} \, ,
\]
so that $S^-_{\phi} \subset E_-$ is a negative normal graph with this convention for $\nu_{\mathbf{C}}$.

Let $H[S^{\pm}_{\phi}]$ denote the mean curvature of $S^{\pm}_{\phi}$, computed with respect to the induced upward-pointing normal vector.
The linearization of $-H[S^{\pm}_{\phi}]$ at the cone is given by the Jacobi operator
\[
\cL_{\mathbf{C}} = \Delta_{\mathbf{C}} + |A_{\mathbf{C}}|^2.
\]
A standard estimate analogous to~\eqref{eqn:MC-bound} (for the mean curvature operator on graphs) shows that
\begin{equation}\label{eqn:mean-curvature-operator}
    |H[S^{\pm}_{\phi}] \pm \cL_{\mathbf{C}} \phi| \leq C(n) ( |\nabla_{\mathbf{C}} \phi|^2 + |\phi ||\nabla^2_{\mathbf{C}} \phi|)
\end{equation}
for any function $\phi \in C^{\infty}(\mathbf{C} \setminus B_R)$.
The computation for $H[S^-_{\phi}]$ comes from writing $S^-_{\phi}$ as the normal graph of $(-\phi)$ in the direction of $\nu_{\mathbf{C}}$, so $H[S^-_{\phi}] + \cL_\mathbf{C} (-\phi) = H[S^-_{\phi}] - \cL_{\mathbf{C}} \phi$ is consistent with the upward orientation.

For a minimal cone $\mathbf{C} \subset \bR^{n+1}$ in Euclidean space, we denote by $\lambda_1 , \varphi_1$ the first eigenvalue and corresponding positive eigenfunction for the Jacobi operator $-\cL_{\Sigma} := -(\Delta_{\Sigma} + |A_{\Sigma}|^2)$ on the spherical link $\Sigma := \mathbf{C} \cap \bS^n$.
Then, the solutions of $\cL_{\mathbf{C}} w =0 $ with $w> 0$ on all of ${\mathbf{C}}$ are given by
\begin{equation}\label{eqn:HS-indicial-roots}
w = ( c_- R^{- \underline{\gamma}} + c_+ R^{- \overline{\gamma}} ) \varphi_1, \qquad \text{where } \; \underline{\gamma},\overline{\gamma} = \frac{n-2 \pm \sqrt{(n-2)^2 + 4 \lambda_1}}{2}
\end{equation}
for some constants $c_{\pm}$.
For the free-boundary cones $C(\bS^{n-k-1} \times \bS^{k}_+)$, a standard computation shows that the link $\Sigma_0$ is a halved Clifford hypersurface with $|A_{\Sigma_0}| = \sqrt{n-1}$.
The first eigenfunction of the operator $- \cL_{\Sigma_0}$ is $\varphi_1 = 1$ (constant functions) with $\lambda_1 = - (n-1)$, whereby
\[
\underline{\gamma},\overline{\gamma} = \frac{n-2 \pm \sqrt{n^2 - 8n + 8}}{2}.
\]
For small $\ve \in (0,\ve_{n,k})$, Theorem~\ref{thm:capillary-cones} and Lemma~\ref{lemma:graph-over-extended-link-pi/2} show that the spherical link $\Sigma_{\ve} := \mathbf{C}_{n,k,\frac{\pi}{2} - \ve} \cap \bS^n_+$ of the capillary cone $\mathbf{C}_{n,k,\frac{\pi}{2} - \ve}$ is expressible as an $O(\ve)$-small normal $C^\infty$ graph over $\Sigma_0$, hence
\begin{equation}\label{eqn:sff-norm-epsilon}
    |A_{\Sigma_{\ve}}|^2 = |A_{\Sigma}|^2 + O (\ve) = (n-1) + O(\ve).
\end{equation}
For the remainder of this section, we will denote
\[
R := |p| = \sqrt{|x|^2 + |y|^2 + z^2}
\]
for a point $p = (x , y, z) \in \bR^{n-k} \times \bR^k \times \bR$.
Notably, $|A_{\mathbf{C}}|^2 = R^{-2} |A_{\Sigma}|^2$ for any cone, whereby
\[
\cL_{\mathbf{C}} = \partial_{R}^2 + \frac{n-1}{R} \partial_{R} + \frac{1}{R^2} (\Delta_{\Sigma} + |A_{\Sigma}|^2).
\]
On the Lawson cones $C(\bS^{n-k-1} \times \bS^{k})$, we have $|A_{\Sigma_0}|^2 = n-1$ and the Laplace operator $\Delta_{\Sigma_0}$ admits a further decomposition on the product $\Sigma_0 = \bS^{n-k-1}(\sqrt{\tfrac{n-k-1}{n-1}}) \times \bS^k_+(\sqrt{\tfrac{k}{n-1}})$.
Consequently, the spectrum of $\cL_{\mathbf{C}}$ for these cones is well-understood, see for example~\cite{SimonSolomon}.
Using the $C^{\infty}_{\on{loc}}$ convergence $\hat{\Sigma}_{\ve} \to \Sigma_0$, which are small normal graphs over $\Sigma_0$, we may express the metric on $\Sigma_{\ve}$ as a $C^{\infty}$-small perturbation of the metric on $\Sigma_0$ by applying Lemma~\ref{lemma:graph-over-extended-link-pi/2}.
As a result, the spectrum and eigenfunctions of the Laplacian $\Delta_{\hat{\Sigma}_{\ve}}$ remain $O(\ve)$-close to those of $\Delta_{\Sigma_0}$; this follows, for example, from the results developed in~\cite{complete-cmc}*{Appendix B} (see also~\cite{chavel}).

In view of the above discussion, the corresponding indicial roots $\underline{\gamma}^\ve, \overline{\gamma}^\ve$ on $\mathbf{C}_{n,k,\frac{\pi}{2} - \ve}$ satisfy 
\begin{equation}\label{eqn:indicial-roots-close}
( - \tfrac{29}{10} , - \tfrac{21}{10}) \subset (-\overline{\gamma}^{\ve} , -\underline{\gamma}^\ve) \quad \text{for } \; n=7, \qquad ( - \tfrac{22}{5}, - \tfrac{8}{5}) \subset (- \overline{\gamma}^{\ve} , - \underline{\gamma}^{\ve}) \quad \text{for } \; n \geq 8
\end{equation}
for $\ve \in (0,\ve_0(n,k))$ sufficiently small.
The non-triviality of the indicial interval $(- \overline{\gamma}^{\ve}, - \underline{\gamma}^{\ve})$ is equivalent to the strict stability of the cones $\mathbf{C}_{n,k,\frac{\pi}{2}-\ve}$ for $n \geq 7$.

The property~\eqref{eqn:indicial-roots-close} will be the only information needed near infinity for constructing the surfaces $S_{\pm}$ outside a large ball.
Specifically, on $E_+$ we produce such mean-convex surfaces asymptotic to $\mathbf{C}_{n,k,\theta}$ outside a large ball, with the improvement that the contact angle is smaller than $\theta$, of the form $\theta - O (R^{-(a+1)})$, when $\theta = \frac{\pi}{2} - O(\ve)$ is sufficiently small.
An identical construction applies on the $E_-$ side, producing a surface that is mean-concave with respect to the normal vector pointing into $E_+$ (``upward'') and forms a contact angle larger than $\theta$ along the free boundary on $\Pi$.

\begin{lemma}\label{lemma:infinity-piece-small-epsilon}
Consider $n \geq 7, 1 \leq k \leq n-2$, and some $a \in (\underline{\gamma} , \overline{\gamma})$ in the indicial interval of~\eqref{eqn:HS-indicial-roots}.
For small $\ve > 0$, let $\hat{\Sigma}_{\ve}$ denote the extended link of the extended cone $\hat{\mathbf{C}}_{n,k,\frac{\pi}{2} - \ve}$, given by~\eqref{eqn:s-hat-eps-extended-link} and~\eqref{eqn:extended-cone}.
There exist $R_0 \gg 1$ and $0 < \tau_0 \ll 1$, depending only on $n,k,a$, and $0 < \ve_0 \ll 1$, depending only on $n,k,a$, and $\tau \in (0,\tau_0)$, such that the following holds for $\ve \in (0,\ve_0)$.

Let $S_{a, \psi}$ be the portion inside $\bR^{n+1}_+ = \{ z \geq 0 \}$ of the positive normal graph
    \[
    S_{a,\psi} := \Bigl\{ p + \ell(R) \psi(\omega) \nu_{\mathbf{C}}(p) : p = (R, \omega) \in (R_0, \infty) \times \hat{\Sigma}_{\ve} \subset \hat{\mathbf{C}}_{n,k,\frac{\pi}{2} - \ve} \setminus \bar{B}_{R_0} \Bigr\} \cap \bR^{n+1}_+
    \]
    of the function $u(R,\omega) = \ell(R) \psi(\omega)$ over $\hat{\mathbf{C}}_{n,k,\frac{\pi}{2} - \ve}$, where
    \[
    \psi(\omega) := 1 - \tau \la \omega, e_{n+1} \rg \quad \text{for } \; \omega \in \hat{\Sigma}_{\ve}, \qquad \ell(R) = R^{-a} + O(R^{-(a+1)}) 
    \]
Then, $S_{a,\psi}$ has the following properties:
    \begin{enumerate}[$(i)$]
        \item At a point $(R,\omega) \in S_{a,\psi} \cap \Pi$, the surface meets $\Pi = \{ z = 0 \}$ at an angle satisfying
        \[
        \theta - c_1 \tau R^{-(a+1)} < \theta_S(R,\omega) < \theta - \tfrac{99}{100} c_1 \tau R^{-(a+1)}
        \]
        for a constant $c_1$ depending on $n,k,a$.
        Here, $\theta = \frac{\pi}{2} - O(\ve)$ is the capillary angle of $\mathbf{C}_{n,k,\frac{\pi}{2} - \ve}$.
        \item With respect to the upward-pointing normal vector, the hypersurface $S_{a,\psi}$ has mean curvature 
    \[
    \tfrac{99}{100} c_2 R^{-a-2} < H[S_{a,\psi}](R,\omega) < c_2 R^{-a-2} 
    \]
    for a constant $c_2$ depending on $n,k,a$.
    \end{enumerate}
\end{lemma}
\begin{proof}
    The result will be obtained by carrying out the computations near the model case of $\mathbf{C}_{n,k,\frac{\pi}{2}}$ and tracking the error at each step in terms of $\ve \in (0,\ve_0)$ and $\theta = \frac{\pi}{2} - O(\ve)$.
    Throughout, we will apply the $\ve$-closeness properties for the profile functions $f_{\ve}$ obtained in~\eqref{eqn:g-eps-g0-closeness}~--~\eqref{eqn:t-hat-t-eps} and Lemma~\ref{lemma:graph-over-extended-link-pi/2}.
    For each $\ve \in [0,\ve_0)$, the function $f_{\ve}$ solves the ODE~\eqref{eqn:cone-ODE} on the extended interval $[0,\hat{t}_{\ve}]$, so the extended link $\hat{\Sigma}_{\ve} \subset \bS^n$ is a minimal hypersurface-with-boundary.
    Therefore, the ambient height coordinate satisfies $\Delta_{\hat{\Sigma}_{\ve}}(z) = - (n-1) z$, whereas the second fundamental form has $|A_{\hat{\Sigma}_{\ve}}|^2 = (n-1) + O(\ve)$ by~\eqref{eqn:sff-norm-epsilon}.
    Denoting $\omega_{n+1} = \la \omega, e_{n+1} \rg$, the Jacobi operator $\cL_{\hat{\Sigma}_{\ve}} := \Delta_{\hat{\Sigma}_{\ve}} + |A_{\hat{\Sigma}_{\ve}}|^2 = \Delta_{\hat{\Sigma}_{\ve}} + (n-1) + O(\ve)$ applied to $\psi = 1 - \tau \omega_{n+1}$ has
    \[
    \cL_{\hat{\Sigma}_{\ve}} \psi = \cL_{\hat{\Sigma}_{\ve}} ( 1 - \tau \omega_{n+1}) = (n-1 + O(\ve)) \psi + (n-1) \tau \omega_{n+1}.
    \]
    Using $|\omega_{n+1}| = |z (\omega)| \leq 1$ on $\hat{\Sigma}_{\ve}$, we conclude that 
    \begin{equation}\label{eqn:jacobi-double-bound}
        (n-1-C(n) \ve) \psi \leq \cL_{\hat{\Sigma}_{\ve}} \psi \leq (n-1+n \tau + C(n) \ve) \psi.
    \end{equation}
    We will abbreviate the operators on the extended cone $\hat{\mathbf{C}}_{n,k,\frac{\pi}{2} - \ve}$ to $\nabla_{\mathbf{C}}, \cL_{\mathbf{C}}$ in what follows.
    The Jacobi operator $\cL_{\mathbf{C}}$ of the cone satisfies
\[
\cL_{\mathbf{C}} ( R^{-a} \psi) = R^{-(a+2)} \bigl( a(a+1) - (n-1) a + \psi^{-1} \cL_{\Sigma} \psi \bigr) \psi.
\]
Using $\ell(R) = R^{-a} + O( R^{-(a+1)})$, we deduce that $u(R,\omega) = \ell(R) \psi(\omega)$ has
\[
\cL_{\mathbf{C}} u = R^{-(a+2)} \bigl( a(a+1) - (n-1) a + \psi^{-1} \cL_{\Sigma} \psi \bigr) \psi + O( R^{-(a+3)}).
\]
Consequently, the relation~\eqref{eqn:jacobi-double-bound} implies that for $R \geq R_0$ sufficiently large, we may bound
\begin{align*}
   (a-\overline{\gamma})(a-\underline{\gamma}) - C(n) \ve \leq R^{a+2} \psi^{-1}\cL_{\mathbf{C}} u \leq (a-\overline{\gamma})(a-\underline{\gamma}) + n \tau + C(n) \ve.
\end{align*}
For any $a \in (\underline{\gamma} , \overline{\gamma})$, we may produce some sufficiently small $\tau_0(n,k,a)$, followed by choosing $\ve_0(n,k,a,\tau)$ sufficiently small, so that
\[
C(n) \ve + n \tau < 10^{-3} \tilde{c}_2(n,k,a), \qquad \tilde{c}_2(n,k,a) := (\overline{\gamma}-a)(a - \underline{\gamma}) > 0
\]
to obtain $\tilde{c}_2 R^{-(a+2)} < -\cL_{\mathbf{C}} u < (1 + 10^{-3}) \tilde{c}_2 R^{-(a+2)}$.
Observe that
\[
|\nabla^2_{\mathbf{C}} u| = O( R^{-a-2}), \qquad |\nabla_{\mathbf{C}} u| = O ( R^{-a-1})
\]
whereby the estimate~\eqref{eqn:mean-curvature-operator} together with $\cL_{\mathbf{C}} u - \cL_{\mathbf{C}} ( R^{-a} \psi) = O(R^{-(a+3)})$ implies that
\[
\left| H[S_{a,\psi}] + \cL_{\mathbf{C}} u \right| \leq C(n) R^{-(a+3)}.
\]
Combining the above bounds proves $(ii)$ for $R \geq R_0$, upon choosing $R_0(n,k,a)$ sufficiently large.
Notably, the final estimate holds on $S_{a,\psi}$ for any $R \geq R_0$, independently of $\ve$, provided that $\ve_0(n,k,a,\tau)$ is taken sufficiently small and $\ve \in (0,\ve_0)$.

To obtain the property $(i)$, we first study the outward-pointing normal vector $\nu_S$ along the surface $S_{a,\psi}$.
For any surface obtained as a small normal graph via a function $u$, the normal vector $\nu_u$ of the resulting surface (at the translated point) satisfies
\begin{align*}
\nu_u( p + u(R,\omega) \nu_{\mathbf{C}}(p)) &= \nu_{\mathbf{C}}(p) - \nabla_{\mathbf{C}} u(R,\omega) + E_u(R,\omega), \\
\text{where } \; |E_u(R,\omega)| &\leq C(n) ( |\nabla_{\mathbf{C}} u|^2 + |u| \, |\nabla^2_{\mathbf{C}} u|).
\end{align*}
In our situation, $u = (R^{-a} + O(R^{-(a+1)})) \psi$ implies that $|E_{a,\psi}| \leq C(n) R^{-2(a+1)}$ for $R \geq R_0$, since $\tau \ll 1$.
(See also the discussion in Mazzeo-Pacard~\cite{mazzeo-pacard} for graphs in the direction of an almost-normal vector field).
This allows us to estimate the normal vector along $S_{a,\psi} \cap \Pi$ by
 \begin{equation}\label{eqn:normal-vector-equality-Pi}
    \begin{split}
    & \left| \nu_S \bigl( p + u(p) \nu_{\mathbf{C}}(p) \bigr) - \nu_{\mathbf{C}}(p) + \nabla_{\mathbf{C}} u(p) \right| \leq C(n) R^{-2(a+1)}, \\
    & \quad \text{where } \qquad p + u(p) \nu_{\mathbf{C}}(p) \in S_{a,\psi} \cap \Pi.
    \end{split}
    \end{equation}
In view of the discussion preceding~\eqref{eqn:s-hat-eps-extended-link} and~\eqref{eqn:extended-cone}, if $\ve>0$ then this property requires $p \in \hat{\mathbf{C}}_{n,k,\frac{\pi}{2} - \ve} \setminus \mathbf{C}_{n,k,\frac{\pi}{2} - \ve}$ to lie in the extended cone, strictly below the separating plane $\Pi$.
If $\ve = 0$, we have $p \in \mathbf{C}_{n,k,\frac{\pi}{2}} \cap \Pi$, since the normal graph over the $\frac{\pi}{2}$-cone again has free boundary on $\Pi$.
If $p=(R,\omega)$ has the property that $p+ \ell(R)\psi(\omega )\nu_{\mathbf{C}}(p) \in \{z=0\}$, we compute that
    \begin{align*}
    0 &= \la R \omega + \ell(R) \psi(\omega) \nu_{\mathbf{C}}(p), e_{n+1} \rg \\
    &= R \bigl( \omega_{n+1} + R^{-1} \ell(R) (1 - \tau \omega_{n+1}) \la \nu_{\mathbf{C}}(p), e_{n+1} \rg \bigr).
    \end{align*}
    This implies that
    \begin{equation}\label{eqn:omega-(n+1)}
    \begin{split}
    \omega_{n+1} &= -\frac{ R^{-1} \ell(R) \la \nu_{\mathbf{C}}(p),e_{n+1}\rg}{1- \tau R^{-1} \ell(R) \la \nu_{\mathbf{C}}(p),e_{n+1}\rg} \\
    &= - R^{-1} \ell(R) \la \nu_{\mathbf{C}}(p), e_{n+1} \rg + O ( \tau R^{-2(a+1)} \la \nu_{\mathbf{C}}(p), e_{n+1} \rg^2)
    \end{split}
    \end{equation}
    where $R \geq R_0$ is sufficiently large, independently of $\ve \in (0,\ve_0)$.
    In the last step, we used the fact that $\ell(R) = O(R^{-a})$, so $R^{-1} \ell(R) = O(R^{-(a+1)})$.
    In particular, for $R \geq R_0$, 
    \begin{equation}\label{eqn:nu-C-cos-theta}
        |\la \nu_{\mathbf{C}}(p),e_{n+1}\rg -\cos \theta| \leq C(n) R^{-(a+1)} \la \nu_{\mathbf{C}}(p),e_{n+1}\rg \leq C(n) R^{-(a+1)} \cos \theta
    \end{equation}
    where $\theta = \frac{\pi}{2} - O(\ve)$ is the capillary angle of $\mathbf{C}_{n,k,\frac{\pi}{2} - \ve}$.
    Next, we compute
    \begin{equation}\label{eqn:nabla-c-u-computation}
        \nabla_{\mathbf{C}} u = \nabla_{\mathbf{C}} ( \ell(R) \psi(\omega)) = \ell'(R) (1 - \tau \omega_{n+1}) \omega - \tau R^{-1}\ell(R) e_{n+1}^{\top}.
    \end{equation}
    where $e_{n+1}^{\top}$ denotes the orthogonal projection of the vector $e_{n+1}$ onto $T_{\omega}\hat{\Sigma}_{\ve}$.
    It follows that
    \begin{align*}
    \la \nabla_{\mathbf{C}} u, e_{n+1} \rg &= \ell'(R) (1 - \tau \omega_{n+1} ) \omega_{n+1} - \tau R^{-1} \ell(R) \bigl(1 - \omega_{n+1}^2 - \la \nu_{\mathbf{C}}(p), e_{n+1} \rg^2 \bigr),
    \end{align*}
    upon using $|e_{n+1}^{\top}|^2 = 1 - \omega_{n+1}^2 - \la \nu_{\mathbf{C}}(p), e_{n+1}\rg^2$.
    Applying~\eqref{eqn:omega-(n+1)}, we deduce that
    \begin{equation}\label{eqn:nabla-c-capillary}
    \left| \la \nabla_{\mathbf{C}} u , e_{n+1} \rg + (1 - \cos^2 \theta) \, \tau R^{-(a+1)} \right| \leq C(n)R^{-(a+2)}
    \end{equation}
    for $R \geq R_0$.
    In the last step, we used $\ell(R) = R^{-a} + O( R^{-(a+1)})$ and $\ell'(R) = - a R^{-(a+1)} + O(R^{-(a+2)})$, whereby
    \[
    R^{-1}\ell'(R) (1 - \tau \omega_{n+1}) \omega_{n+1} = O ( \ell'(R) R^{-2} \ell(R) \cos \theta) = O ( \ve R^{-1-2(a+1)}) = o (R^{-(a+2)})
    \]
    due to $a>1$.
    Combining~\eqref{eqn:nu-C-cos-theta} and~\eqref{eqn:nabla-c-capillary} into~\eqref{eqn:normal-vector-equality-Pi}, we arrive at
    \begin{equation}\label{eqn:nu-S-e-n+1-estimate}
       \left| \la \nu_S, e_{n+1} \rg - \cos \theta - (1- \cos^2 \theta) \tau R^{-(a+1)} \right| \leq C(n) R^{-(a+2)}
    \end{equation}
    for $R \geq R_0$.
    For $\theta = \frac{\pi}{2} - O(\ve)$ and any $\delta \ll 1$ sufficiently small (independently of $\ve$), we have
    \[
    \cos (\theta - \delta) = \cos \theta + \delta \sin \theta + O (\delta^2)
    \]
    In particular, for $\ve \in (0,\ve_0)$, where $\ve_0(n,k,a,\tau)$ is sufficiently small, we may write 
    \begin{equation}\label{eqn:cos-theta-delta-estimate}
        \cos \theta + (1 - 10^{-3}) \delta < \cos ( \theta - \delta) < \cos \theta + (1 + 10^{-3}) \delta.
    \end{equation}
    Consequently, the contact angle $\theta_S(R,\omega)$ at $(R,\omega) \in S_{a,\psi} \cap \Pi$ satisfies
    \[
    \cos \theta + (1 - 10^{-3}) (\theta - \theta_S) < \cos \theta_S = \la \nu_S, e_{n+1} \rg < \cos \theta + (1 + 10^{-3}) (\theta - \theta_S).
    \]
    Combining this property with~\eqref{eqn:nu-S-e-n+1-estimate}, and using $\cos^2 \theta < 10^{-3}$ for $\ve \in (0,\ve_0)$, we conclude that
    \[
    \theta - c_1 \tau R^{-(a+1)} < \theta_S(R,\omega) < \theta - \tfrac{99}{100} c_1 \tau R^{-(a+1)}
    \]
    as claimed.
    This establishes the property $(i)$, completing the proof.
\end{proof}

\begin{remark}\label{remark:piece-at-infinity-general}
    The construction of Lemma~\ref{lemma:infinity-piece-small-epsilon} can be generalized significantly, while the argument given above continues to apply.
    For a capillary cone ${\mathbf{C}}$ with contact angle $\theta$, the first eigenfunction of the Jacobi operator $\cL_{\Sigma}$ on the spherical link $\Sigma := \mathbf{C} \cap \bS^n_+$ is defined as the solution of the Robin eigenvalue problem 
    \begin{equation}\label{eqn:robin-eigenvalue-problem}
\begin{cases}
        \cL_{\Sigma} \varphi_1 + \lambda_1 \varphi_1 = 0 & \text{in } \; \Sigma, \\ 
        \partial_{\nu} \varphi_1 - \cot \theta \, A_{\Sigma}(\nu,\nu)\varphi_1 = 0 & \text{on } \; \partial \Sigma.
\end{cases}
    \end{equation}
    See, for example,~\cite{ros-souam}*{Lemma~3.1} and~\cite{singular-capillary}*{Lemma~2.1}.
    For free-boundary minimal cones, the contact angle $\theta = \frac{\pi}{2}$ means that the Robin boundary condition of~\eqref{eqn:robin-eigenvalue-problem} reduces to a Neumann eigenvalue problem.
    The construction of Lemma~\ref{lemma:infinity-piece-small-epsilon} implies that for contact angle $\theta \in ( \frac{\pi}{2} - \ve_{n,k} , \frac{\pi}{2}]$ and $n \geq 7$, we may construct a strict subsolution of the linearized problem with first eigenvalue $\lambda_1 > - ( \frac{n-2}{2})^2$, which implies the strict stability of the corresponding cone.
    We now provide one formulation of the general result:
    \begin{lemma}
        Consider a family of capillary minimal cones $\{ \mathbf{C}_i \}$ with an isolated singularity whose links satisfy $\Sigma_i \to \Sigma_{\infty}$ in $C^{\infty}_{\on{loc}}$ as hypersurfaces-with-boundary in $\bS^n_+$.
        Suppose that $\Sigma_{\infty}$ is the link of a capillary cone $\mathbf{C}_{\infty}$ with an isolated singularity and contact angle $\theta_{\infty} = \lim_i \theta_i$, and let $\varphi_1$ denote the first Jacobi eigenfunction on $\Sigma_{\infty}$ for the Robin boundary condition~\eqref{eqn:robin-eigenvalue-problem} along $\partial \Sigma_{\infty}$.
        If $\mathbf{C}_{\infty}$ is strictly stable, then each $\mathbf{C}_i$ is strictly stable for $i \geq N_0$.
        Moreover, there exist $R_0 \gg 1$ and $0 < \tau_0 \ll 1$ such that for $i \geq N_0(\tau)$, taking $\ell(R) = R^{-a} + O(R^{-(a+1)})$ for $a \in ( \underline{\gamma}, \overline{\gamma})$ and
        \[
        S^i_{a,\psi} := \on{graph}_{\mathbf{C}_i \setminus \bar{B}_{R_0}} ( \ell(R) ( \varphi_1(\omega) - \tau \omega_{n+1}) ) \cap \bR^{n+1}_+,
        \]
        produces a surface $S^i_{a,\psi} \subset E_+(\mathbf{C}_i)$ with angle $\theta^i_S$ along $\Pi$ and mean curvature $H[S^i_{a,\psi}]$ satisfying
        \[
        \tfrac{99}{100} c_1 < \tau^{-1} R^{a+1} \bigl( \theta_i - \theta^i_S(R,\omega) \bigr) < c_1, \qquad \tfrac{99}{100} c_2 < R^{a+2}H[S^i_{a,\psi}](R,\omega) < c_2.
        \]
    \end{lemma}
    \noindent
    We recover Lemma~\ref{lemma:infinity-piece-small-epsilon} for $\theta_i \to \frac{\pi}{2}$ and $\mathbf{C}_{\infty}= \mathbf{C}_{n,k,\frac{\pi}{2}}$.
    This result has more general consequences for the strict minimality of capillary cones, which we explore in upcoming work~\cite{FTW_MinimizingII}.
\end{remark}

Following the above general discussion, we obtain the following analogue of Lemma~\ref{lemma:infinity-piece-small-epsilon} for negative normal graphs over $\hat{\mathbf{C}}_{n,k,\frac{\pi}{2}-\ve}$ with respect to the vector $\nu_{\mathbf{C}}$ pointing into $E_+$.
The resulting surface $S^-_{a,\psi} \subset E_-$ will form the piece of $S_- \subset E_-$ at infinity.
\begin{lemma}\label{lemma:infinity-piece-on-E-minus}
Consider $n \geq 7, 1 \leq k \leq n-2$, and some $a \in (\underline{\gamma} , \overline{\gamma})$ in the indicial interval of~\eqref{eqn:HS-indicial-roots}.
There exist $R_0 \gg 1$ and $0 < \tau_0 \ll 1$, depending only on $n,k,a$, and $0 < \ve_0 \ll 1$, depending only on $n,k,a$, and $\tau \in (0,\tau_0)$, such that the following holds for $\ve \in (0,\ve_0)$.

Let $S^-_{a, \psi}$ be the portion inside $\bR^{n+1}_+ = \{ z \geq 0 \}$ of the negative normal graph
    \[
    S^-_{a,\psi} := \Bigl\{ p - \ell(R) \psi(\omega) \nu_{\mathbf{C}}(p) : p = (R, \omega) \in (R_0, \infty) \times \Sigma_{\ve} \subset \mathbf{C}_{n,k,\frac{\pi}{2} - \ve} \setminus \bar{B}_{R_0} \Bigr\} \cap \bR^{n+1}_+
    \]
    of the function $u(R,\omega) = \ell(R) \psi(\omega)$ over $\mathbf{C}_{n,k,\frac{\pi}{2} - \ve}$, where
    \[
    \psi(\omega) := 1 - \tau \la \omega, e_{n+1} \rg  \quad \text{for } \; \omega \in \Sigma_{\ve}, \qquad \ell(R) = R^{-a} + O(R^{-(a+1)}) 
    \]
Then, $S^-_{a,\psi}$ has the following properties:
    \begin{enumerate}[$(i)$]
        \item At a point $(R,\omega) \in S^-_{a,\psi} \cap \Pi$, the surface meets $\Pi = \{ z = 0 \}$ at an angle satisfying
        \[
        \theta + \tfrac{99}{100} c_1 \tau R^{-(a+1)} < \theta_S(R,\omega) < \theta + c_1 \tau R^{-(a+1)}
        \]
        for a constant $c_1$ depending on $n,k,a$.
        Here, $\theta = \frac{\pi}{2} - O(\ve)$ is the capillary angle of $\mathbf{C}_{n,k,\frac{\pi}{2} - \ve}$.
        \item With respect to the upward-pointing normal vector, the hypersurface $S^-_{a,\psi}$ has mean curvature
    \[
    - c_2 R^{-a-2} < H[S^-_{a,\psi}](R,\omega) < - \tfrac{99}{100} c_2 R^{-a-2} 
    \]
    for a constant $c_2$ depending on $n,k,a$.
    \end{enumerate}    
\end{lemma}
\begin{proof}
    The argument is a straightforward adaptation of Lemma~\ref{lemma:infinity-piece-small-epsilon}.
    We observe that $S^-_{a,\psi}$ is a normal graph over (a smaller portion of) $\mathbf{C}_{n,k,\frac{\pi}{2}-\ve}$, so we did not need to work with the extended cone $\hat{\mathbf{C}}$ for this step.
    This follows from the discussion preceding~\eqref{eqn:s-hat-eps-extended-link}, since for any $\ve>0$, the normal vector to $\mathbf{C}_{n,k,\frac{\pi}{2}-\ve}$ pointing into $E_+$ has everywhere-positive vertical component, hence any negative normal graph of the cone automatically produces a free boundary along $\Pi$.

    Computing as in Lemma~\ref{lemma:infinity-piece-small-epsilon} and applying~\eqref{eqn:mean-curvature-operator}, we similarly deduce 
    \[
    \left| H[S^-_{a,\psi}] - \cL_{\mathbf{C}} u \right| \leq C(n) R^{-(a+3)},
    \]
    where $u = \ell(R) \psi(\omega)$ was shown to satisfy $\tilde{c}_2 < - R^{a+2} \cL_{\mathbf{C}} u < (1 + 10^{-3}) \tilde{c}_2$.
    Moreover, 
\[
\left|\nu_u( p - u(R,\omega) \nu_{\mathbf{C}}(p)) - \nu_{\mathbf{C}}(p) -\nabla_{\mathbf{C}} u(R,\omega)\right|\leq  O(R^{-2(a+1)})
\]
holds by virtue of~\eqref{eqn:normal-vector-equality-Pi}.
At a point of $S^-_{a,\psi} \cap \Pi$, we have $\la p - \ell(R) \psi(\omega) \nu_{\mathbf{C}}(p) ,e_{n+1}\rg= 0$, which is solved for $\omega_{n+1} = R^{-1} \la p, e_{n+1} \rg$ as in~\eqref{eqn:nu-S-e-n+1-estimate} to produce
\begin{align*}
    \omega_{n+1} &= \frac{ R^{-1} \ell(R) \la \nu_{\mathbf{C}}(p),e_{n+1}\rg}{1+{\tau}R^{-1} \ell(R) \la \nu_{\mathbf{C}}(p),e_{n+1}\rg} \\
    &=  R^{-1} \ell(R) \la \nu_{\mathbf{C}}(p), e_{n+1} \rg + O ( \tau R^{-2(a+1)} \la \nu_{\mathbf{C}}(p), e_{n+1} \rg^2), \\
    |\la \nu_{\mathbf{C}}(p),e_{n+1}\rg -\cos \theta| &\leq  C(n) R^{-(a+1)} \cos \theta, \\
    \left| \la \nabla_{\mathbf{C}} u , e_{n+1} \rg + (1 - \cos^2 \theta) \, \tau R^{-(a+1)} \right| &\leq C(n)R^{-(a+2)}.
\end{align*}
Together, these properties imply that
\[
       \left| \la \nu_S, e_{n+1} \rg - \cos \theta + (1- \cos^2 \theta) \tau R^{-(a+1)} \right| \leq C(n) R^{-(a+2)}.
\]
The remainder of the proof follows similarly.
\end{proof}

\subsection{Compact caps}\label{section:compact-caps}

We now construct the compact ``cap'' profile at distance $1$ from the origin whose positive phase encloses an enlarged cone together with a neighborhood of the origin.
Crucially, this cap is mean-convex, it has non-empty intersection with a large sphere $\partial B_{R_2}$, and it meets the plane $\Pi = \{ z = 0 \}$ at an angle $< \theta = \tfrac{\pi}{2} - O(\ve)$, for $\ve$ sufficiently small (depending on $R_2$).
The construction of such a cap only relies on the proximity of the cones $\mathbf{C}_{n,k,\frac{\pi}{2}-\ve}$ to $\mathbf{C}_{n,k,\frac{\pi}{2}}$, in the sense of Lemma~\ref{lemma:graph-over-extended-link-pi/2}, and therefore admits a fairly general construction, which we discuss in upcoming work~\cite{FTW_MinimizingII}.
In the present situation, we can leverage the geometry of the Lawson cones $C(\bS^{n-k-1} \times \bS^k_+)$ and work with an explicit ansatz for the cap, which we present now.

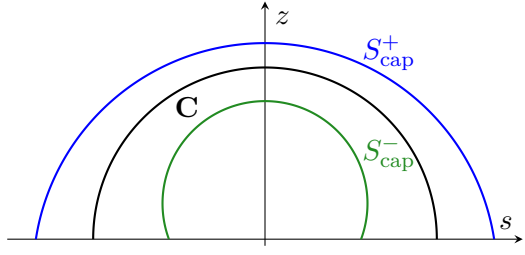
\begin{figure}
    \centering
    
    \begin{tikzpicture}
    \pgfmathsetmacro{\a}{1.3}
    \pgfmathsetmacro{\b}{0.27}
    \pgfmathsetmacro{\c}{2.5}

    \pgfmathsetmacro{\Ablue}{sqrt((\c+\b*\b)^2 + (\a*\a*\a*\a))}

    \pgfmathsetmacro{\xmaxblue}{sqrt(\Ablue - \b*\b)}

    \pgfmathsetmacro{\Agreen}{sqrt((\a*\a*\a*\a) + (\b*\b*\b*\b) - (\c))}
    \pgfmathsetmacro{\rgreen}{sqrt(\Agreen)}

    \pgfmathsetmacro{\alphagreen}{asin(\b/\rgreen)}
    \pgfmathsetmacro{\tstartgreen}{360 - \alphagreen}
    \pgfmathsetmacro{\tendgreen}{540 + \alphagreen}

    \pgfmathsetmacro{\xlab}{0.8*\rgreen}
    \pgfmathsetmacro{\yblueLab}{sqrt(\Ablue - \xlab*\xlab) - \b}
    \pgfmathsetmacro{\ygreenLab}{\b + sqrt(\Agreen - \xlab*\xlab)-0.05}

    \begin{axis}[
      axis equal image,
      xmin=-1.95, xmax=1.95,
      ymin=-0.05, ymax=1.8,      
      axis lines=middle,
      ticks=none,
      xlabel={$s$},
      ylabel={$z$},
      clip=false,                
      clip mode=individual,
    ]

      \addplot[black, thick, domain=0:180, samples=400]
        ({\a*cos(x)},{\a*sin(x)})
        node[pos=0.65, below] {$\mathbf{C}$};

      \addplot[blue, thick, domain=-\xmaxblue:\xmaxblue, samples=400]
        {sqrt(\Ablue - x^2) - \b};

      \addplot[ForestGreen, thick, domain=\tstartgreen:\tendgreen,
               samples=800, variable=\t]
        ({\rgreen*cos(\t)}, {\b + \rgreen*sin(\t)});

      \node[anchor=west, xshift=2pt, yshift=2pt, text=blue]
        at (axis cs:\xlab,\yblueLab) {$S_{\mathrm{cap}}^{+}$};

      \node[anchor=west, xshift=2pt, yshift=-2pt, text=ForestGreen]
        at (axis cs:\xlab,\ygreenLab) {$S_{\mathrm{cap}}^{-}$};

    \end{axis}
  \end{tikzpicture}
    \caption{We depict the cross-sections of the cone $\mathbf{C}$ and the upper and lower barrier surface caps $S^\pm_{\mr{cap}}$ when $|x|$ is equal to a (large) constant.
    The upper cap, drawn in blue, has smaller contact angle and the lower cap, drawn in green, is non-graphical and has greater contact angle.
    The $x$-axis represents the variable $s = |y|$, including a negative component for illustrative purposes.}
    \label{fig:Caps}
\end{figure}

For the area-minimizing Lawson cones $C(\bS^{n-k-1} \times \bS^k)$, the constructions of De~Philippis-Paolini, De~Philippis-Maggi, and Z.~Liu~\cites{dephilippis-paolini, sharp-stability-plateau, lawson-liu} establish a property stronger than strict minimality: there exists a homogeneous function $\varphi_{n,k}$ on $\bR^{n+1}$ such that the unit vector field $\frac{\nabla \varphi_{n,k}}{|\nabla \varphi_{n,k}|}$ provides a quantitative sub-calibration for the cones.
This means that
\begin{equation}\label{eqn:quantitative-sub-calibration}
\begin{aligned}
    \frac{\nabla \varphi_{n,k}}{|\nabla \varphi_{n,k}|} &= \nu_{C(\bS^{n-k-1} \times \bS^k)} & \qquad & \text{on } \; C(\bS^{n-k-1} \times \bS^k), \\
    \pm \on{div} \frac{\nabla \varphi_{n,k}}{|\nabla \varphi_{n,k}|} &\geq c_{n,k} \frac{\on{dist}(p, C(\bS^{n-k-1} \times \bS^k))}{R^2}, &\qquad &\text{where } R := |p|.
    \end{aligned}
\end{equation}
for a positive constant $c_{n,k} > 0$, where the choice of sign $\pm$ corresponds to the two connected components $E_{\pm}$ of $\bR^{n+1} \setminus C(\bS^{n-k-1} \times \bS^k)$.
This construction implies a global quadratic stability inequality for these cones that cannot be deduced from strict minimality alone, cf.~Lemma~\ref{lemma:one-sided}.

The explicit potentials satisfying~\eqref{eqn:quantitative-sub-calibration} will be useful in our setting, as their level sets produce surfaces with a signed mean curvature asymptotic to the free-boundary cone $\mathbf{C}_{n,k,\frac{\pi}{2}}$ from both sides $E_{\pm}$.
We first record these constructions on the sides 
\begin{equation}\label{eqn:E+-for-Lawson}
\begin{split}
E_+ &:= \left\{ (x,y,z) \in \bR^{n-k} \times \bR^k \times \bR : \frac{|x|^2}{n-k-1} < \frac{|y|^2 + z^2}{k} \right\} \\
&\;= \left\{ (x,y,z) \in \bR^{n-k} \times \bR^k \times \bR : z^2 > \rho^2 \hat{f}_{n,k}(t)^2 \right\}
\end{split}
\end{equation}
and $E_- := \bR^{n+1} \setminus \overline{E}_+$.
Here, we recalled our notation of $\hat{f}_{n,k}(t) = \sqrt{\frac{k - (n-1) t^2}{n-k-1}}$ from Lemma~\ref{lemma:lawson-cones}, allowing us to rearrange the above inequality.
Let us write $(r,s) := ( |x|, |y|)$ in what follows.
For points $p = (r,s,z)$ on the $E_+$ side, we define
    \begin{equation}\label{eqn:record-varphi-n,k}
        \begin{aligned}
            \mathrm{I} &:  (n,k) \in \{ (7,1), (7,2), (7,4) \}, \qquad & \varphi^+_{7,k}(r,s,z) &:= \bigl( (s^2 + z^2) - \tfrac{k}{n-k-1} r^2 \bigr) (s^2 + z^2)^{\frac{3}{4}}.\\
            \mathrm{II} &:  (n,k) = (7,3), \qquad & \varphi^+_{7,3}(r,s,z) &:= (s^2 + z^2)^{\frac{7}{4}} - r^{\frac{7}{2}}.\\
            \mathrm{III} &: 8 \leq n \leq 12, \; k=1, \qquad & \varphi^+_{n,1} (r,s,z) &:= \bigl( (s^2 + z^2) - \tfrac{1}{n-2} r^2 \bigr)  (s^2 + z^2).\\
            \mathrm{IV} &: \text{all other pairs with } n \geq 8, \qquad & \varphi_{n,k}^+(r,s,z) &:= (s^2 + z^2)^2 - ( \tfrac{k}{n-k-1})^2 r^4.
        \end{aligned}
    \end{equation}
The construction for $p = (r,s,z) \in E_-$ is analogous, letting
    \begin{equation}\label{eqn:record-varphi-n,k-on-E-minus}
        \begin{aligned}
            \mathrm{I} &:  (n,k) \in \{ (7,2), (7,4), (7,5) \}, \qquad & \varphi^-_{7,k}(r,s,z) &:= \bigl( (s^2 + z^2) - \tfrac{k}{n-k-1} r^2 \bigr) \, r^{\frac{3}{2}}.\\
            \mathrm{II} &:  (n,k) = (7,3), \qquad & \varphi^-_{7,3}(r,s,z) &:= (s^2 + z^2)^{\frac{7}{4}} - r^{\frac{7}{2}}.\\
            \mathrm{III} &: 8 \leq n \leq 12, \; k=1, \qquad & \varphi^-_{n,1} (r,s,z) &:= \bigl( (s^2 + z^2) - \tfrac{1}{n-2} r^2 \bigr) \, r^3.\\
            \mathrm{IV} &: \text{all other pairs with } n \geq 8, \qquad & \varphi_{n,k}^-(r,s,z) &:= (s^2 + z^2)^2 - ( \tfrac{k}{n-k-1})^2 r^4.
        \end{aligned}
    \end{equation}
It is clear from this definition that $\varphi^+_{n,k} = \varphi^-_{n,k}$ in cases II and IV, while in cases I and III, the definition of $\varphi_{n,k}$ component-wise is continuous along the interface $\mathbf{C} = \bar{E}_+ \cap \bar{E}_-$ and defines a potential function $\varphi_{n,k} \in W^{2,1}_{\text{loc}}$.

The functions~\eqref{eqn:record-varphi-n,k} and~\eqref{eqn:record-varphi-n,k-on-E-minus} are the adaptations of~\cite{sharp-stability-plateau}*{\S~4} and~\cite{lawson-liu}*{\S~2.1} to our setting.
We note the inclusion of the pair $(7,1)$ for the $E_+$ side and of the pair $(7,5)$ for the $E_-$ side, corresponding to a strict sub-calibration of the strictly stable cones $\mathbf{C}_{7,1,\frac{\pi}{2}}$ and $\mathbf{C}_{7,5,\frac{\pi}{2}}$ on their strictly minimizing side ($E_+$ and $E_-$, respectively).
The cone $C(\bS^5 \times \bS^1)$ is strictly minimizing on the side $E_+$, but not globally minimizing, by work of Sim\~oes and F.H.~Lin~\cites{simoes, lin}.
To see that the functions $\varphi_{7,1}^+$ and $\varphi_{7,5}^-$ provide the required one-sided strict sub-calibrations, we use~\cite{sharp-stability-plateau}*{Lemma~4.1} and~\cite{lawson-liu}*{Lemma~2 and \S~3} to see that the property
    \[
    \on{div} \frac{\nabla \varphi^+_{7,1}}{|\nabla \varphi^+_{7,1}|} \geq c_{7,1} \frac{\on{dist} (p, \mathbf{C}_{7,1,\frac{\pi}{2}})}{R^2} \qquad \text{for } \; p \in E_+(\mathbf{C}_{7,1,\frac{\pi}{2}})
    \]
for the function $\varphi^+_{7,1}$ of~\eqref{eqn:record-varphi-n,k} (respectively, $- \on{div} \frac{\nabla \varphi^-_{7,5}}{|\nabla \varphi^-_{7,5}|} \geq c_{7,5} \frac{\on{dist} ( p, \mathbf{C}_{7,5,\frac{\pi}{2}})}{R^2}$ for $p \in E_-( \mathbf{C}_{7,5,\frac{\pi}{2}} )$) is equivalent to the positivity of the cubic
    \[
    (t^{3}-3\sqrt{5}t^{2}-15t+25) \qquad \text{for } \; t \in [0,1]
    \]
upon substituting $A = 1$ and $B = \sqrt{5}$ into their notation.
    We see that this expression attains its minimum at $t=1$, equal to $11-3 \sqrt{5} > 4$, whereby the condition~\eqref{eqn:quantitative-sub-calibration} is satisfied.
Consequently, $\varphi^+_{7,1}$ (resp.~$\varphi^-_{7,5}$) produces a one-sided strict sub-calibration which exhibits the strict stability and one-sided strict minimality of the cone $\mathbf{C}_{7,1,\frac{\pi}{2}}$ (resp.~$\mathbf{C}_{7,5,\frac{\pi}{2}}$) on its $E_+$ side (resp.~$E_-$ side).

The potential functions $\varphi_{n,k}^{\pm}$ described in~\eqref{eqn:record-varphi-n,k} and~\eqref{eqn:record-varphi-n,k-on-E-minus} satisfy the property~\eqref{eqn:quantitative-sub-calibration}, meaning that their level sets are mean-convex (resp.~mean-concave) and have contact angle $\frac{\pi}{2}$ with $\Pi = \{z = 0\}$.
For the remainder of this section, we focus on the construction of a cap $S_{\text{cap}}^+ \subset E_+$ for the surface $S_+$.
Accordingly, we work with the functions $\varphi^+_{n,k}$ from~\eqref{eqn:record-varphi-n,k}, which we denote by $\varphi_{n,k}$ for the remainder of this section.
At the end, we will highlight the necessary adaptations for the construction of the piece $S^-_{\text{cap}} \subset E_-$ of the surface $S_-$.

An illustration of the cap profiles enclosing $\mathbf{C}_{n,k,\theta}$ for $n \geq 8$ (corresponding to Case IV of~\eqref{eqn:record-varphi-n,k} and~\eqref{eqn:record-varphi-n,k-on-E-minus}) is given in Figure~\ref{fig:Caps}.

We observe that the level sets of the degree-$\beta$ homogeneous functions $\varphi_{n,k}$ over $E_+$ at distance $1$ from the origin are expressible as normal graphs over the corresponding cone $\mathbf{C}_{n,k,\frac{\pi}{2}}$ asymptotic to the function $c_{n,k} R^{1-\beta}$.
For $\ve \in (0,\ve_0)$ sufficiently small, we construct the desired cap surface $S_{\text{cap}}$ as the level set of a function $\tilde{\varphi}_{n,k,\delta}(r,s,z)$, obtained by modifying $\varphi_{n,k}$ over the ball $B_{R_2}$ into
\begin{equation}\label{eqn:varphi-n,k-and-S-cap}
\begin{split}
        \tilde{\varphi}_{n,k,\delta}(r,s,z) &:= \varphi_{n,k}(r,s,z + \delta R_2^{-(\beta+1)} ), \\ 
        S_{\on{cap}} &:= \{ \tilde{\varphi}_{n,k,\delta} (r,s,z) = \varphi_{n,k}(0,1, \delta R_2^{-(\beta+1)} ) \},
    \end{split}
    \end{equation}
for some small $\delta>0$ to be determined.
    This construction amounts to vertically ``lowering'' a level set of $\varphi_{n,k}$ by $\delta R_2^{-(\beta+1)} >0$, thereby producing a contact angle $<\frac{\pi}{2}$ along $\Pi = \{ z=0\}$.
    Moreover, this definition ensures that $(0,1,0) \in S_{\on{cap}} \cap \Pi$, meaning that the free boundary of $S_{\on{cap}}$ is at distance $1$ from the origin.
    We now obtain a description of $S_{\text{cap}}$ as a normal graph over the cone $\mathbf{C}_{n,k,\frac{\pi}{2}}$ in an annular region $A(R_1, R_2) = B_{R_2} \setminus \bar{B}_{R_1}$.

\begin{lemma}\label{lemma:normal-graph-varphi-n,k}
    Consider the function $\varphi_{n,k}$, defined for $n \geq 7$ and $1 \leq k \leq n-2$ with $(n,k) \neq (7,5)$.
    Let $\beta$ be the degree of homogeneity of $\varphi_{n,k}$, so $\beta(n,k) = \frac{7}{2}$ for $n=7$ and $\beta(n,k) = 4$ for $n \geq 8$.
    There exists a small $\delta_{n,k}>0$ such that for all sufficiently large $R_1 \geq R_0(n,k)$, all $R_2 \geq \delta_{n,k}^{-1} R_1$, and all $\delta \in (0,\delta_{n,k})$, the level set
    \[
    S_{\normalfont{\text{cap}}} := \left\{ \varphi_{n,k}(r,s,z + \delta R_2^{-(\beta+1)} ) = \varphi_{n,k}(0,1, \delta R_2^{-(\beta+1)}) : (r,s) = (|x|,|y|) \right\} \cap \bR^{n+1}_+
    \]
    is expressible as the normal graph of a function $\hat{u}_1$ over $\mathbf{C}_{n,k,\frac{\pi}{2}} \cap A(R_1, R_2)$, satisfying 
    \begin{equation}\label{eqn:scap-over-the-cone}
    \begin{split}
        & R^{-1} | \hat{u}_1 - \bar{u}_1 | + |\nabla_{\mathbf{C}} (\hat{u}_1 - \bar{u}_1)| \leq C(n,k) R^{2(1-\beta)} \\
        & \quad \text{where } \quad  \bar{u}_1(p) := c^+_{n,k} R^{1-\beta} - a_{n,k}^{-1} \delta R_2^{-(\beta+1)} \omega_{n+1}(p), \qquad a_{n,k}^{-1} \omega_{n+1}(p) = \la \nu_{\mathbf{C}}(p), e_{n+1} \rg.
    \end{split}
    \end{equation}
    for $a_{n,k} = \sqrt{\frac{k}{n-k-1}}$ and $c^+_{n,k}$ the constant from~\eqref{eqn:constant-cases-I-through-IV}.
    Moreover, $|\nabla^{\ell}_{\mathbf{C}} \hat{u}_1| \leq C(n,k,\ell) R^{1-\beta-\ell}$.
\end{lemma}
\begin{proof}
For $p \in \mathbf{C}$, we let $\omega := \frac{p}{|p|}$ be the link variable, so $p = R \omega$ and $\nu_{\mathbf{C}}(p) = \nu(\omega)$.
We recall the notation $a_{n,k} := \sqrt{\frac{k}{n-k-1}}$, so that $\mathbf{C}_{n,k,\frac{\pi}{2}}$ is given by $s^2+z^2 =a_{n,k}^2 r^2$, with normal vector $\nu_{\mathbf{C}} (x,y,z) = \frac{(- a^2_{n,k} x, y,z)}{a_{n,k} R}$.
In each of the cases I~--~IV from~\eqref{eqn:record-varphi-n,k}, we use $|\nabla \varphi|^2 = \varphi_r^2 + \varphi_s^2 + \varphi_z^2$ to compute $|\nabla \varphi|$ along $\mathbf{C}_{n,k,\frac{\pi}{2}}$ as $c_{n,k}^+|\nabla \varphi_{n,k}| =  R^{\beta-1}$, where
\begin{equation}\label{eqn:constant-cases-I-through-IV}
\mathrm{I: } \; c_{7,k}^+ = \frac{(1 + a^2_{7,k})^{\frac{3}{4}}}{2 a_{7,k}^{\frac{5}{2}}}, \quad \mathrm{II: } \; c_{7,3}^+ = \frac{2^{\frac{7}{4}}}{7}, \quad \mathrm{III: } \; c_{n,1}^+=\frac{(n-1)\sqrt{n-2}}{2}, \quad \mathrm{IV: } \; c_{n,k}^+ = \frac{1 + a_{n,k}^2}{4 a_{n,k}^3}.
\end{equation}
Let us denote $\hat
\delta := \delta R_2^{-(\beta+1)}$, for brevity, so $\tilde{\varphi}_{n,k,\delta} = \varphi_{n,k}(r,s,z+\hat{\delta})$.
By the homogeneity of $\varphi$, we know that $\varphi \sim R^{\beta}$ and $|\nabla^{\ell} \varphi| \sim R^{\beta-\ell}$.
Using the expressions~\eqref{eqn:record-varphi-n,k}, we define 
\begin{equation}\label{eqn:c-hat-delta-estimate}
c_{\hat{\delta}} := \varphi_{n,k}(0,1,\hat{\delta}) = (1 + \delta^2 R_2^{-2(\beta+1)})^{\frac{\beta}{2}} = 1 + \tfrac{\beta}{2} \delta^2 R_2^{-2(\beta+1)} + O (\delta^4 R_2^{-4(\beta+1)}).
\end{equation}
We also denote, for brevity, $q(p,\tau) := \tau \nu_{\mathbf{C}}(p) + \hat{\delta} e_{n+1}$ and set $\Psi(p,\tau) := \varphi(p + q(p,\tau)) - c_{\hat{\delta}}$, so that $S_{\text{cap}} = \{ \Psi = 0 \}$.
Using $|q(p,\tau)| = O (\tau + \hat{\delta})$, we may bound
\[
\Psi_{\tau}(p, \tau) = \la \nabla \varphi( p + q(p,\tau)) , \nu_{\mathbf{C}} \rg \geq \la \nabla \varphi, \nu_{\mathbf{C}} \rg - C(n) |D^2 \varphi ||q(p,\tau)| \geq c R^{\beta-1}.
\]
The construction of $\bar{u}_1$ in~\eqref{eqn:scap-over-the-cone} is made to ensure that $\la \nabla \varphi, q(p, \bar{u}_1) \rg = 1$, since
\begin{align*}
& \la \nu_{\mathbf{C}}(p), e_{n+1} \rg = a_{n,k}^{-1} \omega_{n+1}(p) \implies \bar{u}_1(p) = c_{n,k}^+R^{1-\beta} - \hat{\delta} \la \nu_{\mathbf{C}} , e_{n+1} \rg \\
& \implies \la \nabla \varphi(p), q(p, \bar{u}_1) \rg= \la \nabla \varphi(p) , c^+_{n,k} R^{1-\beta} \nu_{\mathbf{C}}(p) \rg = c^+_{n,k} R^{1-\beta} \Big\la \nabla \varphi , \frac{\nabla \varphi}{|\nabla \varphi|} \Big\rg = c^+_{n,k} R^{1-\beta} |\nabla \varphi| = 1,
\end{align*}
by using the above computation $c^+_{n,k} |\nabla \varphi| =  R^{\beta-1}$.
We also observe that
\[
|\nabla^{\ell}_{\mathbf{C}} \bar{u}_1| \leq C_{\ell} ( R^{1-\beta-\ell} + \hat{\delta} R^{-\ell}).
\]
by construction.
Therefore, we can Taylor-expand $\varphi$ near a point $p \in \mathbf{C}$ (where $\varphi(p) = 0$) to write
\[
\varphi(p + q(p, \bar{u}_1) ) - 1 = \tfrac{1}{2} D^2 \varphi [ q(p, \bar{u}_1), q(p,\bar{u}_1) ] + O ( \hat{\delta}^2 R^{\beta-2}).
\]
Since $\varphi$ is homogeneous of degree $\beta$, we have $|D^{\ell} \varphi| \sim R^{\beta-\ell}$; on the other hand, $|q(p, \bar{u}_1)| \lesssim R^{1-\beta} + \hat{\delta}$ was observed above.
In what follows, we make repeated use of the fact that $\hat{\delta} = \delta R_2^{-(\beta+1)} \leq \delta R^{-(\beta+1)}$ in $A(R_1, R_2)$, so $\hat{\delta} = o (R^{1-\beta})$.
Consequently,
\[
D^2 \varphi [ q(p, \bar{u}_1) , q(p, \bar{u}_1)] \leq C(n) R^{\beta-2} ( R^{2(1-\beta)} + \hat{\delta}^2 + \hat{\delta} R^{1-\beta}) \leq C(n) R^{-\beta}.
\]
Using $c_{\delta} = 1 + \tfrac{\beta}{2} \hat{\delta}^2 + O ( \hat{\delta}^4)$, due to~\eqref{eqn:c-hat-delta-estimate}, we may expand $\Psi(p,\bar{u}_1 + \cR)$ as
\[
\Psi( p, \bar{u}_1 + \cR) = ( \Psi(p, \bar{u}_1 + \cR) - \Psi(p, \bar{u}_1) ) + O ( R^{-\beta})
\]
Since $\Psi(p, \tau)$ is increasing in $\tau$, for each fixed $p \in \mathbf{C} \cap A(R_1, R_2)$, the mean value theorem gives, for some $\sigma = \sigma(p) \in(0,1)$,
\[
\Psi( p, \bar{u}_1 + \cR) = \Psi(p, \bar{u}_1) + \cR \Psi_{\tau}(p, \bar{u}_1) + \tfrac{1}{2} \cR^2 \Psi_{\tau \tau}(p, \bar{u}_1 + \sigma \cR),
\]
where $\Psi_{\tau} \geq c R^{\beta-1}$ was shown earlier, and we likewise obtain
\[
|\nabla_{\mathbf{C}} \Psi_{\tau}| \leq C(n) R^{\beta-2}, \qquad |\Psi_{\tau \tau}| \leq C(n) R^{\beta-2}.
\]
Applying the implicit function theorem, we deduce the existence of a smooth solution $\cR$ to $\Psi ( p, \bar{u}_1 + \cR) = 0$, which satisfies $R^{-1} |\cR| \leq C(n) R^{-2\beta}$.

Applying total covariant differentiation to the equality $\Psi( p, \bar{u}_1 + \cR) = 0$, we obtain 
\[
\nabla_{\mathbf{C}} \Psi (p, \bar{u}_1 + \cR) + \Psi_{\tau} \nabla_{\mathbf{C}} ( \bar{u}_1 + \cR) = 0 \implies  \nabla_{\mathbf{C}} \cR = -\frac{\nabla_{\mathbf{C}} \Psi ( p , \bar{u}_1 + \cR)  }{\Psi_{\tau} ( p, \bar{u}_1 + \cR)} - \nabla_{\mathbf{C}} \bar{u}_1
\]
where $\nabla_{\mathbf{C}} \Psi(p, \tau)$ denotes the covariant gradient of the function $\Psi$ only with respect to the $p$-variables, evaluated at a point $(p,\tau) \in \bR^{n+1} \times \bR$.
Expanding as before and using the fact that $\frac{\nabla_{\mathbf{C}} \Psi( p, \bar{u}_1)}{\Psi_{\tau}(p, \bar{u}_1) } = - \nabla_{\mathbf{C}} \bar{u}_1 + O( R^{1-\beta}  \cR)$, we may write
\[
\frac{\nabla_{\mathbf{C}} \Psi ( p , \bar{u}_1 + \cR)  }{\Psi_{\tau} ( p, \bar{u}_1 + \cR)} = - \nabla_{\mathbf{C}} \bar{u}_1 + \frac{1}{\Psi_{\tau}(p, \bar{u}_1)} \left( \nabla_{\mathbf{C}} \Psi_{\tau}(p, \bar{u}_1) - \Psi_{\tau \tau}(p, \bar{u}_1) \, \nabla_{\mathbf{C}} \bar{u}_1 \right) \cR + O (R^{1-\beta}\cR+\cR^2)
\]
where $|\nabla_{\mathbf{C}} \Psi_{\tau} - \Psi_{\tau \tau} \nabla_{\mathbf{C}} \bar{u}_1| \leq |\nabla_{\mathbf{C}} \Psi_{\tau}| + |\Psi_{\tau \tau}| \, |\nabla_{\mathbf{C}} \bar{u}_1|$.
Combining these bounds with $|\nabla_{\mathbf{C}} \bar{u}_1| \leq C(n,k) R^{-\beta}$, we arrive at
\begin{align*}
|\nabla_{\mathbf{C}} \cR| &\leq C(n,k) R^{1-\beta} ( |\nabla_{\mathbf{C}} \Psi_{\tau}| + |\Psi_{\tau \tau}| \, |\nabla_{\mathbf{C}} \bar{u}_1|) \, |\cR| + O(\cR^2) \leq C(n,k) R^{2(1-\beta)}
\end{align*}
which establishes the property~\eqref{eqn:scap-over-the-cone}.
The derivation of higher-order estimates for $\Psi$ is analogous, upon differentiating $\Psi(p, \bar{u}_1 + \cR) = 0$ by $\nabla^{\ell}_{\mathbf{C}}$ and inductively iterating the bounds obtained at each step.
We again apply the homogeneity estimates for covariant derivatives of $\varphi$ and $\nu_{\mathbf{C}}$, together with $|D^{\ell} \varphi| \sim R^{\beta-\ell}$, $\partial^{\ell}_{\tau} \Psi \sim R^{\beta-\ell}$, and $|\nabla_{\mathbf{C}}^{\ell}q(p, \bar{u}_1)| \lesssim (R^{1-\beta-\ell} + \hat{\delta} R^{-\ell})$.
In particular, we obtain 
\[
|\nabla^{\ell}_{\mathbf{C}} \hat{u}_1| \leq C(n,k,\ell) |\nabla^{\ell}_{\mathbf{C}} \bar{u}_1| \leq C(n,k,\ell) ( R^{1-\beta-\ell} + \delta R_2^{-(\beta+1)} R^{1-\ell}) \leq C(n,k,\ell) R^{1-\beta-\ell}
\]
as desired.
This completes the proof.
\end{proof}

We will also use the following standard result, allowing us to relate the expressions of a given surface as a small normal graph over two nearby surfaces; see, for example,~\cite{uniqueness-gabor}*{Lemma~5.9}.

\begin{lemma}[\cite{uniqueness-gabor}*{Lemma~5.9}]\label{lemma:iterated-normal-graph}
    Let $S \subset \bR^{n+1}$ be a hypersurface with boundary and second fundamental form $A_S$ satisfying $\sup_S \|\nabla^{\ell}_S A_S \| < c_{\ell}$.
    We consider $C^m$-small functions $u,v$ and define $S_u$ to be the normal graph of $u$ over $S$, then take $S_{u+v}$ to be the normal graph of $v$ over $S_u$.
    Then, $S_{u+v}$ is expressible as a normal graph over $S$, given by a function $w$ satisfying
    \[
    w = u + v + O \bigl( |u| |\nabla_S v|^2 + |v| |\nabla_S u|^2 + |v| |\nabla_S u| |\nabla_S v| \bigr).
    \]
    In particular, if $\| u \|_{C^m(S)}, \| v\|_{C^m(S)}$ are sufficiently small (depending on the geometry of $S$), then
    \begin{equation}\label{eqn:w-u-v-bounds}
    \| w - (u+v) \|_{C^{m-1}(S)} < C(n,S) \| u \|_{C^m(S)} \| v \|_{C^m(S)}.
    \end{equation}
\end{lemma}

For $\ve$ sufficiently small, we construct the caps at contact angle $\theta = \frac{\pi}{2} - O(\ve)$ as level sets of appropriate perturbations $\tilde{\varphi}_{n,k,\ve}$ of the function $\varphi_{n,k}$.
\begin{proposition}\label{prop:mean-convex-cap}
Consider any $n \geq 7$ and $1 \leq k \leq n-2$ except $(n,k) = (7,5)$.
Let $\beta(n) = \frac{7}{2}$ for $n=7$ and $\beta(n) = 4$ for $n \geq 8$.
There exists a $\Lambda = \Lambda(n,k)$ such that for any $R_1 > 0$ sufficiently large, any $R_2 > \Lambda R_1$, and $\ve \in (0, \Lambda^{-1} R_2^{-4\beta})$, we can produce a function $U_{\on{cap}}(r,s)$ over the ball $B^n_{R_2} \subset \Pi$ with graph $S_{\text{cap}} := \{ z = U_{\text{cap}}(r,s) \}$ satisfying the following properties: 
\begin{enumerate}[$(i)$]
\item The mean curvature operator satisfies $\cM ( U_{\on{cap}}(r,s) ) < - d_{n,k} R^{-(1+\beta)}$, meaning that $S_{\text{cap}}$ is mean-convex with respect to the upward-pointing normal vector $\nu_S$.
\item The inequality $\la \nu_S , e_{n+1} \rg \geq \cos \theta$ is satisfied along the free boundary of $U_{\on{cap}}$, where $\theta = \frac{\pi}{2} - O(\ve)$ denotes the contact angle of the capillary cone $\mathbf{C}_{n,k,\frac{\pi}{2}-\ve}$.
\item The free boundary of $U_{\on{cap}}(r,s)$ is contained in the cone region $E_+(\mathbf{C}_{n,k,\frac{\pi}{2} - \ve}) \cap \Pi$, given by $\{ |y| > \tfrac{t_{\ve}}{\sqrt{1-t_{\ve}^2}} |x| \}$, and has distance $1$ from the origin.
\item The cap function satisfies $U_{\on{cap}} > \rho f_{\ve}(t)$ over $B_{R_2}$.
\item The surface $S_{\text{cap}} \cap A(R_1, R_2)$ over the annulus $A(R_1, R_2) := B_{R_2} \setminus \bar{B}_{R_1}$ is expressible as a normal graph over the extended cone $\hat{\mathbf{C}}_{n,k,\frac{\pi}{2}-\ve}$, with 
\begin{equation}\label{eqn:phi-ck-bounds}
    \begin{split}
        & R^{-1} | u_1 - \bar{u}_1 | + |\nabla_{\mathbf{C}} (u_1 - \bar{u}_1)| \leq C(n,k) R^{2(1-\beta)}, \\
        & \quad \text{where } \quad  \bar{u}_1(p) := c_{n,k} R^{1-\beta} - a_{n,k}^{-1} \delta R_2^{-(\beta+1)} \omega_{n+1}(p).
    \end{split}
    \end{equation}
    for $a_{n,k} = \sqrt{\frac{k}{n-k-1}}$ and $c_{n,k}$ the constant from~\eqref{eqn:constant-cases-I-through-IV}.
    Moreover, $|\nabla^2_{\mathbf{C}} u_1| \leq C(n,k) R^{-(\beta+1)}$.
\end{enumerate}
\end{proposition}

\begin{proof}
We consider the function $\tilde{\varphi}_{n,k,\delta}(r,s,z) := \varphi_{n,k}(r,s,z+\delta R_2^{-(\beta+1)})$ over the ball $B_{R_2}$, defined in~\eqref{eqn:varphi-n,k-and-S-cap} by modifying the explicit profile $\varphi_{n,k}$.
    We see directly from~\eqref{eqn:record-varphi-n,k}~--~\eqref{eqn:record-varphi-n,k-on-E-minus} that $\partial_z \tilde{\varphi}_{n,k,\delta} \neq 0$, for $\ve \in (0,\ve_0)$ sufficiently small.
    Consequently, the level set $S_{\text{cap}} = \{ \tilde{\varphi} = \tilde{\varphi}(0,1,0) \}$ is expressible as the graph of a function $z = U_{\on{cap}}(r,s)$.
    Recall the explicit forms of $\varphi_{n,k}$ from~\eqref{eqn:record-varphi-n,k}, corresponding to Cases I~--~IV; in case IV, the corresponding graph function $z = U_{\text{cap}}(r,s)$ is given by
    \[
    z = U_{\on{cap}(n,k,\delta)}(r,s) := \sqrt{\sqrt{a^4_{n,k} r^4 + \bigl(1+\delta^2 R_2^{-2(\beta+1)} \bigr)^2} - s^2} - \delta R_2^{-(\beta+1)}, \qquad a_{n,k} := \sqrt{\tfrac{k}{n-k-1}}
    \]
    coming from~\eqref{eqn:varphi-n,k-and-S-cap}.
    The computations in Cases I~--~III are analogous.
    We will now verify that the above surfaces $S_{\on{cap}}$ and corresponding graph functions $U_{\on{cap}}(r,s)$ have the required properties $(i)$~--~$(v)$.
    The discussion following~\eqref{eqn:varphi-n,k-and-S-cap} shows that $(0,1,0) \in S_{\text{cap}} \cap \Pi$, so $S_{\text{cap}}$ has free boundary at distance $1$ from the origin in all cases.
    The constructions in Cases I and II are small modifications of the more general constructions in Cases III and IV, respectively.
    Therefore, their validity will follow from analogous arguments.

    \smallskip

    \noindent \textbf{Property $(i)$.}
    The functions $\tilde{\varphi}_{n,k,\delta}$ are obtained from $\varphi_{n,k}$ under a translation, which leaves the mean curvature invariant.
    Consequently, the computations of~\eqref{eqn:quantitative-sub-calibration} imply that
    \begin{equation}\label{eqn:div-varphi-tilde}
    \on{div} \frac{\nabla \tilde{\varphi}_{n,k,\delta}}{|\nabla \tilde{\varphi}_{n,k,\delta}|} \geq  d_{n,k} \frac{s^2 + (z+\delta R_2^{-(\beta+1)})^2 - a^2_{n,k} r^2}{R^3} \geq \tilde{d}_{n,k} \frac{(1+\delta^2 R_2^{-2(\beta+1)})^2}{R^{1+\beta}}.
    \end{equation}
    In the last step, we used the fact that $\tilde{\varphi}_{n,k,\delta} = (1+\delta^2 R_2^{-2(\beta+1)})^2$ along $S_{\on{cap}}$, whereby 
    \[
    (s^2 + (z+\delta R_2^{-(\beta+1)})^2 - a^2_{n,k}r^2) \cdot \frac{\tilde{\varphi}_{n,k,\delta}}{s^2 + (z+\delta R_2^{-(\beta+1)})^2 - a^2_{n,k}r^2} = (1+\delta^2 R_2^{-2(\beta+1)})^2.
    \]
    For each function $\varphi^{\pm}_{n,k}$ from~\eqref{eqn:record-varphi-n,k}~--~\eqref{eqn:record-varphi-n,k-on-E-minus}, we observe that $\varphi_{n,k}$ is a homogeneous function of degree $\beta$ vanishing along $\{ s^2 + z^2 = a_{n,k}^2 r^2 \}$, so that $\frac{\varphi_{n,k}}{s^2 + z^2 - a_{n,k}^2 r^2}$ is a well-defined positive homogeneous function of degree $\beta-2$.
    Consequently, $\tilde{\varphi}_{n,k,\delta}$ vanishes along $\{ s^2 + (z+\delta R_2^{-(\beta+1)} )^2 = a^2_{n,k}r^2 \}$, and $\frac{\tilde{\varphi}_{n,k,\delta}}{s^2 + (z+\delta R_2^{-(\beta+1)})^2 - a^2_{n,k}r^2}$ is a well-defined positive function with $\frac{\tilde{\varphi}_{n,k,\delta}}{s^2 + (z+\delta R_2^{-(\beta+1)})^2 - a^2_{n,k}r^2} \leq d_{n,k} R^{\beta-2}$.
    Rearranging terms in~\eqref{eqn:div-varphi-tilde}, we find
    \begin{align*}
    \on{div} \frac{\nabla \tilde{\varphi}_{n,k,\delta}}{|\nabla \tilde{\varphi}_{n,k,\delta}|} &\geq  d_{n,k} \frac{s^2 + (z+\delta R_2^{-(\beta+1)})^2 - a^2_{n,k} r^2}{R^3} \\
    &\geq d_{n,k} \bigl( 1+ \delta^2 R_2^{-2(\beta+1)} \bigr)^2 R^{-3} \frac{s^2 + (z+\delta R_2^{-(\beta+1)})^2 - a^2_{n,k} r^2}{\tilde{\varphi}_{n,k,\delta}}
    \end{align*}
    where the last term is $\geq d_{n,k} R^{-(1+\beta)}$, as claimed.
    This proves~\eqref{eqn:div-varphi-tilde}, and we obtain
    \[
    H[S_{\text{cap}}] = - \cM( U_{\text{cap}}) = \on{div} \frac{\nabla \tilde{\varphi}_{n,k,\delta}}{|\nabla \tilde{\varphi}_{n,k,\delta}|} \geq d_{n,k} R^{-(1+\beta)}
    \]
    as desired.
    This proves the required mean-convexity estimate $(i)$ in each case.

    \smallskip

\noindent \textbf{Property $(ii)$.}
    Along the free boundary of $U_{\on{cap}}$, we know that $z = 0$ and 
    \begin{equation}\label{eqn:level-set-equation}
        (s^2+\delta^2 R_2^{-2(\beta+1)})^2 - a_{n,k}^4 r^4 = (1+\delta^2 R_2^{-2(\beta+1)} )^2.
    \end{equation}
    Recall that $|\nabla \tilde{\varphi}|^2 = \tilde{\varphi}_r^2 + \tilde{\varphi}_s^2 + \tilde{\varphi}_z^2$, hence the contact angle along $S_{\text{cap}} \cap \Pi$ is computed by
    \begin{align*}
    \la\nu_{S_{\on{cap}}} , e_{n+1} \rg &= \frac{\partial_{n+1} \tilde{\varphi}}{|\nabla \tilde{\varphi}|}\bigg\rvert_{z =0} = \frac{\delta R_2^{-(\beta+1)} ( s^2 + \delta^2 R_2^{-2(\beta+1)} ) }{\sqrt{a_{n,k}^8 r^6 + \bigl(s^2+\delta^2 R_2^{-2(\beta+1)} \bigr)^3}} \geq \frac{\delta R_2^{-(\beta+1)}}{\sqrt{(1 + a^2_{n,k})(s^2 + \delta^2 R_2^{-2(\beta+1)})}}.
    \end{align*}
    Using~\eqref{eqn:level-set-equation} together with the bound $s \leq R_2$ in the ball $B_{R_2}$, we may estimate
    \[
    r^2 \leq a_{n,k}^{-2}(s^2 + \delta^2 R_2^{-2(\beta+1)} ), \qquad \sqrt{s^2 + \delta^2 R_2^{-2(\beta+1)}} \leq 2 R_2,
    \]
    which yields the last step above.
    The above bound becomes
    \begin{equation}\label{eqn:dot-product-bound}
        \la\nu_{S_{\on{cap}}} , e_{n+1} \rg \geq \frac{\delta R_2^{-(\beta+1)}}{\sqrt{(1 + a_{n,k}^2)(s^2 + \delta^2 R_2^{-2(\beta+1)})}} \geq \frac{\delta R_2^{-(\beta+1)}}{2 R_2 \sqrt{ \frac{n-1}{n-k-1}}} = C(n,k) \delta R_2^{-(\beta+2)}.
    \end{equation}
    On the other hand, the contact angle satisfies $\cos \theta < \frac{1}{\tan \theta} < 2 a^{-1}_{n,k} \ve$, in view of~\eqref{eqn:small-epsilon-relation}.
    Combining the two inequalities, we may enforce $\la \nu_{S_{\on{cap}}} , e_{n+1} \rg \geq \cos \theta$ inside $B_{R_2}$ upon taking $\delta \in (0,\delta_{n,k})$ and $\ve \in (0,\ve_0)$, where $\ve_0 < C(n,k) \delta R_2^{-(\beta+2)}$.
    This establishes $(ii)$.

        \smallskip 
    
    \noindent \textbf{Property $(iii)$.}
    The free boundary of $U_{\on{cap}}$ is given by the expression~\eqref{eqn:level-set-equation}.
    To ensure that $\partial \{ U_{\on{cap}} > 0 \} \subset \{ s \geq \frac{t_{\ve}}{\sqrt{1-t_{\ve}^2}} r \}$, we will check that $\partial \{ U_{\on{cap}} > 0 \} \subset \{ s \geq a_{n,k} (1+ \frac{R_2^4+1}{R_2^4} \kappa \ve) r\}$, where $\kappa= \kappa_{n,k}$ is as in~\eqref{eqn:kappa-n-k}.
    In view of the expansion $\frac{t_{\ve}}{\sqrt{1-t_{\ve}^2}} = a_{n,k}(1+ \kappa_{n,k} \ve + O(\ve^2))$, this will imply the desired containment for $\ve \in (0,\ve_0)$ with $\ve_0(n,k,R_2)$ sufficiently small.
    For the latter property, we equivalently want
    \begin{align*}
    & \bigl( a_{n,k}^2 (1 + \tfrac{R_2^4+1}{R_2^4} \kappa \ve)^2 R_2^2 + \delta^2 R_2^{-2(\beta+1)} \bigr)^2 - a_{n,k}^4 R_2^4 \leq \bigl(1+\delta^2 R_2^{-2(\beta+1)} \bigr)^2 \\
    & \quad\iff \Bigl( \bigl(1+ \tfrac{R_2^4+1}{R_2^4} \kappa \ve \bigr)^4 - 1 \Bigr) (a_{n,k} R_2)^4 + 2 \Bigl(1+ \tfrac{R_2^4+1}{R_2^4} \kappa \ve\Bigr)^2 ( a_{n,k} \delta R_2^{-\beta})^2 \leq 1 + 2 \delta^2 R_2^{-2(\beta+1)}.
    \end{align*}
    We observe that $\bigl(1+ \tfrac{R_2^4+1}{R_2^4} \kappa \ve\bigr)^4 - 1 < 16 \kappa \ve$ and $\bigl(1 + \tfrac{R_2^4+1}{R_2^4} \kappa \ve\bigr)^2 < 2$, for all $\ve \in (0,\ve_0(\kappa))$ sufficiently small.
    It will therefore suffice to ensure that
    \[
    2 \left[8 \kappa \ve(a_{n,k} R_2)^2 + 2 \delta^2 R_2^{-2(\beta+1)} \right] (a_{n,k} R_2)^2 \leq 1 + 2 \delta^2 R_2^{-2(\beta+1)}.
    \]
    Finally, we observe that this inequality is satisfied whenever $\ve < \ve_0 = C(n,k) \delta R_2^{-(\beta+2)}$ and $R_2 \gg_{n,k}1$ is sufficiently large, where $\beta=4$ in our case.
    Indeed, the left-hand side becomes
    \begin{align*}
    2 \left[ 8 \kappa \ve(a_{n,k} R_2)^2 + 2 \delta^2 R_2^{-2(\beta+1)} \right] (a_{n,k} R_2)^2 < 2 \left( 8 C(n,k) \kappa \ve_0 a_{n,k}^2 R_2^2 + 2 R_2^{-3} \right) a_{n,k}^2 R_2^2 < C(n,k) \delta
    \end{align*}
    which is satisfied whenever $\delta \in (0,\delta_{n,k})$ is sufficiently small, establishing $(iii)$.
    We note that the choice of the factor $\frac{R_2^4+1}{R_2^4}$ is made to produce $(1 + \frac{R_2^4+1}{R_2^4} \kappa \ve) = 1 + \kappa \ve + \kappa R_2^{-4} \ve$, which is $1 + \kappa \ve + O(\ve^2)$ for $\ve = O (R_2^{-(\beta+2)})$ as determined above.

    \smallskip

\noindent\textbf{Property $(iv)$.}
    For the function~\eqref{eqn:varphi-n,k-and-S-cap}, the property $(iv)$ amounts to $U_{\on{cap}}(r,s) > \sqrt{r^2 + s^2} \, f_{\ve}(t)$.
    Since $U_{\on{cap}}$ is defined by $\{ \tilde{\varphi}_{n,k,\delta} = (1+\delta^2 R_2^{-2(\beta+1)} )^2 \}$, we equivalently want 
    \[
    \tilde{\varphi}_{n,k,\delta} ( r,s, \rho f_{\ve}(t)) < (1+\delta^2 R_2^{-2(\beta+1)})^2.
    \]
    We write $s^2 = \rho^2 t^2$ and $r^4 = \rho^4 (1-t^2)^2$, so we want
    \[(\rho^2 t^2 + ( \rho f_{\ve}(t) + \delta R_2^{-(\beta+1)})^2)^2 - a_{n,k}^4 \rho^4 (1-t^2)^2 < (1+\delta^2 R_2^{-2(\beta+1)})^2.
    \]
    Expanding this expansion, we equivalently need
    \[
    \rho^4 ( t^2 + f_{\ve}(t)^2)^2 + 4 \delta R_2^{-(\beta+1)} \rho^3 f_{\ve}(t) (t^2 + f_{\ve}(t)^2) - a_{n,k}^4 \rho^4 (1-t^2)^2 < 1 + o (R^{-2\beta})
    \]
    upon using $R \leq R_2$ to absorb the terms
    \[
    4 (\delta R_2^{-(\beta+1)})^2 \rho \bigl( \delta R_2^{-(\beta+1)} f_{\ve}(t) + \rho  ( t^2 + 2 f_{\ve}(t)^2) \bigr) = o ( R^{-2\beta}).
    \] 
    Note that $f_0(t)^2 = \hat{f}_{n,k}(t)^2 = a_{n,k}^2 - \tfrac{n-1}{n-k-1} t^2$ and $t^2 + f_0(t)^2 = a_{n,k}^2 (1-t^2)$, so we may write
    \begin{align*}
    (t^2 + f_{\ve}(t)^2)^2 - a_{n,k}^4 (1-t^2)^2 &= \bigl( t^2 + f_{\ve}(t)^2 - a_{n,k}^2(1-t^2) \bigr) \bigl( t^2 + f_{\ve}(t)^2 + a_{n,k}^2(1-t^2) \bigr) \\
    &\leq C_1(n,k) ( f_{\ve}(t)^2 - f_0(t)^2).
    \end{align*}
    The last expression is $\leq C_2(n,k) \ve$, in view of $\| f_{\ve} - f_0 \|_{C^0} \leq C(n,k) \ve$ and~\eqref{eqn:g-eps-g0-closeness}.
    Therefore, it suffices to ensure that
    \[
    C_1(n,k) R_2^4 \ve + C(n,k,\kappa) \delta R_2^{2-\beta} < 1 \qquad \text{for } \; \ve \in (0,\ve_0)
    \]
    where $\beta>3$ in all cases.
    Therefore, the required inequality is guaranteed by taking $\delta_{n,k} < \frac{1}{10 C(n,k,\kappa)}$ and $\delta \in (0,\delta_{n,k})$, followed by $\ve_0 < \tilde{C}(n,k) \delta R_2^{-(\beta+2)}$.
    This proves the property $(iv)$.

    \smallskip

    \noindent \textbf{Property $(v)$.}
    Finally, for property $(v)$, we recall from Lemma~\ref{lemma:normal-graph-varphi-n,k} that the surface $S_{\text{cap}}$ is expressible as the normal graph of a function $\hat{u}_1$ over the surface $\mathbf{C}_{n,k,\frac{\pi}{2}} \cap A(R_1, R_2)$ satisfying~\eqref{eqn:scap-over-the-cone}.
    By Lemma~\ref{lemma:graph-over-extended-link-pi/2}, we know that $\mathbf{C}_{n,k,\frac{\pi}{2}} \cap A(R_1, R_2)$ is expressible as the normal graph of a function $v_{\ve}^\mathbf{C}$ over the extended cone $\hat{\mathbf{C}}_{n,k,\frac{\pi}{2}-\ve} \cap A(R_1, R_2)$ satisfying~\eqref{eqn:v-eps-C-estimate}.
    For $\ve \in (0, C(n,k) R^{-4 \beta} )$ sufficiently  small, the properties of~\eqref{eqn:sff-norm-epsilon} and Lemma~\ref{lemma:graph-over-extended-link-pi/2} imply that the surface $\hat{\mathbf{C}}_{n,k,\frac{\pi}{2}-\ve} \cap A(R_1, R_2)$ have bounded geometry, uniformly in $\ve$.
    We may therefore apply Lemma~\ref{lemma:iterated-normal-graph} to see that when $\| \hat{u}_1 \|_{C^m} , \| v_{\ve}^\mathbf{C} \|_{C^m}$ are sufficiently small (in terms of $n,k$), then $S_{\text{cap}}$ is expressible as a normal graph of a function $u_1$ over $\hat{\mathbf{C}}_{n,k,\frac{\pi}{2}-\ve} \cap A(R_1, R_2)$, which satisfies~\eqref{eqn:w-u-v-bounds} in the form
    \[
    \| u_1 - ( \hat{u}_1 + v^{\mathbf{C}}_{\ve}) \|_{C^{m-1}} \leq C(n,k) \| \hat{u}_1 \|_{C^m} \| v^{\mathbf{C}}_{\ve} \|_{C^m}.
    \]
    In particular, we obtain
    \begin{align*}
    |\nabla^2_{\mathbf{C}} u_1| &\leq |\nabla^2_{\mathbf{C}} \hat{u}_1| + |\nabla^2_{\mathbf{C}} v_{\ve}^\mathbf{C}| + C(n,k) \ve R^{1-\beta} \leq C(n,k) ( R^{-(\beta+1)} + \ve R_1^{-1} + \ve R^{1-\beta})
    \end{align*}
    by combining the explicit form of $\bar{u}_1$ with the bound~\ref{lemma:graph-over-extended-link-pi/2}.
    Finally, taking $\delta < C(n,k)^{-1} \Lambda^{-1}$ and $\ve_0 = C(n,k) \delta R_2^{-(\beta+1)}$ ensures that $\ve <\ve_0 < \Lambda^{-1} R_2^{-4 \beta}$ and $\ve (R_1^{-1} + R^{1-\beta}) < R^{-(\beta+1)}$ over $A(R_1, R_2)$ and $|\nabla^2_{\mathbf{C}} u_1| < C(n,k) R^{-(\beta+1)}$, as desired.
    The bounds~\eqref{eqn:phi-ck-bounds} follows similarly, by writing
    \begin{align*}
        R^{-1} &|u_1 - \bar{u}_1| + |\nabla_{\mathbf{C}} (u_1 - \bar{u}_1)| \\
    &\leq R^{-1} ( |u_1 - \hat{u}_1| + |\hat{u}_1 - \bar{u}_1|) + |\nabla_{\mathbf{C}} ( u_1 - \hat{u}_1) | + |\nabla_{\mathbf{C}} ( \hat{u}_1 - \bar{u}_1)| \\
        &\leq C(n,k) ( R^{2(1-\beta)} + \delta R_2^{-(\beta+1)} R^{-(\beta+1)}) + R^{-1} |v^{\mathbf{C}}_{\ve}| + |\nabla_{\mathbf{C}} v^{\mathbf{C}}_{\ve}| + C(n,k) \| \hat{u}_1 \|_{C^2} \| v_{\ve}^\mathbf{C} \|_{C^2}  ).
    \end{align*}
    In the last step, we bounded the term $R^{-1} |\tilde{u}_1 - \bar{u}_1| + |\nabla_\mathbf{C} (\tilde{u}_1 - \bar{u}_1)|$ from~\eqref{eqn:scap-over-the-cone}.
    Moreover,
    \[
    R^{-1} |v^{\mathbf{C}}_{\ve}| + |\nabla_{\mathbf{C}} v_{\ve}^\mathbf{C}| + C(n,k) \| \hat{u}_1 \|_{C^2} \| v_{\ve}^\mathbf{C} \|_{C^2} \leq C(n,k) \ve + C(n,k) \ve R^{1-\beta} R_2
    \]
    from Lemma~\ref{lemma:graph-over-extended-link-pi/2}, which is $\leq C(n,k) R^{2(1-\beta)}$ for $\ve \in (0,\ve_0)$ with $\ve_0 \leq C(n,k) \delta R^{-4\beta}$.
    Combining these steps, we obtain the bounds~\eqref{eqn:phi-ck-bounds}; this establishes property $(v)$.

    The above arguments conclude the construction of the compact cap $S_{\on{cap}}$ in case IV, corresponding to $n \geq 8$.
    The same arguments extended to Case II, upon replacing the exponent $4$ with $\frac{7}{2}$ at each step.
    The verification of Cases I and III, via~\eqref{eqn:varphi-n,k-and-S-cap}, is also a straightforward adaptation of the above discussion.
    In this case, the free boundary of the cap profile is given by
    \begin{equation}\label{eqn:cap-exceptional}
        \bigl( (s^2 + \delta^2 R_2^{-2(\beta+1)}) - \tfrac{1}{n-2} r^2 \bigr) \bigl( s^2 + \delta^2 R_2^{-2(\beta+1)}) = (1+\delta^2 R_2^{-2(\beta+1)})^2.
    \end{equation}
    Indeed, property $(i)$ does not invoke the particular form of $\tilde{\varphi}_{n,k,\delta}$, and therefore follows verbatim; for property $(ii)$, we similarly compute that
    \begin{align*}
    \la \nu_{S_{\on{cap}}} , e_{n+1} \rg &= \frac{\partial_{n+1} \tilde{\varphi}}{|\nabla \tilde{\varphi}|} \bigg\rvert_{z = 0} = \frac{\delta \bigl( 2(s^2+\delta^2 R_2^{-2(\beta+1)}) - \tfrac{1}{n-2} r^2)}{\sqrt{(s^2+\delta^2 R_2^{-2(\beta+1)}) \bigl[ \tfrac{1}{(n-2)^2} r^4 + (2 (s^2+\delta^2 R_2^{-2(\beta+1)}) - \tfrac{1}{n-2} r^2)^2 \bigr]}} \\
    &> \frac{\delta R_2^{-(\beta+1)}}{5 \sqrt{s^2+\delta^2 R_2^{-2(\beta+1)}}}
    \end{align*}  
    which is $> \frac{1}{10} \delta R_2^{-(\beta+2)}$ on $B_{R_2}$.
    Here, we used the rough bound coming from~\eqref{eqn:cap-exceptional},
    \[
    s^2 + \delta^2 R_2^{-2(\beta+1)} <  2 ( s^2 + \delta^2 R_2^{-2(\beta+1)} ) - \tfrac{1}{n-2} r^2  < 2 (s^2 + \delta^2 R_2^{-2(\beta+1)}).
    \]
    The inequality $\cos \theta < 2 a_{n,k}^{-1} \ve$ used above shows that $\la \nu_{S_{\on{cap}}} , e_{n+1} \rg \geq \cos \theta$ is be satisfied for $\delta \in (0,\delta_{n,k})$ and $\ve < \ve_0 = C(n,k) \delta R_2^{-(\beta+2)}$.
    Property $(iii)$ is satisfied for $\ve \leq \frac{1}{10 C(n,k)} R_2^{-(\beta+2)}$ via the same manipulations, and property $(iv)$ may simply be checked along $s=0$, where it amounts to
    \[
    \bigl( \bigl( r f_{\ve}(0) + \delta R_2^{-(\beta+1)}\bigr)^2 - \tfrac{1}{n-2} r^2 \bigr) (r f_{\ve}(0) + \delta R_2^{-(\beta+1)})^2 < (1+\delta^2 R_2^{-2(\beta+1)})^2
    \]
    The same computation as in $(iv)$ above shows that the property is satisfied upon taking $\delta \in (0,\delta_{n,k})$ and $\ve < \ve_0 = C(n,k) \delta R_2^{-(\beta+2)}$, for dimensional constants $\delta_{n,k}, C(n,k)$.
    Finally, property $(v)$ follows from an application of Lemma~\ref{lemma:iterated-normal-graph} and the bound~\ref{eqn:w-u-v-bounds} identical to the above argument.

    To conclude the proof of the result as stated, we note that the above constructions apply to any $\ve \in (0,\ve_0)$ with $\ve_0 = C(n,k) \delta R_2^{-4\beta}$.
    We may therefore redefine $\Lambda(n,k)^{-1}$ to be the minimum between the previous $\Lambda(n,k)^{-1}$ (from the annular region) and this $C(n,k) \delta$.
    We finally take $R_0 \gg_{n,k} 1$ sufficiently large to ensure that $\Lambda(n,k) R_2^{-4\beta} < \ve(n,k)$, where $\ve(n,k)$ are the ambient dimensional bounds within which the indicial root interval satisfies~\eqref{eqn:indicial-roots-close} and the estimates of~\eqref{eqn:g-eps-g0-closeness}~--~\eqref{eqn:t-hat-t-eps} an Lemma~\ref{lemma:graph-over-extended-link-pi/2} are valid.
    This completes our construction.
\end{proof}

We note that the existence of the mean-convex caps described in Proposition~\ref{prop:mean-convex-cap} is essentially equivalent to the strict minimality of the cone $\mathbf{C}_{n,k,\frac{\pi}{2}}$ on the side $E_+$.
Indeed, taking $R_0 \to + \infty$ makes $\ve \downarrow 0$ and recovers the free boundary mean-convex surface $S_+$ on the $E_+$ side of $\mathbf{C}_{n,k,\frac{\pi}{2}}$, which exists (by~\cites{hardt-simon, dephilippis-maggi, lawson-liu}) if and only if $n \geq 7 $ and $1 \leq k \leq n-2$, with $(n,k) \neq (7,5)$.

We now discuss the adaptation of Proposition~\ref{prop:mean-convex-cap} to the side $E_-$.
We require $n \geq 7$ and $1 \leq k \leq n-2$, with $(n,k) \neq (7,1)$, corresponding to the strictly minimizing cones $C(\bS^{n-k-1} \times \bS^k_+)$ on $E_-$.
The previous arguments carry over to this side, with one important caveat: we need to \textit{raise} the cap above its plane of symmetry by a small amount, in order to produce a definite increase to the contact angle along $\Pi$.
While the resulting surface will no longer be graphical over the plane $\Pi$, the necessary constructions transfer in the normal graph gauge over the cone $\mathbf{C}$.

Let us record the construction of the resulting surface:
\begin{proposition}\label{prop:mean-concave-cap-on-E-minus}
Consider any $n \geq 7$ and $1 \leq k \leq n-2$ except $(n,k) = (7,1)$.
Let $\beta(n) = \frac{7}{2}$ for $n=7$, $\beta(n) = 5$ for $\{ 8 \leq n \leq 12, k= 1\}$, and $\beta(n) = 4$ all other $(n,k)$.
There exists a $\Lambda = \Lambda(n,k)$ such that for any $R_1 > 0$ sufficiently large, any $R_2 > \Lambda R_1$, and $\ve \in (0, \Lambda^{-1} R_2^{-4\beta})$, there exists a surface-with-boundary $S^-_{\text{cap}} \subset E_-$, with $S^-_{\text{cap}} \cap \partial B_{R_2} \neq \varnothing$, satisfying the following properties: 
\begin{enumerate}[$(i)$]
\item We have $H[S_{\text{cap}}^-] < - d_{n,k} R^{-(1+\beta)}$, meaning that $S^-_{\text{cap}}$ is mean-concave with respect to the upward-pointing normal vector $\nu_S$.
\item The inequality $\la \nu_S , e_{n+1} \rg \leq 0$ is satisfied along the free boundary of $S^-_{\text{cap}}$.
\item The free boundary of $S^-_{\text{cap}}$ is contained in the cone region $E_-(\mathbf{C}_{n,k,\frac{\pi}{2} - \ve}) \cap \Pi$, given by $\{ |y| < \tfrac{t_{\ve}}{\sqrt{1-t_{\ve}^2}} |x| \}$, and has distance $1$ from the origin.
\item The surface $S^-_{\text{cap}} \cap A(R_1, R_2)$ over the annulus $A(R_1, R_2) := B_{R_2} \setminus \bar{B}_{R_1}$ is expressible as a negative normal graph over the cone $\mathbf{C}_{n,k,\frac{\pi}{2}-\ve}$, i.e., in the direction opposite the cone normal $\nu_{\mathbf{C}}$ pointing into $E_+$, with 
\[
 R^{-1} | u_1 - \underbar{u}_1 | + |\nabla_{\mathbf{C}} (u_1 - \underbar{u}_1)| \leq C(n,k) R^{2-2\beta}\;\text{ where }  \;\underbar{u}_1(p) := c^-_{n,k} R^{1-\beta} - a_{n,k}^{-1} \delta R_2^{-(\beta+1)} \omega_{n+1}(p)
\]
    for $a_{n,k} = \sqrt{\frac{k}{n-k-1}}$ and $c^-_{n,k}$ given by
    \[
    \mathrm{I: } \; c_{7,k}^- = \frac{a_{7,k}(1 + a^2_{7,k})^{\frac{3}{4}}}{2 }, \quad \mathrm{II: } \; c_{7,3}^- = \frac{2^{\frac{7}{4}}}{7}, \quad \mathrm{III: } \; c_{n,1}^-=\frac{(n-1)^{\frac{3}{2}}}{2(n-2)^2}, \quad \mathrm{IV: } \; c_{n,k}^- = \frac{a_{n,k}(1 + a^2_{n,k})^{2}}{4 }.
    \]
    in each of the Cases I-IV from~\eqref{eqn:record-varphi-n,k-on-E-minus}.
    Moreover, $|\nabla^2_{\mathbf{C}} u_1| \leq C(n,k) R^{-(\beta+1)}$.
\end{enumerate}
\end{proposition}
\begin{proof}
We work with the functions $\varphi_{n,k}^-$ from~\eqref{eqn:record-varphi-n,k-on-E-minus}, which we abbreviate to $\varphi_{n,k}$ in what follows.
For some small $\delta>0$, to be determined, we define
\begin{equation}\label{eqn:s-cap-on-e-minus}
    \begin{split}
        \tilde{\varphi}_{n,k,\delta} (r,s,z) &:= \varphi_{n,k} ( r , s , z - \delta R_2^{-(\beta+1)} ) , \\
        S^-_{\text{cap}} &:= \{ \tilde{\varphi}_{n,k,\delta}(r,s,z) = \varphi_{n,k}(1,0,-\delta R_2^{-(\beta+1)} ) \}
    \end{split}
\end{equation}
    as in~\eqref{eqn:varphi-n,k-and-S-cap}.
    This construction ensures that $(1,0,0) \in S^-_{\text{cap}} \subset \Pi$, and amounts to ``raising'' the level set of $\varphi_{n,k}$ inside $E_-$ by $\delta R_2^{-(\beta+1)}$ to form an obtuse angle along the free boundary.
    Then, the property $(ii)$ is automatic, while $(i)$ again follows from the property~\eqref{eqn:quantitative-sub-calibration} of the $\varphi_{n,k}$.
    The property $(iii)$, along with $S^-_{\text{cap}} \subset E_-$, is established as in Proposition~\ref{prop:mean-convex-cap}, upon taking $\ve < C(n,k) \delta R^{-4\beta}$ sufficiently small.
    To obtain $(iv)$, we use the fact that $S^-_{\text{cap}}$ is expressible as the negative normal graph of a function $\hat{u}_1$ over the cone $\mathbf{C}_{n,k,\frac{\pi}{2}}$, which satisfies 
    \[
    R^{-1} |\hat{u}_1 - \underbar{u}_1| + |\nabla_{\mathbf{C}}( \hat{u}_1 -\underbar{u}_1) | \leq C(n,k) R^{2(1-\beta)}, \qquad |\nabla^{\ell}_{\mathbf{C}} \hat{u}_1| \leq C(n,k,\ell) R^{1-\beta-\ell}
    \]
    as in Lemma~\ref{lemma:normal-graph-varphi-n,k}.
    The expansion of $\bar{u}_1$ is obtained by applying the arguments therein, and the constants $c_{n,k}^-$ are likewise computed as $c_{n,k}^-=(-\varphi_{n,k}^- (1,0,0))\frac{R^{\beta-1}}{|\nabla \varphi^-_{n,k}|}$.
    The expression of $S^-_{\text{cap}}$ as a normal graph over $\mathbf{C}_{n,k,\frac{\pi}{2}-\ve}$ is then obtained by combining Lemmas~\ref{lemma:graph-over-extended-link-pi/2} and~\ref{lemma:iterated-normal-graph}, as done above; this proves $(iv)$.
    We conclude that $S^-_{\text{cap}}$ has all the claimed properties $(i)$~--~$(iv)$.
\end{proof}

\subsection{Construction of a foliation}\label{section:construct-foliation}

Combining the constructions of Lemma~\ref{lemma:infinity-piece-small-epsilon} and Proposition~\ref{prop:mean-convex-cap}, we may construct a barrier surface $S_{\pm} \subset E_{\pm}(\mathbf{C}_{n,k,\frac{\pi}{2}-\ve})$, for $\ve$ sufficiently small.

\begin{proposition}\label{prop:mean-convex-surface}
    For every $n \geq 7$ and $1 \leq k \leq n-2$, except $(n,k) = (7,5)$, there exists an $\ve_{n,k} > 0$ such that the following holds.
    For all $\theta \in ( \frac{\pi}{2} - \ve_{n,k} , \frac{\pi}{2}]$, there exists a graphical surface-with-boundary $S_+ \subset E_+ ( \mathbf{C}_{n,k,\theta})$ asymptotic to the cone $\mathbf{C}_{n,k,\theta}$ that is strictly mean-convex (with respect to the upward-pointing normal vector), forms a contact angle $ < \theta$ with the plane $\Pi = \{ z = 0 \}$ along its free boundary, and its positive homotheties foliate $E_+(\mathbf{C}_{n,k,\theta})$.
\end{proposition}
\begin{proof}
We will construct the surface $S_+$ in three pieces.
For sufficiently large radii $R_1 < R_2$, to be determined, we define the annulus $A(R_1,R_2) := B_{R_2} \setminus \bar{B}_{R_1}$.
We will also abbreviate the cone $\mathbf{C}_{n,k,\frac{\pi}{2}-\ve}$ (resp.~the extended cone $\hat{\mathbf{C}}_{n,k,\frac{\pi}{2}-\ve}$) to ${\mathbf{C}}$ (resp.~$\hat{\mathbf{C}}$) in what follows.
Using $\beta(n) = \frac{7}{2}$ for $n = 7$ and $\beta(n) = 4$ for $n \geq 8$, we find $\beta-1 \in (\underline{\gamma}, \overline{\gamma})$ in the corresponding indicial interval of the cone $\mathbf{C}_{n,k,\frac{\pi}{2}}$, so $1- \beta \in (- \overline{\gamma}^{\ve}, - \underline{\gamma}^{\ve})$, for $\ve$ sufficiently small, in view of~\eqref{eqn:indicial-roots-close}.
    Consequently, for $\tilde{R}_1 \geq R_0(n,k)$, we may construct
    \[
    \tilde{S}_{\beta-1,\psi} := \text{graph}_{\hat{\mathbf{C}} \setminus B_{\tilde{R}_1}} \tilde{u}_2, \qquad \tilde{u}_2 := ( R^{1-\beta} - R^{-\beta}) (1 - \tau R_2^{-2} \omega_{n+1}) 
    \]
    to satisfy the properties discussed in Lemma~\ref{lemma:infinity-piece-small-epsilon}.
    We consider the constants $c_{n,k}^+$ of Proposition~\hyperref[prop:mean-convex-cap]{\ref{prop:mean-convex-cap} $(v)$}, which we abbreviate to $c_{n,k}$ for this construction.
    We then define the rescaled surface $S_{\beta-1,\psi} := c_{n,k}^{\frac{1}{\beta}} \tilde{S}_{\beta-1,\psi}$, which is again mean-convex, with contact angle $< \theta$ along $\Pi$; moreover, $c_{n,k}^{\frac{1}{\beta}} \Bigl(c_{n,k}^{- \frac{1}{\beta}} R \Bigr)^{1-\beta} = c_{n,k} R^{1-\beta}$ means that we can write
    \[
    S_{\beta-1,\psi} = \text{graph}_{\hat{\mathbf{C}} \setminus B_{R_1}} u_2, \qquad u_2 := ( c_{n,k} R^{1-\beta} - c_1 R^{-\beta}) (1- \tau R_2^{-2} \omega_{n+1}), \qquad R_1 \geq R_0(n,k)
    \]
    by redefining $R_0(n,k)$ as $c_{n,k}^{\frac{1}{\beta}}R_0(n,k)$.
    We then define a smooth hypersurface $S_+ \subset E_+$ by
    \begin{equation}\label{eqn:define-glued-surface}
    S_+ \cap B_{R_1} = \text{graph}_{B_{R_1}}(U_{\on{cap}}), \qquad S_+ \cap A(R_1,R_2) = S_{\on{glue}}, \qquad S_+ \cap B(R_2)^{\text{\sffamily{C}}} = S_{\beta-1,\psi}.
    \end{equation}
    Here, $U_{\on{cap}}$ is the function obtained in Proposition~\ref{prop:mean-convex-cap}, with $\on{graph}_{B_{R_2}}(U_{\on{cap}}) = S_{\on{cap}}$.
    By Proposition~\ref{prop:mean-convex-cap}, we may choose $\ve < \ve_0(R_1, R_2)$ sufficiently small so that $S_{\on{cap}} \cap A(R_1, R_2)$ is expressible as the normal graph of a function $u_1$ over the annular region $\hat{\mathbf{C}} \cap A(R_1, R_2)$ with 
    \[
    R^{-1} |u_1 - \bar{u}_1| + |\nabla_{\mathbf{C}} (u_1 - \bar{u}_1)| \leq C(n,k) R^{2(1-\beta)}, \quad \text{where } \; \bar{u}_1(p) = c_{n,k} R^{1-\beta} - a_{n,k}^{-1} \delta R_2^{-(\beta+1)} \omega_{n+1}
    \]
    as in~\eqref{eqn:scap-over-the-cone} in Lemma~\ref{lemma:normal-graph-varphi-n,k}.
    We recall that $\beta(n) \geq \frac{7}{2} > 3$ in all cases, so the term $R^{2(1-\beta)}$ is genuinely lower-order and may be absorbed to write $u_1 = \bar{u}_1 + O(R^{3-2\beta}) = c_{n,k} R^{1-\beta} + O(R^{-(\beta+1)})$ inside the annular region $\hat{\mathbf{C}} \cap A(R_1, R_2)$.
    On the other hand, we have
    \[
    u_2 (p) = c_{n,k} R^{1-\beta} - c_1 R^{-\beta} (1 - \tau R_2^{-2} \omega_{n+1}) - c_{n,k} \tau R_2^{-2} R^{1-\beta} \omega_{n+1}
    \]
    which means that, over $\hat{\mathbf{C}}\cap A(R_1, R_2)$, we may write
    \begin{equation}\label{eqn:u2-R-1-beta}
    R^{-1} |u_2 - c_{n,k} R^{1-\beta}| + |\nabla_{\mathbf{C}} (u_2 - c_{n,k} R^{1-\beta})| \leq C(n,k) R^{-(\beta+1)}.
    \end{equation}
    Combining the expressions for $u_1$ and $u_2$ shows that $u_1 - u_2 = c_1 R^{-\beta} + O ( R^{-(\beta+1)})$, with
    \begin{equation}\label{eqn:u1-u2-difference-estimates}
        R^{-1} |u_1 - u_2| + |\nabla_{\mathbf{C}} (u_1 -u_2)| \leq C(n,k) R^{-(\beta+1)}.
    \end{equation}
    Moreover, we can choose $\tau < \frac{1}{4}$ and $R_1$ sufficiently large to satisfy
    \begin{equation}\label{eqn:u1-u2-lower-order}
    \tfrac{1}{2} c_1 R^{-\beta} < u_1 - u_2 < 2 c_1 R^{-\beta}
    \end{equation}
    over $\hat{\mathbf{C}} \cap A(R_1, R_2)$.
    We now take $R_2 = M R_1$, where $M(n,k)$ is a large constant, and $R_1 \gg_{n,k} 1$ is sufficiently large (to be determined later).
    Let $\tilde{\zeta}: [0,1] \to [0,1]$ be a smooth cutoff function satisfying $\tilde{\zeta} \equiv 1$ for $s \in [0,\frac{1}{4}]$, $\tilde{\zeta} \equiv 0$ for $s \in [ \frac{3}{4}, 1]$, and $-10 \leq \tilde{\zeta}' \leq 0$ on $[ \frac{1}{4}, \frac{3}{4}]$.
    We then define $\zeta : [R_1, R_2] \to [0,1]$ by $\zeta(R) := \tilde{\zeta} ( \frac{R - R_1}{R_2 - R_1} )$ and take $\zeta(p) := \zeta(|p|)$ so that $\zeta \equiv 1$ on $B_{\frac{1}{4} R_2}$ and $\zeta \equiv 0$ outside $B_{\frac{4}{5} R_2}$.
    By construction, $\nabla \zeta = \tfrac{1}{R_2-R_1} \tilde{\zeta}' \frac{p}{R}$, which implies that $|\nabla \zeta| \leq \frac{C(n)}{MR_1} \leq \tilde{C}(n) R^{-1}$ on $A(R_1,R_2)$.
    Iterating this computation shows that $|\nabla^{\ell} \zeta| \leq C(n,\ell) R^{-\ell}$ for all $\ell \geq 0$.
    We then let
    \[
    u := \zeta(R) u_1 + (1 - \zeta(R)) u_2
    \]
    and define $S_{\text{glue}}$ over $A(R_1, R_2)$ to be the normal graph of $u$ over the extended cone $\hat{\mathbf{C}}$,
    \[
    S_\text{glue} := \on{graph}_{\hat{\mathbf{C}} \cap A(R_1, R_2)} u= \{ p + u(p) \nu_{\mathbf{C}}(p) : p \in \hat{\mathbf{C}} \cap A(R_1, R_2) \} \cap \bR^{n+1}_+.
    \]
    Consequently, the definition~\eqref{eqn:define-glued-surface} produces a smooth hypersurface $S_+ \subset E_+$.
    We now prove that $S_+$ has the desired barrier properties.

\smallskip \noindent \textbf{Mean convexity.}
    The pieces $S_{\on{cap}}$ and $S_{\beta-1,\psi}$ are mean-convex, by Proposition~\ref{prop:mean-convex-cap} and Lemma~\ref{lemma:infinity-piece-small-epsilon}, respectively.
    We verify the mean-convexity of $S_{\text{glue}}$ by a computation analogous to Lemma~\ref{lemma:infinity-piece-small-epsilon}.
    For the glued function $u$, the linearized operator satisfies
    \begin{equation}\label{eqn:linearized-operator}
        \cL_{\mathbf{C}} u = \zeta \cL_{\mathbf{C}} u_1 + (1- \zeta) \cL_{\mathbf{C}} u_2 + \Bigl( \zeta'' + \frac{n-1}{R} \zeta' \Bigr) (u_1 - u_2) + 2 \zeta' \partial_R (u_1 - u_2).
    \end{equation}
    We observe that
    \begin{equation}\label{eqn:nabla-C-u}
        \nabla_{\mathbf{C}} u = \zeta'(R) (u_1 - u_2) \omega + \zeta \nabla_{\mathbf{C}} u_1 + (1-\zeta) \, \nabla_{\mathbf{C}} u_2
    \end{equation}
    and computing similarly, we may bound
    \begin{align*}
    |\nabla_{\mathbf{C}} u| &\leq |\zeta'| |u_1 - u_2| + |\nabla_{\mathbf{C}} u_1| + |\nabla_{\mathbf{C}} u_2| , \\
    |\nabla^2_{\mathbf{C}} u| &\leq ( |\zeta''| + n R^{-1} |\zeta'|) |u_1 - u_2| + 2 |\zeta'| \, |\nabla_{\mathbf{C}} (u_1 - u_2)| + |\nabla^2_{\mathbf{C}} u_1| + |\nabla^2_{\mathbf{C}} u_2|.
    \end{align*}
    Using the bounds $|\nabla^{\ell}_{\mathbf{C}} u_i| < c_{\ell} R^{1-\beta-\ell}$ in the above computations, we therefore obtain $|\nabla^{\ell}_{\mathbf{C}} u| < \tilde{c}_{\ell} R^{1-\beta-\ell}$.
    Arguing as in Lemma~\ref{lemma:infinity-piece-small-epsilon}, we may therefore use the estimate~\eqref{eqn:mean-curvature-operator} to obtain
    \[
    |H[S_{\on{glue}}] + \cL_{\mathbf{C}} u| \leq C(n) R^{-2 \beta}.
    \]
    On the other hand, the estimate~\eqref{eqn:phi-ck-bounds} becomes
    \[
    R^{-1} |u_1 - \bar{u}_1| + |\nabla_{\mathbf{C}} (u_1 - \bar{u}_1)| + R |\nabla^2_{\mathbf{C}} (u_1 - \bar{u}_1)| < C(n,k) R^{2-2 \beta}
    \]
    which implies, in particular, that 
    \[
    |\Delta_{\mathbf{C}} (u_1 - \bar{u}_1) | + |A_{\mathbf{C}}|^2 |u_1 - \bar{u}_1| < C(n,k) R^{-2 \beta+1} \implies |\cL_{\mathbf{C}} (u_1 - \bar{u}_1)| < C(n,k) R^{-2 \beta+1}.
    \]
    Using the bound~\eqref{eqn:u1-u2-difference-estimates} together with $|\nabla^{\ell} \zeta | \leq C(n) R^{-\ell}$ in~\eqref{eqn:linearized-operator}, we obtain
    \[
    \cL_{\mathbf{C}} u \leq \zeta \cL_{\mathbf{C}} u_1 + (1- \zeta) \cL_{\mathbf{C}} u_2 + C(n) R^{-(\beta+2)}.
    \]
    On the other hand, Lemma~\hyperref[lemma:infinity-piece-small-epsilon]{\ref{lemma:infinity-piece-small-epsilon} $(ii)$} showed that $\cL_{\mathbf{C}} u_2 < - c_2 R^{-(\beta+1)}$, while Proposition~\hyperref[prop:mean-convex-cap]{\ref{prop:mean-convex-cap} $(ii)$} obtained the same result for $u_1$.
    We conclude that
    \[
    \cL_{\mathbf{C}} u \leq - ( c_1 \zeta + c_2(1-\zeta) ) R^{-(\beta+1)} + C(n) R^{-(\beta+2)} \leq - \tfrac{1}{2} \min \{ c_1, c_2 \} R^{-(\beta+1)} + C(n) R^{-(\beta+2)}
    \]
    on $A(R_1, R_2)$.
    Finally, we may choose $R_1 \gg_{n,k} 1$ sufficiently large, so that on $A(R_1, R_2)$ we obtain
    \begin{align*}
    H[S_{\text{glue}}] &\geq - \cL_C u - C(n) R^{-2 \beta} \geq \tfrac{1}{2} \min \{ c_1, c_2 \} R^{-(\beta+1)} -  C(n) R^{-(\beta+2)} - C(n) R^{-2\beta} \\
    & > \tfrac{1}{4} \min \{ c_1, c_2 \} R^{-(\beta+1)},
    \end{align*}
    meaning that $S_{\text{glue}}$ (and therefore all of $S_+$) is strictly mean-convex, as desired.

\smallskip \noindent \textbf{Contact angle along $\Pi$.}
The pieces $S_{\on{cap}}$ and $S_{\beta-1,\psi}$ have contact angle $< \theta$ with the plane $\Pi$, by Proposition~\ref{prop:mean-convex-cap} and Lemma~\ref{lemma:infinity-piece-small-epsilon}, respectively.
Using the expression~\eqref{eqn:normal-vector-equality-Pi}, we may write
\[
|\nu_{S_\text{glue}}(p + u(p) \nu_{\mathbf{C}}(p)) - \nu_{\mathbf{C}}(p) + \nabla_{\mathbf{C}} u| \leq C(n) \bigl( |\nabla_{\mathbf{C}} u|^2 + |u| \, |\nabla^2_{\mathbf{C}} u| \bigr) \leq C(n) R^{-2 \beta}
\]
by applying the bound $|\nabla^{\ell}_{\mathbf{C}} u| < \tilde{c}_{\ell} R^{1-\beta-\ell}$ obtained above.
Applying this computation to the points $p + u(p) \nu_{\mathbf{C}}(p) \in S_{\text{glue}} \cap \Pi$, as in Lemma~\ref{lemma:infinity-piece-small-epsilon}, we obtain $|\la \nu_{\mathbf{C}}(p) , e_{n+1} \rg - \cos \theta| \leq C(n) R^{-\beta} \cos \theta$ and may compute $\nabla_{\mathbf{C}} u$ from the expression~\eqref{eqn:nabla-C-u}.
Consequently, 
\begin{align*}
     \la \nu_{S_\text{glue}}&(p + u(p) \nu_{\mathbf{C}}(p)), e_{n+1} \rg \\
    &\geq  \la \nu_{\mathbf{C}} (p), e_{n+1} \rg - \la \nabla_{\mathbf{C}} u, e_{n+1} \rg - C(n) R^{-2\beta} \\
    &\geq \bigl( \cos \theta - C(n) R^{-\beta} \cos \theta) - \la \zeta \nabla_{\mathbf{C}} \bar{u}_1 + (1-\zeta) \nabla_{\mathbf{C}} u_2 + \zeta'(R) (\bar{u}_1 - u_2) \omega, e_{n+1} \rg - C(n) R^{-2\beta} \\
    &\geq \cos \theta - \la \zeta \nabla_{\mathbf{C}} \bar{u}_1 + (1-\zeta) \nabla_{\mathbf{C}} u_2, e_{n+1} \rg - C(n) R^{-\beta} ( \cos \theta + R^{-\beta}) - C(n) R^{-(\beta+1)} \omega_{n+1}(p).
\end{align*}
Throughout the computation, we replaced $u_1$ by $\bar{u}_1$ using the estimate~\eqref{eqn:phi-ck-bounds} with an error term $C(n) R^{-2\beta}$.
In the last step, we applied the bound~\eqref{eqn:u2-R-1-beta} combined with $|\zeta'(R)| \lesssim R^{-1}$ on $\hat{\mathbf{C}} \cap A(R_1, R_2)$, so $|\zeta'(R) (\bar{u}_1-u_2)| < C(n) R^{-(\beta+1)}$ by applying~\eqref{eqn:u1-u2-lower-order}.
    By Proposition~\ref{prop:mean-convex-cap}, we may take $\delta < \Lambda (n,k)^{-1}$ and $\ve_0 < \Lambda^{-1} R^{-4 \beta}_2 < C(n,k) \delta R_2^{-4 \beta}$ to ensure that
    \[
    \cos \theta < 2 a_{n,k}^{-1} \ve < C(n,k) \delta R_2^{-4\beta} \leq C(n,k) \delta R^{-4 \beta} \qquad \text{on } \; A(R_1, R_2),
    \]
    allowing us to collect the above bounds into
    \[
    \la \nu_{S_\text{glue}}, e_{n+1} \rg \geq \cos \theta - \la \zeta \nabla_{\mathbf{C}} \bar{u}_1 + (1-\zeta)  \nabla_{\mathbf{C}} u_2, e_{n+1} \rg - C(n) R^{-\beta} ( R^{-\beta} + R^{-1} \omega_{n+1})
    \]
    along $S_{\text{glue}} \cap \Pi$.
    We observe that
    \begin{align*}
    u &=  \zeta \bar{u}_1 + (1-\zeta) u_2 \\
    &= c_{n,k} R^{1-\beta} - (1-\zeta) c_1 R^{-\beta} - \left[ \zeta a_{n,k}^{-1} \delta R_2^{-(\beta+1)} + (1-\zeta)(c_{n,k} R^{1-\beta} - c_1 R^{-\beta})\tau R_2^{-2} \right] \omega_{n+1},
    \end{align*}
    whereby $c_1 R^{1-\beta} \leq u \leq c_2 R^{1-\beta}$ on $A(R_1, R_2)$.
    Computing as in~\eqref{eqn:omega-(n+1)} from Lemma~\ref{lemma:infinity-piece-small-epsilon}, we find
    \[
    p + u(p) \nu_{\mathbf{C}}(p) \in \Pi \implies R\,\omega_{n+1}(p) + u(p) \la \nu_{\mathbf{C}}, e_{n+1} \rg = 0
    \]
    where $|\la \nu_{\mathbf{C}}, e_{n+1} \rg - \cos \theta| \leq C(n) R^{-\beta}$ as in~\eqref{eqn:nu-C-cos-theta}.
    We conclude that $- \tilde{c}_2 R^{1-\beta} \cos \theta < \omega_{n+1}(p) < - \tilde{c}_1 R^{1-\beta} \cos \theta$ holds along $S_{\text{glue}} \cap \Pi$.
    Finally, to compute $\nabla_{\mathbf{C}}$, observe that we may write $\nabla_{\mathbf{C}} u_i = \nabla u_i - \la \mathbf{\nu}_C, \nabla u_i \rg \nu_{\mathbf{C}} = ( I - \nu_{\mathbf{C}} \otimes \nu_{\mathbf{C}}) \nabla u_i$, so the above expression of $u = \zeta \bar{u}_1 + (1-\zeta) u_2$ gives
    \allowdisplaybreaks{
    \begin{align*}
         \la \zeta &\nabla_\mathbf{C} \bar{u}_1 + (1-\zeta) \nabla_{\mathbf{C}} u_2, e_{n+1} \rg &\\
        &= - \zeta a_{n,k}^{-1} \delta R_2^{-(\beta+1)} R^{-1} - (1-\zeta) c_{n,k} \tau R_2^{-2} R^{-\beta} & (\mr{L}_1)\\
       &\quad+ (1-\zeta) c_1 \tau R_2^{-2} R^{-(\beta+1)} \left[ 1 - \la \nu_{\mathbf{C}}, e_{n+1} \rg^2 - (\beta+1) \omega_{n+1}^2 \right] &(\mr{L}_2)\\
        &\quad+ \Bigl( - c_{n,k} (\beta-1) R^{-\beta} + (1-\zeta) c_1 \beta R^{-(\beta+1)} \Bigr) \omega_{n+1} & (\mr{L}_3)\\
        &\quad+ \zeta a_{n,k}^{-1} \delta R_2^{-(\beta+1)} R^{-1} ( \la \nu_{\mathbf{C}}(p), e_{n+1} \rg^2 + \omega_{n+1}^2 ) + (1-\zeta) c_{n,k} \tau R_2^{-2} R^{-\beta} ( \la \nu_{\mathbf{C}}(p), e_{n+1} \rg^2 + \beta \omega_{n+1}^2).& (\mr{L}_4)
    \end{align*}}
    Using the above bounds $|\omega_{n+1}| \leq C(n) R^{-4\beta}$ and $|\la \nu_{\mathbf{C}}, e_{n+1} \rg | \leq C(n) \cos \theta \leq C(n) R^{-4\beta}$, we may estimate the terms $(\mathrm{L}_i)$ of each line of the above computation by
    \[
    (\mathrm{L}_2) \leq C(n) R^{-(\beta+3)}, \qquad (\mathrm{L}_3) \leq C(n) R^{-5\beta}, \qquad (\mathrm{L}_4) \leq C(n) R_2^{-4\beta}
    \]
    over $A(R_1, R_2)$.
    Moreover, using $R_1 \leq R \le M R_1$ on $A(R_1, R_2)$ along with $\max \{ \zeta a_{n,k}^{-1} \delta , (1-\zeta) c_{n,k} \tau \} \geq \tfrac{1}{2} \max \{ a_{n,k}^{-1} \delta, c_{n,k} \tau \} \geq C(n,k,\delta,\tau)$ due to $\max \{ \zeta, 1-\zeta \} \geq \frac{1}{2}$, we can bound
    \begin{align*}
    (\mathrm{L}_1) : \quad - \zeta a_{n,k}^{-1} \delta R_2^{-(\beta+1)} R^{-1} - (1-\zeta) c_{n,k} \tau R_2^{-2} R^{-\beta} &\leq - C(n,k,\delta, \tau) \max \{ R_2^{-(\beta+1)} R^{-1}, R_2^{-2} R^{-\beta} \} \\
    &\leq - C(n,k,\delta, \tau, M) R^{-(\beta+2)}.
    \end{align*}
    We therefore obtain
    \begin{equation}\label{eqn:zeta-u1-u2-gradient-bound}
    \la \zeta \nabla_\mathbf{C} \bar{u}_1 + (1-\zeta) \nabla_{\mathbf{C}} u_2, e_{n+1} \rg \leq -  \tfrac{1}{2} C(n,k,\delta, \tau,M) R^{-(\beta+2)}
    \end{equation}
    where we used $R_1 \leq R \leq M R_1$ for $A(R_1, R_2)$ chose $R_1 \gg_{n,k} 1$ sufficiently large so that
    \[
    C(n,k,\delta, \tau, M) R^{-(\beta+2)} > \tfrac{1}{2} C(n,k,\delta, \tau, M) R^{-(\beta+2)} + C(n) R^{-(\beta+3)} + C(n) R^{-2\beta}.
    \]
    Finally, combining the above lower bound on $\la \nu_{S_{\text{glue}}}, e_{n+1} \rg$ with the estimates $(\mathrm{L}_1) - (\mathrm{L}_4)$ and the bounds $|\la \nu_{\mathbf{C}}, e_{n+1} \rg| ,\cos \theta, |\omega_{n+1}(p)| \leq C(n) R^{-4\beta}$, we conclude that 
    \begin{align*}
        \la \nu_{S_\text{glue}}, e_{n+1} \rg &\geq \cos \theta - \la \zeta \nabla_{\mathbf{C}} \bar{u}_1 + (1-\zeta)  \nabla_{\mathbf{C}} u_2, e_{n+1} \rg - C(n) R^{-2\beta} \\
        &\geq \cos \theta + \tfrac{1}{2} C(n,k,\delta, \tau, M) R^{-(\beta+2)} - C(n) R^{-2\beta}
    \end{align*}
    on $A(R_1, R_2)$.
    Taking $R_1 \geq R_0(n,k)$ sufficiently large in the last step ensures that $\la \nu_S, e_{n+1} \rg > \cos \theta$ holds everywhere along $S_+ \cap \Pi$, which establishes the desired contact angle condition.

\smallskip \noindent \textbf{Graphical property.}
The graphical property of the constructed surface $S_+$ is a direct consequence of the above discussion.
Indeed, since $\omega_{n+1}(p) \geq - R^{-1} u(p) \la \nu_{\mathbf{C}}(p), e_{n+1} \rg$ is satisfied for any $p \in \mathbf{C}$ with $p + u(p) \nu_{\mathbf{C}}(p) \in S_{\text{glue}}$, the computation along $S_{\text{glue}}$ shows that $\la \nu_{S_{\text{glue}}} (p + u(p) \nu_{\mathbf{C}}) , e_{n+1} \rg \geq \cos \theta > 0$ holds everywhere along $S_{\text{glue}}$.
The same computation holds on $S_{\beta-1, \psi}$ (as computed in Lemma~\ref{lemma:infinity-piece-small-epsilon}) and on $S_{\text{cap}}$ (by Proposition~\ref{prop:mean-convex-cap}), meaning that $\la \nu_S , e_{n+1} \rg > 0$ holds everywhere on $S_+$.
We deduce from this that $S_+$ is expressible as the graph of a function over $\Pi = \{ z = 0\}$.

\smallskip \noindent \textbf{Foliation by homotheties.}
We finally prove that the positive homotheties of $S_+$ foliate $E_+$.
The surface $S_+$ is smooth, properly embedded, connected, and asymptotic to ${\mathbf{C}}$, so Lemma~\ref{lemma:get-a-foliation} shows that the leaves $\{ \lambda S_+ \}_{\lambda>0}$ foliate $E_+$ if and only if $\la p, \nu_{S_+}(p) \rg > 0$ for all $p \in S_{\pm}$.
We verify this condition separately on the three pieces of $S_{\pm}$ from~\eqref{eqn:define-glued-surface}.

First, on $S_{\on{cap}}$, the level set expression $\nu_{S_{\on{cap}}} = \frac{\nabla \tilde{\varphi}}{|\nabla \tilde{\varphi}|}$ makes $\la p, \nu_{S_{\on{cap}}} \rg > 0$ equivalent to
\begin{align*}
    p \cdot \nabla \tilde{\varphi} &= (x,y,z) \cdot \bigl( \tilde{\varphi}_r \frac{x}{r}, \tilde{\varphi}_s \frac{y}{s}, \tilde{\varphi}_z\bigr) = r \tilde{\varphi}_r + s \tilde{\varphi}_s + z \tilde{\varphi}_z > 0.
\end{align*}
We compute this expression using $\tilde{\varphi}_{n,k}(r,s,z) := \varphi_{n,k}^+(r,s,z+ \delta R_2^{-(\beta+1)}$ from~\eqref{eqn:record-varphi-n,k}; in the general Case IV, we have $\beta=4$ and $S_{\text{cap}} = \{ \tilde{\varphi}(r,s,z) = \tilde{\varphi}(0,1,0) \}$.
We then obtain, along $S_{\text{cap}}$
\begin{align*}
    \tilde{\varphi}_r &= - 4 a_{n,k}^4 r^3, \qquad \tilde{\varphi}_s = 4 s ( s^2 + (z+\delta R_2^{-5})^2), \qquad \tilde{\varphi}_z = 4 (z+\delta R_2^{-5}) ( s^2 + (z+\delta R_2^{-5})^2), \\
    p \cdot \nabla \tilde{\varphi} &= - 4 a_{n,k}^4 r^4 + 4 ( s^2 + (z+\delta R_2^{-5})^2) \bigl(s^2 + z(z+\delta R_2^{-5}) \bigr) \\
    &= 4(1+\delta^2 R_2^{-6})^2 -  4 \delta R_2^{-5}(z+\delta R_2^{-5}) ( s^2 + (z+\delta R_2^{-5})^2) \qquad \qquad \text{on } \; S_{\on{cap}}.
\end{align*}
The second term above is $\geq - 4 \delta R_2^{-5} \cdot 2 R_2 \cdot 2 R_2^2 = - 16 \delta R_2^{-2}$ on $B_{R_2}$.
For any $R_2$, we may therefore ensure that $p \cdot \nabla \tilde{\varphi} > 0$ in $B_{R_2}$ by taking $\delta \in (0,\delta_{n,k})$ sufficiently small (followed by taking $\ve < \ve_0 = C(n,k) \delta R_2^{-6}$ sufficiently small) consistent with the construction of Proposition~\ref{prop:mean-convex-cap}.
Consequently, $S_{\text{cap}}$ will be star-shaped.
In Case II, we recall the definition
\[
\tilde{\varphi}(r,s,z) = (s^2 + (z+\delta R_2^{-5})^2 - \tfrac{1}{n-2} r^2) ( s^2 + (z+\delta R_2^{-5})^2), \qquad S_{\text{cap}} := \{ \tilde{\varphi}(r,s,z) = \tilde{\varphi}(0,1,0) \}.
\]
In this situation, we obtain
\begin{align*}
    p \cdot \nabla \tilde{\varphi} &= - \tfrac{2}{n-2} r^2 ( s^2 + (z+\delta R_2^{-5})^2 ) + 2 \Bigl( 2 ( s^2 + (z+\delta R_2^{-5})^2 ) - \tfrac{1}{n-2} r^2 \Bigr) (s^2 + z(z+\delta R_2^{-5}) ) \\
    &= 4 ( 1 + \delta^2 R_2^{-10})^2 - 2 \delta R_2^{-5}(z+\delta R_2^{-5}) \Bigl( s^2 + (z+\delta R_2^{-5})^2 + \frac{(1+\delta^2 R_2^{-10})^2}{s^2 + (z+\delta R_2^{-5})^2} \Bigr)
\end{align*}
on $S_{\text{cap}}$, by computing as above.
Since $\on{dist} ( S_{\text{cap}} \cap \Pi, 0) = 1$, the second term above is again bounded by $\geq - C(n) \delta R_2^{-2}$, so taking $\delta \in (0,\delta_{n,k})$ ensures that $p \cdot \nabla \tilde{\varphi} > 0$ in $B_{R_2}$.
In the exceptional cases I and III, the same computation is valid, consistent with Proposition~\ref{prop:mean-convex-cap}.
Therefore, $S_{\text{cap}}$ is star-shaped, for $\ve<\ve_1(n,k,R_1, R_2)$ sufficiently small.

For the pieces $S_{\text{glue}} \cup S_{\beta-1,\psi}$, we may write
\[
S_{\text{glue}} \cup S_{\beta-1,\psi} = \text{graph}_{\hat{\mathbf{C}} \setminus \bar{B}_{R_1} } \tilde{u} = \{ p + \tilde{u}(p) \nu_{\mathbf{C}}(p) : p \in \hat{\mathbf{C}} \} \cap \bR^{n+1}_+,
\]
where $\tilde{u}$ agrees with $u$ or $u_2$ in the respective region.
The points on $S_{\text{glue}} \cup S_{\beta-1,\psi}$ have the form $q = p + \tilde{u}(p) \nu_{\mathbf{C}}(p)$, and we can express $\nu_{S_+}(q) = \nu_{\mathbf{C}}(p) - \nabla_{\mathbf{C}} \tilde{u} + E_{\tilde{u}}$, with
\[
|E_{\tilde{u}}| = |\nu_{S_+}(q) - \nu_{\mathbf{C}}(p) + \nabla_{\mathbf{C}} \tilde{u}| \leq C(n) \bigl( |\nabla_{\mathbf{C}} \tilde{u}|^2 + |\tilde{u}| \, |\nabla^2_{\mathbf{C}} \tilde{u}| \bigr)
\]
as in~\eqref{eqn:normal-vector-equality-Pi}.
We therefore obtain
\begin{align*}
    \la q, \nu_{S_+} (q) \rg &= \la p + \tilde{u}(p) \nu_{\mathbf{C}} , \nu_{\mathbf{C}} - \nabla_{\mathbf{C}} \tilde{u} \rg + \la p + \tilde{u}(p) \nu_{\mathbf{C}} , E_{\tilde{u}} \rg \\
    &= \tilde{u}(p) - \la p, \nabla_{\mathbf{C}} \tilde{u} \rg + \la p + \tilde{u}(p) \nu_{\mathbf{C}}, E_{\tilde{u}} \rg
\end{align*}
because $\la \nu_{\mathbf{C}}, \nabla_{\mathbf{C}} \tilde{u} \rg = 0$.
Using $|\nabla^{\ell}_{\mathbf{C}} \tilde{u}| \leq c_{\ell} R^{1-\beta-\ell}$, as above, we obtain $|E_{\tilde{u}}| \leq C(n) R^{-2\beta}$ and 
\[
|\la p + \tilde{u}(p) \nu_{\mathbf{C}}, E_{\tilde{u}} \rg| \leq |p + \tilde{u}(p) \nu_{\mathbf{C}}| \cdot |E_{\tilde{u}}| \leq C(n) R^{1-2\beta}.
\]
Along $S_{\beta-1, \psi}$, we have $\tilde{u} = u_2$ and the computation~\eqref{eqn:nabla-c-u-computation} gives
\[
\la p, \nabla_{\mathbf{C}} ( \ell(R) \psi(\omega) ) \rg = R \ell'(R) \psi(\omega)
\]
because the second summand is orthogonal to $p = R \omega$.
Applying this to $u_2$, we obtain
\begin{align*}
u_2 - \la p , \nabla_{\mathbf{C}} u_2 \rg &= ( \ell(R) - R \ell'(R)) (1 - \tau \omega_{n+1}) \\ 
&= \bigl( c_{n,k} \beta R^{1-\beta} - c_1 (\beta+1) R^{-\beta} \bigr) (1 - \tau \omega_{n+1}),
\end{align*}
which is $> \tilde{c}_{n,k} R^{1-\beta}$ for $\tau < \frac{1}{4}$ and $R \geq R_1(n,k)$.
Combined with the above bound, this yields
\[
\la q, \nu_{S_{\beta-1,\psi}}(q) \rg > \tilde{c}_{n,k} R^{1-\beta} - C(n) R^{-\beta} > 0
\]
for $R \geq R_1(n,k)$.
Similarly, on $S_{\text{glue}}$, we use~\eqref{eqn:nabla-C-u} to write
\begin{align*}
    \tilde{u} - \la p, \nabla_{\mathbf{C}} \tilde{u} \rg &= \zeta \bigl( u_1 - \la p, \nabla_{\mathbf{C}} u_1 \rg \bigr) + (1-\zeta) \bigl( u_2 - \la p, \nabla_{\mathbf{C}} u_2 \rg \bigr) + R \zeta'(R) (u_1 - u_2) \\
    & > \zeta \bigl( u_1 - \la p, \nabla_{\mathbf{C}} u_1 \rg \bigr) + (1-\zeta) \bigl( u_2 - \la p, \nabla_{\mathbf{C}} u_2 \rg \bigr) - C(n) R^{-\beta}
\end{align*}
by applying~\eqref{eqn:u1-u2-lower-order} and $|R \zeta'(R)| \leq C(n)$.
Using the property~\eqref{eqn:phi-ck-bounds} of $u_1$, we obtain
\[
u_1 - \la p, \nabla_{\mathbf{C}} u_1 \rg > \tilde{c}_{n,k} R^{1-\beta}
\]
by an analogous computation to the one for $u_2$.
We conclude that
\[
\la q, \nu_{S_{\on{glue}}}(q) \rg > \tilde{c}_{n,k} R^{1-\beta} - C(n) R^{- \beta} > 0
\]
holds as well, upon taking $R_1(n,k)$ sufficiently large.
Therefore, $S_+$ is star-shaped, as claimed.

\smallskip \noindent \textbf{Conclusion of the construction.}
Summarizing the above discussion, we can produce a smooth surface $S_+$ with the desired properties provided that the gluing annulus $A(R_1, R_2)$ is sufficiently large and $\tau$ is sufficiently small.
Applying Lemma~\ref{lemma:infinity-piece-small-epsilon} ensures the validity of the construction for $\ve < \ve_1(n,k,\tau, R_2)$ sufficiently small, and Proposition~\ref{prop:mean-convex-cap} is applicable for $\ve < \ve_2(n,k,R_1, R_2)$.
We may therefore take $\ve_{n,k} := \min \{ \ve(n,k), \ve_1, \ve_2 \}$ to ensure that all the above perturbative arguments are valid.
We conclude that $S_+ \subset E_+ (\mathbf{C}_{n,k,\theta})$ is a star-shaped, strictly mean-convex surface asymptotic to $\mathbf{C}_{n,k,\theta}$, with $\la \nu_{S_+}, e_{n+1} \rg > 0$.
We can also find a ray $\ell_q = \{ \lambda q > 0 \}$ through some $q \in E_+$ intersecting $S_+$ exactly once; for example, any point with $(r,s,z) = (0,1,0) \in S_{\text{cap}}$ has this property.
Applying Lemma~\ref{lemma:get-a-foliation} shows that the positive homotheties of $S_+$ foliate $E_+$, as desired.
\end{proof}

The adaptation of the above construction to the side $E_-$ is straightforward, allowing us to smoothly glue the piece $S^-_{\beta-1,\psi}$ at infinity (from Lemma~\ref{lemma:infinity-piece-on-E-minus}) with the compact cap $S^-_{\text{cap}}$ (from Proposition~\ref{prop:mean-concave-cap-on-E-minus}) as negative normal graphs over $\mathbf{C}_{n,k,\frac{\pi}{2}-\ve}$.
The relevant properties of the glued surface $S_- \subset E_-$ are verified in the same manner as for $S_+$ changed appropriately by negating $\cL_{\mathbf{C}} u$ and reversing the bounds on $\la \nu_S, e_{n+1} \rg - \cos \theta$ at each step.
The star-shapedness of $E_-$ is unchanged.
As remarked in the prelude to Proposition~\ref{prop:mean-concave-cap-on-E-minus}, the cap region forms an obtuse angle with $\Pi$, meaning that the barrier surface $S_-$ is not graphical.

We record the construction and properties of $S_-$ in the following statement:
\begin{proposition}\label{prop:mean-convex-surface-on-E-minus}
    For every $n \geq 7$ and $1 \leq k \leq n-2$, except $(n,k) = (7,1)$, there exists an $\ve_{n,k} > 0$ such that the following holds.
    For all $\theta \in ( \frac{\pi}{2} - \ve_{n,k} , \frac{\pi}{2}]$, there exists a surface-with-boundary $S_- \subset E_- ( \mathbf{C}_{n,k,\theta})$ asymptotic to the cone $\mathbf{C}_{n,k,\theta}$ that is strictly mean-concave (with respect to the upward-pointing normal vector), forms a contact angle $> \theta$ with the plane $\Pi = \{ z = 0 \}$ along its free boundary, and its positive homotheties foliate $E_-(\mathbf{C}_{n,k,\theta})$.
\end{proposition}

Using the construction of Propositions~\ref{prop:mean-convex-surface} and~\ref{prop:mean-convex-surface-on-E-minus}, we may now conclude the first part of Theorem~\ref{thm:cones-are-minimizing-nearpi/2}.
Namely, we deduce that for all $n \geq 7$ and $1 \leq k \leq n-2$, except $(n,k) \in \{ (7,1) , (7,5) \}$, the cones $\mathbf{C}_{n,k,\theta}$ are strictly minimizing for $\cA^{\theta}$, for $\theta \in ( \frac{\pi}{2} - \ve_{n,k}, \frac{\pi}{2}]$.

\begin{proof}[Proof of Theorem~\ref{thm:cones-are-minimizing-nearpi/2}: Part I]
The strict stability of the cones $\mathbf{C}_{n,k,\frac{\pi}{2}-\ve}$, for $\ve \in (0,\ve_{n,k})$ and any $n \geq 7$, follows from the existence of a strict subsolution of the linearized Robin boundary value problem~\eqref{eqn:robin-eigenvalue-problem} on the cone, as obtained in Lemma~\ref{lemma:infinity-piece-small-epsilon} and the subsequent discussion.
To prove the strict minimality of the cone $\mathbf{C}_{n,k,\frac{\pi}{2}-\ve}$, we first consider the side $E_+ := \{ z > \rho f(t) \}$.
Proposition~\ref{prop:mean-convex-surface} ensures that for every $n \geq 7$ and $1 \leq k \leq n-2$, except $(n,k) = (7,5)$, there is an $\ve_{n,k} > 0$ such that for all $\theta \in ( \frac{\pi}{2} - \ve_{n,k}, \frac{\pi}{2}]$, there exists a surface-with-boundary $S_+ \subset E_+$ that is graphical, strictly mean-convex, has contact angle $< \theta$ with the plane $\Pi$ along its free boundary, and its positive homotheties foliate $E_+(\mathbf{C}_{n,k,\theta})$.
We may therefore apply Lemma~\ref{lemma:get-a-foliation} to obtain a strict capillary sub-calibration of $E_+ ( \mathbf{C}_{n,k,\theta}) $ from the rescalings $\{ \lambda S_+\}_{\lambda>0}$, which proves that the cone $\mathbf{C}_{n,k,\theta}$ is strictly minimizing for $\cA^{\theta}$ on its $E_+$ side.

Likewise, for every $n \geq 7$ and $1 \leq k \leq n-2$ except $(n,k) = (7,1)$, there exists an $\ve'_{n,k} > 0$ for which Proposition~\ref{prop:mean-convex-surface-on-E-minus} provides a graphical, mean-concave surface-with-boundary $S_- \subset E_-$ asymptotic to $\mathbf{C}_{n,k,\theta}$, with $\theta \in ( \frac{\pi}{2} - \ve'_{n,k}, \frac{\pi}{2}]$, having contact angle $> \theta$ with the plane $\Pi$ along its free boundary.
The positive homotheties of $S_-$ foliate $E_-(\mathbf{C}_{n,k,\theta})$, whereby applying Lemma~\ref{lemma:get-a-foliation} as above proves the strict minimality of $\mathbf{C}_{n,k,\theta}$ for $\cA^{\theta}$ on its $E_-$ side.
Combining the two properties implies the strict global minimality of the cones when $n \geq 7$ and $(n,k) \neq \{ (7,1), (7,5) \}$.
This completes the proof of Theorem~\ref{thm:cones-are-minimizing-nearpi/2} for $(n,k) \not\in \{ (7,1), (7,5) \}$.
\end{proof}

\begin{remark}
    As discussed following the proofs of Lemma~\ref{lemma:infinity-piece-small-epsilon} and Proposition~\ref{prop:mean-convex-cap}, the strategy used in the above construction is very general and can be used to prove the minimality of capillary cones close to a given strictly stable and strictly minimizing one, by gluing a compact piece of the foliation of the cone near the origin to the normal graph of a radial function $\sim r^{-a}$ determined by the indicial interval of the cone.
    This construction is applicable on either side of the cone; we explore it in greater generality in upcoming work~\cite{FTW_MinimizingII}.
\end{remark}

\subsection{Non-minimizing cones}\label{section:non-minimizing}

Finally, we show the second part of Theorem~\ref{thm:cones-are-minimizing-nearpi/2}, namely that for $(n,k) \in \{ (7,1), (7,5) \}$, there exist some $\ve^* > 0$ such that for all $\theta \in ( \frac{\pi}{2} - \ve^*, \frac{\pi}{2}]$, the cones $\mathbf{C}_{7,1,\theta}$ and $\mathbf{C}_{7,5,\theta}$ are strictly stable and one-sided area minimizing, but not minimizing for $\cA^{\theta}$.

\begin{proof}[Proof of Theorem~\ref{thm:cones-are-minimizing-nearpi/2}: Part II]
    We recall that $\mathbf{C}_{7,1,\frac{\pi}{2}}$ corresponds to a halved cone $C(\bS^5 \times \bS^1)$, which was proved by Lin~\cite{lin} to be one-sided strictly area-minimizing on its $E_+$ side.
    Likewise, the cone $\mathbf{C}_{7,5,\frac{\pi}{2}}$ is a halved cone $C(\bS^1 \times \bS^5)$, which is strictly minimizing on its $E_-$ side.
    Applying Proposition~\ref{prop:mean-convex-surface} proves the strict minimality of $\mathbf{C}_{7,1,\theta}$ in $E_+$, for $\theta \in ( \frac{\pi}{2} - \ve_{7,1}, \frac{\pi}{2}]$.
    Likewise, applying Proposition~\ref{prop:mean-convex-surface-on-E-minus} shows that the cones $\mathbf{C}_{7,5,\theta}$ are strictly minimizing on $E_-$, for $\theta \in ( \frac{\pi}{2} - \ve_{7,5}, \frac{\pi}{2}]$.
    Both results also imply the strict stability of the cones $\mathbf{C}_{7,1,\theta}$ and $\mathbf{C}_{7,5,\theta}$ due to the existence of a strict subsolution of the linearized Robin boundary value problem~\eqref{eqn:robin-eigenvalue-problem} on the cone, as obtained in Lemma~\ref{lemma:infinity-piece-small-epsilon} and the subsequent discussion.

    We will now prove that the above cones are not globally minimizing as $\theta \uparrow \frac{\pi}{2}$.
    In what follows, we will work with $\mathbf{C}_{7,1,\theta}$, for concreteness and notational simplicity; the argument for $\mathbf{C}_{7,5,\theta}$ is entirely analogous.
We recall that the cones $C(\bS^5 \times \bS^1)$ and $C(\bS^1 \times \bS^5)$ are not area-minimizing, by the work of Sim\~oes~\cite{simoes}.
    A computation of Lawlor~\cite{lawlor}*{p.~3} shows that the area-minimizing surface $S$ in the unit ball with boundary $\bS^5 \times \bS^1$ has area $(1 - \sigma_{1,5}) \cdot \textbf{M}( C(\bS^5 \times \bS^1) \cap B_1)$, where $\sigma_{1,5} \approx 2 \cdot 10^{-8}$.
    Let $\hat{S}$ denote the hypersurface-with-boundary obtained as the piece of $S$ contained in the half-space $\{ z \geq 0 \}$, and let $\hat{E} \subset B_1^+$ be either of the connected components of $B_1^+ \setminus \hat{S}$.
    Note that $S$ inherits the $O(6)$-equivariance of its boundary, therefore $\hat{S}$ is a free boundary minimal hypersurface in $B_1^+ = B_1 \cap \{ z \geq 0 \}$ which satisfies
    \begin{equation}\label{eqn:competitor-area-lawlor}
    \textbf{M}( \hat{S} \cap B_1^+) = (1 - \sigma_{1,5}) \cdot \textbf{M} ( \mathbf{C}_{7,1 , \frac{\pi}{2}} \cap B_1^+), \qquad \sigma_{1,5} \approx 2 \cdot 10^{-8}.
    \end{equation}
    For angles $\theta$ close to $\frac{\pi}{2}$, we can directly use normal graphs over $\hat{S}$ with small gradient to construct area-decreasing competitors to $\mathbf{C}_{7,1,\theta}$ for the capillary problem.
    Using Lemma~\ref{lemma:graph-over-extended-link-pi/2}, we may express the spherical links of the cones $\mathbf{C}_{7,1,\frac{\pi}{2}-\ve}$, with $\ve \in (0,\ve_*)$, as normal graphs over $\mathbf{C}_{7,1,\frac{\pi}{2}} \cap \bS^n_+$ with norm $\| \nabla^{\ell} v^{\Sigma}_{\ve} \|_{C^2} \leq C_{\ell} \ve$.
    Consequently,
    \[
    \textbf{M}(\mathbf{C}_{7,1,\theta} \cap B_1^+) \geq \textbf{M}(\mathbf{C}_{7,1,\frac{\pi}{2}} \cap B_1^+) - C \ve
    \]
    for $\theta \in (\frac{\pi}{2} - \ve_*, \frac{\pi}{2}]$.
    On the other hand, for any $\ve \in (0, \ve_*)$ sufficiently small, we can produce a function $\hat{v}_{\ve}$, with $\| \nabla^{\ell} \hat{v}_{\ve} \|_{C^2} \leq C_{\ell} \ve$ such that the normal graph $\hat{S}_{\theta}$ of its positive phase has spherical trace $\partial ( \mathbf{C}_{7,1,\theta} \cap B_1^+)$ but is not a cone.
    As before, let $\hat{E}_{\theta} \subset B_1^+$ denote the connected component of $B_1^+ \setminus \hat{S}_{\theta}$, chosen to vary continuously in $\theta$ such that $\hat{E}_{\theta} \xrightarrow{H} \hat{E}$ as $\theta \uparrow \frac{\pi}{2}$.
    Note that $\hat{S}_{\theta}$ is not required to be $\cA^{\theta}$-stationary, though such a construction is also possible.
    We therefore obtain
    \[
    \textbf{M} ( \hat{S}_{\theta} \cap B_1^+) \leq \textbf{M}( \hat{S} \cap B_1^+) + C \ve
    \]
    for a constant depending only on the ambient parameters.

    Let $E_{7,1,\theta}$ denote the connected component of $\bR^{n+1}_+ \setminus \mathbf{C}_{7,1,\theta}$ lying on the same side as the set $\hat{E}_{\theta}$ obtained by the above construction.
    Combining the previous steps, we obtain
    \begin{align*}
         \cA^{\theta}& (\hat{E}_{\theta} ;  B_1^+) - \cA^{\theta} (E_{7,1,\theta} ; B_1^+) \\
        &= \left[ \textbf{M}(\hat{S}_{\theta} \mres B_1^+) - \cos \theta |\partial^* \hat{E}_{\theta} \cap B_1 \cap \Pi| \right] - \left[ \textbf{M}(\mathbf{C}_{7,1,\theta} \cap B_1^+) - \cos \theta |\partial^* E_{7,1,\theta} \cap B_1 \cap \Pi| \right] \\
        &\leq \textbf{M}( \hat{S}_{\theta} \mres B_1^+) - \textbf{M}(\mathbf{C}_{7,1,\theta} \mres B_1^+) + C  |\partial^* E_{7,1,\theta} \cap B_1 \cap \Pi| \cdot \ve \\
        &\leq \left( \textbf{M}(\hat{S} \mres B_1^+) - \textbf{M}(\mathbf{C}_{7,1,\frac{\pi}{2}} \mres B_1^+) \right) + C_1 \ve + C_2 |\partial^* E_{7,1,\frac{\pi}{2}} \cap B_1 \cap \Pi| \cdot \ve + C_3 \ve^2
    \end{align*}
    for constants $C_i$ depending only on the ambient parameters.
    The computation~\eqref{eqn:competitor-area-lawlor} implies
    \[
    \textbf{M}(\hat{S} \mres B_1^+) - \textbf{M}(\mathbf{C}_{7,1,\frac{\pi}{2}} \mres B_1^+) = - \sigma_{1,5} \cdot \textbf{M}(\mathbf{C}_{7,1,\frac{\pi}{2}} \mres B_1^+).
    \]
    We therefore conclude that 
    \[
    \cA^{\theta} (\hat{E}_{\theta} ; B_1^+) < \cA^{\theta} (E_{7,1,\theta} ; B_1^+)
    \]
    for all $\theta \in (\frac{\pi}{2} - \ve, \frac{\pi}{2}]$, where $\ve \in (0, \ve_*)$ is sufficiently small.
    This produces a capillary competitor of strictly lower $\cA^{\theta}$-energy, showing that $\mathbf{C}_{7,1,\theta}$ is not $\cA^{\theta}$-minimizing.
    \end{proof}

Recall that Theorem~\ref{thm:cones-are-minimizing-smalltheta} shows that for some maximal $\ve_*$ and any $\theta \in (0,\ve_*)$, the cone $\mathbf{C}_{7,1,\theta}$ is strictly minimizing.
    We conjecture that $\ve_* = \frac{\pi}{2} - \ve^*$, so that all cones $\mathbf{C}_{7,1,\theta}$ are strictly area-minimizing for $\theta < \ve_*$ and not area-minimizing for $\theta > \ve_*$.

\begin{remark}
Observe that the above proof makes no use of the precise values of $(n,k)$ and $|\theta - \tfrac{\pi}{2}| = \ve$ beyond the fact that the cones $\mathbf{C}_{7,1,\frac{\pi}{2}}$ and $\mathbf{C}_{7,5,\frac{\pi}{2}}$ are not minimizing for $\cA^{\frac{\pi}{2}}$ and admits a strictly area-decreasing competitor with the same spherical trace, by Lawlor's construction.
The above argument is applicable for any values of $(n,k,\theta)$, showing that:
\begin{corollary}\label{cor:not-minimizing}
Suppose that for some $\theta_0 \in (0,\tfrac{\pi}{2}]$, the cone $\mathbf{C}_{n,k,\theta_0}$ is not minimizing for $\cA^{\theta_0}$, meaning that there exists some $F \neq E_{n,k,\theta_0}$ with 
\[
E_{n,k,\theta_0} \triangle F \Subset B_1, \qquad \cA^{\theta_0}(F;B_1) < \cA^{\theta_0}(E_{n,k,\theta_0} ; B_1).
\]
Then, there exists some $\ve_0 = \ve_0(n,k,\theta_0) > 0$ such that for all $\theta \in (\theta_0 - \ve_0, \theta_0 + \ve_0)$, the cone $\mathbf{C}_{n,k,\theta}$ is not minimizing for $\cA^{\theta}$.
\end{corollary}
Indeed, the non-minimizing assumption implies the existence of some $\sigma_0 > 0$ such that 
\[
\cA^{\theta_0}(F;B_1) = (1-\sigma_0) \cA^{\theta_0}(E_{n,k,\theta_0} ; B_1).
\]
For $|\theta - \theta_0| < \ve_0$ we may expand $\cos \theta = \cos \theta_0 + O(\ve)$ as above and approximate the capillary energy $\cA^{\theta}$ (in view of~\eqref{eqn:capillary-energy}) as
\begin{align*}
\bigl| \cA^{\theta}(E;B_1) - \cA^{\theta_0}(E;B_1) \bigr| &= |\cos \theta - \cos \theta_0| \, \cH^n ( \partial^* E \cap \Pi \cap B_1) \leq C(n,E) \ve 
\end{align*}
We may then construct an area-decreasing competitor $F_{\theta}$ with the same spherical trace as $\mathbf{C}_{n,k,\theta}$ and $P(F_{\theta};B_1) - P(F_{\theta_0};B_1) = O(|\theta-\theta_0|)$; applying the above computation, we then deduce that $\cA^{\theta}(F_{\theta}; B_1) < \cA^{\theta}(E_{n,k,\theta};B_1)$ for $|\theta -\theta_0| < \ve_0$.
In particular,~\eqref{cor:not-minimizing} implies that, for given $(n,k)$, the property of $\mathbf{C}_{n,k,\theta}$ being non-minimizing for $\cA^{\theta}$ is an open condition in $\theta \in (0, \frac{\pi}{2}]$.
\end{remark}

\begin{remark}\label{rem:minimizingNotStrict}
    For a regular cone $\mathbf{C} \subset \bR^{n+1}$ with an isolated singularity, being area-minimizing implies that it is uniquely minimizing, i.e., the unique solution for the Plateau problem with spherical trace $\mathbf{C} \cap \bS^n$; this follows from the existence of a Hardt-Simon foliation of $\bR^{n+1} \setminus \mathbf{C}$ by smooth minimal hypersurfaces~\cite{hardt-simon}.
    In upcoming work, we study the analogue of this construction for minimizing capillary cones.
    In particular, this will let us deduce that the set
    \[
    \vartheta_{n,k} := \{ \theta \in (0, \tfrac{\pi}{2}] : \mathbf{C}_{n,k,\theta} \; \text{ is not area-minimizing} \},
    \]
    which is open due to Corollary~\ref{cor:not-minimizing}, is not closed unless $\vartheta_{n,k} = ( 0, \frac{\pi}{2}]$.
    For $(n,k) = (7,1)$, we have $\vartheta_{n,k} \neq (0, \frac{\pi}{2}]$, due to Theorem~\ref{thm:cones-are-minimizing-smalltheta}; consequently, $\theta_* := \inf \vartheta_{7,1} \not\in \vartheta_{7,1}$.
    The cone $\mathbf{C}_{7,1,\theta_*}$ will then be minimizing, but not strictly minimizing, for $\cA^{\theta_*}$.
    This transition exhibits the first instance of such a non-strict minimizer; the corresponding question of area-minimizing but not strictly area-minimizing minimal cones in Euclidean space remains open.
\end{remark}

\newpage
\appendix

\section{Supersolution exponents}\label{subsection:exponents}

In Section~\ref{subsection:supersolution}, we constructed supersolutions proving the minimizing property of $\mathbf{C}_{n,k,\theta}$ for given pairs $(n,k)$ based on the decay rate $\beta$ and reduced the verification of their properties (obtained in Proposition~\ref{prop:supersolution}) to the conditions of Lemmas~\ref{lemma:W-computation} and~\ref{lemma:K-computation} for $\beta$.
A representative sample of viable exponents $\beta_{n,k}$ for pairs $(n,k)$ covered by Theorem~\ref{thm:cones-are-minimizing-smalltheta} is presented in the following Table~\ref{table:bigTableOfBetas}.
We record the quantities $\hat{\cQ}(t) := \cQ(t) + \frac{1}{(A^2+1)(1 - \frac{\bar{r}}{A})}$ and $\{ K_0(\xi), K_1 (\xi) \} := \{ K(0,\xi), K(1,\xi) \}$.

\begin{table}[H]
    \centering
     \begin{tabular}{|c|c|c||c|c|c|c|c|c|}      % B + C table v3              -- only differs from R by K computation precision                          
      \hline                                         
      $n$ & $k$ & $\beta$ & $\bar{r} -A$ & $\max_{[0,\tau]} \hat{\cQ}(t) $ & $\max_{ [0,1]} K_0(\xi) $ & $\max_{[0,1]} K_1(\xi) $ & $\min_{[0,1]} \cP(\xi)$ \\                                
      \hline \hline                       
      $ 7 $ & $ 1 $ & $ -2 $ & $ -3.33 $ & $ -0.011 $ & $ -1.55 $ & $ -1.39 $ & $2.34$ \\            
      $ 7 $ & $ 2 $ & $ -2.5 $ & $ -0.58 $ & $ -0.034 $ & $ -0.81 $ & $ -0.62 $ & $4.29$ \\            
      \hline                                        
      $ 8 $ & $ 1 $ & $ -4 $ & $ -1.67 $ & $ -0.237 $ & $ -1.68 $ & $ -1.27 $ & $5.37$ \\                 
      $ 8 $ & $ 2 $ & $ -4 $ & $ -0.66 $ & $ -0.645 $ & $ -1.71 $ & $ -1.01 $ & $5.65$ \\                 
      $ 8 $ & $ 3 $ & $ -3 $ & $ -1.29 $ & $ -1.024 $ & $ -2.52 $ & $ -2.51 $ & $5.95$ \\                 
      \hline                                       
      $ 9 $ & $ 1 $ & $ -5 $ & $ -2.66 $ & $ -0.307 $ & $ -2.00 $ & $ -1.69 $ & $7.22$ \\                 
      $ 9 $ & $ 2 $ & $ -5 $ & $ -0.98 $ & $ -0.967 $ & $ -2.79 $ & $ -1.94 $ & $7.11$ \\                 
      $ 9 $ & $ 3 $ & $ -5 $ & $ -0.60 $ & $ -1.478 $ & $ -1.97 $ & $ -1.27 $ & $8.10$ \\                 
      $ 9 $ & $ 4 $ & $ -4 $ & $ -1.26 $ & $ -1.720 $ & $ -3.06 $ & $ -2.98 $ & $11.03$ \\                 
      $ 9 $ & $ 5 $ & $ -3 $ & $ -1.45 $ & $ -1.452 $ & $ -2.81 $ & $ -3.09 $ & $13.52$ \\                 
      \hline                                       
      $ 12 $ & $ 1 $ & $ -8 $ & $ -5.47 $ & $ -0.345 $ & $ -2.85 $ & $ -2.64 $ & $13.90$ \\                
      $ 12 $ & $ 2 $ & $ -8 $ & $ -1.73 $ & $ -1.266 $ & $ -3.56 $ & $ -3.00 $ & $12.29$ \\                
      $ 12 $ & $ 3 $ & $ -8 $ & $ -1.03 $ & $ -2.206 $ & $ -4.73 $ & $ -3.39 $ & $12.69$ \\                
      $ 12 $ & $ 4 $ & $ -8 $ & $ -0.73 $ & $ -2.617 $ & $ -4.10 $ & $ -2.55 $ & $16.73$ \\                
      $ 12 $ & $ 5 $ & $ -8 $ & $ -0.56 $ & $ -2.920 $ & $ -3.13 $ & $ -1.46 $ & $25.21$ \\                
      $ 12 $ & $ 6 $ & $ -7 $ & $ -1.20 $ & $ -2.903 $ & $ -4.95 $ & $ -4.09 $ & $26.59$ \\                
      $ 12 $ & $ 7 $ & $ -7 $ & $ -0.95 $ & $ -2.917 $ & $ -2.77 $ & $ -1.45 $ & $27.07$ \\                
      $ 12 $ & $ 8 $ & $ -6 $ & $ -1.33 $ & $ -2.380 $ & $ -3.60 $ & $ -2.71 $ & $30.4$ \\                
      $ 12 $ & $ 9 $ & $ -4 $ & $ -2.01 $ & $ -1.440 $ & $ -5.42 $ & $ -5.27 $ & $28.09$ \\                
      \hline                                       
      $ 20 $ & $ 16 $ & $-12 $ & $-1.55 $ & $ -3.777 $ & $ -10.37 $ & $ -3.81 $ & $74.18$ \\             
      $ 20 $ & $ 17 $ & $ -11 $ & $ -1.71 $ & $ -2.720 $ & $ -10.65 $ & $ -4.59 $ & $72.39$  \\             
      $ 20 $ & $ 18 $ & $ -7 $ & $ -3.01 $ & $ -0.997 $ & $ -18.14 $ & $ -16.41 $ & $56.85$ \\             
      \hline                                       
      $ 30 $ & $ 1 $ & $ -26 $ & $ -19.38 $ & $ -0.349 $ & $ -5.77 $ & $ -5.66 $ & $77.62$ \\              
      $ 30 $ & $ 15 $ & $ -24 $ & $ -1.32 $ & $ -8.985 $ & $ -19.73 $ & $ -10.03 $ & $119.39$ \\            
      $ 30 $ & $ 28 $ & $ -16 $ & $ -2.16 $ & $ -1.909 $ & $ -40.33 $ & $ -24.82 $ & $109.76$\\            
      \hline                                       
      $ 40 $ & $ 1 $ & $ -36 $ & $ -25.43 $ & $ -0.348 $ & $ -6.87 $ & $ -6.77 $ & $125.76$\\              
      $ 40 $ & $ 38 $ & $ -24 $ & $ -1.91 $ & $ -2.621 $ & $ -98.47 $ & $ -31.33 $ & $159.46$\\            
      \hline                                       
      $ 50 $ & $ 1 $ & $ -46 $ & $ -30.71 $ & $ -0.347 $ & $ -7.81 $ & $ -7.71 $ & $181.08$\\              
      $ 50 $ & $ 25 $ & $ -42 $ & $ -1.74 $ & $ -14.545 $ & $ -48.89 $ & $ -28.40 $ & $245.12$\\           
      $ 50 $ & $ 48 $ & $ -31 $ & $ -1.96 $ & $ -3.105 $ & $ -282.76 $ & $ -109.81 $ & $207.44$\\          
      \hline                                       
      $ 60 $ & $ 1 $ & $ -56 $ & $ -35.43 $ & $ -0.347 $ & $ -8.65 $ & $ -8.55 $ & $242.74$\\              
      $ 60 $ & $ 58 $ & $ -38 $ & $ -1.99 $ & $ -3.539 $ & $ -757.64 $ & $ -326.85 $ & $255.70$\\          
      \hline                                       
      $ 70 $ & $ 1 $ & $ -66 $ & $ -39.71 $ & $ -0.347 $ & $ -9.42 $ & $ -9.32 $ & $310.14$\\              
      $ 70 $ & $ 35 $ & $ -60 $ & $ -2.02 $ & $ -19.879 $ & $ -104.13 $ & $ -57.47 $ & $386.63$\\          
      $ 70 $ & $ 68 $ & $ -46 $ & $ -1.86 $ & $ -4.069 $ & $ -1847.27 $ & $ -484.63 $ & $306.18$\\       
      
      \hline                                       
      $ 100 $ & $ 1 $ & $ -96 $ & $ -50.76 $ & $ -0.347 $ & $ -11.42 $ & $ -11.30 $ & $542.80$\\           
      $ 100 $ & $ 50 $ & $ -88 $ & $ -1.98 $ & $ -28.832 $ & $ -271.42 $ & $ -59.27 $ & $625.82$\\         
      $ 100 $ & $ 98 $ & $ -46 $ & $ -26.56 $ & $ -3.544 $ & $ -9680.10 $ & $ -9670.75 $ & $381.46$\\      
      \hline                     
  \end{tabular}  
    \caption{Representative $\beta_{n,k}$ sample for pairs $(n,k)$.}
    \label{table:bigTableOfBetas}
\end{table}

\newpage 

\nocite{*}
\bibliography{ref}

\end{document}